\documentclass[12pt,a4paper]{article}
\usepackage[utf8]{inputenc}
\usepackage[english]{babel}
\usepackage{xcolor}

\usepackage{amssymb,mathtools,amsmath,amsthm,stmaryrd}
\usepackage[all,cmtip,2cell]{xy} 

\usepackage[draft=false, hidelinks, linkcolor = blue]{hyperref}
\usepackage{cleveref}
\usepackage{tensor}
\usepackage{url}
\usepackage{bookmark}

\usepackage{graphicx}
\usepackage{lmodern}
\usepackage{enumitem}
\usepackage{array}

\usepackage[left=2.5cm,right=2.5cm,top=2.5cm,bottom=2.5cm]{geometry}
\def\defthm#1#2#3#4{
  \newtheorem{#1}[theorem]{#3}
  \newtheorem*{#1*}{#3}
  \newtheorem{#2}[theorem]{#4}
  \newtheorem*{#2*}{#4}
  \crefname{#1}{#3}{#4}
  \crefname{#2}{#4}{#4}  
}

\newtheoremstyle{mythm}%
{10pt}
{}
{\itshape}
{}
{\bf}
{.}
{.5em}
{}%

\newtheoremstyle{mydef}%
{10pt}
{3pt}
{}
{}
{\bf}
{.}
{.5em}
{}%

\newtheoremstyle{myrmk}%
{10pt}
{3pt}
{}
{}
{\bf}
{.}
{.5em}
{}%

\theoremstyle{mythm}
\newtheorem{theorem}{Theorem}[section]
\newtheorem*{theorem*}{Theorem}

\defthm{corollary}{corollaries}{Corollary}{Corollaries}
\defthm{lemma}{lemmata}{Lemma}{Lemmata}
\defthm{proposition}{propositions}{Proposition}{Propositions}
\defthm{axiom}{axioms}{Axiom}{Axioms}
\defthm{propdef}{props/defs}{Proposition/Definition}{Prositions/Definitions}
\defthm{defcor}{defscors}{Corollary/Definition}{Corollaries/Definitions}
\defthm{conjecture}{conjectures}{Conjecture}{Conjectures}

\theoremstyle{mydef}
\defthm{definition}{definitions}{Definition}{Definitions}

\theoremstyle{myrmk}
\defthm{acknowledgment}{acknowledgments}{Acknowledgment}{Acknowledgments}
\defthm{remark}{remarks}{Remark}{Remarks}
\defthm{motivation}{motivations}{Motivation}{Motivations}
\defthm{example}{examples}{Example}{Examples}
\defthm{question}{questions}{Question}{Questions}
\defthm{observation}{observations}{Observation}{Observations}
\defthm{claim}{claims}{Claim}{Claims}
\defthm{notation}{notations}{Notation}{Notations}

\newtheorem*{replemmax}{\reptitle}
 {\end{replemmax}}

\newtheorem*{repthmx}{\reptitle}
\newenvironment{repthm}[1]{%
 \def\reptitle{Theorem \ref*{#1}}%
 \begin{repthmx}}%
 {\end{repthmx}}

\newtheorem*{repcorx}{\reptitle}
 {\end{repcorx}}

\crefname{section}{Section}{Sections}
\crefname{theorem}{Theorem}{Theorems}

\makeatletter
\renewenvironment{proof}[1][\proofname] {\par\pushQED{\qed}\normalfont\topsep6\p@\@plus6\p@\relax\trivlist\item[\hskip\labelsep\bf#1\@addpunct{.}]\ignorespaces}{\popQED\endtrivlist\@endpefalse}
\makeatother

\newcommand{\blank}{\mbox{\hspace{3pt}\underline{\ \ }\hspace{2pt}}}
\newcommand{\sprime}{^{\prime}}

\newcommand{\pbs}{\scalebox{1.5}{\rlap{$\cdot$}$\lrcorner$}}

\newcommand{\adj}{\rotatebox[origin=c]{180}{$\vdash$}\hspace{0.1pc}}

\newcommand{\Address}{{
  \bigskip
  \footnotesize

\textsc{Max Planck Institute for Mathematics, Vivatsgasse 7, 53111 Bonn, Germany}\par\nopagebreak
  \textit{E-mail address}: \texttt{stenzel@mpim-bonn.mpg.de}

}}

\title{Notions of $(\infty,1)$-sites and related formal structures}
\author{Raffael Stenzel}
\renewcommand\footnotemark{}

\begin{document}
\maketitle

\abstract{
We study various characterizations of higher sites over a given $\infty$-category $\mathcal{C}$ which are conceptually in line with their 
classical ordinary categorical counterparts, and extract some new results about $\infty$-topos theory from them. First, in terms of formal
$(\infty,2)$-category theory, we define a notion of higher Lawvere-Tierney operators on $\infty$-toposes which 
internalizes a parametrized version of the left exact modalities in \cite{abjfsheavesI} and the left exact modalities in \cite{rss_hottmod}. 
Second, in the spirit of Lawvere's hyperdoctrines, we describe the $\infty$-toposes embedded in the $\infty$-category $\hat{\mathcal{C}}$ of 
presheaves over $\mathcal{C}$ as the sheaves of ideals of what we call the logical structure sheaf on $\mathcal{C}$. This naturally induces a 
notion of ``geometric kernels'' on $\mathcal{C}$ which play the part of higher Grothendieck topologies from the given perspective.
Lastly, we study the $\infty$-category of cartesian $(\infty,1)$-sites. We generalize the notion of canonical Grothendieck topologies 
from \cite[Section 6.2.4]{luriehtt} appropriately to all geometric kernels and show an according ``Comparison Lemma'' in the best case 
scenario.
However, we show that a corresponding topological version of the lemma in the context of Grothendieck 
topologies fails.
}
\tableofcontents

\section{Introduction}

The notion of a site in (ordinary) category theory comes in various guises, some motivated by topological practice and others by notions 
naturally occurring in the context of categorical logic. The aim of this paper is, first, to recover these guises in an appropriately 
adjusted manner for the notion of $(\infty,1)$-site as developed in \cite{abjfsheavesI}, \cite{luriehtt}, and \cite{rss_hottmod}, and 
second, to extract some additional new theory from the structures studied along the way.

Any Grothendieck site $(\mathcal{C},J)$, thought of as a categorified 
topological space, comes canonically equipped with a notion of associated sheaf theory, that is, the associated ``petit'' topos of
$J$-sheaves on $\mathcal{C}$. It turns out that every topos (over a fixed base $\mathcal{C}$) is a petit topos with respect to some 
Grothendieck site (with base $\mathcal{C}$), and so the topological practice of sheaf theory is a generic aspect of the categorical logic 
of toposes.
A faithful development of the corresponding $\infty$-categorical landscape has been found out to be somewhat of a false friend in recent 
years, given that the translation of intuition as well as of results from ordinary topos theory subtly fails in eventually crucial places. 
One of those crucial differences is that the notion of a Grothendieck topology on an $\infty$-category $\mathcal{C}$ as defined in
\cite[Section 6.2.2]{luriehtt} is insufficient to describe all $\infty$-toposes generated by $\mathcal{C}$. A theory of notions to classify 
the poset of $\infty$-subtoposes of a given $\infty$-topos $\mathcal{B}$ in terms of topologically motivated or logically formal notions
associated to $\mathcal{B}$ is still work in progress. Towards this end, the authors of \cite{abjfsheavesI} introduced classes of 
factorizations systems (the left exact modalities of small generation) in $\mathcal{B}$ which stand in 1-1 correspondence to the
$\infty$-subtoposes of (and hence the ``higher sheaf theories'' obtained as quotients from) $\mathcal{B}$. These translate to left exact 
modalities in Homotopy Type Theory imposed on top of the internal language of
$\mathcal{B}$ via their semantics as ``reflective subfibrations''  of $\mathcal{B}$ as lined out in the appendix
\cite[Section A.2]{rss_hottmod}.

In this paper, we take up on this literature and introduce various equivalent formal $(\infty,2)$-categorical (specifically
$\infty$-cosmological) structures fibered over, indexed over, and internal to
$\mathcal{B}$, respectively, which mirror the ordinary topos theoretic landscape appropriately. In the special case when $\mathcal{B}$ is the
$\infty$-category of presheaves on a small $\infty$-category $\mathcal{C}$, we furthermore introduce an equivalent 
notion indexed over $\mathcal{C}$ itself which we call sheaves of $\mathcal{O}_{\mathcal{C}}$-ideals, and show that these ideals are fully 
determined by their ``geometric kernels'' which play the role of a corresponding notion of higher Grothendieck topology. The main results of 
this paper may hence be summarized by the following theorem.

\begin{theorem}\label{thmmain}
Let $\mathcal{B}$ be an $\infty$-topos. Then the following (partially ordered) classes are isomorphic to one another.
\begin{enumerate}
\item Equivalence classes of accessible left exact reflective localizations of $\mathcal{B}$.
\item Higher closure operators on $\mathcal{B}$.
\item Higher Lawvere-Tierney operators on $\mathcal{B}$.
\end{enumerate}
Whenever $\mathcal{B}$ is the $\infty$-topos of presheaves over a small $\infty$-category $\mathcal{C}$, all of these (partially ordered)
classes furthermore are isomorphic to any of the following classes as well.
\begin{enumerate}
\item[4.] Accessible sheaves of $\mathcal{O}_{\mathcal{C}}$-ideals.
\item[5.] Accessible geometric kernels on $\mathcal{C}$.
\end{enumerate}
\end{theorem}

We will study pairwise equivalent non-accessible versions of 1., 4.\ and 5.\ as well.
One of the benefits of these notions is that the classes $2.$, $3.$ and $4.$ in Theorem~\ref{thmmain} are formal constructions in associated 
$\infty$-cosmoses which translate 
to one another along (almost) cosmological functors basically by construction. They hence can be generalized to other contexts for this 
reason (see e.g.\ Remark~\ref{remelemlvtops} for a note on accordingly defined higher Lawvere-Tierney operators on elementary
$\infty$-toposes). More 
specifically, they are formal structures defined on the target fibration $t\colon\mathcal{B}^{\Delta^1}\twoheadrightarrow\mathcal{B}$ as an 
object in the $\infty$-cosmos $\mathbf{Cart}(\mathcal{B})$ of cartesian fibrations over $\mathcal{B}$ -- or the images/preimages under 
certain cosmological functors thereof, respectively -- which can be replaced for instance by the cartesian fibration
$\mathrm{Sub}_{\mathcal{B}}\in\mathbf{Cart}(\mathcal{B})$ of subobjects in $\mathcal{B}$. That way one recovers the classical equivalences 
between topological localizations, closure operators, Lawvere-Tierney topologies and, in the case of $\mathcal{B}=\hat{\mathcal{C}}$, 
Grothendieck topologies on $\mathcal{C}$ (although there are some technical differences in the proofs, see e.g.\ Remark~\ref{remtoplideals}). 
Therefore, we will start in Section~\ref{secfirstorder} with a concise motivation from ordinary topos theory, and state in 
Proposition~\ref{proplocalglobalfirstorder} how the more familiar notions which represent subtoposes here translate to one another in the 
specific sense which we will generalize to $\infty$-topos theory in the rest of the paper.

Section~\ref{secfibglobal} builds the foundation for those generalizations. In Section~\ref{subsecprelim} we cover the relevant material from 
the literature regarding the framework of left exact modalities as worked out in \cite{abjfsheavesI}, and in Section~\ref{secsubfsrl} we 
develop the theory of reflective subfibrations as considered in \cite[Section A.2]{rss_hottmod} in some more depth and generality. 
Section~\ref{secfibglobal2} is technically entirely independent from Sections~\ref{sectopdef} and \ref{secsitemorphisms}, and hence need not to be read if the reader wishes to proceed directly to Section~\ref{sectopdef}.

In Section~\ref{secfibglobal2}, we prove the equivalence of the classes $1.$, $2.$, and $3.$ in Theorem~\ref{thmmain} and define all the 
notions involved. 
Higher closure operators will be defined as idempotent, left exact, accessible and compositional monads in the
$\infty$-cosmos of cartesian fibrations over $\mathcal{B}$ in Section~\ref{sechigherclosureops}, and will be shown to stand in 1-1 
correspondence to left exact modalities of small generation in Theorem~\ref{thmlocmodopacc}. Higher Lawvere-Tierney operators will be defined 
as ``polymorphic families'' of such idempotent monads in the $\infty$-cosmos of internal $\infty$-categories in $\mathcal{B}$ in
 Section~\ref{secsubhlttops}, and will be proven to be 
equivalent to their external counterparts in Theorem~\ref{thmcharlvtops}. Here, a parametrization of the monads along unbounded sequences of 
regular cardinals is necessary, because the target fibration $t$ of $\mathcal{B}$ is only representable one cardinal segment at a time. 
For such parametrizations to exist, the assumption of accessibility of the monads is crucial. Indeed, an according equivalence of the given 
internal and external structures in a non-presentable setting -- as the one of elementary $\infty$-toposes -- might simply not hold.

In Section~\ref{sectopdef} we consider the special case of presheaf $\infty$-toposes over a fixed $\infty$-category $\mathcal{C}$ and 
prove the equivalence of the classes $1.$, $4.$, and $5.$ in Theorem~\ref{thmmain}. In Section~\ref{secsublogstrsheaf}, we adapt the 
perspective that a topology ought to be an indexed 
logical equipment in the sense of Lawvere, and argue that every $(\infty,1)$-category $\mathcal{C}$ comes equipped with a ``logical structure 
sheaf'' $\mathcal{O}_{\mathcal{C}}$ such that a topology on $\mathcal{C}$ is exactly what qualifies -- in terms borrowed from algebraic 
geometry --  as a sheaf of $\mathcal{O}_{\mathcal{C}}$-ideals. Here, the structure sheaf is valued not in frames, but in $\infty$-toposes (or 
$\infty$-logoi) themselves.
Whenever the $\infty$-category $\mathcal{C}$ itself exhibits a sufficiently impredicative internal language, this structure sheaf is 
representable, and so topologies correspond to internal ideals of the corresponding internalized structure. This is precisely the equivalence 
between 3.\ and 4.\ in Proposition~\ref{proplocalglobalfirstorder} in the ordinary case, and underlies the equivalence of (higher) closure 
operators and (higher) Lawvere-Tierney operators whenever $\mathcal{C}$ is an ($\infty$-)topos. 

In Section~\ref{subsectopdef}, we will use this notion of sheaves of $\mathcal{O}_{\mathcal{C}}$-ideals $\mathcal{E}$ to define the notion of 
a geometric kernel on $\mathcal{C}$, which is a $\mathcal{C}$-indexed collection of objects in $\mathcal{O}_{\mathcal{C}}$ which are to be 
nullified pointwise in a sheaf of $\mathcal{O}_{\mathcal{C}}$-ideals $\mathcal{E}$. We show that all sheaves of
$\mathcal{O}_{\mathcal{C}}$-ideals are uniquely determined by their associated geometric kernels, and classify all geometric kernels in some 
simple examples in Section~\ref{subsecexples} (which all turn out to be generated by Grothendieck topologies, raising the question how 
complicated an $\infty$-category has to be for it to exhibit non-topological geometric kernels).
%
%
%

In Section~\ref{secsitemorphisms} we define an $\infty$-category of cartesian $(\infty,1)$-sites, show how it relates to the
$\infty$-category of $\infty$-toposes, and prove an appropriate version of the Comparison Lemma.

\begin{repthm}{thmcomplemma}[Comparison Lemma]
Let $(\mathcal{D},K)$ be a cartesian $(\infty,1)$-site, $\mathcal{C}$ be a small left exact $\infty$-category and
$F\colon\mathcal{C}\rightarrow\mathcal{D}$ be a cartesian functor.
\begin{enumerate}
\item Whenever the localization $\mathrm{Sh}_K(\mathcal{D})$ is locally presentable, the ordered class of accessible geometric kernels $J$ on
$\mathcal{C}$ such that the functor $F\colon(\mathcal{C},J)\rightarrow(\mathcal{D},K)$ is cover-preserving has a maximal element $J_K$.
\item Whenever $F$ is fully faithful and $K$-dense, there is a geometric kernel $J_K$ on $\mathcal{C}$ such that
$F\colon(\mathcal{C},J_K)\rightarrow(\mathcal{D},K)$ is an equivalence of cartesian $(\infty,1)$-sites.
\item Whenever $F$ is fully faithful and $K$-dense, and $\mathrm{Sh}_K(\mathcal{D})$ is locally presentable as well, the $(\infty,1)$-sites
$(\mathcal{C},J_K)$ from Parts 1 and 2 coincide.
\end{enumerate}
\end{repthm}

As a corollary (Corollary~\ref{corgeoesssmall}) we obtain that 
the canonical cartesian site on any $\infty$-topos exists and is essentially small. This gives an alternative proof of the fact that 
every $\infty$-topos has a left exact and subcanonical presentation (\cite[Proposition 6.1.5.3]{luriehtt}) as well as that the large 
sheafification of a small presheaf on an $\infty$-topos is small again (\cite[Lemma 6.3.5.28]{luriehtt}).
We however show that the Comparison Lemma stated for Grothendieck sites with the usual notion of density -- which is the one considered in 
ordinary topos theory in \cite[Theorem 2.2.3]{elephant} -- fails, 
because the according notion of density is too weak in the given context. This implies that the ``canonical Grothendieck topologies'' from 
\cite[Proposition 6.2.4.6]{luriehtt} generally cannot yield equivalences of sheaf theories.

\begin{notation*}
As is often custom, in all of the following the prefix ``$(\infty,1)$-'' will be abbreviated to ``$\infty$-''. The term
``$\infty$-category'' will mean ``quasi-category'' unless explicitly stated otherwise. The
$\infty$-category of spaces will be denoted by $\mathcal{S}$, the $\infty$-category of functors between two
$\infty$-categories $\mathcal{B}$, $\mathcal{C}$  will be denoted by $\mathrm{Fun}(\mathcal{B},\mathcal{C})$. The $\infty$-category of small 
$\infty$-categories will be denoted by $\mathrm{Cat}_{\infty}$. The (simplicially enriched) category of simplicial sets will be denote by
$\mathbf{S}$.
\end{notation*}

\begin{acknowledgments*}
Part of the paper was written with support of the Grant Agency of the Czech Republic under the grant 19-00902S. Part of the paper was written 
as a guest at the Max Planck Institute for Mathematics in Bonn, Germany.
The author thanks John Bourke and Nima Rasekh for multiple discussions regarding the results of this paper.
\end{acknowledgments*}

\section{The first order logical structure sheaf}\label{secfirstorder}
 
In this section we restrict our attention to Grothendieck sites on ordinary categories.
Let $\mathrm{Frm}$ be the category of frames and frame homomorphisms. Accordingly, its opposite is the category
$\mathrm{Loc}$ of locales. An extensive study of both is given in \cite{johnstone_ss}.
Let $\Omega$ be the subobject classifier in $\mathrm{Set}$ considered as a frame and thus thought of as the
predicate classifier in classical set theory. Let $\mathcal{C}$ be an ordinary category. Given an object
$C\in\mathcal{C}$, we can define the frame $\mathrm{Sv}(\mathcal{C}_{/C})$ of isomorphism classes of sieves on the slice
$\mathcal{C}_{/C}$; that is, up to natural isomorphism, the frame of isomorphism classes of monomorphisms of presheaves 
over the representable $yC\in\mathrm{Fun}(\mathcal{C}^{op},\mathrm{Set})$. The composition
\[\Omega_{\mathcal{C}}\colon\mathcal{C}^{op}\xrightarrow{\mathcal{C}_{/(\cdot)}}\mathrm{Cat}^{op}\xrightarrow{\mathrm{Fun}((\cdot)^{op},\Omega)}\mathrm{Frm}\]
takes an object $C$ to $\mathrm{Fun}((\mathcal{C}_{/C})^{op},\Omega)\simeq\mathrm{Sv}(\mathcal{C}_{/C})$ and a map $f\colon C\rightarrow D$ 
to the pullback functor $(\Sigma_f)^{\ast}\colon\mathrm{Sv}(\mathcal{C}_{/D})\rightarrow\mathrm{Sv}(\mathcal{C}_{/C})$ which will be simply 
denoted by $f^{\ast}$.
We will denote this functor by $\Omega_{\mathcal{C}}$ and call it the \emph{first order logical structure sheaf} on
$\mathcal{C}$. In dual terms, the structure sheaf is an assignment of the form
\[\Omega_{\mathcal{C}}\colon\mathcal{C}\rightarrow\mathrm{Loc}\]
which takes an object $C$ to the (localic) space of predicates over $\mathcal{C}_{/C}$.

\begin{remark}
Given an ordinary category $\mathcal{C}$ with finite products, recall that a first order hyperdoctrine on $\mathcal{C}$ 
is a functor
\[P\colon\mathcal{C}^{op}\rightarrow\mathrm{Frm}\]
which comes equipped with the following logical structure.
\begin{enumerate}
\item For each $f\colon C\rightarrow D$ in $\mathcal{C}$, the frame homomorphism $P(f)\colon P(D)\rightarrow P(C)$ is part of an adjoint triple $(\exists_f,P(f),\forall_f)$.
\item $P$ satisfies the Beck-Chevalley condition for squares associated to the projections
$\pi_1\colon(\cdot)\times C\Rightarrow\mathrm{id}_{\mathcal{C}}$ for $C\in\mathcal{C}$.
\end{enumerate}
Given an object $C\in\mathcal{C}$, the frame $P(C)$ can be considered as a structure of synthetic predicates over $C$. The 
left adjoints serve the function of a $P$-existential quantifier, the right adjoints that of a $P$-universal quantifier. 
$P$ is said to be a first order hyperdoctrine with equality if it satisfies the Beck-Chevalley condition for pullback 
squares in $\mathcal{C}$ of the form
\[\xymatrix{
A\times B\ar[r]^{A\times f}\ar[d]_{\Delta_A\times B} & A\times C\ar[d]^{\Delta_A\times C} \\
A\times A\times B\ar[r]_{A\times A\times f} & A\times A \times C.
}\]
The first order structure sheaf $\Omega_{\mathcal{C}}\colon\mathcal{C}^{op}\rightarrow\mathrm{Frm}$ is a standard 
example for a hyperdoctrine (with equality) for every category $\mathcal{C}$ with finite products
(via \cite{lawverecomp}).
\end{remark}

\begin{definition}\label{defidealsheaf1istorder}
Let $\mathcal{C}$ be a small category. A sheaf $\mathcal{E}$ of first order $\Omega_{\mathcal{C}}$-ideals is a regular
subfunctor
\[\mathcal{E}\subseteq\Omega_{\mathcal{C}}\colon\mathcal{C}\rightarrow\mathrm{Loc}\]
with $\mathcal{C}$-indexed reflector. 
\end{definition}

Here, to be a regular subfunctor $\mathcal{E}\subseteq\Omega_{\mathcal{C}}$ means that for every object $C\in\mathcal{C}$ the 
locale $\mathcal{E}(C)$ is a sublocale (i.e.\ a regular subobject) of $\Omega_{\mathcal{C}}(C)$, and the transition 
maps $\mathcal{E}(f)_{\ast}\colon\mathcal{E}(C)\rightarrow\mathcal{E}(D)$ are the restrictions
$\Omega_{\mathcal{C}}(f)_{\ast}\mid_{\mathcal{E}(C)}$ of the transition maps
$\Omega_{\mathcal{C}}(f)_{\ast}\colon\Omega_{\mathcal{C}}(C)\rightarrow\Omega_{\mathcal{C}}(D)$ (of 
locales). For the reflector to be $\mathcal{C}$-indexed means that the associated regular epic frame morphisms
$j_C\colon\Omega_{\mathcal{C}}(C)\rightarrow \mathcal{E}(C)$ for $C\in\mathcal{C}$ are compatible with the $\mathcal{C}$-action
as well. This additional commutativity of the left adjoints $j_C$ with the cartesian $\mathcal{C}$-action can be thought 
of as a Beck-Chevalley condition on $\mathcal{E}$ with respect to $\Omega_{\mathcal{C}}$. 

A sheaf of first order $\Omega_{\mathcal{C}}$-ideals can be literally thought of as a sheaf of 
exponential $\Omega_{\mathcal{C}}$-ideals via \cite[Exercise II.2]{johnstone_ss}.
By construction, a sheaf $\mathcal{E}$ of first order $\Omega_{\mathcal{C}}$-ideals considered as a contravariant 
functor of the form
\[\mathcal{E}\colon \mathcal{C}^{op}\rightarrow\mathrm{Frm}\]
gives rise to a coherent family of nuclei $j_C$ on the frames $\Omega_{\mathcal{C}}(C)\cong\mathrm{Sub}_{\hat{\mathcal{C}}}(yC)$ 
indexed over $\mathcal{C}$.  Left Kan extension of these nuclei along the Yoneda embedding of $\mathcal{C}$ yields the notion of a 
closure operator on $\hat{\mathcal{C}}$, as explained in more detail in the following Example.

\begin{example}\label{exmplehypdocint}
Whenever $\mathcal{B}$ is a geometric category, it naturally comes equipped with a hyperdoctrine 
\begin{align*}
\Omega_{\mathcal{B}}^{\mathrm{int}}\colon\mathcal{B}^{op} & \rightarrow\mathrm{Frm} \\
B & \mapsto \mathrm{Sub}_{\mathcal{B}}(B)
\end{align*}
of \emph{internal predicates}. By definition, for any small category $\mathcal{C}$ we have $\Omega_{\hat{\mathcal{C}}}^{\mathrm{int}}(yC)\cong\Omega_{\mathcal{C}}(C)$ for all $C\in\mathcal{C}$, using that presheaf categories are geometric. In fact more 
is true: as every topos is effective (and hence satisfies monic descent in the sense of Rezk, \cite[Proposition 2.5]{rezkhtytps}), we have
\[\Omega_{\hat{\mathcal{C}}}^{\mathrm{int}}\cong\mathrm{Lan}_y(\Omega_{\mathcal{C}}).\]
Furthermore, every sheaf $\mathcal{E}$ of first order $\Omega_{\mathcal{C}}$-ideals extends to a sheaf $\mathrm{Lan}_y(\mathcal{E})$ 
of first order $\Omega_{\hat{\mathcal{C}}}^\mathrm{int}$-ideals, by which we again mean a regular subfunctor of
$(\Omega_{\hat{\mathcal{C}}}^\mathrm{int})^{op}$. Such are exactly the \emph{closure operators} on $\hat{\mathcal{C}}$.
\end{example}

For comparison with the $\infty$-categorical expressions in the following sections, in the fibrational framework a closure operator on a geometric category $\mathcal{B}$ is exactly (the fibered monad corresponding to) an elementary subfibration
\[\xymatrix{
\mathcal{E}\ar@{^(->}[r]\ar@{->>}@/_/[dr] & \sum\limits_{B\in\mathcal{B}}\mathrm{Sub}_{\mathcal{B}}(B)\ar@{->>}[d] \\
 & \mathcal{B}
}\]
with a cartesian fibered reflector
$\bar{(\cdot)}\colon\sum_{X\in\mathcal{C}}\mathrm{Sub}_{\mathcal{B}}(B)\rightarrow\mathcal{B}$ which preserves meets 
fiberwise. In the case of presheaf categories $\mathcal{B}=\hat{\mathcal{C}}$, its pullback along the Yoneda embedding is by construction 
exactly its associated sheaf of $\Omega_\mathcal{C}$-ideals represented as an elementary fibration. For technical reasons, it is 
this fibrational expression which we will consider in the $\infty$-categorical setting in the coming sections.

Given that $\mathcal{B}$ is not only geometric, but an elementary topos, the hyperdoctrine $\Omega_{\mathcal{B}}^{\mathrm{int}}$ from 
Example~\ref{exmplehypdocint} is representable by the internal posetal category formed by its subject classifier
$\Omega\in\mathcal{B}$. That is, there is an \emph{externalization} 2-functor
\[\mathrm{Ext}\colon\mathrm{ICat}(\mathcal{B})\rightarrow\mathrm{Fun}(\mathcal{B}^{op},\mathrm{Cat})\]
from the 2-category of internal categories in $\mathcal{B}$ to the 2-category of indexed categories over $\mathcal{B}$ such that
$\mathrm{Ext}(\Omega)\cong\Omega_{\mathcal{B}}^{\mathrm{int}}$. This 2-functor
preserves and reflects various formal categorical structures in the respective 2-categories (\cite[Section 7.3]{jacobsttbook}). In 
particular, it transfers left exact idempotent monads back and forth between $\Omega$ and its externalization. This is captured by the 
equivalence of 4.\ and 5.\ in the following formulation of a classic correspondence.

\begin{proposition}\label{proplocalglobalfirstorder}
Let $\mathcal{C}$ be a small category. Then the following classes stand in bijection to one another.
\begin{enumerate}
\item Grothendieck topologies on $\mathcal{C}$.
\item Sheaves of first order $\Omega_{\mathcal{C}}$-ideals.
\item Closure operators on $\hat{\mathcal{C}}$ (i.e.\ sheaves of first order $\Omega_{\hat{\mathcal{C}}}^{\mathrm{int}}$-ideals).
\end{enumerate}
Whenever $\mathcal{C}$ is an (elementary) topos itself, the following classes stand in bijection to one another as well.
\begin{enumerate}
\item[4.] Sheaves of first order $\Omega_{\mathcal{C}}^{\mathrm{int}}$-ideals.
\item[5.] Lawvere-Tierney topologies on the sub-object classifier $\Omega$ in $\mathcal{C}$.
\item[6.] Equivalence classes of reflective left exact localizations of $\mathcal{C}$.
\end{enumerate}
\end{proposition}
\begin{proof}
The equivalence of Lawvere-Tierney topologies, closure operators and reflective left exact localizations of
toposes (and furthermore of Grothendieck topologies in the case of presheaf toposes) is a fundamental classic result of topos theory, see 
e.g\ \cite{mlmsheaves}. The other equivalences stated here are just reformulations thereof.
\end{proof}

\begin{remark}\label{remlocisnull}
The Grothendieck topology $J$ on a small category $\mathcal{C}$ which corresponds to a given sheaf $\mathcal{E}$ of 
first-order $\Omega_{\mathcal{C}}$-ideals consists at an object $C$ exactly of the objects $m\in\Omega_{\mathcal{C}}(C)$ which 
are nullified by $j_C$. That means a predicate $m\in\Omega_{\mathcal{C}}(C)$ is contained in $J(C)$ if and only 
if $j_C(m)=\top$ in the ideal $\mathcal{E}(C)$. In that sense, the definition of a Grothendieck topology gives necessary and sufficient 
conditions on a collection of predicates $\{J(C)\subseteq\Omega_{\mathcal{C}}(C)\mid C\in\mathcal{C}\}$ for the existence and uniqueness of 
an associated sheaf $\mathcal{E}$ of first order $\Omega_{\mathcal{C}}$-ideals such that each $\mathcal{E}(C)$ nullifies exactly the objects 
contained in $J(C)$. In other words, every such $J$ generates a unique Lawvere-Tierney topology 
$j\colon\Omega\rightarrow\Omega$ such that the inclusion $J\subseteq\Omega$ fits into a pullback square of the form
\[\xymatrix{
J\ar@{^(->}[d]\ar@{}[dr]|(.3){\pbs}\ar[r] & 1\ar@{^(->}[d]^{\top} \\
\Omega\ar[r]_j & \Omega.
}\]
Thus, $j$ is fully determined by the class of predicates $J$ that it nullifies internally.
\end{remark}

The usual B\'{enabou}-Mitchell language internal to a topos $\mathrm{Sh}_J(\mathcal{C})$ associated to a site
$(\mathcal{C},J)$ gives meaning to the interpretation of monomorphisms as predicates in the topos
$\mathrm{Sh}_J(\mathcal{C})$ (\cite[Section VI.5]{mlmsheaves}). The Kripke-Joyal sheaf semantics associated to the site
$(\mathcal{C},J)$ is a canonical way to translate statements in the internal language of $\mathrm{Sh}_J(\mathcal{C})$ to 
the internal language of the category of sets. Its ``local character'' (\cite[Section VI.7]{mlmsheaves}) may be understood as a 
semantic local to global principle of the $\Omega_{\mathcal{C}}$-ideal $\mathcal{E}$ with respect to its associated 
topology $J$. In fact, more generally, the family of frames $\mathcal{E}_C$ for $C\in\mathcal{C}$ itself satisfies a 
categorical local to global principle in the following sense, which justifies the description of such $\mathcal{E}$ as 
\emph{sheaves} of $\Omega_{\mathcal{C}}$-ideals. 

\begin{proposition}\label{prop1stidealsheaf}
Let $(\mathcal{C},J)$ be a site and $\mathcal{E}$ be its associated sheaf of $\Omega_{\mathcal{C}}$-ideals. Then 
the diagram
\[\mathcal{E}\colon\mathcal{C}^{op}\rightarrow\mathrm{Frm}\]
is a $J$-stack of frames (and is in fact the largest -- i.e.\ the terminal -- such among all sheaves of $\Omega_{\mathcal{C}}$-ideals).
\end{proposition}

\begin{proof}
We will show that for every $J$-cover $S$, the sequence
\begin{align}\label{diagsemsheaf}
\xymatrix{
\mathcal{E}(C)\ar[r] & \prod_{f\in S}\mathcal{E}(\mathrm{dom}f)\ar@<.5ex>[r]\ar@<-.5ex>[r] & \prod_{f\in S, \mathrm{dom}f=\mathrm{cod}g}\mathcal{E}(\mathrm{dom}g)
}
\end{align}
is an equalizer diagram of frames. Indeed, the equalizer of the parallel pair in (\ref{diagsemsheaf}) is equivalent to 
the category of monic natural cartesian transformations over the functor
$(\mathcal{C}\downarrow S)\xrightarrow{F}\mathcal{C}\xrightarrow{y}\hat{\mathcal{C}}$ which are componentwise relative 
$J$-sheaves (i.e.\ the isomorphism classes of their components are contained in $\mathcal{E}(\cdot)$).
The functor $y\circ F$ exhibits $S$ as the canonical colimit of representables over $S$. Hence, via monic descent in the presheaf 
category $\mathrm{Fun}(\mathcal{C}^{op},\mathrm{Set})$ (\cite[Proposition 2.5]{rezkhtytps}), the equalizer of the 
parallel pair in (\ref{diagsemsheaf}) is canonically isomorphic to the frame $\mathcal{E}(S)$ of (isomorphism classes 
of) relative monic $J$-sheaves over the presheaf $S$. The inclusion $s\colon S\hookrightarrow yC$ then induces an 
adjoint isomorphism $s^{\ast}\colon\mathcal{E}(C)\rightarrow \mathcal{E}(S)$ (with left adjoint given by $\Sigma_s$ and 
the $\mathcal{E}(C)$-factorization system on monomorphisms) which is exactly the left map in (\ref{diagsemsheaf}).
\end{proof}

\begin{remark}
The local character of the sheaf semantics in \cite[VI.7]{mlmsheaves} follows from Proposition~\ref{prop1stidealsheaf} when applied to the 
elements $\mathcal{E}(C)$ which are mapped to the top element in $\mathcal{E}(\mathrm{dom}(f))$ for every $f$ in some covering sieve $S$.
\end{remark}

\begin{remark}
For later reference, we note that given a site $(\mathcal{C},J)$ and a $J$-cover $S\hookrightarrow yC$, the equalizer of 
the parallel pair in (\ref{diagsemsheaf}) is the general ``equalizer of products'' description of the limit of the 
precomposition
\[(\mathcal{C}\downarrow S)^{op}\xrightarrow{\pi}\mathcal{C}^{op}\xrightarrow{\mathcal{E}}\mathrm{Frm},\]
where $\pi$ is the canonical projection from the category of elements of $S$ to $\mathcal{C}$.
Since $\mathrm{Frm}$ is a complete category, this limit is a presentation of the weighted limit $\{S,\mathcal{E}\}$ when 
$S\colon\mathcal{C}^{op}\rightarrow\mathrm{Set}$ is considered as a weight on $\mathcal{C}^{op}$. Thus, a contravariant 
diagram $F\colon\mathcal{C}^{op}\rightarrow\mathrm{Frm}$ of frames is a stack if and only if for every $J$-cover
$S\hookrightarrow yC$, the induced map 
\[F(C)\cong\{yC,F\}\rightarrow\{S,F\}\]
of weighted limits is an isomorphism.
Analogously, a presheaf $F\in\hat{\mathcal{C}}$ is a $J$-sheaf if and only if 
$F(C)\rightarrow\{S,F\}$ is an isomorphism of sets for every $J$-cover $S\hookrightarrow yC$.
We will use this description of stacks for the analogue of Proposition~\ref{prop1stidealsheaf} in the case of sheaves of
$\mathcal{O}_{\mathcal{C}}$-ideals in Proposition~\ref{propinftysemsheaf}.
\end{remark}

\begin{example}
Proposition~\ref{prop1stidealsheaf} implies that for every small category $\mathcal{C}$ the logical structure sheaf
$\Omega_{\mathcal{C}}$, being an ideal over itself, is (trivially) a sheaf for the minimal topology on $\mathcal{C}$, and that there is no 
finer topology on $\mathcal{C}$ that $\Omega_{\mathcal{C}}$ is a sheaf for.

In contrast, whenever $\mathcal{C}$ is geometric, in Example~\ref{exmplehypdocint} we considered the hyperdoctrine 
$\Omega_{\mathcal{C}}^{\mathrm{int}}$ of internal predicates on $\mathcal{C}$. In this case we may equip $\mathcal{C}$ with its canonical 
(generally not minimal) Grothendieck topology ``$\mathrm{Geo}$'' of geometric covers. Then the hyperdoctrine
$\Omega_{\mathcal{C}}^{\mathrm{int}}$ is a Geo-stack of frames whenever $\mathcal{C}$ is an infinitary pretopos. 
That means, in this case it is a ``propositional'' derivator on $\mathcal{C}$ if we think of $\mathcal{C}$ as a 1-category of propositions. 
Analogously, in Section~\ref{sectopdef} we will see that whenever $\mathcal{C}$ is an $\infty$-logos (that is, the $\infty$-categorical 
version of an infinitary pretopos), then the $\infty$-categorical analogon $\mathcal{O}_{\mathcal{C}}^{\mathrm{int}}$ of internal predicates 
in $\mathcal{C}$ is a higher geometric sheaf as well. In that sense, it is an $\infty$-categorical derivator on $\mathcal{C}$, too, 
of we think of $\mathcal{C}$ as an $\infty$-category of homotopy types.
\end{example}

\section{Fibered structures over $\infty$-toposes I}\label{secfibglobal}

\subsection{Preliminary notions}\label{subsecprelim}

\begin{notation}
Given a reflective localization $\rho\colon\mathcal{B}\rightarrow\mathcal{E}$ of an $\infty$-category $\mathcal{B}$, we adhere to the usual 
conventions and will say that a map $f\in\mathcal{B}$ is an $\mathcal{E}$-local equivalence if $\rho{f}\in\mathcal{E}$ is an equivalence. We 
say that an object $B\in\mathcal{B}$ is $\mathcal{E}$-local if the associated unit $\eta_B$ is an equivalence.
\end{notation}

Recall that every reflective localization $\rho\colon\mathcal{B}\rightarrow\mathcal{E}$ represents the localization of $\mathcal{B}$ at the 
class of all $\mathcal{E}$-local equivalences. As such it is equivalent to the full subcategory spanned by the $\mathcal{E}$-local objects in 
$\mathcal{B}$, which in turn is not only reflective itself but replete in $\mathcal{B}$ as well. It is closed under all limits that exist in 
$\mathcal{B}$ and is (in particular) closed under retracts in $\mathcal{B}$. In the following, the term ``reflective subcategory'' will 
always refer to a subcategory presenting a reflective localization, and hence in particular imply fullness of the subcategory.
Note that ``localization'' in \cite{luriehtt} is synonymous with 
``reflective localization'' (\cite[Definition 5.2.7.2, Warning 5.2.7.3]{luriehtt}).

\begin{definition}[{\cite[Definition 6.1.0.4]{luriehtt}}]
An $\infty$-category $\mathcal{B}$ is an \emph{$\infty$-topos} if it is equivalent to an accessible left exact
localization of the $\infty$-category $\hat{\mathcal{C}}=\mathrm{Fun(\mathcal{C}^{op},\mathcal{S})}$ of presheaves over 
a small $\infty$-category $\mathcal{C}$.
\end{definition}

Given a small $\infty$-category $\mathcal{C}$, any accessible localization
$L\colon\hat{\mathcal{C}}\rightarrow\mathcal{B}$ is reflective and hence may be presented by the subcategory of
$\mathcal{B}$-local objects in $\hat{\mathcal{C}}$ (\cite[Section 5.5.4]{luriehtt}).

\begin{definition}[{\cite[Definition 4.2.1]{abjfsheavesI}}]
Let $\mathcal{B}$ be an $\infty$-category which admits finite limits and small colimits. A class $\mathcal{L}$ of maps 
in $\mathcal{B}$ is called a \emph{congruence} if the following conditions hold.
\begin{enumerate}
\item $\mathcal{L}$ is closed under equivalences and compositions.
\item The full $\infty$-subcategory $\mathcal{L}\subseteq\mathcal{B}^{\Delta^1}$ is closed under small colimits.
\item The full $\infty$-subcategory $\mathcal{L}\subseteq\mathcal{B}^{\Delta^1}$ is closed under finite limits.
\end{enumerate}
A congruence $\mathcal{L}$ is \emph{of small generation} if there is a set $S\subseteq\mathcal{B}^{\Delta^1}$ such that
$\mathcal{L}$ is the strong saturation of $S$ (\cite[Definition 5.5.4.5, Remark 5.5.4.7]{luriehtt}).
\end{definition}

Thus, the notion of small generation of a congruence $\mathcal{L}$ is just the notion of small generation of
$\mathcal{L}$ as a strongly saturated class. Indeed, every congruence is strongly saturated in the first place
(via \cite[Theorem 2.5]{binimkelly} and its dual).

\begin{definition}[{\cite[Definition 5.2.8.8]{luriehtt}}]
Given an $\infty$-category $\mathcal{B}$, a pair $(\mathcal{L},\mathcal{R})$ of classes of maps in
$\mathcal{B}$ is a \emph{factorisation system} whenever
\begin{enumerate}
\item $\mathcal{L}\perp\mathcal{R}$,
\item every map in $\mathcal{B}$ has an $(\mathcal{L},\mathcal{R})$-factorization,
\item each of the two classes $\mathcal{L}$ and $\mathcal{R}$ is closed under retracts.
\end{enumerate}
A factorization system $(\mathcal{L},\mathcal{R})$ is \emph{of small generation} if there is a set
$S\subseteq\mathcal{B}^{\Delta^1}$ such that $\mathcal{R}$ is the class
$S^{\perp}:=\{f\in\mathcal{B}^{\Delta^1}| S\perp f\}$.
\end{definition}

\begin{definition}[{\cite[Definitions 3.2.1, 4.1.1]{abjfsheavesI}}] 
Let $\mathcal{B}$ be an $\infty$-category with pullbacks. 
\begin{enumerate}
\item A factorization system $(\mathcal{L},\mathcal{R})$ on $\mathcal{B}$ is a \emph{modality} on $\mathcal{B}$ if the left class
$\mathcal{L}$ is pullback-stable.
\item Suppose $\mathcal{B}$ has all finite limits. A factorization system $(\mathcal{L},\mathcal{R})$ on
$\mathcal{B}$ is a \emph{left exact modality} if the full subcategory $\mathcal{L}\subseteq\mathcal{B}^{\Delta^1}$ is closed under 
finite limits.
\end{enumerate}
\end{definition}

Whenever the $\infty$-category $\mathcal{B}$ is presentable, a factorization system $(\mathcal{L},\mathcal{R})$ is of 
small generation if and only if there is a set $S\subset\mathcal{B}^{\Delta^1}$ such that $\mathcal{L}$ is the 
\emph{saturation} of $S$ (\cite[Definition 5.5.5.1]{luriehtt}). This follows from the small object argument given in
\cite[Section 2]{as_soa} in standard fashion.

\begin{proposition}\label{thmprelim}
Let $\mathcal{B}$ be an $\infty$-topos. Then the following structures stand in bijective correspondence to one another.
\begin{enumerate}
\item An equivalence class of accessible left exact localizations of $\mathcal{B}$.
\item A congruence of small generation in $\mathcal{B}$.
\item A left exact modality of small generation on $\mathcal{B}$.
\end{enumerate}
\end{proposition}
\begin{proof}
The proof is a straight-forward combination of results in \cite{abjfsheavesI}, \cite{as_soa} and \cite{luriehtt}. First, 
every congruence $\mathcal{L}$ on $\mathcal{B}$ with generating set $G$ yields an accessible localization
$\mathcal{B}\rightarrow\mathcal{B}[G^{-1}]$ by localizing at the set $G$ (\cite[Lemma 5.5.4.14]{luriehtt}). The class of 
arrows which is inverted by this localization is exactly the strong saturation $\mathcal{L}$ of $G$. The localization is 
left exact, because $\mathcal{L}$ is pullback-stable by assumption (via \cite[Proposition 5.6]{rezkhtytps}).

Vice versa, given an accessible left exact localization $\mathcal{B}\rightarrow\mathcal{E}$, 
the class $\mathcal{L}$ of maps inverted by this localization is strongly saturated and closed under finite limits, 
hence a congruence. It is of small generation by accessibility of the localization (again by
\cite[Lemma 5.5.4.14]{luriehtt}). Since a localization is (up to equivalence) uniquely determined by the class of arrows 
it inverts, these two assignments are inverse to each other.

Furthermore, the left class $\mathcal{L}$ of a left exact modality $(\mathcal{L},\mathcal{R})$ on $\mathcal{B}$ is 
always a congruence (via \cite[Theorem 4.2.3]{abjfsheavesI}). Small generation of the factorization system
$(\mathcal{L},\mathcal{R})$ directly implies small generation of $\mathcal{L}$ as a congruence since $\mathcal{L}$ is 
strongly saturated.
Vice versa, given a congruence $\mathcal{L}$ of small generation (as a strongly saturated class) in $\mathcal{B}$, and 
given the equivalence of $1.$ and $2.$, an application of \cite[Proposition 3.4.7]{as_soa} yields a set $G$ of maps in
$\mathcal{B}$ whose saturation is $\mathcal{L}$ by \cite[Theorem 3.4.2]{as_soa}.
\end{proof}

In the next section we express left exact modalities of small generation in an $\infty$-topos $\mathcal{B}$ as certain fibered 
subtoposes of the target fibration $t\colon\mathcal{B}^{\Delta^1}\twoheadrightarrow\mathcal{B}$. Such in 
turn are characterized by their associated fibered idempotent monads on $\mathcal{B}^{\Delta^1}$. 
These may be thought of as higher modal operators on $\mathcal{B}$ (or higher closure operators for that matter). 
Stratifying these operators along an unbounded sequence of regular cardinals and using descent of $\mathcal{B}$, 
we show that these can be internalized to yield a functorial (or ``polymorphic'') sequence of accordingly higher
Lawvere-Tierney operators on $\mathcal{B}$. Much of this can be done in more generality for quasi-categories $\mathcal{B}$ with 
finite limits and other intermediate structure, and hence will be developed accordingly as some of those aspects will turn out to be very 
useful.

To formalize all the involved structures and further transfer them between the various $\infty$-categorical frameworks on, over and in
an $\infty$-topos $\mathcal{B}$, we will make use of Riehl and Verity's work on $\infty$-cosmoses (\cite{riehlverityelements}).
All $\infty$-cosmoses we will consider arise as the $\infty$-cosmos $\mathbb{M}_f$ of fibrant objects associated to a
quasi-categorically enriched model category $\mathbb{M}$ (see \cite[Proposition E.1.1]{riehlverityelements} in the case in which furthermore 
all fibrant objects are cofibrant).

\begin{example}
Given a quasi-category $\mathcal{B}$, the $\infty$-cosmos $(\mathbf{Cat}_{\infty})_{/\mathcal{B}}$ of isofibrations over 
$\mathcal{B}$ is given by the quasi-categorically enriched category of fibrant objects associated to the slice-category
$\mathbf{S}_{/\mathcal{B}}$ equipped with the Joyal model structure over $\mathcal{B}$; its underlying quasi-category is the
over-category $(\mathrm{Cat}_{\infty})_{/\mathcal{B}}$. 

Recall that a fibered adjunction between isofibrations $p\colon\mathcal{E}\twoheadrightarrow\mathcal{B}$ and
$q\colon\mathcal{F}\twoheadrightarrow\mathcal{B}$ over an $\infty$-category $\mathcal{B}$ is an adjunction in the $\infty$-cosmos 
$(\mathbf{Cat}_{\infty})_{/\mathcal{B}}$ (\cite[Definition 3.6.5]{riehlverityelements}). That is, essentially, a pair of functors
$F\colon \mathcal{E}\rightarrow\mathcal{F}$ and $G\colon\mathcal{F}\rightarrow\mathcal{E}$ over $\mathcal{B}$ together with an 
equivalence
\[\mathcal{E}\downarrow G\simeq F\downarrow \mathcal{F}\]
between comma-objects over $\mathcal{E}\times\mathcal{F}$ (\cite[Proposition 4.1.1]{riehlverityelements}).
Such are -- up to a contractible space of higher data -- the formal \emph{homotopy-coherent adjunctions} between $p$ and $q$ in
$(\mathbf{Cat}_{\infty})_{/\mathcal{B}}$ in the sense of \cite{riehlverityadjmon}.
\end{example}

In line with standard notation, we say that a homotopy-coherent adjunction in an $\infty$-cosmos is a \emph{reflective localization} 
whenever its right adjoint is fully faithful (the latter is defined in \cite[Corollary 3.5.6]{riehlverityelements}). Given an 
adjunction, this is the case if and only if the associated counit is a (pointwise) equivalence
(\cite[Proposition 9.4.6]{riehlverityelements}).

\begin{example}
Whenever the isofibrations $p\colon\mathcal{E}\twoheadrightarrow\mathcal{B}$ and $q\colon\mathcal{F}\twoheadrightarrow\mathcal{B}$ 
over a quasi-category $\mathcal{B}$ are cartesian fibrations, we may as well consider the formal homotopy-coherent adjunctions between 
$p$ and $q$ in the $\infty$-cosmos
$\mathbf{Cart}(\mathcal{B})$ of cartesian fibrations over $\mathcal{B}$. Similarly, this $\infty$-cosmos is given by the simplicially 
enriched category of fibrant objects associated to the category $\mathbf{S}^+_{/\mathcal{B}^{\sharp}}$ of marked simplicial sets over 
$\mathcal{B}$ equipped with the cartesian model structure which we denote by ``Cart'' (\cite[Section 3]{luriehtt}). Here, it is 
worth to point out that we consider the model category $(\mathbf{S}^+_{/\mathcal{B}^{\sharp}},\mathrm{Cart})$ as a
quasi-categorically enriched model category via its flat simplicial enrichment (\cite[Section 3.1.3]{luriehtt}) opposed to its sharp 
simplicial enrichment; the latter of which gives it the structure of a simplicial (i.e.\ a $\infty$-groupoidally enriched) model 
category. We will denote these two instances of the cartesian model structure over $\mathcal{B}$ by
$(\mathbf{S}^+_{/\mathcal{B}^{\sharp}},\mathrm{Cart})^{\flat}$ and
$(\mathbf{S}^+_{/\mathcal{B}^{\sharp}},\mathrm{Cart})^{\sharp}$ respectively, so that
$\mathbf{Cart}(\mathcal{B}):=(\mathbf{S}^+_{/\mathcal{B}^{\sharp}},\mathrm{Cart})^{\flat}_f$.
Its underlying quasi-category -- that is equivalently, the underlying quasi-category of
$(\mathbf{S}^+_{/\mathcal{B}^{\sharp}},\mathrm{Cart})^{\sharp}$ -- will be denoted by $\mathrm{Cart}(\mathcal{B})$.

A given fibered homotopy-coherent adjunction $(F,G,\epsilon)$ (with essentially uniquely determined higher adjunction 
data via \cite[Example 4.2.3]{riehlverityadjmon}) between such cartesian fibrations $p$ and $q$ is an adjunction \emph{in}
$\mathbf{Cart}(\mathcal{B)}$ whenever the functors $F\colon p\rightarrow q$ and $G\colon q\rightarrow p$ are 
cartesian functors as well. That is because the forgetful functor
$U\colon\mathbf{Cart}(\mathcal{B})\hookrightarrow(\mathbf{Cat}_{\infty})_{/\mathcal{B}}$ of $\infty$-cosmoi (which arises 
from a right Quillen functor on associated quasi-categorically enriched model structures, see \cite[Proposition 3.1.5.3]{luriehtt}) is 
fully faithful on all $n$-cells for $n\geq 2$ (and hence in particular ``$n$-smothering'' for $n\geq 2$). In some more detail, the 
forgetful functor $U$ yields the forgetful map
\[\mathrm{Map}^{\flat}_{\mathcal{B}}(\mathcal{E}^{\natural},\mathcal{F}^{\natural})\rightarrow\mathrm{Fun}_{\mathcal{B}}(\mathcal{E},\mathcal{F})\]
on associated mapping quasi-categories. It is a bijection for $n\geq 2$ by definition, and a bijection for $n=1$ because the edges 
marked in the product $\mathcal{E}^{\natural}\times(\Delta^{1})^{\flat}$ are tuples of edges which \emph{both} are marked. Thus, a 
functor $\mathcal{E}^{\natural}\times(\Delta^{1})^{\flat}\rightarrow\mathcal{F}^{\natural}$ of marked simplicial sets over
$\mathcal{B}$ is just a functor $\mathcal{E}\times\Delta^{1}\rightarrow\mathcal{F}$ of simplicial sets over $\mathcal{B}$ such that 
the restrictions to $\mathcal{E}\times\{i\}$ for $i=0,1$ are both cartesian functors. That means in particular that squares of the form
\[\xymatrix{
\overline{\{f,g\}}\ar@{^(->}[d]\ar[r]^(.4){(f,g)} & \mathbf{Cart}(\mathcal{B})\ar[d]^U \\
\overline{\{\epsilon\}}\ar[r]^(.4){(f,g,\epsilon)} & (\mathbf{Cat}_{\infty})_{/\mathcal{B}}
}\]
have (unique) fillers. Here, $\overline{\{f,g\}}$ and $\overline{\{\epsilon\}}$ are the simplicial subcomputads of the 
free adjunction $\mathrm{Adj}$ generated by the pair $f,g$ of 1-cells and the 2-cell $\epsilon$, respectively 
(\cite[Sections 3 and 4]{riehlverityadjmon}). A triple $(f,g,\epsilon)$ defined on $\overline{\{\epsilon\}}$ has a homotopically 
unique extension to a fully fledged adjunction defined on $\mathrm{Adj}$ via \cite[Example 4.2.3]{riehlverityadjmon}. 
\end{example}

In this context, let us state a simple closure property of cartesian fibrations for future reference.

\begin{lemma}\label{subfibcart}
Let $\mathcal{B}$ be a quasi-category, $p\colon\mathcal{E}\twoheadrightarrow\mathcal{B}$ be a cartesian fibration and
$\mathcal{F}\subseteq\mathcal{E}$ be the inclusion of a full subcategory such that for all $b\in\mathcal{F}$ and all cartesian 
arrows $f\colon a\rightarrow b$ in $\mathcal{F}$, the domain $a$ is again contained in $\mathcal{F}$. Then the restriction
$q\colon\mathcal{F}\twoheadrightarrow\mathcal{B}$ is again a cartesian fibration, and the inclusion $\iota\colon q\rightarrow p$ is a 
fibration in $\mathbf{Cart}(\mathcal{B})$.
\end{lemma}

\begin{proof}
The first part is one half of \cite[Lemma 4.5]{barwicketalparahct}. The fact that $\iota \colon p\rightarrow q$ preserves cartesian 
arrows readily follows. It is a fibration in $\mathbf{Cart}(\mathcal{B})$ itself, because it has the right lifting property with respect 
to the right marked anodyne morphisms over $\mathcal{B}^{\sharp}$. The latter characterize the fibrations between fibrant objects in
$(\mathbf{S}^+/\mathcal{B}^{\sharp},\mathrm{Cart})$ by \cite[Theorem 2.17, Definition 4.32]{nguyencontracov}). 
\end{proof}

\begin{notation}
Given a class $S$ of maps in an $\infty$-category $\mathcal{B}$ and an object $B\in\mathcal{B}$, we adopt the notation from 
\cite{abjfsheavesI} and denote the full subcategory of the slice $\mathcal{B}_{/B}$ generated by the maps with codomain $B$ which are 
contained in $S$ by $S(B)$. The class of edges $f\in\mathrm{Fun}(\Delta^1,\mathcal{C}_{/B})$ such that the map of domains $\Sigma_B(f)$ is contained in $S$ will be denoted by $S_{/B}.$

Whenever $\mathcal{B}$ has pullbacks and $S$ is pullback-stable, the full cartesian subfibration of the target fibration
$t\colon\mathcal{B}^{\Delta^1}\twoheadrightarrow\mathcal{B}$ generated by the objects in $S$ will be denoted by
\[\sum_{B\in\mathcal{B}}S(B)\twoheadrightarrow\mathcal{B}.\]
A reflective subcategory $\mathcal{E}$ of $\mathcal{B}$ is a reflective localization $\mathcal{B}\rightarrow\mathcal{E}$ whose right adjoint
is an inclusion of quasi-categories. A fibered (cartesian) reflective subcategory $p\colon\mathcal{E}\twoheadrightarrow\mathcal{B}$ of a 
(cartesian) fibration $q\colon\mathcal{F}\twoheadrightarrow \mathcal{B}$ is a fibered (cartesian) reflective localization such that the 
underlying functor $\mathcal{E}\hookrightarrow\mathcal{F}$ of quasi-categories is an inclusion of quasi-categories.
\end{notation}

\subsection{Factorization systems and fibered reflective localizations}\label{secsubfsrl}

The following characterization of factorization systems in an $\infty$-category given in Proposition~\ref{propfibmods} is an
$\infty$-categorical generalization of \cite[Theorem 5.10]{binimkelly} (which in the given reference is stated without the fibrational 
aspects and with an additional assumption on identities which we show to be redundant in the next lemma). Versions of the statement appear 
to be fairly folklore in various categorical contexts, but since a reference for $\infty$-categories also appears to missing to date in this 
generality, the proof will be given in full.

\begin{lemma}\label{lemmaidreduct}
Let $\mathcal{B}$ be an $\infty$-category and $\mathcal{E}\hookrightarrow\mathcal{B}^{\Delta^1}$ be a fibered reflective and replete 
subcategory of the target fibration over $\mathcal{B}$. Then $\mathcal{E}\subseteq\mathcal{B}^{\Delta^1}$ contains all 
identities in $\mathcal{B}$.
\end{lemma}
\begin{proof}
One way to see this is that for every $B\in\mathcal{B}$ the inclusion $\{B\}\colon\Delta^0\rightarrow\mathcal{B}$ induces a cosmological 
functor $\{B\}^{\ast}\colon(\mathbf{Cat}_{\infty})_{/\mathcal{B}}\rightarrow\mathbf{Cat}_{\infty}$ via
\cite[Proposition 1.3.4.(v)]{riehlverityelements}. As cosmological 
functors preserve reflective localizations by \cite[Proposition 10.1.4]{riehlverityelements}, we obtain a reflection of the subcategory
$\mathcal{E}(B)\subseteq\mathcal{B}_{/B}$. But reflective subcategories are closed under limits, and the identity $1_B\in \mathcal{B}_{/B}$ 
is terminal. Thus, $1_B\in\mathcal{E}(B)$ for all objects $B\in\mathcal{B}$.

In fact, more directly, the fibered reflection $\rho\colon\mathcal{B}^{\Delta^1}\rightarrow\mathcal{E}$ applied to an identity $1_B$ exhibits 
the object $1_B$ as a retract of $\iota\rho(1_B)$ via the square
\[\xymatrix{
B\ar@{=}[d]_{1_B}\ar[r]_{\eta_{1_B}} & \bar{B}\ar[d]^{\iota\rho(1_B)}\ar@{-->}@/_1pc/[l]_{\iota\rho(1_B)} \\
B\ar@{=}[r] & B.
}\]
As $\mathcal{E}$ is closed under retracts, it contains all identities.
\end{proof}

\begin{proposition}\label{propfibmods}
For any $\infty$-category $\mathcal{B}$, there is a bijection between factorisation systems on $\mathcal{B}$ and fibered reflective 
and replete subcategories
\begin{align}\label{equpropfibmods}
\begin{gathered}
\xymatrix{
\mathcal{E}\ar@{->>}@/_/[dr]_p\ar@{^(->}[r]_{\iota} & \mathcal{B}^{\Delta^1}\ar@{->>}[d]^{t_{\mathcal{B}}}\ar@/_1pc/[l]_{\rho}^{\rotatebox[origin=c]{90}{$\vdash$}} \\
 & \mathcal{B}
}
\end{gathered}
\end{align}
of $t_{\mathcal{B}}$ such that $\mathcal{E}$
is (objectwise) closed under composition.
This bijection yields an isomorphism of partial orders when considering the class of factorisation systems on $\mathcal{B}$ 
with inclusions of their right classes, and fibered reflective subcategories with functors over $t_{\mathcal{B}}$.	
\end{proposition}

\begin{proof}
Given a factorization system $(\mathcal{L},\mathcal{R})$ on $\mathcal{B}$, we may consider the full subfibration
\[\xymatrix{
\sum\limits_{B\in\mathcal{B}}\mathcal{R}(B)\ar@{^(->}[rr]^{\iota}\ar@{->>}@/_/[dr] & & \mathcal{B}^{\Delta^1}\ar@{->>}@/^/[dl]^t \\
 & \mathcal{B} & 
}\]	
generated by the right class $\mathcal{R}$. This subfibration is clearly replete and
is (objectwise) closed under composition. In the following we show that the inclusion $\iota$ has a fibered left adjoint
\[\rho\colon\mathcal{B}^{\Delta^1}\rightarrow\sum\limits_{B\in\mathcal{B}}\mathcal{R}(B)\]
which takes a map $f$ to the right map $\rho(f)$ according to its (essentially unique) associated
$(\mathcal{L},\mathcal{R})$-factorization. Therefore, let
\[\mathrm{Fun}^{LR}(\Delta^2,\mathcal{B})\subseteq\mathrm{Fun}(\Delta^2,\mathcal{B})\]
be the full $\infty$-subcategory spanned by those triangles $\tau\colon\Delta^2\rightarrow\mathcal{B}$ such that
$d_2(\tau)\in\mathcal{L}$ and $d_0(\tau)\in\mathcal{R}$. Since $\iota$ itself is a fully faithful embedding, we obtain a fully 
faithful embedding of the form
\[\xymatrix{
\mathrm{Fun}^{LR}(\Delta^2,\mathcal{B})\ar@/^1pc/@{^(->}[drr]\ar@/_1pc/@{->>}[ddr]_{(d_1,d_0)}\ar@{^(-->}[dr] &  & \\
 & \mathcal{B}^{\Delta^1}\downarrow\iota\ar@{->>}[d]\ar@{^(->}[r]\ar@{}[dr]|(.3){\pbs}& \mathcal{B}^{\Delta^2}\ar@{->>}[d]^{(d_1,d_0)} \\
 & \mathcal{B}^{\Delta^1}\times_{\mathcal{B}}\sum\limits_{B\in\mathcal{B}}\mathcal{R}(B)\ar@{^(->}[r]_{(1,\iota)} &
 \mathcal{B}^{\Delta^1}\times_{\mathcal{B}}\mathcal{B}^{\Delta^1}.
}\]
depicted by the dotted arrow. Here note that the fibration
$(d_1,d_0)\colon\mathcal{B}^{\Delta^2}\twoheadrightarrow\mathcal{B}^{\Delta^1}\times_{\mathcal{B}}\mathcal{B}^{\Delta^1}$ is the two-sided hom-fibration associated to the 
cotensor $(t\colon\mathcal{B}^{\Delta^1}\twoheadrightarrow\mathcal{B})^{\Delta^1}$ in the $\infty$-cosmos
$(\mathbf{Cat}_{\infty})_{/\mathcal{B}}$. The isofibration
\[d_1\colon\mathrm{Fun}^{LR}(\Delta^2,\mathcal{B})\twoheadrightarrow\mathcal{B}^{\Delta^1}\]
is a trivial fibration (over $\mathcal{B}$) by \cite[Proposition 5.2.8.17]{luriehtt}. Let
$\sigma\colon\mathcal{B}^{\Delta^1}\rightarrow\mathrm{Fun}^{LR}(\Delta^2,\mathcal{B})$ be a section in
$(\mathbf{Cat}_{\infty})_{/\mathcal{B}}$ to this fibration; we claim that the composition
\[\eta\colon\mathcal{B}^{\Delta^1}\rightarrow\mathrm{Fun}^{LR}(\Delta^2,\mathcal{B})\hookrightarrow\mathcal{B}^{\Delta^1}\downarrow\iota\]
is an initial element of $\mathcal{B}^{\Delta^1}\downarrow\iota$ over $\mathcal{B}^{\Delta^1}$ in
$(\mathbf{Cat}_{\infty})_{/\mathcal{B}}$. It then follows that the inclusion
$\iota\colon\sum_{B\in\mathcal{B}}\mathcal{R}(B)\rightarrow\mathcal{B}^{\Delta^1}$ has a left adjoint in
$(\mathbf{Cat}_{\infty})_{/\mathcal{B}}$ by \cite[Proposition 4.1.6]{riehlverityelements}.

Therefore, we consider the following two maps. First, consider the post-composition
\begin{align*}
(\mathcal{B}^{\Delta^1}\downarrow\iota) & \xrightarrow{(\sigma d_1,1)}\mathrm{Fun}^{LR}(\Delta^2,\mathcal{B})\times_{\mathcal{B}^{\Delta^1}}(\mathcal{B}^{\Delta^1}\downarrow\iota)\\ 
 & \hookrightarrow(\mathcal{B}^{\Delta^1}\downarrow\iota)\times_{\mathcal{B}^{\Delta^1}}(\mathcal{B}^{\Delta^1}\downarrow\iota)\\
 & \hookrightarrow\mathcal{B}^{\Delta^2}\times_{\mathcal{B}^{\Delta^1}}\mathcal{B}^{\Delta^2}\\
 &\cong \mathcal{B}^{\Delta^1\times\Delta^1}. 
\end{align*}
Let
\[\mathrm{Fun}^{\mathcal{L}}_{\mathcal{R}}(\Delta^1\times\Delta^1,\mathcal{B})\subseteq\mathcal{B}^{\Delta^1\times\Delta^1}\]
be the full subcategory of squares such that the top map is contained in $\mathcal{L}$ and the bottom map is contained in
$\mathcal{R}$. The inclusion
$\mathrm{Fun}^{LR}(\Delta^1,\mathcal{B})\times_{\mathcal{B}^{\Delta^1}}(\mathcal{B}^{\Delta^1}\downarrow\iota)\subseteq\mathcal{B}^{\Delta^1\times\Delta^1}$ factors through $\mathrm{Fun}^{\mathcal{L}}_{\mathcal{R}}(\Delta^1\times\Delta^1,\mathcal{B})$ by 
construction. Second, let $\mathcal{B}^{\Delta^3}\twoheadrightarrow\mathcal{B}^{\Delta^1\times\Delta^1}$ be the isofibration of 
squares in $\mathcal{B}$ together with a lift; indeed note that $\Delta^3$ is the free shape of a lifting square, and the inclusion
$\Delta^1\times\Delta^1\hookrightarrow\Delta^3$ picks out the according square (without the lift). We thus may consider the following pullbacks.
\begin{align}\label{diagpropfibmods1}
\begin{gathered}
\xymatrix{
\bullet\ar@{^(->}[r]\ar@{->>}[d]\ar@{}[dr]|(.3){\pbs} & (\mathcal{B}^{\Delta^1}\downarrow\iota)^{\Delta^1}\ar@{^(->}[r]\ar@{->>}[d]\ar@{}[dr]|(.3){\pbs} & \mathcal{B}^{\Delta^3}\ar@{->>}[d] \\
\mathrm{Fun}^{LR}(\Delta^2,\mathcal{B})\times_{\mathcal{B}^{\Delta^1}}(\mathcal{B}^{\Delta^1}\downarrow\iota)\ar@{^(->}[r] & (\mathcal{B}^{\Delta^1}\downarrow\iota)\times_{\mathcal{B}^{\Delta^1}}(\mathcal{B}^{\Delta^1}\downarrow\iota)\ar@{^(->}[r] & \mathcal{B}^{\Delta^1\times\Delta^1}
}
\end{gathered}
\end{align}
Here, the fibration
$(\mathcal{B}^{\Delta^1}\downarrow\iota)^{\Delta^1}\twoheadrightarrow(\mathcal{B}^{\Delta^1}\downarrow\iota)\times_{\mathcal{B}^{\Delta^1}}(\mathcal{B}^{\Delta^1}\downarrow\iota)$ is the two-sided hom-fibration of $\mathcal{B}^{\Delta^1}\downarrow\iota$ in
the slice $((\mathrm{Cat}_{\infty})_{/\mathcal{B}})_{/\mathcal{B}^{\Delta^1}}\cong((\mathrm{Cat}_{\infty})_{/\mathcal{B}^{\Delta^1}}$.
The fact that it fits into the pullback square above follows from a straight-forward combinatorial computation of the simplicial 
weights and the respective cotensors. 

The composition of inclusions on the bottom of Diagram~(\ref{diagpropfibmods1}) may also be factored through the subcategory
$\iota_{(\mathcal{L},\mathcal{R})}\colon\mathrm{Fun}^{\mathcal{L}}_{\mathcal{R}}(\Delta^1\times\Delta^1,\mathcal{B})\hookrightarrow\mathcal{B}^{\Delta^1\times\Delta^1}$ instead. Hence, every vertex of the bottom left corner in (\ref{diagpropfibmods1}) is an 
$(\mathcal{L},\mathcal{R})$-lifting problem. Indeed, the pullback
\begin{align}\label{equpropfibmods1}
\iota_{(\mathcal{L},\mathcal{R})}^{\ast}\mathcal{B}^{\Delta^3}\twoheadrightarrow\mathrm{Fun}^{\mathcal{L}}_{\mathcal{R}}(\Delta^1\times\Delta^1,\mathcal{B})
\end{align}
is a trivial fibration: 
first, it is a locally cartesian fibration (\cite[Definition 2.4.2.6]{luriehtt}) as (in a nutshell) for every morphism 
$f\colon\Delta^1\rightarrow\mathrm{Fun}^{\mathcal{L}}_{\mathcal{R}}(\Delta^1\times\Delta^1,\mathcal{B})$ of
$(\mathcal{L},\mathcal{R})$-squares, a lift of $f$ to a morphism $\bar{f}$ of associated lifting squares corresponds to lifts to the square 
between the $\mathcal{L}$-map $s(s(f))$ and the $\mathcal{R}$-map $t(t(f))$ induced by composition. Hence, such lifts $\bar{f}$ exist and are 
unique up to equivalence. In particular, the fibers of the Kan fibration 
\[((\iota_{(\mathcal{L},\mathcal{R})}\circ\{f\})^{\ast}\mathcal{B}^{\Delta^3})_{/\bar{f}}\twoheadrightarrow ((\iota_{(\mathcal{L},\mathcal{R})}\circ\{f\})^{\ast}\mathcal{B}^{\Delta^3})_{/t(\bar{f})}\]
are homotopy equivalent to a slice of the contractible space of lifts to the given square between the $\mathcal{L}$-map $s(s(f))$ and the $\mathcal{R}$-map $t(t(f))$, and hence are contractible themselves. Second, the fibers of the fibration (\ref{equpropfibmods1}) are the spaces of lifts to a given $(\mathcal{L},\mathcal{R})$-square and thus 
are contractible by definition of a factorization system (\cite[Definition 5.2.8.1]{luriehtt}). 
Thus, the map (\ref{equpropfibmods1}) is a locally cartesian fibration with contractible fibers; such are trivial fibrations by 
\cite[Corollary 2.4.4.4]{luriehtt}.

It follows that the left hand side isofibration in Diagram (\ref{diagpropfibmods1}) is a pullback of a trivial fibration and hence a trivial fibration itself.
It hence admits a section which induces a lift of the form
\[\xymatrix{
& \bullet\ar@{^(->}[r]\ar@{->>}[d]\ar@{}[dr]|(.3){\pbs} & (\mathcal{B}^{\Delta^1}\downarrow\iota)^{\Delta^1}\ar@{->>}[d]\\
\mathcal{B}^{\Delta^1}\downarrow\iota\ar[r]_(.25){(\sigma d_1,1)}\ar@{-->}@/^/[ur] & \mathrm{Fun}^{LR}(\Delta^2,\mathcal{B})\times_{\mathcal{B}^{\Delta^1}}(\mathcal{B}^{\Delta^1}\downarrow\iota)\ar@{^(->}[r] & (\mathcal{B}^{\Delta^1}\downarrow\iota)\times_{\mathcal{B}^{\Delta^1}}(\mathcal{B}^{\Delta^1}\downarrow\iota)
}\]
and by composition a lift
\[\xymatrix{
 & (\mathcal{B}^{\Delta^1}\downarrow\iota)^{\Delta^1}\ar@{->>}[d]\\
\mathcal{B}^{\Delta^1}\downarrow\iota\ar[r]_(.3){(\eta d_1,1)}\ar@{-->}@/^/[ur]^{\epsilon} & (\mathcal{B}^{\Delta^1}\downarrow\iota)\times_{\mathcal{B}^{\Delta^1}}(\mathcal{B}^{\Delta^1}\downarrow\iota)
}\]
which is a 2-cell in the $\infty$-cosmos $(\mathbf{Cat}_{\infty})_{/\mathcal{B}^{\Delta^1}}$. Its restriction along
$\eta\colon\mathcal{B}^{\Delta^1}\hookrightarrow(\mathcal{B}\downarrow\iota)$ is pointwise a lift to a square of the form
\[\xymatrix{
\bullet\ar[r]^{\in\mathcal{L}}\ar[d]_{\in\mathcal{L}} & \bullet\ar[d]^{\in\mathcal{R}} \\
\bullet\ar[r]_{\in\mathcal{R}} & \bullet
}\]
and hence a (pointwise) equivalence. It thus is a natural equivalence itself. It follows from
\cite[Lemma 2.2.2
]{riehlverityelements} that the section $\eta$ is initial as was to show.

Vice versa, suppose we are given a fibered reflective localization as in (\ref{equpropfibmods}). Without loss of generality,
$\mathcal{E}$ is a full subcategory of $\mathcal{B}^{\Delta^1}$. Let
$\mathcal{R}\subseteq\mathcal{B}^{\Delta^1}$ be the class of objects in $\mathcal{E}\subseteq\mathcal{B}^{\Delta^1}$, and let
$\mathcal{L}\subseteq\mathcal{B}^{\Delta^1}$ be the class of objects 
$f\in\mathcal{B}^{\Delta^1}$ such that $\iota\rho(f)$ is an equivalence in $\mathcal{B}$. In other words, as $\mathcal{E}$ 
contains the identities in $\mathcal{B}$ by Lemma~\ref{lemmaidreduct}, and hence contains all equivalences, a map $f\colon A\rightarrow B$ is 
contained in $\mathcal{L}(B)$ if and only if $\rho(f)$ is a terminal object in the fiber $\mathcal{E}(B)$.
We show that $(\mathcal{L},\mathcal{R})$ is a factorization system on $\mathcal{B}$.

First, let us show orthogonality. Using terminology from \cite{luriehtt} to denote subsimplices of the standard simplex $\Delta^n$, 
consider the following combinatorial instantiation of the precomposition of 2-cells with 3-cells and squares in $\mathcal{B}$, 
respectively.
\begin{align}\label{diagpropfibmods2}
\begin{gathered}
\xymatrix{
\mathcal{B}^{\Delta^2}\times_{\mathcal{B}^{\Delta^1}}\mathcal{B}^{\Delta^3}\ar@{->>}[d]\ar[r]_(.45){\cong} & \mathcal{B}^{\Delta^{\{0,1,2\}}\cup_{\Delta^1}d_0[\Delta^3]}\ar@{->>}[d] & \mathcal{B}^{\Delta^4}\ar@{->>}[d]\ar@{->>}[l]^(.35){\sim}_(.4){}\ar@{->>}[r]^{d_1} & \mathcal{B}^{\Delta^3}\ar@{->>}[d] \\
\mathcal{B}^{\Delta^2}\times_{\mathcal{B}^{\Delta^1}}\mathcal{B}^{\Delta^1\times\Delta^1}\ar[r]_(.45){\cong} & \mathcal{B}^{\Delta^{\{0,1,2\}}\cup_{\Delta^1}(\Delta^{\{1,2\}}\times\Delta^{\{3,4\}})} & \mathcal{B}^{d_2[\Delta^3]\cup_{\Delta^2}d_3[\Delta^3]}\ar@{->>}[l]^(.4){\sim}\ar@{->>}[r]_(.55){(d_1,d_1)} &  \underset{\cong\mathcal{B}^{\Delta^1\times\Delta^1}}{\mathcal{B}^{\Delta^2\cup_{\Delta^1}\Delta^2}}
}
\end{gathered}
\end{align}
Here, the horizontal trivial fibrations are trivial because they are induced by inner anodyne cofibrations of simplicial sets.

Let $\eta\colon\mathcal{B}^{\Delta^1}\rightarrow\mathcal{B}^{\Delta^2}$ be the fibered unit of the localization, where we 
consider $\mathcal{B}^{\Delta^2}\subseteq\mathcal{B}^{\Delta^1\times\Delta^1}$ via
$(1,s_1d_1)\colon\Delta^1\times\Delta^1\xrightarrow{\cong}\Delta^2\cup_{\Delta^1}\Delta^2\rightarrow\Delta^2$. Now suppose
$f\colon A\rightarrow B$ is contained in $\mathcal{L}$ and $g\colon C\rightarrow D$ is contained in $\mathcal{R}$.
Consider the hom-space
\begin{align}\label{diagpropfibmods3}
\begin{gathered}
\xymatrix{
\mathcal{B}^{\Delta^1}(\iota\rho f,g)\ar@{^(->}[r]\ar@{->>}[d]\ar@{}[dr]|(.3){\pbs} & \mathcal{B}^{\Delta^1\times\Delta^1}\ar@{->>}[d]  \\
\Delta^0\ar@{^(->}[r]_(.4){(\{\iota\rho f\},\{g\})} & \mathcal{B}^{\Delta^1}\times\mathcal{B}^{\Delta^1}
}
\end{gathered}
\end{align}
and the pullback
\begin{align}\label{diagpropfibmods4}
\begin{gathered}
\xymatrix{
\mathcal{B}^{\Delta^3}(\iota\rho f,g)\ar@{^(->}[r]\ar@{->>}[d]\ar@{}[dr]|(.3){\pbs} & \mathcal{B}^{\Delta^3}\ar@{->>}[d]^{(d_{\{0,1\}},d_{\{2,3\}})}  \\
\Delta^0\ar@{^(->}[r]_(.4){(\{\iota\rho f\},\{g\})} & \mathcal{B}^{\Delta^1}\times\mathcal{B}^{\Delta^1}.
}
\end{gathered}
\end{align} 
Define $\mathcal{B}^{\Delta^1}(f,g)$ and $\mathcal{B}^{\Delta^3}(f,g)$ accordingly. Then restriction of the two rows in 
(\ref{diagpropfibmods2}) along the inclusions
\[(\{\eta(f)\},(\ref{diagpropfibmods3}))\colon\mathcal{B}^{\Delta^1}(\iota\rho f,g)\hookrightarrow\mathcal{B}^{\Delta^2}\times_{\mathcal{B}^{\Delta^1}}\mathcal{B}^{\Delta^1\times\Delta^1}\]
and
\[(\{\eta(f)\},(\ref{diagpropfibmods4}))\colon \mathcal{B}^{\Delta^3}(\iota\rho f,g)\hookrightarrow\mathcal{B}^{\Delta^2}\times_{\mathcal{B}^{\Delta^1}}\mathcal{B}^{\Delta^3}\]
yields associated corestrictions of the form
\begin{align}\label{diagpropfibmods5}
\begin{gathered}
\xymatrix{
\mathcal{B}^{\Delta^3}(\iota\rho f,g)\ar@{->>}[d] & \mathcal{B}^{\Delta^4}(f,g)\ar@{->>}[d]\ar@{->>}[l]^(.45){\sim}\ar@{->>}[r]^{d_1} & \mathcal{B}^{\Delta^3}(f,g)\ar@{->>}[d]\\
\mathcal{B}^{\Delta^1}(\iota\rho f,g) & \mathcal{B}^{\Delta^3\cup_{\Delta^2}\Delta^3}(f,g)\ar@{->>}[l]^(.5){\sim}\ar@{->>}[r]_(.55){(d_1,d_1)} & \mathcal{B}^{\Delta^1}(f,g)
}
\end{gathered}
\end{align}
where the quasi-categories in the middle are again defined by the obvious boundary conditions. We use this diagram to argue that $f$ 
is left orthogonal to $g$. Indeed, the top horizontal sequence in (\ref{diagpropfibmods5}) fits into the diagram
\[\xymatrix{
\mathcal{B}^{\Delta^3}(\iota\rho f,g)\ar@{->>}@/_/[dr]_{d_{\{1,2\}}}^{\sim} & \mathcal{B}^{\Delta^4}(f,g)\ar@{->>}[d]|{d_{\{2,3\}}}\ar@{->>}[l]^(.45){\sim}\ar@{->>}[r]^{d_1} & \mathcal{B}^{\Delta^3}(f,g)\ar@{->>}@/^/[dl]^{d_{\{1,2\}}}_{\sim}\\
& \mathcal{B}(B,C), & 
}\]
where the two vertical fibrations $d_{\{1,2\}}$ are trivial in virtue of being pullbacks of the trivial fibration
$\mathcal{B}^{\Delta^3}\twoheadrightarrow\mathcal{B}^{S^3}$ for $S^3\subset\Delta^3$ the 3-spine. It follows that the top fibration $d_1$ is trivial as well.

The bottom horizontal sequence of (\ref{diagpropfibmods5}) consists of trivial fibrations as well, because it is (precisely up to choice 
of a section of the trivial fibration on the bottom of (\ref{diagpropfibmods2})) ) a pullback of the 
(unfibered) adjunction datum
\[\xymatrix{
\rho\downarrow\mathcal{E}\ar[rr]_{\sim}^{\ulcorner(\cdot)\circ\eta\urcorner}\ar@/_/@{->>}[dr] & & \mathcal{B}^{\Delta^1}\downarrow\iota\ar@/^/@{->>}[dl]\ar@{^(->}[r]\ar@{}[dr]|(.3){\pbs} & \mathcal{B}^{\Delta^1\times\Delta^1}\ar@{->>}[d]\\
 & \mathcal{B}^{\Delta^1}\times\mathcal{E}\ar[rr]_{(1,\iota)} & & \mathcal{B}^{\Delta^1}\times\mathcal{B}^{\Delta^1}
}\]
that is given by the precomposition operation from (\ref{diagpropfibmods2}) with the unit 2-cell $\eta$ by
\cite[Proposition 4.1.1, Observation 4.1.2]{riehlverityelements}.

The left vertical fibration in (\ref{diagpropfibmods5}) is a trivial fibration because $\iota\rho(f)$ is an equivalence in
$\mathcal{B}$, and equivalences are left orthogonal to all maps in $\mathcal{B}$ (\cite[Example 5.2.8.9]{luriehtt}). Here we use again that 
this fibration is locally cartesian (being a pullback of the map (\ref{equpropfibmods1}) for $(\mathcal{L},\mathcal{R})$ the pair of 
equivalences and all maps), and so contractibility of its fibers implies triviality of the fibration as noted before. 

Eventually, by 2-out-of-3 it follows that the right vertical fibration in (\ref{diagpropfibmods5}) is trivial, which means that $f$ is 
left orthogonal to $g$ (as its fibers are exactly the mapping spaces considered in \cite[Definition 5.2.8.1]{luriehtt}).

Second, to show that every map $f\colon A\rightarrow B$ in $\mathcal{B}$ admits an $(\mathcal{L},\mathcal{R})$-factorization, it 
suffices to show that the unit 2-cell
\[\xymatrix{
A\ar@/_/[dr]_{f}\ar[rr]^{l(f)}\ar@{}[drr]|{\eta(f)} & & \bar{A}\ar@/^/[dl]^{\iota\rho f} \\
 & B & 
}\]
yields such a factorization. Therefore, we only have to show that the map $l(f):=d_2(\eta(f))\colon A\rightarrow\bar{A}$ is contained 
in $\mathcal{L}$. Thus, consider the (essentially unique) 3-simplex $\tau$ given by
\[\xymatrix{
 & \bar{\bar{A}}\ar[ddd]|(.56)\hole^(.3){}\ar@/^/[ddr]^{\iota\rho(l(f))} & \\
A\ar@/_/[ddr]_f\ar@/_/[drr]^(.3){l(f)}\ar@/^/[ur]^{l(l(f))} & & \\
 & & \bar{A}\ar@/^/[dl]^{\iota\rho(f)} \\
 & B & \\
}\]
in $\mathcal{B}$ where the front face is $\eta(f)$, the top face is $\eta(l(f))$ (where $l(f)$ is considered as an object of
$\mathcal{B}_{/\bar{A}}$) and the vertical map $\bar{\bar{A}}\rightarrow B$ is the composition $\iota\rho(f)\circ\iota\rho(l(f))$.
Since objects in $\mathcal{E}$ are closed under composition by assumption, this composition is contained in $\mathcal{R}$. 
Thus, both the boundaries $d_1[\tau]$ and $d_2[\tau]$ are $\mathcal{E}$-local equivalences from $f\in\mathcal{B}^{\Delta^1}$ into a
$\mathcal{E}$-local object. It follows that the face $d_0[\tau]$ is an $\mathcal{E}$-local equivalence between $\mathcal{E}$-local 
objects, and hence an equivalence in $\mathcal{B}^{\Delta^1}$. Thus, the object $l(f)\in\mathcal{B}^{\Delta^1}$ is contained in
$\mathcal{L}$ by definition of $\mathcal{L}$.

Third, the class $\mathcal{R}$ is closed under retracts, since replete and reflective subcategories are closed under retracts. The 
class $\mathcal{L}$ is closed under retracts since both $\iota$ and $\rho$ preserve retracts and equivalences in $\mathcal{B}$ are 
closed under retracts.

The two assignments are mutually inverse since, on the one hand, factorization systems are fully determined by their right 
class, and on the other hand, two reflective localizations of $\mathcal{B}^{\Delta^1}$ are equivalent if their images in
$\mathcal{B}^{\Delta^1}$ coincide.
Lastly, the fact that this 1-1 correspondence induces an isomorphism of partial orders is straight-forward.

\end{proof}

\begin{lemma}\label{lemmafibmodscons}
Suppose $\mathcal{B}$ is an $\infty$-category with pullbacks and $\mathcal{E}$ is a reflective subcategory of $\mathcal{B}^{\Delta^1}$ as in 
Proposition~\ref{propfibmods}. Let $(\mathcal{L},\mathcal{R})$ be the associated factorization system on $\mathcal{B}$. Then for every 
object $B\in\mathcal{B}$, by pullback along $\{B\}\colon\Delta^0\rightarrow\mathcal{B}$ we obtain a reflective localization
$\rho_B\colon\mathcal{B}_{/B}\rightarrow\mathcal{E}(B)$ of quasi-categories (via
\cite[Proposition 1.3.4.(v), Proposition 10.1.4]{riehlverityelements}). 
\begin{enumerate}
\item The class $\mathcal{L}_{/B}$ -- which by construction consists of those arrows $f\colon A\rightarrow C$ over $B$ which are
$\mathcal{E}(C)$-connected (i.e.\ contractible in $\mathcal{E}(C)$) -- is always contained in the class of
$\mathcal{E}(B)$-local equivalences in $\mathcal{B}_{/B}$. 
\item Whenever the left class $\mathcal{L}\subseteq\mathcal{B}^{\Delta^1}$ is closed under pullbacks, the converse holds as well.
\end{enumerate}
\end{lemma}

\begin{proof}
For Part 1, given a map $f\colon A\rightarrow C$ in $\mathcal{L}_{/B}$ and a 
map $g\in D\rightarrow B$ contained in $\mathcal{E}(B)$, the action
\[f^{\ast}\colon\mathcal{B}_{/B}(C,D)\rightarrow\mathcal{B}_{/B}(A,D)\]
is equivalent to
\[f^{\ast}\colon\mathcal{B}_{/C}(C,g^{\ast}D)\rightarrow\mathcal{B}_{/C}(A,g^{\ast}D).\]
Since $g^{\ast}D$ is again contained in $\mathcal{E}(C)$, the latter map is an equivalence, and hence so is the former.

For Part 2, given a map $f\colon A\rightarrow C$ over $B$, it follows from Part 2 of the proof of Proposition~\ref{propfibmods} that 
the units $\eta_A$ and $\eta_C$ of the adjunction $ \rho_B\adj \iota_B$ are contained in $\mathcal{L}_{/B}$. Thus, whenever $f$ is a
$\mathcal{E}(B)$-local equivalence, the square
\[\xymatrix{
A\ar[r]^{\eta(A)}\ar[d]_f & \bar{A}\ar[d]^{\iota_B\rho_B(A)}_{\simeq}  \\
C\ar[r]_{\eta(C)} & \bar{C}
}\]
shows that $f$ itself is contained in $\mathcal{L}_{/B}$ whenever $\mathcal{L}_{/B}$ is left-exact (and hence left-cancellable, see 
\cite[Theorem 2.5]{binimkelly}).
\end{proof}

\begin{remark}
Whenever $\mathcal{B}$ is left-exact, pullback-stability of a left class $\mathcal{L}\subseteq\mathcal{B}^{\Delta^1}$ is equivalent to
left-exactness of $\mathcal{L}$. In this case, note that left-exactness in Lemma~\ref{lemmafibmodscons}.2 is a necessary condition on
$\mathcal{L}$ at least whenever $(\mathcal{L},\mathcal{R})$ is a modality. That is, because for every object $B\in\mathcal{B}$ the class of
$\mathcal{E}(B)$-local equivalences satisfies the 2-out-of-3 property. Given that the converse of Lemma~\ref{lemmafibmodscons}.1 holds, it 
follows that $\mathcal{L}$ is left-cancellable; but pullback-stability and left-cancellability of $\mathcal{L}$ imply left-exactness.
\end{remark} 

Whenever the base $\mathcal{B}$ has pullbacks, the target fibration $t\colon\mathcal{B}^{\Delta^1}\twoheadrightarrow\mathcal{B}$ is a cartesian fibration. In this case, any fibered reflective localization of the form
\begin{align}\label{equpropfibmodsisol}
\begin{gathered}
\xymatrix{
\mathcal{E}\ar@{->>}@/_/[dr]_p\ar@{^(->}[r]_{\iota} & \mathcal{B}^{\Delta^1}\ar@{->>}[d]^{t_{\mathcal{B}}}\ar@/_1pc/[l]_{\rho}^{\rotatebox[origin=c]{90}{$\vdash$}} \\
 & \mathcal{B}
}
\end{gathered}
\end{align}
is itself automatically a full cartesian subfibration. Indeed, it is straight-forward to show that the subcategory
$\mathcal{E}\subseteq\mathcal{B}^{\Delta^1}$ satisfies the conditions of Lemma~\ref{subfibcart}. In the specific context of 
Proposition~\ref{propfibmods}, this is just the fact that the right class of any factorization system is pullback-stable.
In Lemma~\ref{lemmarefllocbifib} we show a dual statement.

\begin{corollary}\label{corlexfibmods}
The bijection of Proposition~\ref{propfibmods} induces the following bijections as special cases.
\begin{enumerate}
\item  Suppose $\mathcal{B}$ is an $\infty$-category with pullbacks. Then there is a 1-1 correspondence between modalities on
$\mathcal{B}$ and fibered reflective subcategories as in Proposition~\ref{propfibmods} such that the fibered reflector 
is a cartesian functor.\footnote{This is also stated in \cite[Theorem A.7]{rss_hottmod} and \cite[Theorem 4.8]{vergura_loctop}.}
\item Suppose $\mathcal{B}$ has finite limits. Then there is a 1-1 correspondence between left exact modalities on $\mathcal{B}$ and 
fibered reflective subcategories as in Proposition~\ref{propfibmods} such that the fibered reflector is left exact. \footnote{This is 
in essence \cite[Lemma 4.1.2]{abjfsheavesI}.}
\item Suppose $\mathcal{B}$ is presentable with universal colimits. Then there is a 1-1 correspondence between left exact modalities 
on $\mathcal{B}$ of small generation, and fibered reflective subcategories as in Proposition~\ref{propfibmods} with left exact
reflector and fiberwise accessible right adjoint inclusion.
\end{enumerate}
\end{corollary}
\begin{proof}
For Part 1, suppose $(\mathcal{L},\mathcal{R})$ is a modality on $\mathcal{B}$. We are to show that the associated fibered left 
adjoint $\rho\colon\mathcal{B}^{\Delta^1}\rightarrow\sum\limits_{B\in\mathcal{B}}\mathcal{R}(B)$ preserves cartesian squares in
$\mathcal{B}$. Thus, consider its unit 
\begin{align}\label{diagcorlexfibmods1}
\begin{gathered}
\xymatrix{
A\ar[rr]^a\ar[dd]_(.7)f^{\hspace{.3cm}\eta_f}\ar@/^/[dr]^(.6){l(f)}_(.3){\pbs}& & C\ar[dd]_(.7)g^(.4){\hspace{.3cm}\eta_g}\ar@/^/[dr]^{l(g)} & \\
 & \bar{A}\ar[rr]\ar[dl]^(.4){\rho(f)} & & \bar{C}\ar[dl]^{\rho(g)}\\ 
B\ar[rr]_b & & D & \\ 
}
\end{gathered}
\end{align}
applied to the cartesian square in the back. We want to show that the square in the front is cartesian as well; therefore, we note 
that $l(g)\in\mathcal{L}$ and $\rho(g)\in\mathcal{R}$, and so are their pullbacks along the base map $b\colon B\rightarrow D$ since  
both $\mathcal{L}$ and $\mathcal{R}$ are pullback-stable by assumption. Thus both 2-cells $\eta_f$ and $b^{\ast}\eta_g$ yield an
$(\mathcal{L},\mathcal{R})$-factorization of $f$. It follows that $\rho(f)\simeq b^{\ast}\rho(g)$ over $B$ and so the front square in 
(\ref{diagcorlexfibmods1}) is cartesian.

Vice versa, suppose the left adjoint $\rho\colon\mathcal{B}^{\Delta^1}\rightarrow\sum_{B\in\mathcal{B}}\mathcal{R}(B)$ 
preserves cartesian squares, and suppose we are given a cartesian square as in the back of (\ref{diagcorlexfibmods1}) such that the map
$g\colon C\rightarrow D$ is contained in $\mathcal{L}$. It follows that $\rho(g)\colon\bar{C}\rightarrow D$ is an equivalence in
$\mathcal{B}$, and hence so is $\rho(f)\colon\bar{A}\rightarrow B$  as the front square is cartesian as well by assumption. It follows that 
$f\simeq l(f)$ and so the pullback $f$ is contained in $\mathcal{L}$.

Part 2 follows immediately from Proposition~\ref{propfibmods} and \cite[Lemma 4.1.2]{abjfsheavesI}.
%

For Part 3, if the factorisation system $(\mathcal{L},\mathcal{R})$ is of small generation, let $H$ be a generating set of
$\mathcal{L}$-maps. Then an object $f\in\mathcal{B}_{/B}$ is an $\mathcal{R}$-map -- and as such contained in the associated 
fibered reflective localization of $\mathcal{B}^{\Delta^1}\twoheadrightarrow\mathcal{B}$ -- if and only if it is local with respect to 
the set $H_{/B}$. Since the slice $\mathcal{B}_{/B}$ is again presentable, and reflective localizations at sets of maps of such are accessible, it follows that the inclusion $\iota_B\colon\mathcal{R}(B)\hookrightarrow\mathcal{B}_{/B}$ is accessible.

Vice versa, if each inclusion $\mathcal{R}(B)\hookrightarrow\mathcal{B}_{/B}$ is accessible, for $B\in\mathcal{B}$ consider the 
preimage $\rho_B^{-1}[\mathcal{W}]\subseteq(\mathcal{B}_{/B})^{\Delta^1}$ of the class $\mathcal{W}$ of equivalences in
$\mathcal{R}(B)$. The class $\mathcal{W}$ is accessible as the degeneracy $(s_0)^{\ast}\colon\mathcal{R}(B)\rightarrow\mathcal{W}$ which maps 
an object to its identity is part of an equivalence, and the $\infty$-category $\mathcal{R}(B)$ is accessible by assumption. The inclusion $\mathcal{W}\subseteq\mathcal{R}(B)$ 
preserves (filtered colimits), and it follows that the preimage $\rho_B^{-1}[\mathcal{W}]\subseteq(\mathcal{B}_{/B})^{\Delta^1}$ is
accessible by \cite[Corollary A.2.6.5]{luriehtt} (which is stated in more elaborate fashion for $\infty$-categories in
\cite[Proposition 5.4.6.6]{luriehtt}). Thus, for $B\in\mathcal{B}$ let
$H(B)\subseteq(\mathcal{B}_{/B})^{\Delta^1}$ be a set of generators of $\rho_B^{-1}[\mathcal{W}]$. Let $G$ be a set of generators of
the presentable $\infty$-category $\mathcal{B}$. We 
claim that
\[H=\bigcup_{B\in G}d_1[H(B)]\]
generates $(\mathcal{L},\mathcal{R})$ as a factorization system. Thus, first, suppose 
$g\colon C\rightarrow D$ is a map in $\mathcal{B}$ such that $H\perp g$. Since $(\mathcal{L},\mathcal{R})$ is a 
modality and $\mathcal{B}$ is presentable with universal colimits, to show that $g\in\mathcal{R}$ it suffices to show that its pullbacks over 
generators in $G$ are contained in $\mathcal{R}$. Given a map $b\colon B\rightarrow D$
with $B\in G$, the assumption that $g$ is right orthogonal to $d_1[H(B)]$ implies that the pullback $b^{\ast}g\in\mathcal{B}_{/B}$ is 
$H(B)$-local. It follows that $b^{\ast}g\in\mathcal{R}(B)$.

Vice versa, we have to show that $H\subseteq\mathcal{L}$. But this follows from the fact that for every $B\in G$ and every local
$\mathcal{E}(B)$-equivalence $\alpha\in(\mathcal{B}_{/B})^{\Delta^1}$, we have $d_1(\alpha)\in\mathcal{L}$
by Lemma~\ref{lemmafibmodscons}.2 since $\mathcal{L}$ is assumed to be left-exact.
%
%
\end{proof}

\begin{remark}
In Corollary~\ref{corlexfibmods}.3, the proof of fiberwise accessibility of
$\iota\colon\mathcal{E}\hookrightarrow\mathcal{B}^{\Delta^1}$ from small generation of the factorization system
$(\mathcal{L},\mathcal{R})$ does neither require the assumption of universality of colimits in $\mathcal{B}$ nor of left-exactness of 
$\mathcal{L}$. In the other direction, it may be worth to point out the formal role of universality of 
colimits in $\mathcal{B}$: it states that the target fibration $\mathcal{B}^{\Delta^1}\twoheadrightarrow\mathcal{B}$ is a cocomplete 
object in $\mathbf{Cart}(\mathcal{B})$. That is, because for cocomplete $\mathcal{B}$, colimits in $\mathcal{B}$ are universal if and 
only if the Straightening $\mathcal{B}_{/\blank}\colon\mathcal{B}^{op}\rightarrow\mathrm{Cat}_{\infty}$ factors through the
$\infty$-category of cocomplete $\infty$-categories and cocontinuous functors (see Proposition~\ref{propfiblex} and 
Remark~\ref{rempropfiblexgen}). Similarly, fiberwise accessibility of the right adjoints in this context is formal accessibility of the right 
adjoints in the $\infty$-cosmos $\mathbf{Cart}(\mathcal{B})$. 

Furthermore, whenever a factorization system $(\mathcal{L},\mathcal{R})$ on a presentable $\infty$-category $\mathcal{B}$ is of small 
generation (with generating set $G\subset\mathcal{L}$), an object $f\in\mathcal{B}^{\Delta^1}$ is an $\mathcal{R}$-map -- and as such 
contained in the associated fibered reflective localization of $\mathcal{B}^{\Delta^1}\twoheadrightarrow\mathcal{B}$ -- if and only if it is 
local with respect to the class of maps
\[\xymatrix{
C\ar[r]^f\ar[d]_f & D\ar@{=}[d] \\
D\ar@{=}[r] & D
}\]
in $\mathcal{B}^{\Delta^1}$ for $f\in G$. Since the $\infty$-category $\mathcal{B}^{\Delta^1}$ is again presentable, it follows that 
the inclusion 
$\iota\colon\sum_{B\in\mathcal{B}}\mathcal{R}(B)\hookrightarrow\mathcal{B}^{\Delta^1}$ is accessible. This means that fiberwise 
accessibility of the fibered reflective localization implies ``total'' accessibility of the localization in this case.
\end{remark}

Proposition~\ref{propfibmods} and Corollary~\ref{corlexfibmods}.3 are the main results of this section relevant for the constructions in 
Sections~\ref{sechigherclosureops} and \ref{secsubhlttops}. The following is relevant for Section~\ref{sectopdef} and indulges in a few 
more observations about the interplay of the present cartesian and cocartesian structures.

First, we note that fibered reflective localizations $\mathcal{E}\twoheadrightarrow\mathcal{B}$ 
of the target isofibration $t_{\mathcal{B}}$ in $(\mathbf{Cat}_{\infty})_{/\mathcal{B}}$ over any $\infty$-category $\mathcal{B}$ also 
automatically inherit the cocartesian structure from $t_{\mathcal{B}}$ as follows.

\begin{lemma}\label{lemmarefllocbifib}
Let $\mathcal{B}$ be a quasi-category and $p\colon\mathcal{E}\twoheadrightarrow\mathcal{B}$ be a reflective localization of
$t_{\mathcal{B}}\colon\mathcal{B}^{\Delta^1}\twoheadrightarrow\mathcal{B}$ in $(\mathbf{Cat}_{\infty})_{/\mathcal{B}}$ with fibered 
reflector $\rho\colon\mathcal{B}^{\Delta^1}\rightarrow\mathcal{E}$. Then $\mathcal{E}\twoheadrightarrow\mathcal{B}$ is also a 
cocartesian fibration; its cocartesian arrows are exactly those equivalent to $\rho(\alpha)$ for some 
cocoartesian arrow $\alpha\in\mathcal{B}^{\Delta^1}$. In particular, the fibered reflector $\rho$ is a cocartesian functor.  
\end{lemma}

\begin{proof}
Since $\rho$ is essentially surjective, and cocartesian arrows with fixed domain over a given base map in $\mathcal{B}$ are unique up 
to equivalence, it suffices to show that $\rho$ preserves cocartesianness of arrows. Let $f\colon C\rightarrow D$ be an arrow in
$\mathcal{B}$, let $g\in\mathcal{B}_{/C}$ be an object over $C$, and $\gamma\colon g\rightarrow fg$ be the (essentially unique) 
cocartesian arrow in $\mathcal{B}^{\Delta^1}$, given by the canonical square in $\mathcal{B}$ of the form
\begin{align}\label{diaglemmarefllocbifib}
\begin{gathered}
\xymatrix{
A\ar@{=}[r]\ar[d]_{g}\ar@{}[dr]|{\gamma} & A\ar[d]^{f\circ g} \\
C\ar[r]_f & D.
}
\end{gathered}
\end{align}
We are to show that the canonical fibration
\[\mathcal{E}_{\rho(\gamma)/}\rightarrow\mathcal{E}_{\rho(g)/}\times_{\mathcal{B}_{p\rho(g)/}}\mathcal{B}_{p\rho(\gamma)/}\]
is a trivial fibration. Because $\iota$ is fully faithful and right adjoint to $\rho$ over $\mathcal{B}$, this fibration is equivalent 
to the gap map
\begin{align}\label{equlemmarefllocbifib}
\gamma\downarrow\iota\rightarrow g\downarrow\iota\times_{\mathcal{B}_{t(g)/}}\mathcal{B}_{t(\gamma)/}
\end{align}
via the (fibered) unit precomposition $\eta^{\ast}\colon\rho\downarrow\mathcal{E}\rightarrow\mathcal{B}^{\Delta^1}\downarrow\iota$ 
over $\mathcal{B}^{\Delta^1}\times\mathcal{E}$. This gap map fits into the 
concatenation of the two top squares in the following diagram.
\begin{align}\label{diaglemmarefllocbifib2}
\begin{gathered}
\xymatrix{
 & & \gamma\downarrow\iota\ar[dl]\ar@{->>}@/^2pc/[dddll]|(.26)\hole|(.57)\hole\ar@{^(->}[rr]\ar@{}[ddd]^(.1){\pbs}  & & (\mathcal{B}^{\Delta^1})_{\gamma/}\ar[dl]\ar[r]\ar@{->>}@/^1pc/[dddll]|(.3)\hole & \mathcal{B}_{t(\gamma)/}\ar[dl]\ar@{->>}@/^1pc/[dddll] \\
 & g\downarrow\iota\ar@{^(->}[dl]\ar@{->>}@/^/[ddl]|(.48)\hole\ar@{^(->}[rr]\ar@{}[ddr]|(.2){\pbs}  & & (\mathcal{B}^{\Delta^1})_{g/}\ar@{^(->}[dl]\ar@{->>}@/^/[ddl]\ar@{->>}[r] & \mathcal{B}_{t(g)/}\ar@{->>}@/^/[ddl] & \\
\mathcal{B}^{\Delta^1}\downarrow\iota\ar@{->>}[d]\ar@{^(->}[rr]\ar@{}[drr]|(.3){\pbs} & & \mathcal{B}^{\Delta^2}\ar@{->>}[d]^{d_0} & & & \\
\mathcal{E}\ar@{^(->}[rr]_{\iota} & & \mathcal{B}^{\Delta^1}\ar@{->>}[r]_t & \mathcal{B} & & 
}
\end{gathered}
\end{align}
The right hand square on top of (\ref{diaglemmarefllocbifib2}) is homotopy-cartesian by cocartesianness of $\gamma$. The left hand side 
square is homotopy-cartesian by right-cancellation of homotopy-cartesian squares. It follows that the composed square on top of 
(\ref{diaglemmarefllocbifib2}) is homotopy-cartesian. This means that the gap map (\ref{equlemmarefllocbifib}) is a trivial fibration 
since the bottom composition $g\downarrow\iota\rightarrow\mathcal{B}_{t(g)/}$ is still a categorical fibration (and so the
homotopy-pullback of the large top square is presented by the ordinary corresponding pullback).
\end{proof}

\begin{remark}\label{remcocartaction}
Suppose $\mathcal{B}$ has pullbacks and we are given a fibered reflective localization as in (\ref{equpropfibmods}). We have seen that
$\mathcal{E}\twoheadrightarrow\mathcal{B}$ automatically inherits the bifibrational structure from $t_{\mathcal{B}}$, so that the left 
adjoint is cocartesian and the right adjoint inclusion is cartesian. In indexed terms, we obtain a subfunctor
$\mathcal{E}\hookrightarrow\mathcal{B}_{/\blank}$ in $\mathrm{Fun}(\mathcal{B}^{op},\mathrm{Cat}_{\infty})$ on the one hand, and a 
natural transformation $\rho\colon\mathcal{B}_{/\blank}\rightarrow\mathcal{E}$ in $\mathrm{Fun}(\mathcal{B},\mathrm{Cat}_{\infty})$ on the 
other hand, such that each pair
$(\rho_B,\iota_B)\colon\mathcal{B}_{/B}\rightarrow\mathcal{E}(B)$ is a reflective localization of $\infty$-categories.
For arrows $f\colon A\rightarrow B$ in $\mathcal{B}$ we thus obtain homotopy-commutative squares of the form
\begin{align}\label{diagremcocartaction}
\begin{gathered}
\xymatrix{
\mathcal{E}(B)\ar[d]_{f^{\ast}}\ar@{^(->}[r]^{\iota_B} & \mathcal{B}_{/B}\ar[d]^{f^{\ast}}\\
\mathcal{E}(A)\ar@{^(->}[r]_{\iota_A} & \mathcal{B}_{/A}
}
\end{gathered}
& & 
\begin{gathered}
\xymatrix{
\mathcal{E}(B) & \mathcal{B}_{/B}\ar[l]_{\rho_B}\\
\mathcal{E}(A)\ar[u]^{\Sigma_f} & \mathcal{B}_{/A}\ar[l]^{\rho_A}\ar[u]_{\Sigma_f}.
}
\end{gathered}
\end{align}
The fact that each such $f^{\ast}\colon\mathcal{B}_{/B}\rightarrow\mathcal{B}_{/A}$ takes $\mathcal{E}(B)$-local objects to
$\mathcal{E}(A)$-local objects (given by homotopy-commutativity of the left square) is equivalent to the fact that the left adjoints
$\Sigma_f\colon\mathcal{B}_{/A}\rightarrow\mathcal{B}_{/B}$ take $\mathcal{E}(A)$-local equivalences to $\mathcal{E}(B)$-local 
equivalences (given by homotopy-commutativity of the right square). The cocartesian action of a map $f\colon A\rightarrow B$ in
$\mathcal{B}$ on an object $g\in\mathcal{E}(A)$ is given by 
the composition
\begin{align}\label{equcocartformula}
\Sigma_f(g):=\rho_B(\Sigma_f(\iota_A(g))).
\end{align}
Generally, a bifibration $\mathcal{E}\twoheadrightarrow\mathcal{B}$ is said to be Beck-Chevalley if its cartesian and cocartesian 
actions satisfy the according Beck-Chevalley condition over all cartesian squares in $\mathcal{B}$. Since
$t\colon\mathcal{B}^{\Delta^1}\twoheadrightarrow\mathcal{B}$ is always Beck-Chevalley itself, it follows from the formula in 
(\ref{equcocartformula}) that, whenever $\mathcal{E}$ is reflective and hence contains all identities in $\mathcal{B}$ as objects, the 
fibration $\mathcal{E}\twoheadrightarrow\mathcal{B}$ is Beck-Chevalley if and only if the left adjoint $\rho$ is cartesian (and hence the 
whole adjunction $\rho\adj\iota$ is an adjunction in $\mathbf{Cart}(\mathcal{B})$).
\end{remark}

\begin{remark}\label{remcocartcharfacsys}
Let $\mathcal{E}\hookrightarrow\mathcal{B}^{\Delta^1}$ be a reflective subcategory as in Proposition~\ref{propfibmods}.
\begin{enumerate}
\item In this special case, the cocartesian action of a map $f$ in $\mathcal{B}$ on
$\mathcal{E}\simeq\sum_{B\in\mathcal{B}}\mathcal{R}(B)$ described in (\ref{equcocartformula}) is exactly the sharp-construction 
$f_{\sharp}$ given in \cite[Proposition 3.1.22]{abjfsheavesI}. 
\item By Corollary~\ref{corlexfibmods}.1 and Remark~\ref{remcocartaction}, the fibration $\mathcal{E}\twoheadrightarrow\mathcal{B}$ 
is Beck-Chevalley if and only if the factorization system $(\mathcal{L},\mathcal{R})$ is a modality.
\end{enumerate}
\end{remark}

While cocartesianness of the left adjoint in Lemma~\ref{lemmarefllocbifib} is automatic, the fibered inclusion
$\iota\colon\mathcal{E}\hookrightarrow\mathcal{B}^{\Delta^1}$ is generally not cococartesian. Indeed, the inclusion $\iota$ preserves 
cocartesian arrows over maps $g$ in $\mathcal{B}$ \emph{with $g\in\mathcal{E}$} if and only if the objects of $\mathcal{E}$ are closed 
under composition. Equivalently, this is characterized as follows.

\begin{lemma}\label{lemmacharTglue}
Let $\mathcal{B}$ be a quasi-category and $p\colon\mathcal{E}\twoheadrightarrow\mathcal{B}$ be a reflective localization of
$t_{\mathcal{B}}\colon\mathcal{B}^{\Delta^1}\twoheadrightarrow\mathcal{B}$ in $(\mathbf{Cat}_{\infty})_{/\mathcal{B}}$ with fibered 
reflector $\rho\colon\mathcal{B}^{\Delta^1}\rightarrow\mathcal{E}$. Then $\mathcal{E}$ is closed under compositions if and only if
for every $g\colon C\rightarrow D$  contained in $\mathcal{E}$ the cocartesian action
$\Sigma_{g}\colon\mathcal{B}_{/C}\rightarrow\mathcal{B}_{/D}$ reflects $\mathcal{E}$-local equivalences. In turn, that is equivalent to the 
induced cocartesian action $\Sigma_g\colon\mathcal{E}(C)\rightarrow\mathcal{E}(D)$ being conservative.
\end{lemma}

\begin{proof}
First, suppose $\mathcal{E}$ is closed under compositions. Let $f\colon A\rightarrow B$ be a map in $\mathcal{B}_{/C}$ for some object 
$C\in\mathcal{B}$ (via a 2-cell $\sigma$ in $\mathcal{B}$) and let $g\colon C\rightarrow D$ be a map contained in $\mathcal{E}$. Applying the 
unit $\eta$ to $f$ yields a 3-cell in $\mathcal{B}$ of the following shape. 
\begin{align}\label{diagcharTglue1}
\begin{gathered}
\xymatrix{
 & \bar{A}\ar[rr]^{d_1(\iota\rho(\sigma))}\ar@/_/[dd]|(.5)\hole^(.3){\iota\rho(a)}& & \bar{B}\ar@/^2pc/[ddll]^{\iota\rho(b)} \\
A\ar[rr]_(.6)f\ar[ur]^{\eta_a}\ar[dr]_a & & B\ar[dl]^b\ar[ur]_{\eta_b} & \\
 & C & & 
}
\end{gathered}
\end{align}
Postcomposition with $g\colon C\rightarrow D$ yields an according 3-cell as follows.
\begin{align}\label{diagcharTglue2}
\begin{gathered}
\xymatrix{
 & \bar{A}\ar[rr]^{d_1(\iota\rho(\sigma))}\ar@/_/[dd]|(.5)\hole^(.3){\Sigma_g(\iota\rho(a))}& & \bar{B}\ar@/^2pc/[ddll]^{\Sigma_g(\iota\rho(b))} \\
A\ar[rr]_(.6)f\ar[ur]^{\eta_a}\ar[dr]_{\Sigma_g(a)} & & B\ar[dl]^(.3){\Sigma_g(b)}\ar[ur]_{\eta_b} & \\
 & D & & 
}
\end{gathered}
\end{align}
Suppose the front 2-cell is a $\mathcal{E}(D)$-local equivalence. Since the $\Sigma_g$-action preserves $\mathcal{E}$-local 
equivalences (Remark~\ref{remcocartaction}), the 2-cells on the left and right hand side of the 3-cell (\ref{diagcharTglue2}) are
$\mathcal{E}(D)$-local equivalences. By 2-out-of-3 it follows that the back face is a $\mathcal{E}(D)$-local equivalence. Both its 
source and target are $\mathcal{E}(D)$-local objects, since $\mathcal{E}$-local objects are assumed to be closed under composition. It 
follows that the back face of (\ref{diagcharTglue2}) is an equivalence 1-cell in $\mathcal{B}^{\Delta^1}$. That means that the map 
$d_1(\iota\rho(\sigma)))$ is an equivalence in $\mathcal{B}$, which in turn means that the back face of the 3-cell in
(\ref{diagcharTglue1}) is an equivalence in $\mathcal{B}^{\Delta^1}$. In other words, the front face $\sigma$ is an
$\mathcal{E}(C)$-local equivalence. It follows that $\Sigma_g\colon\mathcal{B}_{/C}\rightarrow\mathcal{B}_{/D}$ reflects
$\mathcal{E}$-local equivalences, and so $\Sigma_g\colon\mathcal{E}(C)\rightarrow\mathcal{E}(D)$ is conservative (in virtue of 
homotopy-commutativity of the squares in (\ref{diagremcocartaction})).

Second, suppose the cocartesian action of $\mathcal{E}$-arrows on $\mathcal{E}$ is conservative, and let $e\colon E\rightarrow D$ be 
in $\mathcal{E}(D)$ and  $d\colon D\rightarrow C$ be in $\mathcal{E}(C)$. We want to show that the (essentially unique) composition
$de\colon E\rightarrow C$ is $\mathcal{E}(C)$-local. Therefore, let $f\colon A\rightarrow B$ be an $\mathcal{E}(C)$-local 
equivalence over $C$. Then any map $A\rightarrow E$ over $C$ induces essentially unique pairs of dotted arrows in the diagram
\[\xymatrix{
A\ar[rr]\ar[dr]_(.7)f\ar@/_/[ddr]_a & & E\ar[d]^e \\
 & B\ar[d]^b\ar@{-->}[r]\ar@{-->}[ur] & D.\ar@/^/[dl]^d \\
 & C &
}\] 
They are formally given by the equivalences in the square
\begin{align}\label{equlemmacharTglue} 
\begin{gathered}
\xymatrix{
\mathcal{B}_{/C}(b,\Sigma_d(e))\ar[rr]_(.4){\simeq}^(.4){(e_!,\pi_C(b,e))}\ar[dd]_{f^{\ast}}  & & \mathcal{B}_{/C}(b,d)\times_{\mathcal{B}(B,D)}\mathcal{B}(B,E)\ar[d]_{\simeq}^{(1,f^{\ast})} \\
 & & \mathcal{B}_{/C}(b,d)\times_{\mathcal{B}(A,D)}\mathcal{B}(A,E)\ar[d]_{\simeq}^{(f^{\ast},1)}  \\
\mathcal{B}_{/C}(a,\Sigma_d(e))\ar[rr]_(.4){\simeq}^(.4){(e_!,\pi_C(a,e))} & & \mathcal{B}_{/C}(a,d)\times_{\mathcal{B}(A,D)}\mathcal{B}(A,E)
}
\end{gathered}
\end{align}
where the objects on the right hand side denote homotopy-pullbacks of $\infty$-groupoids. Here, the two horizontal maps are 
equivalences by \cite[Proposition 2.4.4.3]{luriehtt} applied to the right fibration
$\pi_C\colon\mathcal{B}_{/C}\twoheadrightarrow\mathcal{B}$ and the cartesian morphism $e\colon E\rightarrow D$ in
$\mathcal{B}_{/C}$. The lower vertical map $(f^{\ast},1)$ is an equivalence, because $d\in\mathcal{B}_{/C}$ is contained in
$\mathcal{E}(C)$ by assumption, and $f\colon A\rightarrow B$ is $\mathcal{E}(C)$-local (and homotopy-pullbacks preserve equivalences). To see 
that the upper vertical map $(1,f^{\ast})$ is an equivalence, consider the triangle
\[\xymatrix{
\mathcal{B}_{/C}(b,d)\times_{\mathcal{B}(B,D)}\mathcal{B}(B,E)\ar@/_1pc/[dr]_{\pi_1}\ar[rr]^{(1,f^{\ast})} &  & \mathcal{B}_{/C}(b,d)\times_{\mathcal{B}(A,D)}\mathcal{B}(A,E)\ar@/^1pc/[dl]^{\pi_1} \\
 & \mathcal{B}_{/C}(b,d). & 
}\]
Its restrictions to the homotopy-fibers over points $g\in\mathcal{B}_{/C}(b,d)$ are the maps
\[f^{\ast}\colon\mathcal{B}_{/D}(g,e)\rightarrow\mathcal{B}_{/D}(gf,e).\]
These are equivalences, precisely because $e\in\mathcal{B}_{/D}$ is contained in $\mathcal{E}(D)$ and $f\colon A\rightarrow B$ 
considered as a map over $D$ is an $\mathcal{E}(D)$-local equivalence by the conservativity assumption. It follows that the map 
$(1,f^{\ast})$ is an equivalence itself since the $\infty$-category of $\infty$-groupoids is an $\infty$-topos generated by the point 
$\Delta^0$. It follows that the left vertical map $f^{\ast}$ in (\ref{equlemmacharTglue}) is an equivalence, which means that the 
composition $\Sigma_d(e)\colon E\rightarrow D\rightarrow C$ is $f$-local over $C$. Since $f$ was an arbitrary $\mathcal{E}(C)$-local 
equivalence, this means
it is contained in $\mathcal{E}(C)$ itself.
\end{proof}

\begin{corollary}\label{corlexmodetale}
Suppose $\mathcal{B}$ is left exact. Then we obtain a bijection between left exact modalities on $\mathcal{B}$ and  
reflective subcategories $\mathcal{E}\hookrightarrow\mathcal{B}^{\Delta^1}$ in $\mathbf{Cart}(\mathcal{B})$ with left exact left 
adjoint $\rho$ such that the cocartesian action $\Sigma_f\colon\mathcal{E}(A)\rightarrow\mathcal{E}(B)$ is conservative 
for all $f\colon A\rightarrow B$ in $\mathcal{B}$.
\end{corollary}
\begin{proof}
Against the background of Lemma~\ref{lemmacharTglue} and Corollary~\ref{corlexfibmods}.2 we are only left to show that whenever a 
reflective localization $\mathcal{E}\hookrightarrow\mathcal{B}^{\Delta^1}$ in $\mathbf{Cart}(\mathcal{B})$ arises from a left exact 
modality on $\mathcal{B}$, then the cocartesian action of any arrow in $\mathcal{B}$ on $\mathcal{E}$ is conservative.
Thus assume $\mathcal{E}\simeq\sum_{B\in\mathcal{B}}\mathcal{R}(B)$ for an associated left exact modality $(\mathcal{L},\mathcal{R})$ 
on $\mathcal{B}$. Then for all $B\in\mathcal{B}$ the class of $\mathcal{E}(B)$-local equivalences is exactly 
the class $\mathcal{L}_{/B}$ by Lemma~\ref{lemmafibmodscons}.2. Clearly, given a map $f\colon A\rightarrow B$ in $\mathcal{B}$ and a map 
$g\colon C\rightarrow D$ over $A$, we have $\Sigma_f(g)\in\mathcal{L}_{/B}$ if and only if $g\in\mathcal{L}_{/A}$. Thus,
$\Sigma_f\colon\mathcal{B}_{/A}\rightarrow\mathcal{B}_{/B}$ reflects $\mathcal{E}$-local equivalences.
\end{proof}

\begin{remark}\label{remlawverecomp}
Let $\mathcal{E}\hookrightarrow\mathcal{B}^{\Delta^1}$ be a fibered reflective and replete subcategory with fibered left adjoint $\rho$. The 
full $\infty$-subcategory $\mathcal{E}\subseteq\mathcal{B}^{\Delta^1}$ contains all identities if and only if the fibration
$\mathcal{E}\twoheadrightarrow\mathcal{B}$ has a terminal object (as an object of the $\infty$-cosmos
$(\mathbf{Cat}_{\infty})_{/\mathcal{B}}$). We have seen in Lemma~\ref{lemmaidreduct} that this is always automatically satisfied.
Since $g\simeq\Sigma_g(1_C)$ for all arrows $g\colon C\rightarrow D$ in $\mathcal{B}$, the formula 
(\ref{equcocartformula}) exhibits the functor $\rho\colon\mathcal{B}^{\Delta^1}\rightarrow\mathcal{E}$ as the left adjoint 
in Lawvere's comprehension schema (\cite[Section 2]{lawverecomp}). We hence obtain a 1-1 correspondence between factorization systems
(modalities) on $\mathcal{B}$ and those \emph{full (cartesian) Lawvere $\infty$-categories} in the sense of
\cite[Example 4.18]{jacobscompcats} which have sums (as a comprehension $\infty$-category in the sense of \cite[Section 5]{jacobscompcats}).
\end{remark}


\section{Fibered structures over $\infty$-toposes II}\label{secfibglobal2}

Having characterized various kinds of factorization systems in an $\infty$-category $\mathcal{B}$ via associated kinds of fibered 
reflective localizations of its target fibration $t\colon\mathcal{B}^{\Delta^1}\twoheadrightarrow\mathcal{B}$, we may apply the Beck 
Monadicity Theorem to describe these fibered reflective localizations furthermore in terms of associated kinds of fibered idempotent 
monads on $t$. We thus will show in Section~\ref{sechigherclosureops} the following theorem and define all non-standard notions occurring 
therein.

\begin{repthm}{thmlocmodopacc}[Higher closure operators]
Let $\mathcal{B}$ be a presentable $\infty$-category with universal colimits. Then there is a bijection between 
left exact modalities of small generation on $\mathcal{B}$ and higher closure operators on
$\mathcal{B}$: that is, idempotent monads $T\colon t_{\mathcal{B}}\rightarrow t_{\mathcal{B}}$ in the $\infty$-cosmos
$\mathbf{Cart}(\mathcal{B})$ which are (formally) left exact, accessible and compositional.
\end{repthm}

In Section~\ref{secsubhlttops} we internalize the cartesian reflective localizations from Section~\ref{secsubfsrl} (or equivalently the 
higher closure operators from Section~\ref{sechigherclosureops}) whenever the base $\mathcal{B}$ is an $\infty$-topos, and show the following 
theorem as well as define all non-standard notions occurring therein.

\begin{repthm}{thmcharlvtops}[Higher Lawvere-Tierney operators]
Let $\mathcal{B}$ be an $\infty$-topos. Then there is a bijection between left exact modalities of small generation on
$\mathcal{B}$, and equivalence classes of Lawvere-Tierney operators on $\mathcal{N}(\pi_{\bullet})$: that is, up to equivalence,
unbounded sequences of left exact, accessible and compositional idempotent monads
$T_{\kappa}\colon\mathcal{N}(\pi_{\kappa})\rightarrow\mathcal{N}(\pi_{\kappa})$ in the $\infty$-cosmos $\mathbf{Cat}_{\infty}(\mathcal{B})$ 
of internal $\infty$-categories in $\mathcal{B}$, where $\mathcal{N}(\pi_{\kappa})$ denotes the nerve of the (essentially unique) relatively
$\kappa$-compact object classifier $\pi_{\kappa}$ in $\mathcal{B}$ for suitable regular cardinals $\kappa$.
\end{repthm}

\subsection{Higher closure operators}\label{sechigherclosureops}

\subsubsection*{Monads and the Beck Monadicity Theorem}

The aim of this subsection is to state the Beck Monadicity Theorem in a general $\infty$-cosmos as developed by Riehl and Verity in the 
special case of reflective localizations on the one hand and idempotent monads on the other hand (Corollary~\ref{corbeckmonadicity}). Thus, 
the reader familiar with the nature of the result may skip to Corollary~\ref{corbeckmonadicity} right away.

To summarize the most relevant notions from \cite{riehlverityadjmon} in this context, first we 
recall that every $\infty$-cosmos is in particular a simplicially enriched category. A homotopy-coherent adjunction in an
$\infty$-cosmos $\mathcal{V}$ is a simplicial functor $\underline{\mathrm{Adj}}\rightarrow\mathcal{V}$ from the simplicial category
$\underline{\mathrm{Adj}}$ called the \emph{generic adjunction} (\cite[Section 3]{riehlverityadjmon}). The simplicial category
$\underline{\mathrm{Adj}}$ has two objects $+$ and $-$, its hom-objects are given by the (nerve of the) category of augmented 
simplicial sets
\[\underline{\mathrm{Adj}}(+,+)\cong\underline{\mathrm{Adj}}(-,-)^{op}\cong\Delta_+,\]
and the wide subcategory
\[\underline{\mathrm{Adj}}(-,+)\cong\underline{\mathrm{Adj}}(+,-)^{op}\cong\Delta_{\infty}\]
of $\Delta$ whose operators preserve the top element. Both categories come equipped with a monoidal structure 
derived from the ordinal sum operation which induces a simplicially enriched composition operation on $\underline{\mathrm{Adj}}$ 
(\cite[Remark 3.3.8]{riehlverityadjmon}). Similarly, there is a simplicial category
$\underline{\mathrm{Mnd}}$
called the \emph{generic monad}; it has one object $+$, and the according hom-object is given by the category $\Delta_+$ again 
equipped with the join operation (\cite[Section 6]{riehlverityadjmon}). Consequently, a monad $T$ on 
an object $B$ in an $\infty$-cosmos $\mathcal{V}$ is a simplicial functor $T\colon\underline{\mathrm{Mnd}}\rightarrow\mathcal{V}$ 
which maps $+$ to $B$. Its action on hom-quasi-categories $T(+,+)\colon N(\Delta_+)\rightarrow\mathcal{V}(B,B)$, which in the 
following we denote simply by $T$ itself, respects the monoidal structures. Hence, $T_{-1}=\mathrm{id}_B$, $T_n=(T_0)^n$ for
$n\geq 0$, and the boundaries and degeneracies are mapped to according compositions of the multiplication $\mu$ and the unit $\eta$ 
associated to $T$. Indeed, its evaluation at the 1-skeleton of $N(\Delta_+)$ can be depicted by the following diagram in
$\mathcal{V}(B,B)$ (which classically depicts the associated monad resolution, see \cite[Definition 6.1.4]{riehlverityadjmon}).
\begin{align}\label{diagmonres}
\xymatrix{
\mathrm{id}_B\ar[rr]|{\eta} & & T_0\ar@<2ex>[rr]|(.4){\eta T}\ar@<-2ex>[rr]|(.4){T\eta} & & T_0\circ T_0\ar[ll]|(.6){\mu}\ar@<4ex>[rr]|(.4){\eta T T}\ar[rr]|(.4){T \eta T}\ar@<-4ex>[rr]|(.4){T T \eta} & & T_0\circ T_0\circ T_0\ar@<2ex>[ll]|(.6){\mu T}\ar@<-2ex>[ll]|(.6){T \mu} & \dots
}
\end{align}
For every homotopy-coherent adjunction 
\begin{align}\label{equlttop1}
\xymatrix{B\ar@/^/[r]^f_{\rotatebox[origin=c]{90}{$\vdash$}} & A\ar@/^/[l]^{u}}
\end{align}
in an $\infty$-cosmos $\mathcal{V}$ -- given by a simplicial functor $\underline{\mathrm{Adj}}\rightarrow\mathcal{V}$ that takes the 
according $0$-simplices in the sign-switching hom-objects to $u$ and $f$, respectively -- restriction along the canonical inclusion
$\underline{\mathrm{Mnd}}\rightarrow\underline{\mathrm{Adj}}$ gives rise to a monad $T$ on $B$ with underlying endofunctor
$T(+,+)_0=uf$. Its unit is exactly the unit of the adjunction, its multiplication is given by the composition $u\epsilon f$ where
$\epsilon$ is the counit of the adjunction.

In turn, every monad $T$ on an object $B$ in $\mathcal{V}$ gives rise to a homotopy-coherent adjunction by its associated
Eilenberg-Moore object of algebras
\begin{align}\label{equlttop2}
\xymatrix{B\ar@/^/[r]^{f^T}_{\rotatebox[origin=c]{90}{$\vdash$}} & B[T]\ar@/^/[l]^{u^T}}
\end{align}
constructed as a flexible weighted limit of $T$ in \cite[Section 6.1]{riehlverityadjmon}.
Such adjunctions derived from a monad $T$ are called \emph{monadic} (\cite[Definition 6.1.14]{riehlverityadjmon}). The monad 
associated to a monadic adjunction (\ref{equlttop2}) associated to a monad $T$ is equivalent to $T$ itself (this basically follows 
directly from \cite[Definition 6.1.14]{riehlverityadjmon}). We thus obtain an equivalence between monadic adjunctions over an object 
$B$ in an $\infty$-cosmos $\mathcal{V}$ and monads on $B$ in $\mathcal{V}$.

To recognize when a given adjunction is monadic, we may apply the popular Beck Monadicity Theorem.

\begin{theorem}[{Beck Monadicity Theorem, \cite{riehlverityadjmon}}]\label{thmbeckmonadicity}
Let $\mathcal{V}$ be an $\infty$-cosmos. For every homotopy-coherent adjunction of the form (\ref{equlttop1}) in $\mathcal{V}$ with 
associated monad $T$, the associated monadic homotopy-coherent adjunction (\ref{equlttop2}) in $\mathcal{V}$ comes together with a 
canonical comparison functor
\[\xymatrix{
A\ar@{-->}[rr]\ar@<.5ex>[dr]^u & & B[T]\ar@<.5ex>[dl]^{u^T} \\
 & B\ar@<.5ex>[ul]^f\ar@<.5ex>[ur]^{f^T} & 
}\]
that defines the underlying functor of a simplicial natural transformation between the two adjunctions (each considered as a
simplicial functor defined on $\underline{\mathrm{Adj}}$).

If $A$ admits colimits of $u$-split simplicial objects, then the comparison functor admits a left adjoint. If $u$ preserves colimits 
of $u$-split simplicial objects and reflects equivalences, then this adjunction defines an adjoint equivalence $A\simeq B[T]$.
\end{theorem}
\begin{proof}
This is \cite[Section 7]{riehlverityadjmon}. It applies to all $\infty$-cosmoi against the background of
\cite[Remark 6.1.2]{riehlverityadjmon} and \cite[Proposition 6.2.8.(i)]{riehlverityelements}.
\end{proof}

Recall that a homotopy-coherent adjunction in an $\infty$-cosmos is a reflective localization if its associated counit is an 
equivalence. Whenever the adjunction is monadic, this in turn is the case if and only if the associated monad is idempotent, as we 
spell out a quick proof of in the next lemma due to a lack of reference. Therefore, recall that a monad
$T\colon\underline{\mathrm{Mnd}}\rightarrow\mathcal{V}$ with  $T(+)=B$ is \emph{idempotent} if the multiplication
$\mu=T(d_1)\colon T_0\circ T_0\rightarrow T_0$ is an equivalence in the quasi-category $\mathcal{V}(B,B)$.  

\begin{remark}\label{remidpotlocality}
Idempotency of a monad is a locality condition in the following sense. Let $T\colon\underline{\mathrm{Mnd}}\rightarrow\mathcal{V}$ be 
a monad on an object $B$ in an $\infty$-cosmos $\mathcal{V}$. Then $T$ is idempotent if and only if the underlying cosimplicial object 
$T\colon\Delta\rightarrow\mathcal{V}(B,B)$ is homotopically constant, i.e.\ it maps every edge to an equivalence in the quasi-category 
$\mathcal{V}(B,B)$. This essentially follows from Diagram~(\ref{diagmonres}), the cosimplicial identities and the 2-out-of-3 property 
for equivalences in a quasi-category.
\end{remark}

\begin{lemma}
Let $\mathcal{V}$ be an $\infty$-cosmos. Then a monadic adjunction of the form (\ref{equlttop2}) in $\mathcal{V}$ is a reflective 
localization if and only if the monad $T$ is idempotent.
\end{lemma}
\begin{proof}
The multiplication 1-cell $\mu\colon T_0\circ T_0\rightarrow T_0$ of the monad $T$ in $\mathcal{V}(B,B)$ is exactly 
the composition $u\epsilon f\colon fufu\rightarrow fu$ for $\epsilon$ the counit of its associated monadic adjunction $f\adj u$. 
It follows that whenever the counit $\epsilon$ is a natural equivalence, then so is the multiplication $\mu$. 
Vice versa, if $\mu = u\epsilon f$ is an equivalence, then so is $\epsilon f$, because the right adjoint $u$ is conservative whenever 
the adjunction is monadic (\cite[Corollary 6.2.3]{riehlverityadjmon}). 
That means the counit is an equivalence on every free algebra. But every algebra $A$ in $B[T]$ is the colimit of a simplicial object whose 
vertices $A_n$ are free algebras (\cite[Theorem 6.3.17]{riehlverityadjmon}) and which is essentially constant in virtue of idempotency of the 
monad $T$ (as noted in Remark~\ref{remidpotlocality}). It follows that every algebra $A\simeq A_0$ is free.
\end{proof}

Thus, the equivalence between monadic adjunctions over an object $B\in\mathcal{V}$ and monads on $B$ restricts to an equivalence 
between monadic reflective localizations of $B$ and idempotent monads on $B$. 

\begin{corollary}
In an $\infty$-cosmos, every homotopy-coherent reflective localization is monadic.
\end{corollary}
\begin{proof}
Given a reflective localization of the form (\ref{equlttop1}) in an $\infty$-cosmos $\mathcal{V}$, by the Monadicity Theorem we only 
have to show that $A$ admits colimits of $u$-split simplicial objects, that $u$ preserves colimits of $u$-split 
simplicial objects, and that $u$ reflects equivalences. But $A$ admits all colimits that exist in $B$, and $B$ always admits colimits of 
split augmented simplicial objects (given by the augmentation itself, see \cite[Proposition 2.3.15]{riehlverityelements}). A $u$-split 
augmented simplicial object in $A$ in fact yields a split \emph{augmented} simplicial object in $A$ by application of the reflector to the 
corresponding split augmented simplicial object in $B$. But colimits of such are absolute by the same Proposition in Riehl and Verity's book, 
and so it follows that $u$ preserves colimits of $u$-split simplicial objects. 
It furthermore reflects equivalences, because $f$ preserves equivalences and the counit $\epsilon\colon fu\rightarrow \mathrm{id}_A$ is 
a natural equivalence itself.
\end{proof}

We summarize this general discussion of monadicity in $\infty$-cosmoses by the following corollary which is exactly what we want to 
use in our case of interest.

\begin{corollary}\label{corbeckmonadicity}
Given an $\infty$-cosmos $\mathcal{V}$ together with an object $B\in\mathcal{V}$, the Beck Monadicity Theorem induces a bijection between 
equivalence classes of reflective localizations of $B$ in $\mathcal{V}$ and equivalence classes of idempotent monads $T$ on $B$ in
$\mathcal{V}$.
\end{corollary}\qed

\subsubsection*{Application to the $\infty$-cosmos of fibrations over $\mathcal{B}$}

The Monadicity Theorem applied to the $\infty$-cosmoses $(\mathbf{Cat}_{\infty})_{/\mathcal{B}}$ and $\mathbf{Cart}(\mathcal{B})$ 
yields a description of all notions of reflective localizations featured in Proposition~\ref{propfibmods} and 
Corollary~\ref{corlexfibmods} in terms of suitable classes of monads.

\begin{definition}\label{defcompmonads}
Let $\mathcal{B}$ be a quasi-category and
\[\xymatrix{
\mathcal{F}\ar@{^(->}[r]\ar@{->>}@/_/[dr] & \mathcal{B}^{\Delta^1}\ar@{->>}[d]^{t_{\mathcal{B}}} \\
 & \mathcal{B}
}\]
be a full and replete subfibration whose objects are closed under composition.
Say a fibered pointed endofunctor $(T,\eta)$ on $\mathcal{F}$ is \emph{compositional} if for all concatenable arrows $f,g\in\mathcal{F}$ 
there is an equivalence
\[T(g\circ f)\simeq T(g)\circ T(\eta_g\circ f)\]
in $\mathcal{B}_{/t(g)}$ (for any choice of respective compositions). We say that a monad $T$ on $\mathcal{F}$ in
$(\mathbf{Cat}_{\infty})_{/\mathcal{B}}$ is compositional if its underlying pointed endofunctor is so.
\end{definition}

The reason to relativize Definition~\ref{defcompmonads} to full and replete subfibrations $\mathcal{F}$ of $t_{\mathcal{B}}$ will become 
evident in Section~\ref{secsubhlttops}. 

\begin{proposition}\label{propcompunitmnds}
Let $\mathcal{B}$ be a quasi-category. Let $\mathcal{F}\subseteq\mathcal{B}^{\Delta^1}$ be a full and replete subfibration which is closed 
under compositions and contains all identities. Then there is a bijection between
\begin{enumerate}
\item 
\begin{enumerate}
\item fibered reflective and replete subcategories $\mathcal{E}\hookrightarrow\mathcal{F}$ such that $\mathcal{E}$ is closed under 
compositions, and 
\item compositional 
fibered idempotent monads $T$ on $\mathcal{F}$.
\end{enumerate}
\item Suppose furthermore $\mathcal{B}$ has pullbacks and $\mathcal{F}\subseteq\mathcal{B}^{\Delta^1}$ is a full cartesian subfibration. Then 
there is a bijection between
\begin{enumerate}
\item cartesian reflective and replete subcategories $\mathcal{E}\hookrightarrow\mathcal{F}$ such that $\mathcal{E}$ is closed under 
compositions, and 
\item compositional 
cartesian idempotent monads $T$ on $\mathcal{F}$.
\end{enumerate}
\end{enumerate}
\end{proposition}

\begin{proof}
In virtue of Corollary~\ref{corbeckmonadicity}, we only have to show that a fibered (cartesian) reflective subcategory
$\mathcal{E}\subseteq\mathcal{B}^{\Delta^1}$ 
is closed under composition if and only if its associated idempotent monad on $t_{\mathcal{B}}$ is 
compositional. This however is a fairly straight-forward computation:
if $T$ is compositional, let $f$ and $g$ be concatenable and contained in $\mathcal{E}$. Consider the diagram
\begin{align}\label{diagpropmodop}
\begin{gathered}
\xymatrix{
A\ar[d]_f\ar[r]^(.35){\eta_{\eta_g\circ f}} & T_{T_C(B)}(A)\ar[d]^{T(\eta_g\circ f)}  \\
B\ar[d]_g\ar[r]^{\eta_g} & T_C(B)\ar[dl]^{T_C(g)} \\
C & 
}
\end{gathered}
\end{align}
in $\mathcal{B}$. Since the vertical composition $T_C(g)\circ T(\eta_g\circ f)$ is equivalent to $T(g\circ f)$ over $C$ by assumption, 
it is contained in $\mathcal{E}$. Since $\mathcal{E}$ is replete, to show that $g\circ f$ is again contained in $\mathcal{E}$, it thus 
suffices to show that the unit $\eta_{\eta_g\circ f}$ is an equivalence in $\mathcal{B}^{\Delta^1}$. But the unit
$\eta_g$ is an equivalence, because $g$ is contained in $\mathcal{E}$. Furthermore, since $\mathcal{E}$ is replete, it is always closed under 
postcomposition with equivalences (via the equivalence $f\xrightarrow{\sim}\eta_g\circ f$ in $\mathcal{B}^{\Delta^1}$ with boundaries
$(1_A,\eta_g)$ in this case).
It follows that the unit $\eta_{\eta_g\circ f}$ is an equivalence.

Vice versa, suppose $\mathcal{E}$ is closed under compositions, suppose 
$f,g\in\mathcal{F}$ are concatenable, and consider Diagram~(\ref{diagpropmodop}). Since the composition $T_C(g)\circ T(\eta_g\circ f)$ 
is contained in $\mathcal{E}(C)$ by assumption, it suffices to show that the unit $\eta_{\eta_g\circ f}$ as a map in $\mathcal{F}(C)$ 
is an $\mathcal{E}(C)$-local equivalence. But it is an $\mathcal{E}(T_C(B))$-local equivalence by construction. Furthermore, as
$\mathcal{F}$ is closed under compositions, it admits a cocartesian arrow over every arrow $f$ in the base $\mathcal{B}$ that is 
itself contained in $\mathcal{F}$. The left adjoint to the inclusion $\mathcal{E}\hookrightarrow\mathcal{F}$ preserves the cocartesian 
squares in $\mathcal{F}$ by the proof of Lemma~\ref{lemmarefllocbifib} (applied to the full subcategory $\mathcal{F}$ rather than
$\mathcal{B}^{\Delta^1}$ itself). Via Remark~\ref{remcocartaction}, it follows that the cocartesian action
$\Sigma_{T_C(g)}\colon\mathcal{F}(T_C(B))\rightarrow\mathcal{F}(C)$ takes $\mathcal{E}(T_C(B))$-local equivalences to $\mathcal{E}(C)$-local 
equivalences. In particular, the unit $\eta_{\eta_g\circ f}$ is an $\mathcal{E}(C)$-local equivalence.

\end{proof}

\begin{corollary}\label{corcompunitmnds}
Let $\mathcal{B}$ be a quasi-category.
\begin{enumerate}
\item Then there is a bijection between factorization systems on $\mathcal{B}$ and 
fibered, compositional and idempotent monads on $t_{\mathcal{B}}$. 
\item Suppose $\mathcal{B}$ has pullbacks. Then there is a bijection between modalities on $\mathcal{B}$ and 
idempotent and compositional monads on $t_{\mathcal{B}}$ in $\mathbf{Cart}(\mathcal{B})$. 
\end{enumerate}
\end{corollary}
\begin{proof}
Against the background of Proposition~\ref{propfibmods} and Corollary~\ref{corlexfibmods}.1, respectively, this is immediate by 
Proposition~\ref{propcompunitmnds} applied to $\mathcal{F}=\mathcal{B}^{\Delta^1}$ over $\mathcal{B}$.
\end{proof}

To incorporate the additional structures featured in Corollary~\ref{corlexfibmods}, we recall that there is a formal definition of 
limits and colimits of simplicial diagrams \emph{in} a cartesian fibration $p\colon\mathcal{E}\twoheadrightarrow\mathcal{B}$ as an 
object of $\mathbf{Cart}(\mathcal{B})$ (\cite[Chapters 2 and 4]{riehlverityelements}). In that sense, a cartesian fibration
$p\colon\mathcal{E}\twoheadrightarrow\mathcal{B}$ has finite (simplicially indexed) limits in $\mathbf{Cart}(\mathcal{B})$ if for 
every finite simplicial set $K$, the diagonal $\Delta\colon p\rightarrow p^K$ of cotensors over $\mathcal{B}$ has a cartesian right 
adjoint. 

\begin{proposition}\label{propfiblex}
Let $\mathcal{B}$ be an $\infty$-category.
\begin{enumerate}
\item Suppose $p\colon\mathcal{E}\twoheadrightarrow\mathcal{B}$ is a cartesian fibration.
Then the fibration $p$ has finite limits in $\mathbf{Cart}(\mathcal{B})$ if and only if its Straightening
$\mathrm{St}(p)\colon\mathcal{B}^{op}\rightarrow\mathrm{Cat}_{\infty}$ factors through the $\infty$-category 
$\mathrm{Cat}_{\infty}^{\mathrm{lex}}$ of left exact $\infty$-categories and left exact functors between them.
\item  Suppose $p\colon\mathcal{E}\twoheadrightarrow\mathcal{B}$ and $q\colon\mathcal{F}\twoheadrightarrow\mathcal{B}$ are cartesian 
fibrations with finite limits in $\mathbf{Cart}(\mathcal{B})$. Then a functor $F\colon q\rightarrow p$ in $\mathbf{Cart}(\mathcal{B})$ 
preserves finite limits if and only if the straightened natural transformation
$\mathrm{St}(F)\colon\mathrm{St}(q)\rightarrow\mathrm{St}(p)$ is contained in
$\mathrm{Fun}(\mathcal{B}^{op},\mathrm{Cat}_{\infty}^{\mathrm{lex}})$. 
\end{enumerate}
\end{proposition}

\begin{proof}
We show Part 1. For one direction, suppose $p\colon\mathcal{E}\twoheadrightarrow\mathcal{B}$ has finite limits. Let $K$ be a finite 
simplicial set, and let
\[(\Delta,\mathrm{lim},\epsilon)\colon\underline{\mathrm{Adj}}\rightarrow\mathbf{Cart}(\mathcal{B})\]
be the (homotopically uniquely induced) adjunction with left adjoint $\Delta\colon F\rightarrow F^K$. 
Then for any object $B\in\mathcal{B}$, the pullback
$\{B\}^{\ast}\colon\mathbf{Cart}(\mathcal{B})\rightarrow\mathbf{Cart}(\Delta^0)\xrightarrow{\cong}\mathbf{Cat}_{\infty}$ is induced 
by an enriched right Quillen functor of quasi-categorically enriched model categories and hence cosmological
(the first one is fairly straight-forward and the second one follows from \cite[Proposition 3.1.5.3]{luriehtt}). Push-forward of the 
simplicial functor $(\Delta,\mathrm{lim},\epsilon)$ along $\{B\}^{\ast}$ hence yields a right adjoint $\mathrm{lim}(B)$ to the 
diagonal 
\[\xymatrix{
 & \mathrm{St}(p)(B)^K\\
\mathrm{St}(p)(B)\ar[r]_(.45){\Delta(B)}\ar[ur]^{\Delta_{\mathrm{St}(p)(B)}} & \mathrm{St}(p)^K(B)\ar[u]_{\simeq}.
}\]
This means that the $\infty$-categories $\mathrm{St}(p)(B)$ are left-exact (as left-exactness is stable under equivalence). Furthermore, as 
the right adjoint $\mathrm{lim}\colon p^K\rightarrow p$ is a cartesian functor, for every arrow $f\colon A\rightarrow B$ in $\mathcal{B}$ the 
square
\[\xymatrix{
\mathrm{St}(p)(A)\ar@/^/[r]^{\Delta(A)}_{}& \mathrm{St}(p)(A)^K\ar@/^/[l]^{\mathrm{lim(A)}}  \\
\mathrm{St}(p)(B)\ar@/^/[r]^{\Delta(B)}_{}\ar[u]^{f^{\ast}} & \mathrm{St}(p)(B)^K\ar[u]_{f^{\ast}}\ar@/^/[l]^{\mathrm{lim(B)}}
}\]
commutes up to equivalence in both directions. In other words, the functor $f^{\ast}$ is left-exact as well.

For the other direction, let $\mathbf{Cat}_{\infty}^{\mathrm{lex}}\subset\mathbf{Cat}_{\infty}$ be the simplicial subcategory spanned 
by the left exact quasi-categories and left exact functors between them. Thus,
$\mathrm{Cat}_{\infty}^{\mathrm{lex}}=N_{\Delta}(\mathbf{Cat}_{\infty}^{\mathrm{lex}})$ is its underlying quasi-category. Assume the 
functor $\mathrm{St}(p)\colon\mathcal{B}^{op}\rightarrow\mathrm{Cat}_{\infty}$ factors through the canonical inclusion
$\iota\colon\mathrm{Cat}_{\infty}^{\mathrm{lex}}\hookrightarrow\mathrm{Cat}_{\infty}$. The resulting pullback functor
\[\mathrm{St}(p)^{\ast}\colon\mathbf{Cart}((\mathrm{Cat}_{\infty}^{\mathrm{lex}})^{op})\rightarrow\mathbf{Cart}(\mathcal{B})\]
is again induced by an enriched right Quillen functor of quasi-categorically enriched model categories, and hence is a cosmological functor 
of associated $\infty$-cosmoses itself. Let $\pi\colon\mathrm{Dat}_{\infty}^{op}\twoheadrightarrow\mathrm{Cat}_{\infty}^{op}$ be the 
universal cartesian fibration (\cite[Section 3.3.2]{luriehtt}). Then the pullback functor $\mathrm{St}(p)^{\ast}$ maps the cartesian 
fibration $\iota^{\ast}\pi\in\mathbf{Cart}((\mathrm{Cat}_{\infty}^{\mathrm{lex}})^{op})$ to a fibration equivalent to $p$; it hence 
suffices to show that the the object $\pi_{\mathrm{lex}}:=\iota^{\ast}\pi$ has finite limits as an object of the $\infty$-cosmos
$\mathbf{Cart}((\mathrm{Cat}_{\infty}^{\mathrm{lex}})^{op})$ by \cite[Proposition 10.1.4]{riehlverityelements}.

Therefore, let $K$ be a finite simplicial set, and let $\Delta^0\ast K$ be its (fat) join defined as the pushout
$\Delta^0\cup_{(\{1\}\times K)}(\Delta^1\times K)$. For a quasi-category $\mathcal{C}\in\mathbf{Cat}_{\infty}$, let
\[j_{\mathcal{C}}\colon(\mathcal{C}^{\Delta^0\ast K})_{\mathrm{lim}}\subset\mathcal{C}^{\Delta^0\ast K}\]
be the full subcategory spanned by the limit cones in $\mathcal{C}$. This inclusion gives rise to a simplicial natural transformation
$j\colon(\cdot)^{\Delta^0\ast K}_{\mathrm{lim}}\rightarrow(\cdot)^{\Delta^0\ast K}$ of functors from the simplicially enriched 
subcategory $\mathbf{Cat}_{\infty}^{\mathrm{lex}}$ to $\mathbf{Cat}_{\infty}$, simply by the fact that the functors in
$\mathbf{Cat}_{\infty}^{\mathrm{lex}}$ are limit preserving and hence factor pointwise through the inclusions $j_{\mathcal{C}}$. By composition with the simplicial restriction functor
$\mathrm{res}_K\colon(\cdot)^{\Delta^0\ast K}\rightarrow(\cdot)^{K}$, we obtain a natural transformation
\[\lambda\colon(\cdot)^{\Delta^0\ast K}_{\mathrm{lim}}\rightarrow(\cdot)^K\]
in the category $\mathrm{Fun}(\mathbf{Cat}_{\infty}^{\mathrm{lex}},\mathbf{Cat}_{\infty})$ of simplicially enriched functors.
It is pointwise a trivial fibration of quasi-categories by \cite[Proposition 4.3.2.15]{luriehtt}.
It follows that the natural transformation $\lambda$ is a trivial fibration of fibrant objects in the model category of simplicially 
enriched functors $\mathrm{Fun}(\mathbf{Cat}_{\infty}^{\mathrm{lex}},(\mathbf{S},\mathrm{QCat}))_{\mathrm{proj}}$.
Since Unstraightening is a Quillen right equivalence, and as such preserves both trivial fibrations and limits strictly as well as simplicial 
cotensors up to equivalence, we obtain a diagram as follows.
\[\xymatrix{
(\pi_{\mathrm{lex}})^{\Delta^0\ast K}_{\mathrm{lim}}\ar@{^(->}[dr]\ar@{->>}@/_1pc/[ddr]^(.6){\sim}\ar[rr]^{\simeq} & & \mathrm{Un}((\cdot)^{\Delta^0\ast K}_{\mathrm{lim}})\ar@{^(->}[dr]^{\mathrm{Un}(j)}\ar@{->>}@/_1pc/[ddr]|(.4)\hole^(.6){\sim}_(.6){\mathrm{Un}(\lambda)} & \\
& (\pi_{\mathrm{lex}})^{\Delta^0\ast K}\ar[rr]_{\simeq}\ar@{->>}[d] & & \mathrm{Un}((\cdot)^{\Delta^0\ast K})\ar@{->>}[d]\\
& (\pi_{\mathrm{lex}})^{K}\ar[rr]_{\simeq} & & \mathrm{Un}((\cdot)^{K})
}\]
Here, the vertex on the top left shall denote the pullback of the top span by definition. As the horizontal equivalences are chosen 
individually, the front square a priori only commutes up to equivalence. This homotopy-commutative square however may be replaced by a 
strictly commutative square in $\mathbf{Cart}((\mathrm{Cat}_{\infty}^{\mathrm{lex}})^{op})$ e.g.\ by hand via the homotopy lifting 
property (as the right hand side vertical map is a fibration). We may thus choose a section
$\mathrm{lim}_+\colon(\pi_{\mathrm{lex}})^K\rightarrow(\pi_{\mathrm{lex}})^{\Delta^0\ast K}_{\mathrm{lim}}$, and consider the composite 
map
\begin{align}\label{diagpropfiblex0}
\epsilon_K\colon(\pi_{\mathrm{lex}})^K\xrightarrow{\mathrm{lim}_+} (\pi_{\mathrm{lex}})^{\Delta^0\ast K}_{\mathrm{lim}}\hookrightarrow(\pi_{\mathrm{lex}})^{\Delta^0\ast K}\xrightarrow{} (\pi_{\mathrm{lex}})^{\Delta^1\times K}\xrightarrow{\simeq}((\pi_{\mathrm{lex}})^K)^{\Delta^1}
\end{align}
in $\mathbf{Cart}((\mathrm{Cat}_{\infty}^{\mathrm{lex}})^{op})$. Here, the map
$(\pi_{\mathrm{lex}})^{\Delta^0\ast K}\rightarrow(\pi_{\mathrm{lex}})^{\Delta^1\times K}$ is induced by the coprojection
$\Delta^1\times K\rightarrow\Delta^0\ast K$. This constitutes a 2-cell of the form
\[\xymatrix{
 & & \pi_{\mathrm{lex}}\ar[dd]^{\Delta} \\
 & (\pi_{\mathrm{lex}})^{\Delta^0\ast K}\ar[ur]^{\mathrm{res}_{\Delta^0}}\ar@{}[d]^{\Downarrow\epsilon_K} & \\
(\pi_{\mathrm{lex}})^K\ar@{=}[rr]\ar[ur]^{\mathrm{lim}_+}\ar@/^3pc/[uurr]^{:=\mathrm{lim}_K} & & (\pi_{\mathrm{lex}})^K
}\]
in $\mathbf{Cart}((\mathrm{Cat}_{\infty}^{\mathrm{lex}})^{op})$. To show that the triple $(\Delta,\mathrm{lim}_K,\epsilon_K)$ indeed forms an 
adjunction, by \cite[Theorem 3.5.8]{riehlverityelements} it suffices to show that the associated map
\begin{align}\label{diagpropfiblex1}
\begin{gathered}
\xymatrix{
\pi_{\mathrm{lex}}\downarrow\mathrm{lim}_K\ar[rr]^{\ulcorner\epsilon_K\circ\Delta\urcorner}\ar@{->>}@/_/[dr] & & \Delta\downarrow(\pi_{\mathrm{lex}})^K\ar@{->>}@/^/[dl] \\
 & \pi_{\mathrm{lex}}\times(\pi_{\mathrm{lex}})^K & 
}
\end{gathered}
\end{align}
over $\pi_{\mathrm{lex}}\times(\pi_{\mathrm{lex}})^K$ in $\mathbf{Cart}((\mathrm{Cat}_{\infty}^{\mathrm{lex}})^{op})$ is an equivalence.
Therefore it suffices to show that it is so pointwise over every $\mathcal{C}\in(\mathrm{Cat}_{\infty}^{\mathrm{lex}})^{op}$ (e.g.\ by \cite[Proposition 3.3.1.5]{luriehtt}). But pullback along an object
$\{\mathcal{C}\}^{\ast}\colon\Delta^0\rightarrow(\mathrm{Cat}_{\infty}^{\mathrm{lex}})^{op}$ is 
a cosmological functor and hence takes the triangle (\ref{diagpropfiblex1}) to the triangle
\[\xymatrix{
\mathcal{C}\downarrow\mathrm{lim}_K(\mathcal{C})\ar[rr]^{\ulcorner\epsilon_K(\mathcal{C})\circ\Delta(\mathcal{C})\urcorner}\ar@{->>}@/_/[dr] & & \Delta(\mathcal{C})\downarrow\mathcal{C}^K\ar@{->>}@/^/[dl] \\
 & \mathcal{C}\times\mathcal{C}^K. & 
}
\]
This is an equivalence precisely because the triple $(\Delta_{\mathcal{C}},\mathrm{lim}_K(\mathcal{C}),\epsilon_K(\mathcal{C}))$ is 
(equivalent to) an adjunction in $\mathbf{Cat}_{\infty}^{\mathrm{lex}}\subset\mathbf{Cat}_{\infty}$ as $\mathcal{C}$ is left exact. This 
finishes the proof of Part 1.

For Part 2, let $K$ again be a finite simplicial set. We want to show that the square
\begin{align}\label{diagpropfiblex2}
\begin{gathered}
\xymatrix{
q\ar[rr]^F\ar@/^/[d]^{\Delta} & & p\ar@/^/[d]^{\Delta} \\
q^K\ar[rr]_{F^K}\ar@/^/[u]^{\mathrm{lim}_K} & & p^K\ar@/^/[u]^{\mathrm{lim}_K}
}
\end{gathered}
\end{align}
commutes up to equivalence (in both directions) if and only if for all $B\in\mathcal{B}$ the square 
\begin{align}\label{diagpropfiblex3}
\begin{gathered}
\xymatrix{
q(B)\ar[rr]^{F(B)}\ar@/^/[d]^{\Delta(B)} & & p(B)\ar@/^/[d]^{\Delta(B)} \\
q(B)^K\ar[rr]_{F(B)^K}\ar@/^/[u]^{\mathrm{lim}_K(B)} & & p(B)^K\ar@/^/[u]^{\mathrm{lim}_K(B)}
}
\end{gathered}
\end{align}
commutes up to equivalence (in both directions). Again by applying pointwise evaluation, homotopy commutativity of 
(\ref{diagpropfiblex2}) implies homotopy commutativity of (\ref{diagpropfiblex3}). For the converse direction, the construction of the 
map $\epsilon_K$ in (\ref{diagpropfiblex0}) gives rise to a sequence of commutative squares in $\mathbf{Cart}(\mathcal{B})$ as 
follows.
\[\xymatrix{
p^K& (p^{\Delta^0\ast K})_{\mathrm{lim}}\ar[r]^{\mathrm{Un}(j)}\ar@{->>}[l]^(.6){\sim}_(.6){\mathrm{Un}(\lambda(p))} & p^{\Delta^0\ast K}\ar[r] & p^{\Delta^1\times K}\ar[r]^{\simeq} & (p^K)^{\Delta^1} \\
q^K\ar[u]^{F^K} & (q^{\Delta^0\ast K})_{\mathrm{lim}}\ar[r]^{\mathrm{Un}(j)}\ar@{->>}[l]^(.6){\sim}_(.6){\mathrm{Un}(\lambda(q))}\ar[u]^{(F^{\Delta^0\ast K})_{\mathrm{lim}}}& q^{\Delta^0\ast K}\ar[r]\ar[u]^{F^{\Delta^0\ast K}} & q^{\Delta^1\times K}\ar[r]^{\simeq}\ar[u]^{F^{\Delta^1\times K}} & (q^K)^{\Delta^1}\ar[u]^{(F^{K})^{\Delta^1}}
}\]
The natural transformation $(F^{\Delta^0\ast K})_{\mathrm{lim}}$ exists precisely because $\mathrm{St}(F)$ is assumed to be pointwise 
finite limit preserving (without loss of generality we may assume $F=\mathrm{Un}(\mathrm{St}(F))$ since left exactness is
equivalence-invariant). Then any section $\mathrm{lim}_+$ to $\mathrm{Un}(\lambda)$ as chosen in (\ref{diagpropfiblex0}) gives rise to sections
\[\xymatrix{
p^K\ar@/^1pc/@{-->}[r]^{\mathrm{lim}_+(p)} & (p^{\Delta^0\ast K})_{\mathrm{lim}}\ar@{->>}[l]_(.6){\sim}^(.6){\mathrm{Un}(\lambda(p))} \\
q^K\ar@/_1pc/@{-->}[r]_{\mathrm{lim}_+(q)}\ar[u]^{F^K} & (q^{\Delta^0\ast K})_{\mathrm{lim}}\ar@{->>}[l]^(.6){\sim}_(.6){\mathrm{Un}(\lambda(q))}\ar[u]_{(F^{\Delta^0\ast K})_{\mathrm{lim}}}.
}\]
The resulting square given by the pair $(\mathrm{lim}_+(p),\mathrm{lim}_+(q))$ commutes up to homotopy, because $\mathrm{Un}(\lambda(p))$ is 
a trivial fibration and so lifts against the (cofibrant) object $q^K$ along $F^K$ are homotopically unique. By definition of
$\mathrm{lim}_K$ from $\mathrm{lim}_+$, it readily follows that the square (\ref{diagpropfiblex2}) commutes up to homotopy as well.
\end{proof}

We say that a cartesian fibration $p\colon\mathcal{E}\twoheadrightarrow\mathcal{B}$ over an $\infty$-category $\mathcal{B}$ 
is \emph{left exact} whenever any of the conditions in Proposition~\ref{propfiblex} is satisfied. 

\begin{remark}
Let $\pi\colon\mathrm{Dat}_{\infty}\twoheadrightarrow\mathrm{Cat}_{\infty}$ denote the universal small cocartesian fibration 
(\cite[Section 3.3.3]{luriehtt}). Then it follows from Proposition~\ref{propfiblex} that the pullback
\[\xymatrix{
\mathrm{Dat}_{\infty}^{\mathrm{lex}}\ar@{->>}[d]_{\pi_{\mathrm{lex}}}\ar[r]\ar@{}[dr]|(.3){\pbs} & \mathrm{Dat}_{\infty}\ar@{->>}[d]^{\pi} \\
\mathrm{Cat}_{\infty}^{\mathrm{lex}}\ar@{^(->}[r]_{\iota} & \mathrm{Cat}_{\infty}
}\]
along the canonical inclusion is the universal small left exact cocartesian fibration (here note that the $\infty$-category of cocartesian 
fibrations over $\mathrm{Cat}_{\infty}^{\mathrm{lex}}$ is equivalent to the $\infty$-category of cartesian fibrations over
$(\mathrm{Cat}_{\infty}^{\mathrm{lex}})^{op}$). The fibration $\pi_{\mathrm{lex}}$ is univalent by \cite[Theorem 6.29]{rasekh}, since the 
inclusion $\iota$ is $(-1)$-truncated. In fact, as the inclusion $\iota$ has a left adjoint $F$ give by the free finite limit 
completion of a quasi-category, the cartesian fibration $\pi_{\mathrm{lex}}^{op}$ is represented (via externalization) by the 
complete Segal object $F_{\bullet}(\Delta^{\bullet})$ in $(\mathrm{Cat}_{\infty}^{\mathrm{lex}})^{op}$ for
$\Delta^{\bullet}\colon\Delta\rightarrow\mathrm{Cat}_{\infty}$ the canonical inclusion by \cite[Section 4]{rs_comp}.
\end{remark}

\begin{remark}\label{rempropfiblexgen}
With an according transcription of the proof of \cite[B1.4, Lemma 1.4.1]{elephant} one can furthermore show that a fibration 
$p\colon\mathcal{E}\twoheadrightarrow\mathcal{B}$ over a left exact base $\mathcal{B}$ is left exact itself if and only if the
$\infty$-category $\mathcal{E}$ has finite limits and $p$ is finite limit preserving. Both the proof of Proposition~\ref{propfiblex} 
and this remark generalize to all sorts of shapes of limits, and dually for all sorts of shapes of colimits, too.
\end{remark}

\begin{remark}\label{rempropfiblexgen2}
Proposition~\ref{propfiblex} is an instance of an equivalence between certain cartesian fibered adjunctions and according indexed
($\mathrm{Cat}_{\infty}$-valued) adjunctions over a quasi-category $\mathcal{B}$. Despite its length, the proof is in essence just an
application of the Straightening construction together with the fact that formal indexed adjunctions are  corepresentable and hence computed 
pointwise:
\[\mathrm{Fun}(\mathrm{Adj},\mathrm{Fun}(\mathfrak{C}(\mathcal{B}),\mathbf{S}))\cong\mathrm{Fun}(\mathrm{Adj}\times \mathfrak{C}(\mathcal{B}),\mathbf{S}))\cong\mathrm{Fun}(\mathfrak{C}(\mathcal{B}),\mathrm{Fun}(\mathrm{Adj},\mathbf{S})).\]
\end{remark}

\begin{corollary}\label{corbasecompleteness}
Whenever $\mathcal{B}$ is a quasi-category with pullbacks, the target fibration
$t\colon\mathcal{B}^{\Delta^1}\twoheadrightarrow\mathcal{B}$ -- as well as every full cartesian subfibration
$\mathcal{F}\subseteq\mathcal{B}^{\Delta^1}$ which is fiberwise closed under finite limits -- is left exact in
$\mathbf{Cart}(\mathcal{B})$, and the inclusion $\mathcal{F}\subseteq\mathcal{B}^{\Delta^1}$ preserves those limits.

Whenever $\mathcal{B}$ has furthermore universal $\mathcal{K}$-shaped colimits for some class
$\mathcal{K}$ of simplicial sets, then the target fibration $t\colon\mathcal{B}^{\Delta^1}\twoheadrightarrow\mathcal{B}$ -- as well as 
every full subfibration $\mathcal{F}\subseteq\mathcal{B}^{\Delta^1}$ which is fiberwise closed under $\mathcal{K}$-shaped colimits -- 
has all $\mathcal{K}$-shaped colimits in the $\infty$-cosmos $\mathbf{Cart}(\mathcal{B})$. Furthermore, the inclusion
$\mathcal{F}\subseteq\mathcal{B}^{\Delta^1}$ preserves those colimits.
\end{corollary}
\begin{proof}
Since $\mathcal{B}$ has pullbacks, all slices of $\mathcal{B}$ have pullbacks as well as a terminal object, and hence are left exact.
The transition maps are limit preserving as they are given by pullback action in $\mathcal{B}$, and limits commute among each other.
For the second statement, we just note that universality of $K$-shaped colimits implies that the canonical indexing of $\mathcal{B}$ 
factors through the $\infty$-category of $\mathcal{K}$-cocomplete $\infty$-categories and $\mathcal{K}$-cocontinuous functors. Hence, both 
statements follow directly from Proposition~\ref{propfiblex} (and Remark~\ref{rempropfiblexgen}).
\end{proof}

Given these formal characterizations, we can extend Proposition~\ref{propcompunitmnds} as follows.

\begin{corollary}\label{cormodoplex}
Let $\mathcal{B}$ be a quasi-category with pullbacks.  Let $\mathcal{F}\subseteq\mathcal{B}^{\Delta^1}$ be a full and replete cartesian
subfibration which is closed under compositions and fiberwise closed under finite limits. Then there is a bijection between 
\begin{enumerate}
\item cartesian reflective and replete subcategories $\iota\colon\mathcal{E}\hookrightarrow\mathcal{F}$ with left exact reflector $\rho$ such 
that $\mathcal{E}$ is closed under compositions, and
\item cartesian left exact compositional idempotent monads $T$ on $\mathcal{F}$ in $\mathbf{Cart}(\mathcal{B})$.
\end{enumerate}
\end{corollary}

\begin{proof}
By Proposition~\ref{propfiblex}, the cartesian reflector in 1.\ is left exact (as a functor of quasi-categories and thus fiberwise 
left exact) if and only if it is formally left exact as a functor in $\mathbf{Cart}(\mathcal{B})$. In virtue of 
Proposition~\ref{propcompunitmnds}, we are hence only left to show that formal left exactness of the left adjoint $\rho$ is equivalent to 
formal left exactness of the underlying endofunctor $T\simeq \iota\rho$ of 
the associated idempotent monad in $\mathbf{Cart}(\mathcal{B})$. One direction follows immediately from the fact that the right adjoint is 
left exact in $\mathbf{Cart}(\mathcal{B})$. The other direction follows from the fact that $\iota$ is fully faithful and hence reflects 
limits in $\mathbf{Cart}(\mathcal{B})$ (\cite[Proposition 2.4.7]{riehlverityelements}).
\end{proof}

\begin{corollary}\label{cormodopacc}
Suppose $\mathcal{B}$ is a quasi-category with pullbacks and universal $K$-shaped colimits for some class $K$ of simplicial sets, 
and suppose $\mathcal{F}\subseteq\mathcal{B}^{\Delta^1}$ is a full and replete cartesian subfibration whose domain
$\mathcal{F}\subseteq\mathcal{B}^{\Delta^1}$ is fiberwise closed under $K$-shaped colimits (and finite limits). Then there is a bijection 
between
\begin{enumerate}
\item cartesian (left exact) full reflective and replete subcategories $\mathcal{E}\hookrightarrow\mathcal{F}$ which are fiberwise closed 
under $K$-colimits and closed under compositions, and
\item cartesian (left exact) compositional and $K$-colimits preserving idempotent monads $T$ on $\mathcal{F}$  in
$\mathbf{Cart}(\mathcal{B})$.
\end{enumerate}
\end{corollary}
\begin{proof}
Immediate from Corollary~\ref{corbasecompleteness}, Corollary~\ref{cormodoplex} (and \cite[Proposition 2.4.7]{riehlverityelements} applied to 
$K$-shaped colimits).
\end{proof}

We will say that a functor $f\colon A\rightarrow B$ in an $\infty$-cosmos $\mathcal{V}$ is \emph{$\kappa$-accessible} for some regular 
cardinal $\kappa$ whenever it preserves all (simplicially indexed) $\kappa$-filtered colimits that exist in $A$. It is \emph{accessible} 
whenever it is $\kappa$-accessible for some regular cardinal $\kappa$.
In particular, whenever $A$ has all $\kappa$-filtered colimits, the map $f\colon A\rightarrow B$ is
$\kappa$-accessible if and only if $f\circ \mathrm{colim}_A\simeq\mathrm{colim}_B\circ f^K$ for all $\kappa$-filtered simplicial sets $K$.

\begin{definition}
Let $\mathcal{B}$ be a quasi-category with pullbacks. An \emph{(accessible) higher closure operator} $T$ on $\mathcal{B}$ is a 
left exact, accessible and compositional idempotent monad on $t_{\mathcal{B}}\colon\mathcal{B}^{\Delta^1}\twoheadrightarrow\mathcal{B}$ in
$\mathbf{Cart}(\mathcal{B})$.
\end{definition}

In the following, we will implicitly assume accessibility whenever we refer to higher closure operators.

\begin{theorem}\label{thmlocmodopacc}
Suppose $\mathcal{B}$ is a presentable quasi-category with universal colimits. Then there is a bijection between
\begin{enumerate}
\item left exact modalities of small generation on $\mathcal{B}$, and
\item higher closure operators on $\mathcal{B}$.
\end{enumerate}
\end{theorem}
\begin{proof}
Immediate from Corollary~\ref{corlexfibmods}.3 
and Corollary~\ref{cormodopacc} applied to the class $K$ of $\kappa$-filtered simplicial sets for the according regular cardinal $\kappa$.
\end{proof}

\subsection{Higher Lawvere-Tierney operators}\label{secsubhlttops}

In this section we internalize the cartesian reflective localizations from Section~\ref{secsubfsrl} (or equivalently the closure 
operators from Section~\ref{sechigherclosureops}) whenever the base $\mathcal{B}$ is an $\infty$-topos. Therefore we use, first, that 
the target fibration $t$ over $\mathcal{B}$ is the externalization of an internal category object in $\mathcal{B}$ (up to a cofinal sequence 
of truncations), and, second, that the externalization functor is (close enough to be) a cosmological embedding and so both preserves and 
reflects all relevant higher categorical structures.

\begin{remark}
Although technically we want to internalize higher closure operators in an $\infty$-topos to obtain a notion of higher Lawvere-Tierney 
operators on its object classifiers, we will instead internalize the according reflective localizations over $\mathcal{B}$ and again apply 
the Beck Monadicity Theorem to the $\infty$-cosmos of internal $\infty$-categories within $\mathcal{B}$. The reason for this additional 
gymnastic exercise is that the theory of homotopy-coherent monads per se does not lift well along cosmological biequivalences: a
homotopy-coherent monad is defined to be a simplicial functor out of the free monad, a lift along a cosmological biequivalence however generally does not preserve strict simplicial action. 
Instead, we can lift the according adjunction however (which is determined by a finite initial segment of simplicial data and lifts perfectly 
fine) and recover a corresponding monad by the Monadicity Theorem applied to the corresponding $\infty$-cosmos. 
\end{remark}

%

\subsubsection*{Polymorphic families of accessible cartesian endofunctors}
Whenever $\mathcal{B}$ is presentable, every sufficiently large regular cardinal $\kappa$
defines a fiberwise left exact and $\kappa$-closed full and replete cartesian subfibration
\[\xymatrix{
(\mathcal{B}^{\Delta^1})_{\kappa}\ar@{^(->}[r]\ar@/_/@{->>}[dr]_{t_{\kappa}} & \mathcal{B}^{\Delta^1}\ar@{->>}[d]^t \\
 & \mathcal{B}
}\]
spanned by the relatively $\kappa$-compact maps. Every fiberwise accessible endofunctor on $t$ over $\mathcal{B}$ restricts to an accessible 
endofunctor on $t_{\kappa}$ for suitably large regular cardinals $\kappa$ by the following lemma.

\begin{lemma}\label{lemmaaccsmall}
\begin{enumerate}
\item Let $F\colon\mathcal{A}\rightarrow\mathcal{B}$ be an accessible functor between presentable $\infty$-categories. 
Then there is a cardinal $\mu$ such that $F$ preserves $\kappa$-compact objects for all $\kappa\gg\mu$ (\cite[Definition 5.4.2.8]{luriehtt}).
\item Let $\mathcal{B}$ be presentable and $T\colon t_{\mathcal{B}}\rightarrow t_{\mathcal{B}}$ be a cartesian endofunctor that is fiberwise 
accessible. Then there is a cardinal $\mu$ such that for all $\kappa\gg\mu$, the functor $T$ preserves relatively $\kappa$-compact maps 
(objectwise).
\end{enumerate}
\end{lemma}
\begin{proof}
Part 1 is a version of the ``Uniformization Theorem'' (\cite[Theorem 2.19]{adamekrosicky}) and is stated in \cite[Remark 5.4.2.15]{luriehtt}. 
A full proof is given in \cite[Lemma 8.3.4]{thesis}. 

For Part 2, let $\mathcal{B}$ be $\lambda$-presentable and $G$ be a set of $\lambda$-compact generators of
$\mathcal{B}$. For every $C\in G$ the slice $\mathcal{B}_{/C}$ is $\lambda$-presentable and as such generated 
by the maps with domain in $G$, and so for each $\kappa\gg\lambda$ the $\kappa$-compact objects in $\mathcal{B}_{/C}$ are 
exactly the maps with $\kappa$-compact domain in $\mathcal{B}$. It follows that
$T_C\colon\mathcal{B}_{/C}\rightarrow\mathcal{B}_{/C}$ preserves maps with $\kappa$-compact domain for all $C\in G$ and all
$\kappa\gg\mu\gg\lambda$ where $\mu$ is chosen so that $T_C$ is $\lambda$-accessible for all $C\in G$. Given such $\kappa$ and a 
relative $\kappa$-compact map $f\colon A\rightarrow B$ between two arbitrary objects $A,B$ in $\mathcal{B}$, we want to show that
$T(f)\in\mathcal{B}_{/B}$ is relatively $\kappa$-compact again. Therefore, let $C\in G$ be a generator and $b\colon C\rightarrow B$ a 
map. Then $b^{\ast}T(f)\simeq T(b^{\ast}f)$ over $C$ since $T$ is a cartesian functor. Now, $b^{\ast}A$ is $\kappa$-compact by assumption, 
and hence $b^{\ast}f\in\mathcal{B}_{/C}$ is $\kappa$-compact as well by the above. It follows that $T(b^{\ast}f)\in\mathcal{B}_{/C}$ is
$\kappa$-compact, and thus so is its domain. This finishes the proof.
\end{proof}

Since every map in a presentable $\infty$-category is relatively $\kappa$-compact for some regular cardinal $\kappa$, the target fibration
$t\colon\mathcal{B}^{\Delta^1}\twoheadrightarrow\mathcal{B}$ allows a filtration
\[\mathcal{B}^{\Delta^1}=\bigcup_{\kappa}(\mathcal{B}^{\Delta^1})_{\kappa}\]
indexed by any cofinal sequence of regular cardinals
$\kappa$. Thus, we may scatter a given accessible idempotent cartesian monad $T$ on $t$ across initial segments $t_{\kappa}$ and glue them 
back together without any essential loss of information. More precisely, we can do the following. 

Suppose $\mathbb{M}$ is a quasi-categorically enriched model category. Then for any small Reedy category $I$, the 
functor category $\mathrm{Fun}(I,\mathbb{M})$ comes equipped with the Reedy model structure which is again quasi-categorically 
enriched. Its full simplicial subcategory spanned by the fibrant objects hence gives rise to an $\infty$-cosmos as defined in
\cite[Definition 2.1.1]{rvyoneda}, but fails to give rise to an $\infty$-cosmos as used so far in this paper (following 
\cite{riehlverityelements}), because fibrant objects in $(\mathrm{Fun}(I,\mathbb{M}),\mathrm{Reedy})$ are generally not cofibrant.
While this still may be enough to develop much of the theory of $\infty$-cosmoses generally in a suitable sense, in some special cases the 
model category $(\mathrm{Fun}(I,\mathbb{M}),\mathrm{Reedy})$ does give rise to an honest $\infty$-cosmos after all. For one, this 
is the case whenever $I$ is an elegant Reedy category and $\mathbb{M}$ is a Cisinski model category. For the constructions in this section we 
exploit a different but similar situation (see the following lemma).

Therefore, first, we have to address that given an $\infty$-category $\mathcal{B}$ which is not necessarily small (but only locally small 
instead), the $\infty$-category $\mathrm{Cart}(\mathcal{B})$ is not even locally small. To manage the usual issues associated with size, we 
will assume the existence of a Grothendieck universe $\mathcal{V}$ -- for instance given by an inaccessible cardinal $\nu$ -- which contains 
$\mathcal{B}$. More precisely, 
whenever $\mathcal{B}$ is presentable, we will assume that the model categorical presentation $\mathbb{B}$ is given by a left Bousfield 
localization of $\nu$-small simplicial sets valued simplicial presheaves over an indexing simplicial $\infty$-category of size less than
$\nu$. Thus in the following the class ``Ord'' of ordinals will refer to the set $\nu$ of ordinals smaller than $\nu$. 
In particular, the category $\mathrm{Fun}(\mathrm{Ord},\mathrm{Cart}(\mathcal{B}))$ is again locally small in the metatheory, and so is
$\mathrm{Fun}(\mathrm{Ord},\mathbb{M})$ for any $\nu$-small model category $\mathbb{M}$.

Given the standing assumption of such a metatheoretical universe $\mathcal{V}_{\nu}$, for the constructions associated to an $\infty$-topos 
$\mathcal{E}=\hat{\mathcal{C}}[T^{-1}]$ in this section we have two choices: one deducing structure from the top down and one inducing 
structure from the bottom up. 
On the one hand, the top-down approach considers the ``large'' object classifier $\pi_{\nu}$ of relatively $\nu$-small maps in the enlarged
$\infty$-topos $\mathcal{E}^+:=\mathrm{Fun}(\mathcal{C}^{op},\mathcal{S}^+)[T^{-1}]$ of sheaves valued in (not necessarily $\nu$-small) 
spaces, and uses it basically as a logical framework for $\mathcal{E}\simeq\mathcal{E}^+_{\nu}$ (which however lies 
outside of $\mathcal{E}$) to derive the desired functorial contructions in $\mathcal{E}$. This is an approach in essence taken for instance 
in the work of \cite{martiniwolf_ihtt} to internalize various notions of higher topos theory. The ``bottom-up'' approach on the other hand 
works in $\mathcal{E}$ directly and coheres pointwise notions to the desired diagrammatic structures by various inductive arguments.
Both approaches have their subtleties, and in large rely on the same methods. In the following, we will pursue the ``bottom-up'' 
approach.



\begin{lemma}\label{lemmadirectedmc}
Let $\mathbb{M}$ be a model category such that the cofibrations are exactly the monomorphisms and such that the underlying category of
$\mathbb{M}$ is finitely presentable. Let $I$ be an ordinal (or the class of all ordinals).
\begin{enumerate}
\item 
A diagram $B\colon I\rightarrow\mathbb{M}$ is Reedy cofibrant if and only if $B(0)$ is cofibrant and for all $i\leq j$ the designated map 
$B(i)\rightarrow B(j)$ is a cofibration. A diagram $B\colon I\rightarrow\mathbb{M}$ is Reedy fibrant if and only if each $B(i)\in\mathbb{M}$ 
is fibrant.
\item The full simplicial subcategory $\mathbf{Fun}(I,\mathbf{M})$ spanned by the 
bifibrant objects in $(\mathrm{Fun}(I,\mathbb{M}),\mathrm{Reedy})$ is an $\infty$-cosmos. 
\end{enumerate}
\end{lemma}
\begin{proof}
For Part 1, we note that all matching objects exist and are trivial, and so the fibrancy condition follows readily. By definition, a diagram
$B\colon I\rightarrow\mathbb{M}$ is Reedy cofibrant if and only if
\begin{itemize}
\item $B(0)$ is cofibrant,
\item all designated maps $B(i)\rightarrow B(i+1)$ are cofibrations in $\mathbb{M}$, and
\item for all limit ordinals $\lambda\in I$, the map $\mathrm{colim}_{i<\lambda}B(i)\rightarrow B(\lambda)$ is a cofibration.
\end{itemize}
Thus, whenever $B$ is Reedy cofibrant, we can show that $B(i)\rightarrow B(j)$ is a cofibration for all $i\leq j\leq\lambda$ for all
$\lambda\in I$ by induction. 
Clearly the identity $B(0)\rightarrow B(0)$ is a cofibration, and whenever $B(i)\rightarrow B(j)$ is a cofibration for all
$i\leq j\leq\lambda$, then so is the composition $B(i)\rightarrow B(j)\rightarrow B(\lambda+1)$ for all $i\leq\lambda+1$. For limit 
ordinals $\lambda$, assume that $B(i)\rightarrow B(j)$ is a cofibration 
for all $i\leq j<\lambda$. Then for every $i<\lambda$ the cocone map $B(i)\rightarrow\mathrm{colim}_{i<\lambda}B(i)$ is monic due to finite 
presentability of $\mathbb{M}$ (\cite[Proposition 1.62]{adamekrosicky}); it follows that the composition
$B(i)\hookrightarrow\mathrm{colim}_{i<\lambda}B(i)\hookrightarrow B(\lambda)$ is monic as well. This finishes the induction.

For the other direction, whenever all designated maps $B(i)\rightarrow B(j)$ are monic, we are left to show 
that the map $\mathrm{colim}_{i<\lambda}B(i)\rightarrow B(\lambda)$ is monic for $\lambda\in I$ a limit ordinal. This again however follows 
from finite representability of $\mathbb{M}$.

For Part 2, as $\mathbb{M}$ is quasi-categorically enriched (as is $I$), the Reedy model structure on $\mathrm{Fun}(I,\mathbb{M})$ is again 
enriched over ($\nu$-small) quasi-categories. Given the proof of \cite[Proposition E.1.1]{riehlverityelements} we are hence only left to show 
that the class of cofibrant objects in $(\mathrm{Fun}(I,\mathbb{M}),\mathrm{Reedy})_f$ is closed under all operations required in 
\cite[Definition 1.2.1]{riehlverityelements}. The terminal object $1\in\mathrm{Fun}(I,\mathbb{M})$ however is the constant terminal 
object functor and as such is bifibrant 
by Part 1. The other properties all follow from Part 1, from the fact that limits and cotensors in $\mathrm{Fun}(I,\mathbb{M})$ are computed 
pointwise, and from the fact that the class of monomorphisms in $\mathbb{M}$ is closed under all limits and cotensors.
\end{proof}

\begin{remark}
Given that we only assume finite presentability of $\mathbb{M}$ in Lemma~\ref{lemmadirectedmc}, it may be worth to point out that 
``monomorphism'' here does not mean regular monomorphism (or any other variation of the notion).
\end{remark}

\begin{example}
Crucially, for a regular cardinal $\mu$ we may consider the unbounded class $\mathrm{Shl}(\mu)$ of regular cardinals which are sharply 
larger than $\mu$, together with its according monotone $\mathrm{Ord}$-indexed enumeration $\mathrm{Ord}\rightarrow \mathrm{Shl}(\mu)$.
Given a quasi-category $\mathcal{B}$, we may apply 
Lemma~\ref{lemmadirectedmc} to the model category $(\mathbf{S}^+/\mathcal{B}^{\sharp},\mathrm{Cart})$. Indeed, the category $\mathbf{S}^+/\mathcal{B}^{\sharp}$ is finitely presentable by the fact that $\mathbf{S}^+$ is so and by \cite[Proposition 1.57]{adamekrosicky}. The cofibrations in
$(\mathbf{S}^+/\mathcal{B}^{\sharp},\mathrm{Cart})$ are by definition exactly the maps which induce cofibrations of underlying simplicial 
sets; those are exactly the monomorphisms in $\mathbf{S}^+$ since the forgetful functor $U\colon \mathbf{S}^{+}\rightarrow \mathbf{S}$ both 
preserves and reflects monomorphisms. Furthermore, the target fibration $t\in\mathbf{Cart}(\mathbf{\mathcal{B}})$ induces the diagram
$t_{\bullet}\colon \mathrm{Ord}\rightarrow \mathrm{Shl}(\mu)\rightarrow\mathbf{Cart}(\mathcal{B})$. This diagram is Reedy bifibrant by Lemma~\ref{lemmadirectedmc}.1. 
As such it is an object of the $\infty$-cosmos $\mathbf{Fun}(\mathrm{Shl}(\mu),\mathbf{Cart}(\mathcal{B}))$.
\end{example}

\begin{definition}\label{defpolyfam}
Let $\mu$ be a regular cardinal and $\mathrm{Shl}(\mu)$ be the class of regular cardinals sharply larger than $\mu$. Let
$\mathbb{M}$ be a model category as in Lemma~\ref{lemmadirectedmc}, and let $B_{\bullet}\colon \mathrm{Shl}(\mu)\rightarrow\mathbb{M}$ be a 
Reedy bifibrant diagram. Then a \emph{$\mu$-polymorphic family of (left exact/accessible/\dots) reflective localizations of $B_{\bullet}$} 
in $\mathbf{M}$ is a (left exact/accessible/\dots) reflective localization of the diagram $B_{\bullet}$ in the $\infty$-cosmos
$\mathbf{Fun}(\mathrm{Shl}(\mu),\mathbf{M})$. 

A pair of $\mu_i$-polymorphic families $E_i\rightarrow B_i$ $(i\in\{0,1\})$ of reflective localizations in $\mathbf{M}$ are 
\emph{equivalent} if there is a cardinal $\mu\gg\mu_1,\mu_2$ and a $\mu$-polymorphic family $E\rightarrow B$ of reflective localizations in
$\mathbf{M}$ such that $(E_i\rightarrow B_i)|_{\mathrm{Shl}(\mu)}\simeq(E\rightarrow B)$ in $\mathbf{Fun}(\mathrm{Shl}(\mu),\mathbf{M})$. 

Given a quasi-category $\mathcal{B}$ with pullbacks, a \emph{polymorphic family of reflective localizations of
$t\in\mathbf{Cart}(\mathcal{B})$} is regular cardinal $\mu$ together with a $\mu$-indexed polymorphic family of 
reflective localizations of $t_{\bullet}\colon \mathrm{Shl}(\mu)\rightarrow\mathbf{Cart}(\mathcal{B})$. 

\end{definition}

Definition~\ref{defpolyfam} induces an equivalence relation on the class of polymorphic families of reflective localizations 
in $\mathbf{M}$ of the form $(\mu,\rho)$ for $\mu$ a regular cardinal and $\rho$ a $\mu$-polymorphic family of reflective localizations in
$\mathbf{M}$. 
Here, note that the restriction of Reedy bifibrant diagrams to subclasses of regular cardinals in that context
preserves Reedy bifibrancy by Lemma~\ref{lemmadirectedmc}. 

\begin{proposition}\label{propmonadscatter}
Let $\mathcal{B}$ be a presentable quasi-category. Then every accessible (left exact) reflective localization of the target fibration 
$t\in\mathbf{Cart}(\mathcal{B})$ yields a polymorphic family of accessible (left exact) reflective localizations of 
$t\in\mathbf{Cart}(\mathcal{B})$ by restriction. Vice versa, every regular cardinal $\mu$ together with a polymorphic family of accessible 
(left exact) reflective localizations of $t_{\bullet}\colon \mathrm{Shl}(\mu)\rightarrow\mathbf{Cart}(\mathcal{B})$ defines an accessible 
(left exact) cartesian reflective localization of $t\in\mathbf{Cart}(\mathcal{B})$ via union. On equivalence classes, the two assignment are 
mutually inverse to one another.
\end{proposition}
\begin{proof}
Let $\iota\colon\mathcal{E}\hookrightarrow t$ be an accessible cartesian reflective subcategory with cartesian 
(left exact) left adjoint $\rho$. As $\iota$ is accessible, let $\mu$ be the regular cardinal such that Lemma~\ref{lemmaaccsmall}.2 applies 
to all cardinals in $\mathrm{Shl}(\mu)$ to the the composition $\iota\rho\colon t\rightarrow t$.
For $\lambda\in \mathrm{Shl}(\mu)$, let $\mathcal{E}_{\lambda}\subseteq\mathcal{E}$ be the full subcategory given by the (strict) pullback
\[\xymatrix{
\mathcal{E}_{\lambda}\ar@{^(->}[r]\ar@{^(->}[d]^{\iota_{\lambda}}\ar@{}[dr]|(.3){\pbs} & \mathcal{E}\ar@{^(->}[d]^{\iota} \\
(\mathcal{B}^{\Delta^1})_{\lambda}\ar@{^(->}[r]\ar@/^1pc/@{-->}[u]^{\rho_{\lambda}} & \mathcal{B}^{\Delta^1}\ar@/^1pc/@{-->}[u]^{\rho}
}\]
We obtain natural transformations $\iota_{\bullet}\colon\mathcal{E}_{\bullet}\hookrightarrow(\mathcal{B}^{\Delta^1})_{\bullet}$ and 
$\rho_{\bullet}\colon(\mathcal{B}^{\Delta^1})_{\bullet}\hookrightarrow\mathcal{E}_{\bullet}$ together with 2-cells
$\eta_{\bullet}\colon 1\rightarrow \iota_{\bullet}\rho_{\bullet}$ and $\epsilon_{\bullet}\colon\rho_{\bullet}\iota_{\bullet}\rightarrow 1$ in 
the simplicial functor category $\mathrm{Fun}(\mathrm{Shl}(\mu),\mathbf{Cart}(\mathcal{B}))$ induced by restriction of the unit and counit of 
the adjunction $\rho\mbox{ }\adj\mbox{ }\iota$ in $\mathbf{Cart}(\mathcal{B})$. 
The diagrams $\mathcal{E}_{\bullet}$ and $(\mathcal{B}^{\Delta^1})_{\bullet}$ are both Reedy bifibrant by Lemma~\ref{lemmadirectedmc}.1. It 
follows that the tuple $(\rho_{\bullet},\iota_{\bullet},\eta_{\bullet},\epsilon_{\bullet})$ yields a reflective localization in the
$\infty$-cosmos $\mathbf{Fun}(\mathrm{Shl}(\mu),\mathbf{Cart}(\mathcal{B}))$ for example via \cite[Lemma B.4.2]{riehlverityelements}, as the 
homotopy equivalences in the $\infty$-cosmos of diagrams are exactly the natural transformations which are componentwise equivalences in
$\mathbf{Cart}(\mathcal{B)}$. Furthermore, as simplicial cotensors and limits in $\mathbf{Fun}(S,\mathbf{Cart}(\mathcal{B}))$ are computed 
pointwise as well, the reflective localization $(\rho_{\bullet},\iota_{\bullet})$ in $\mathbf{Fun}(S,\mathbf{Cart}(\mathcal{B}))$ is left 
exact and $\kappa$-accessible for some $\kappa$ whenever each pair $(\rho_{\lambda},\iota_{\lambda})$ in $\mathbf{Cart}(\mathcal{B})$ is so 
(via \cite[Theorem 3.5.3]{riehlverityelements}). This in turn is the case whenever the pair $(\rho,\iota)$ is left exact and
$\kappa$-accessible.

Vice versa, suppose we are given a regular cardinal $\mu$ together with a $\mu$-polymorphic family
$\iota_{\bullet}\colon\mathcal{E}_{\bullet}\rightarrow(\mathcal{B}^{\Delta^1})_{\bullet}$ of accessible (left exact) reflective localizations 
with left adjoint $\rho_{\bullet}$. First, we may replace the localization with an equivalent diagram of pointwise reflective subcategories: 
every $\lambda\in\mathrm{Shl}(\mu)$ yields an accessible (left exact) reflective localization
\[\xymatrix{
\iota_{\lambda}\colon\mathcal{E}_{\lambda}\ar@{^(->}@<-.5ex>[r] & (\mathcal{B}^{\Delta^1})_{\lambda}\colon\rho_{\lambda}\ar@<-.5ex>@/_/[l]
}\]
by evaluation at $\lambda$, each of which is equivalent (in $\mathbf{Cart}(\mathcal{B})$) to the cartesian full reflective subcategory 
$\bar{\mathcal{E}}_{\lambda}$ spanned by the objects that are local with respect to the class 
of $\rho_{\lambda}$-equivalences in $(\mathcal{B}^{\Delta^1})_{\lambda}$. The factorizations
$\mathcal{E}_{\lambda}\xrightarrow{\simeq}\bar{\mathcal{E}}_{\lambda}\hookrightarrow(\mathcal{B}^{\Delta^1})_{\lambda}$ are natural in $\mathrm{Shl}(\mu)$, so we obtain a factorization of diagrams
$\mathcal{E}_{\bullet}\xrightarrow{\simeq}\bar{\mathcal{E}}_{\bullet}\overset{\bar{\iota}_{\bullet}}{\hookrightarrow}(\mathcal{B}^{\Delta^1})_{\bullet}$ together with a natural left adjoint $\bar{\rho}_{\bullet}\colon(\mathcal{B}^{\Delta^1})_{\bullet}\xrightarrow{\rho}\mathcal{E}_{\bullet}\xrightarrow{\simeq}\bar{\mathcal{E}}_{\bullet}$ of $\bar{\iota}_{\bullet}$.
Taking colimits of the diagram 
\[\xymatrix{
\bar{\iota}_{\bullet}\colon\bar{\mathcal{E}}_{\bullet}\ar@{^(->}@<-.5ex>[r] & (\mathcal{B}^{\Delta^1})_{\bullet}\colon\bar{\rho}_{\bullet}\ar@<-.5ex>@/_/[l]
}\] 
yields a diagram 
\[\xymatrix{
\bar{\iota}\colon\bar{\mathcal{E}}\ar@{^(->}@<-.5ex>[r] & \mathcal{B}^{\Delta^1}\colon\bar{\rho}\ar@<-.5ex>@/_/[l]
}\]
in $\mathbf{Cart}(\mathcal{B})$, where $\bar{\mathcal{E}}\subseteq\mathcal{B}^{\Delta^1}$ is simply the union of the full subcategories
$\mathcal{E}_{\lambda}\subset\mathcal{B}^{\Delta^1}$. The unit and counit 2-cells of the pair $(\bar{\iota}_{\bullet},\bar{\rho}_{\bullet})$ 
in $\mathbf{Fun}(S,\mathbf{Cart}(\mathcal{B}))$ yield 2-cells $\eta\colon1\rightarrow\bar{\iota}\bar{\rho}$ and 
$\epsilon\colon\bar{\rho}\bar{\iota}\rightarrow 1$ in $\mathbf{Cart}(\mathcal{B})$ accordingly. To show that this data constitutes a 
reflective localization in $\mathbf{Cart}(\mathcal{B})$, again by \cite[Lemma B.4.2]{riehlverityelements} it suffices to show that
the 2-cells $\epsilon$, $\bar{\rho}\eta$ and $\eta \bar{\iota}$ are natural equivalences. But the diagrams $\bar{\mathcal{E}}$ and
$\mathcal{B}^{\Delta^1}_{\bullet}$ are Reedy (and hence projectively) cofibrant, and so it follows that their colimits $\bar{\mathcal{E}}$ 
and $\mathcal{B}^{\Delta^1}$ are their homotopy-colimits as well. As the 2-cells $\epsilon$, $\bar{\rho}\eta$ and $\eta \bar{\iota}$ are thus 
the homotopy-colimits of $\epsilon_{\bullet}$, $\bar{\rho}_{\bullet}\eta_{\bullet}$ and $\eta_{\bullet}\bar{\iota}_{\bullet}$, respectively, 
and all three of these 2-cells are pointwise equivalences by assumption,  it follows that $\epsilon$, $\bar{\rho}\eta$ and $\eta \bar{\iota}$ 
are so, too. This shows that the pair $(\bar{\iota},\bar{\rho})$ is part of a reflective localization in $\mathbf{Cart}(\mathcal{B})$. The 
left adjoint $\bar{\rho}$ is left exact whenever $\rho_{\bullet}$ -- and hence $\bar{\rho}_{\bullet}$ -- is (pointwise) left exact, and the 
right adjoint $\bar{\iota}$ is $\kappa$-accessible whenever $\iota_{\bullet}$ -- and hence $\bar{\iota}_{\bullet}$ -- is 
(pointwise) $\kappa$-accessible for some regular cardinal $\kappa$. That is, because for every simplicial set $K$ (with small cardinality), 
the cotensor $(\cdot)^K\colon\mathbf{Cart}(\mathcal{B})\rightarrow\mathbf{Cart}(\mathcal{B})$ preserves large sequential colimits (as such 
are $\kappa$-filtered for all small regular cardinals $\kappa$).
\end{proof}

\begin{definition}\label{defcomppolyfam}
We say that a polymorphic family $\mathcal{E}_{\bullet}\colon\mathrm{Shl}(\mu)\rightarrow\mathrm{Cart}(\mathcal{B})$ of reflective 
localizations of $t_{\mathcal{B}}$ is \emph{compositional} if for each $\lambda\in\mathrm{Shl}(\mu)$, the associated class
$\bar{\mathcal{E}}_{\lambda}\subseteq(\mathcal{B}^{\Delta^1})_{\lambda}$ of local objects for $\rho_{\lambda}$-equivalences is 
(objectwise) closed under compositions.
\end{definition}

We obtain the following stratification of Theorem~\ref{thmlocmodopacc}.

\begin{proposition}\label{prophcopsstrat}
Let $\mathcal{B}$ be a presentable $\infty$-category and $\kappa$ be a regular cardinal. Then there is a bijection between the 
following two collections.
\begin{enumerate}
\item Equivalence classes of $\kappa$-accessible (left exact and compositional) idempotent monads $T$ on $t\in\mathbf{Cart}(\mathcal{B})$.
\item Equivalence classes of polymorphic families of $\kappa$-accessible (left exact and compositional) reflective 
localizations of $t\in\mathbf{Cart}(\mathcal{B})$.
\end{enumerate}
\end{proposition}
\begin{proof}
The proof of Proposition~\ref{propmonadscatter} yields a bijection between equivalence classes of (left exact and)
$\kappa$-accessible reflective localizations of $t\in\mathbf{Cart}(\mathcal{B})$, and equivalence classes of polymorphic families 
of (left exact and) $\kappa$-accessible reflective localizations of $t$. The fact that compositionality translates back and forth as well is fairly immediate.
\end{proof}

\subsubsection*{Internalizing accessible reflective localizations of an $\infty$-topos}

Whenever the presentable quasi-category $\mathcal{B}$ is an $\infty$-topos, we can internalize the notion of 
polymorphic families of reflective localizations of $t\in\mathbf{Cart}(\mathcal{B})$ -- and hence obtain an according notion of 
Lawvere-Tierney operators -- as follows.\\

For any presentable quasi-category $\mathcal{B}$ there is a category $\mathbb{B}$ of simplicial presheaves (over an associated small 
simplicial category which for the following may remain anonymous) equipped with the injective model structure, together with a set
$T\subset\mathbb{B}^{\Delta^1}$ of arrows such that the left Bousfield localization $\mathbb{B}[T^{-1}]$ presents $\mathcal{B}$ 
(\cite[Proposition A.3.7.6]{luriehtt}). The model category $\mathbb{B}[T^{-1}]$ is combinatorial, simplicial and cartesian, its cofibrations 
are exactly the monomorphisms. In the following, $\mathbb{B}$ will denote the simplicial category $\mathbb{B}$ equipped with the model 
structure $\mathbb{B}[T^{-1}]$.

The quasi-category $\mathrm{Cat}_{\infty}(\mathcal{B})$ of complete Segal objects in $\mathcal{B}$ is again presented by a 
combinatorial, simplicial and cartesian model category $(\mathbb{B}^{\Delta^{op}},\mathrm{CS})$ whose cofibrations are exactly the 
monomorphisms (constructed in this generality in Barwick's thesis \cite{barwickthesis}, or subsequently for example in 
\cite[Proposition E.3.7]{riehlverityelements}). Here the category $\mathbb{B}^{\Delta^{op}}$ is equipped with the Reedy model structure over
$\mathbb{B}$ and again left Bousfield localized at a suitable set of arrows. Now, we observe that the model 
category $(\mathbb{B}^{\Delta^{op}},\mathrm{CS})$ is in fact not only again simplicial as provided by the general theory of left Bousfield localizations (say with hom-functor $\mathrm{Hom}_{\mathrm{Kan}}$), but it comes equipped with a canonical quasi-categorical enrichment
$\mathrm{Hom}_{\mathrm{QCat}}$ whose core $\mathrm{Hom}_{\mathrm{QCat}}(A,B)^{\simeq}$ for complete Segal objects
$A,B\in\mathbb{B}^{\Delta^{op}}$ is naturally homotopy equivalent to the original simplicial enrichment
$\mathrm{Hom}_{\mathrm{Kan}}(A,B)$ (see \cite{rs_mcext}; for Segal spaces this is \cite[Proposition E.2.2]{riehlverityelements}). The 1-cells in these quasi-categories correspond exactly to the internal natural 
transformations between functors of the complete Segal objects $A,B$. We thus obtain an $\infty$-cosmos
$\mathbf{Cat}_{\infty}(\mathbb{B})=(\mathbb{B}^{\Delta^{op}},\mathrm{CS})_f$ which represents an $(\infty,2)$-category of complete Segal 
objects in $\mathcal{B}$ -- rather than the underlying $(\infty,1)$-category $\mathrm{Cat}_{\infty}(\mathcal{B})$ only, which in turn is 
presented by the locally $\infty$-groupoidal $\infty$-cosmos given by the $\mathrm{Hom}_{\mathrm{Kan}}$-enrichment of
$\mathbb{B}^{\Delta^{op}}$. The latter is the $\infty$-cosmos of Rezk objects considered in \cite[Proposition E.3.7]{riehlverityelements} and 
\cite[Proposition 2.2.9]{rvyoneda}). 

Every complete Segal object in a left exact $\infty$-category $\mathcal{B}$ may be \emph{externalized} to yield a fibered
$\infty$-category over $\mathcal{B}$. This defines the externalization functor
\[\mathrm{Ext}\colon\mathrm{Cat}_{\infty}(\mathcal{B})\rightarrow\mathrm{Cart}(\mathcal{B})\]
from the $\infty$-category of internal $\infty$-categories in $\mathcal{B}$ to the $\infty$-category of cartesian fibrations over
$\mathcal{B}$ (\cite[Section 4]{rs_comp}). By definition, $\mathrm{Ext}(X)$ is the Unstraightening of the indexed quasi-category $\mathrm{Ext}(X)\colon\mathcal{C}^{op}\rightarrow\mathrm{Cat}_{\infty}$ given pointwise by the horizontal hom-quasi category
\begin{align}\label{motivdefext}
\mathrm{Ext}(X)(C):=\mathcal{B}(C,X_{\bullet})_0
\end{align}
up to a Reedy fibrant replacement of the simplicial space $\mathcal{B}(C,X_{\bullet})$.
We will make crucial use of the following $\infty$-cosmological version of \cite[Proposition 7.3.8]{jacobsttbook} in the special case of 
presentable bases $\mathcal{B}$, together with a Yoneda lemma for ``Segal-representable indexed quasi-categories'' stated in the next lemma. 
The proofs of both statements are fairly straight-forward but will be found in \cite{rs_mcext}.

\begin{proposition}[{\cite{rs_mcext}}]\label{propcosmicext}
Suppose $\mathcal{B}$ is a presentable quasi-category. Then the $(\infty,1)$-categorical externalization functor
\[\mathrm{Ext}\colon\mathrm{Cat}_{\infty}(\mathcal{B})\rightarrow\mathrm{Cart}(\mathcal{B})\]
is induced by a pseudo-cosmological embedding
\begin{align}\label{equcosmicext}
\mathrm{Ext}\colon\mathbf{Cat}_{\infty}(\mathbb{B})\rightarrow\mathbf{Cart}(\mathcal{B})
\end{align}
of $\infty$-cosmoses. Here, that means that the functor $\mathrm{Ext}$ is a simplicially enriched functor which induces locally categorical 
equivalences of hom-quasi-categories, and furthermore
\begin{enumerate}
\item preserves fibrations and all ordinary limits, and
\item preserves simplicial cotensors up to natural homotopy equivalence.
\end{enumerate}
Every such fully faithful pseudo-cosmological functor preserves and reflects $\infty$-cosmological structures such as equivalences, 
adjunctions, and the property of an arrow to be limit or colimit preserving for any class of simplicial shapes.
\end{proposition}\qed

\begin{remark}
A pseudo-cosmological embedding as in Proposition~\ref{propcosmicext} is almost a cosmological biequivalence onto its essential 
image in the sense of \cite[Section 10.2]{riehlverityelements}. The only difference is that the simplicial cotensors are not necessarily
preserved up to natural isomorphism but rather up to natural homotopy equivalence only. This tweak of the definition is enforced to 
accommodate the fact that $\mathrm{Ext}$ is defined in terms of a precomposition with Lurie's Unstraightening functor; and Unstraightening is 
simplicially enriched but preserves simplicial cotensors only up to natural homotopy equivalence indeed (via
\cite[Corollary 3.2.1.15]{luriehtt}).
\end{remark}

%
The externalization functor is a direct generalization of the Yoneda embedding via the pullback square
\[\xymatrix{
\mathcal{B}\ar[r]^y\ar@{^(->}[d] & \mathrm{RFib}(\mathcal{B})\ar@{^(->}[d] \\
\mathrm{Cat}_{\infty}(\mathcal{B})\ar[r]_{\mathrm{Ext}} & \mathrm{Cart}(\mathcal{B}) 
}\]
where the vertical inclusion $\mathcal{B}\hookrightarrow\mathrm{Cat}_{\infty}(\mathcal{B})$ identifies an object $B\in\mathcal{B}$ with its 
associated internal complete Segal groupoid (\cite[Section 1.1]{lgood}). Indeed, one can prove a generalized Yoneda Lemma for 
complete Segal objects and cartesian fibrations which we state for later use as well.

\begin{lemma}[The Segal-Yoneda Lemma]\label{lemmaextyoneda}
Let $\mathcal{B}$ be a quasi-category with pullbacks, $X\in\mathrm{Cat}_{\infty}(\mathcal{B})$ be a complete Segal object in
$\mathcal{B}$ and $F\colon\mathcal{B}^{op}\rightarrow\mathrm{Cat}_{\infty}(\mathcal{S})$ be a $\mathcal{B}$-indexed complete Segal space. Then there is a binatural equivalence
\[\mathrm{Fun}(\mathcal{B}^{op},\mathrm{Cat}_{\infty}(\mathcal{S}))(\mathcal{B}(\blank,X),F)\xrightarrow{\simeq}\int\limits_{n\in\Delta^{op}}F_n(X_n)\]
of spaces, which is pointwise induced by the Yoneda Lemma for presheaves. Given the equivalences
$((\cdot)^{\Delta^{\bullet}})^{\simeq}\colon\mathrm{Cat}_{\infty}\xleftrightarrow{\simeq}\mathrm{Cat}_{\infty}(\mathcal{S})\colon\mathbb{R}(\cdot)_{\bullet 0}$ induced by Joyal and Tierney's Quillen equivalence in \cite[Theorem 4.12]{jtqcatvsss},
for every $\mathcal{B}$-indexed quasi-category $F\colon\mathcal{B}^{op}\rightarrow\mathrm{Cat}_{\infty}$ we obtain a binatural equivalence
\[\mathrm{Fun}(\mathcal{B}^{op},\mathrm{Cat}_{\infty})(\mathrm{Ext}(X),F)\xrightarrow{\simeq}\int\limits_{n\in\Delta^{op}}(F(X_n)^{\Delta^n})^{\simeq}\]
of spaces. Informally, given the formula (\ref{motivdefext}), the equivalence 
maps a natural transformation $\alpha\colon\mathrm{Ext}(X)\rightarrow F$ to the tuple 
$(\alpha_{X_{n}}(1_{X_{n}})|n\geq 0)\in\int\limits_{n\in\Delta^{op}}(F(X_n)^{\Delta^n})^{\simeq}$.
\end{lemma}
\begin{proof}
See {\cite{rs_mcext}} for details.
\end{proof}

Now, as simplicial presheaf categories are finitely presentable (\cite{kellyfinlimits}), we may apply Lemma~\ref{lemmadirectedmc} 
to the model category $\mathbb{B}$. For every regular cardinal $\mu$ we thus obtain an $\infty$-cosmos
$\mathbf{Fun}(\mathrm{Shl}(\mu),\mathbf{Cat}_{\infty}(\mathbb{B}))$ together with the following corollary.

\begin{corollary}\label{corcosmicext}
Suppose $\mathcal{B}$ is a presentable quasi-category. For every regular cardinal $\mu$, the externalization functor (\ref{equcosmicext}) 
induces a pseudo-cosmological embedding
\[\mathrm{Ext}_{\ast}\colon\mathbf{Fun}(\mathrm{Shl}(\mu),\mathbf{Cat}_{\infty}(\mathbb{B}))\rightarrow\mathbf{Fun}(\mathrm{Shl}(\mu),\mathbf{Cart}(\mathcal{B})),\]
by postcomposition.
\end{corollary}
\begin{proof}
As the externalization in (\ref{equcosmicext}) is a simplicially enriched functor, so is the postcomposition $\mathrm{Ext}_{\ast}$. The 
fibrations and equivalences in both $\infty$-cosmoses are exactly the pointwise fibrations and equivalences, respectively, hence they are 
preserved by push-forward along any functor that preserves fibrations and equivalences. In the same vein, 
all simplicial cotensors and limits are computed pointwise in both cases. The fact that the push-forward $\mathrm{Ext}_{\ast}$ 
preserves Reedy cofibrant objects is given by Lemma~\ref{lemmadirectedmc} as $\mathrm{Ext}$ preserves monomorphisms. It follows
that $\mathrm{Ext}_{\ast}$ is a pseudo-cosmological functor. The fact that fully faithfulness of $\mathrm{Ext}$ implies fully faithfulness of 
$\mathrm{Ext}_{\ast}$ follows readily from the bifibrancy conditions and the fact that homotopy-limits preserve pointwise weak 
equivalences.
\end{proof}

Thus, in particular, any formal theory of adjunctions between diagrams in the image of the functor $\mathrm{Ext}_{\ast}$ in
$\mathbf{Fun}(\mathrm{Shl}(\mu),\mathbf{Cart}(\mathcal{B}))$ is equivalent to an associated theory of adjunctions between according objects 
in $\mathbf{Fun}(\mathrm{Shl}(\mu),\mathbf{Cat}_{\infty}(\mathbb{B}))$ (\cite[Proposition 10.1.4]{riehlverityelements}). The notions of 
invertibility, left exactness and accessibility as featured in Proposition~\ref{prophcopsstrat}.2 are purely formal altogether and hence are 
all preserved and reflected along such pseudo-cosmological embeddings.

The remaining crucial fact in the case when $\mathcal{B}$ is an $\infty$-topos is that there is a regular cardinal $\mathfrak{b}$ such that 
the ordinal sequence $t_{\bullet}\colon \mathrm{Shl}(\mathfrak{b})\rightarrow\mathbf{Cart}(\mathcal{B})$ and any reflective localization 
thereof is contained in the essential image of the externalization functor.
Therefore, given a quasi-category $\mathcal{B}$ with pullbacks and small colimits, recall that $\mathcal{B}$ satisfies \emph{descent} if
the canonical indexing (of large quasi-categories)
\begin{align}\label{equcanind}
\mathrm{\mathcal{B}_{/\blank}}\colon\mathcal{B}^{op}\xrightarrow{\mathcal{B}_{/\blank}}\mathrm{CAT}_{\infty}
\end{align}
preserves small limits (\cite{aneljoyaltopos}). In this generality, the canonical indexing (\ref{equcanind}) is defined as the Straightening 
of the cartesian target fibration $t\colon\mathcal{B}^{\Delta^1}\twoheadrightarrow\mathcal{B}$. For each sufficiently large regular cardinal 
$\kappa$ we obtain a subfunctor $(\mathcal{B}_{\blank})_{\kappa}:=\mathrm{Un}(t_{\kappa}\hookrightarrow t)$. It follows from the explicit 
construction of the Straightening functor (\cite[Section 3.2.1]{luriehtt}) that each
$(\mathcal{B}_{\blank})_{\kappa}\subseteq\mathcal{B}_{/\blank}$ is pointwise the full subcategory spanned by the relatively $\kappa$-compact 
maps. Furthermore, each inclusion $t_{\lambda}\subseteq t_{\kappa}$ for $\lambda\leq\kappa$ is mapped to the natural inclusion
$(\mathcal{B}_{\blank})_{\lambda}\subseteq(\mathcal{B}_{\blank})_{\kappa}$.

\begin{remark}\label{remaltslice}
Whenever the quasi-category $\mathcal{B}$ is presentable and hence arises as the underlying quasi-category
$N_{\Delta}(\mathbb{B}_{\mathrm{cf}})$ of a simplicial right proper model category
$\mathbb{B}=\mathcal{L}_T\mathbf{sPsh}(\mathbf{C})_{\mathrm{inj}}$ of simplicial presheaves, we may instead consider a canonical indexing of 
$\mathcal{B}$ defined as the functor given by
\[\mathbb{B}^{op}_{\mathrm{cf}}\xrightarrow{(\mathbb{B}_{/\blank})_{\mathrm{cf}}}(\mathrm{Cat}_{\Delta})_{\mathrm{f}}\xrightarrow{N_{\Delta}}\mathrm{QCat}\]
evaluated at the right Quillen functor $N_{\Delta}\colon\mathrm{Cat}_{\Delta}\rightarrow(\mathbf{S},\mathrm{QCat})$. It follows from 
\cite[Proposition 3.1.2, Corollary 6.1.2]{harpazprasmamodfib} that the two functors
$\mathrm{Ho}_{\infty}(\mathbb{B}_{/\blank})=N_{\Delta}(N_{\Delta}\circ(\mathbb{B}_{/\blank})_{\mathrm{cf}})$ and
$\mathcal{B}_{/\blank}$ are naturally equivalent to one another. The same holds for the truncated indexings
$\mathrm{Ho}_{\infty}((\mathbb{B}_{/\blank})_{\kappa})=N_{\Delta}(N_{\Delta}\circ((\mathbb{B}_{/\blank})_{\mathrm{cf}})_{\kappa})$ and
$(\mathcal{B}_{/\blank})_{\kappa}$, respectively, as well as for their respective inclusions for all sufficiently large 
cardinals $\kappa$. Here, $((\mathbb{B}_{/B})_{\mathrm{cf}})_{\kappa}$ denotes the full simplicial subcategory of
$(\mathbb{B}_{/B})_{\mathrm{cf}}$ spanned by the 1-categorically relatively $\kappa$-compact fibrations. ``Sufficiently large'' in this 
context is specified concretely in Notation~\ref{remdefB} in terms of an associated regular cardinal $\mathfrak{b}$.
\end{remark}

A presentable quasi-category $\mathcal{B}$ satisfies descent if and only if it is an
$\infty$-topos (\cite[Section 6.1.3]{luriehtt}). In that case, given any sufficiently large regular cardinal $\kappa$, 
the small presheaf 
\[(\mathrm{\mathcal{B}_{/\blank}})_{\kappa}^{\simeq}\colon\mathcal{B}^{op}\xrightarrow{(\mathcal{B}_{/\blank})_{\kappa}}\mathrm{Cat}_{\infty}\xrightarrow{(\cdot)^{\simeq}}\mathcal{S}\]
spanned by the relatively $\kappa$-compact maps is small limit preserving as well and hence is represented by the relatively $\kappa$-compact
\emph{object classifier} $U_{\kappa}\in\mathcal{B}$ (\cite[Section 6.1.6]{luriehtt}). This object comes with a canonically associated 
relatively $\kappa$-compact map
\[\pi_{\kappa}\colon \tilde{U}_{\kappa}\rightarrow U_{\kappa}\]
such that every relatively $\kappa$-compact map in $\mathcal{B}$ arises by pullback from $\pi_{\kappa}$ in an essentially unique way. Since
$\infty$-toposes are locally cartesian closed, representability of the indexed core
$(\mathrm{\mathcal{B}_{/\blank}})_{\kappa}^{\simeq}\colon\mathcal{B}^{op}\rightarrow\mathcal{S}$ is in fact equivalent to 
representability of the $\kappa$-small canonical indexing
\[(\mathrm{\mathcal{B}_{/\blank}})_{\kappa}\colon\mathcal{B}^{op}\rightarrow\mathrm{Cat}_{\infty}\]
itself by an internal $\infty$-category in $\mathcal{B}$ itself. More precisely, the \emph{nerves}
$\mathcal{U}_{\kappa}:=\mathcal{N}(\pi_{\kappa})$ of the object classifiers $\pi_{\kappa}$ (\cite[Definition 6.21]{rasekh}) are complete 
Segal objects in $\mathcal{B}$. 
Here, the object $(\mathcal{U}_{\kappa})_0$ is equivalent to the base $U_{\kappa}$. The object
$(\mathcal{U}_{\kappa})_1\simeq\mathrm{Fun}(\pi_{\kappa})$ is the ``generic function type'' associated to $\pi_{\kappa}$. The higher objects 
$(\mathcal{U}_{\kappa})_n$ are all equivalent to the pullback
$((\mathcal{U}_{\kappa})_1)_{/(\mathcal{U}_{\kappa})_0})^n:=(\mathcal{U}_{\kappa})_1\times_{(\mathcal{U}_{\kappa})_0}\dots\times_{(\mathcal{U}_{\kappa})_0}(\mathcal{U}_{\kappa})_1$ via their corresponding Segal maps.
That means, for all sufficiently large regular cardinals $\kappa$, there is an equivalence
\begin{align}\label{equextobjclass}
\mathrm{Ext}(\mathcal{U}_{\kappa}))\simeq t_{\kappa}
\end{align}
in the quasi-category $\mathrm{Cart}(\mathcal{B})$.
Indeed, the $\kappa$-small object classifier $\pi_{\kappa}$ in $\mathcal{B}$ is a univalent map which witnesses that the indexed 
$\infty$-category $(\mathcal{B}_{/\blank})_{\kappa}\colon\mathcal{B}^{op}\rightarrow\mathrm{Cat}_{\infty}$ is 
\emph{small} in the sense of \cite[Section 3]{rs_comp}. Local cartesian closedness of $\mathcal{B}$ hence implies the statement by 
\cite[Proposition 4.21.2]{rs_comp}. 

\begin{notation}\label{remdefB}
Given furthermore the model categorical presentation $\mathbb{B}$ of $\mathcal{B}$ from \cite[Proposition A.3.7.6]{luriehtt} (whose 
associated Bousfield localization of the injective model structure on the category $\mathbb{B}$ is left exact assuming
$\mathcal{B}$ an $\infty$-topos), there is regular cardinal $\mathfrak{b}$ such that we can represent the object 
classifiers $\pi_{\kappa}$ by the universal relatively $\kappa$-compact fibration
$p_{\kappa}\colon\tilde{V}_{\kappa}\twoheadrightarrow V_{\kappa}$ in $\mathbb{B}$ for every $\kappa\in\mathrm{Shl}(\mathfrak{b})$(\cite{shulmanuniverses}, \cite[Corollary 4.1]{rs_small} ). Here, note that the given cardinal bound $\mathfrak{b}$ itself 
is by definition sharply larger than the presentability index of both $\mathcal{B}$ and $\mathbb{B}$. Here, relative $\kappa$-compactness of 
a map in $\mathbb{B}$ refers to the according 1-categorical notion. For the following, let $\beta\colon\mathrm{Ord}\rightarrow\mathrm{Shl}(\mathfrak{b})$ be a strictly increasing enumeration of $\mathrm{Shl}(\mathfrak{b})$. 
\end{notation}

Let $\mathcal{F}_{\mathbb{B}}\subseteq\mathbb{B}^{\Delta^1}$ 
denote the full subcategory of the category of arrows in $\mathbb{B}$ spanned by the fibrations in $\mathbb{B}$.
Let $\mathcal{F}_{\mathbb{B}}^{\times}\subset\mathcal{F}_{\mathbb{B}}$ denote the wide subcategory spanned by the homotopy-cartesian squares 
between fibrations (which includes strictly cartesian squares between fibrations due to right-properness of $\mathbb{B}$).

\begin{lemma}\label{lemmafunctsequenceofclass}
Let $\mathbb{B}$ be a type theoretic model topos with associated notion of fibered structure $\mathbb{F}$
(\cite[Section 3]{shulmanuniverses}). Then there is a diagram
\[p_{\bullet}\colon \mathrm{Shl}(\mathfrak{b})\rightarrow\mathcal{F}_{\mathbb{B}}^{\times}\]
such that
\begin{enumerate}
\item each $p_{\beta_i}$ for $i\in \mathrm{Ord}$ is a univalent universe with fibrant base for the notion of fibered 
structure $\mathbb{F}_{\beta_i}$ spanned by the relatively $\beta_i$-compact fibrations in $\mathbb{B}$, and
\item each map $p_{i\leq j}\colon p_{\beta_i}\rightarrow p_{\beta_j}$ is strictly cartesian in $\mathbb{B}$ with a cofibration for base map.
\end{enumerate}
\end{lemma}
\begin{proof}
The proof uses the results about universes in \cite[Chapter 5]{shulmanuniverses} pointwise for every $\beta_i$, together with a suitable
ordinal recursion to achieve the functorial properties as stated in 1.\ and 2. Thus, in the following we show that for all ordinals $j$ 
there is a diagram 
\[(p_{\beta_{\bullet}})_{j}\colon j\rightarrow\mathcal{F}_{\mathbb{B}}^{\times}\]
which satisfies Properties 1.\ and 2., and such that for all pairs of ordinals $i\leq j$, we have
$(p_{\beta_{\bullet}})_i\subseteq(p_{\beta{\bullet}})_j$. By union we hence obtain a diagram $p_{\beta_{\bullet}}$ as stated.

Therefore, first we simply take the fibration $p_{\beta_0}\colon\tilde{U}_{\beta_0}\twoheadrightarrow U_{\beta}$ in
$\mathcal{F}_{\mathbb{B}}$ obtained by \cite[Theorem 5.22]{shulmanuniverses} (specifically via \cite[Proposition 5.12]{shulmanuniverses}). 
Whenever $j\in\mathrm{Ord}$ is a limit ordinal, we let $(p_{\beta_{\bullet}})_j\colon j\rightarrow\mathcal{F}_{\mathbb{B}}^{\times}$ be the 
union $\bigcup_{i<j}(p_{\beta_{\bullet}})_i$. For the successor case, suppose $j=k+1$ for some $k\in\mathrm{Ord}$ and we have constructed a 
sequence
\begin{align}\label{diaglemmafunctsequenceofclass1}
(p_{\beta_{\bullet}})_k\colon k\rightarrow\mathcal{F}_{\mathbb{B}}^{\times}
\end{align}
for every $i<j$ as stated. We denote the fibrations $(p_{\beta_i})_k$ simply by
$p_{\beta_i}\colon\tilde{U}_{\beta_i}\twoheadrightarrow U_{\beta_i}$. First, assume $k$ itself is a successor ordinal, so $k=i+1$ for 
some ordinal $i$.\footnote{This step is not really necessary, but decompresses the chain of arguments for the sake of readability.} In 
particular, we have a fibration $p_{\beta_{i}}\colon\tilde{U}_{\beta_{i}}\twoheadrightarrow U_{\beta_{i}}$ as 
in 1. To construct a suitable fibration $p_{\beta_{i+1}}\colon\tilde{U}_{\beta_{i+1}}\twoheadrightarrow U_{\beta_{i+1}}$, choose a universe 
$q_{\beta_{i+1}}\colon\tilde{V}_{\beta_{i+1}}\twoheadrightarrow V_{\beta_{i+1}}$ for the class of
$\beta_{i+1}$-relatively compact fibrations via \cite[Theorem 5.22]{shulmanuniverses}. By construction, it comes with an acyclic fibration
$\mathbb{B}(\blank,V_{\beta_{i+1}})\overset{\sim}{\twoheadrightarrow}\mathbb{F}_{\beta_{i+1}}$ (in the sense of
\cite[Definition 5.1]{shulmanuniverses}) in the (very large) 2-category $\mathcal{PSH}(\mathbb{B},\mathcal{GPD})$ of (large) stacks. Here, 
the codomain $\mathbb{F}_{\beta_{i+1}}$ is defined as the pullback
$\mathbb{F}\times_{\mathbb{B}_{/\blank}}(\mathbb{B}_{/\blank})_{\beta_{i+1}}$ where $\mathbb{F}$ is the associated notion of fibered 
structure of $\mathbb{B}$, and $(\mathbb{B}_{/\blank})_{\beta_{i+1}}$ is the fibered structure of relatively $\beta_{i+1}$-compact maps in
$\mathbb{B}$. The square 
\[\xymatrix{
\mathbb{B}(\blank,\emptyset)\ar@{^(->}[rr]\ar@{^(->}[d] & &  \mathbb{B}(\blank,V_{\beta_{i+1}})\ar@{->>}[d]^{\sim} \\
\mathbb{B}(\blank,U_{\beta_{i}})\ar@{->>}[r]^(.6){\sim} & \mathbb{F}_{\beta_{i}}\ar@{^(->}[r] & \mathbb{F}_{\beta_{i+1}}
}\]
thus admits a lift $b_{\beta_i}\colon U_{\beta_{i}}\rightarrow V_{\beta_{i+1}}$ which in turn gives a cartesian square
\[\xymatrix{
\tilde{U}_{\beta_{i}}\ar@{->>}[d]_{p_{\beta_{i}}}\ar[r] & \tilde{V}_{\beta_{i+1}}\ar@{->>}[d]^{q_{\beta_{i+1}}}\\
U_{\beta_{i}}\ar[r] & V_{\beta_{i+1}}
}\]
in $\mathbb{B}$ by evaluation in $\mathbb{F}_{\beta_{i+1}}$. Factoring the base map into a cofibration followed by a trivial fibration
$U_{\beta_{i}}\hookrightarrow U_{\beta_{i+1}}\overset{\sim}{\twoheadrightarrow}V_{\beta_{i+1}}$ induces a cartesian square
\begin{align}\label{diaglemmafunctsequenceofclass2}
\begin{gathered}
\xymatrix{
\tilde{U}_{\beta_{i}}\ar@{->>}[d]_{p_{\beta_{i}}}\ar[r] & \tilde{U}_{\beta_{i+1}}\ar@{->>}[d]^{p_{\beta_{i+1}}}\\
U_{\beta_{i}}\ar@{^(->}[r] & U_{\beta_{i+1}}
}
\end{gathered}
\end{align}
by pullback of $q_{\beta_{i+1}}$ along $U_{\beta_{i+1}}\overset{\sim}{\twoheadrightarrow}V_{\beta_{i+1}}$ with a cofibration as base. The 
base $U_{\beta_{i+1}}$ is again fibrant by construction, the fibration $p_{\beta_{i+1}}$ is equivalent to $q_{\beta_{i+1}}$ and hence again 
univalent. We obtain a composite acyclic fibration
$\mathbb{B}(\blank,U_{\beta_{i+1}})\overset{\sim}{\twoheadrightarrow}\mathbb{B}(\blank,V_{\beta_{i+1}})\overset{\sim}{\twoheadrightarrow}\mathbb{F}_{\beta_{i+1}}$. We thus may define
\[(p_{\beta_{\bullet}})_{i+2}\colon i+2\rightarrow\mathcal{F}_{\mathbb{B}}^{\times}\]
by extension of (\ref{diaglemmafunctsequenceofclass1}) by (\ref{diaglemmafunctsequenceofclass2}) in the obvious way. This finishes the case 
if $k$ itself is a successor ordinal.

Given that $k$ is a limit ordinal, we have defined $(p_{\beta_{\bullet}})_k=\bigcup_{i<k}(p_{\beta_{\bullet}})_i$. We note that the colimit
\begin{align}\label{equlemmafunctsequenceofclass1}
\mathrm{colim}_{i<k}p_{\beta_{i}}\colon\mathrm{colim}_{i<k}\tilde{U}_{\beta_{i}}\rightarrow\mathrm{colim}_{i<k}U_{\beta_{i}}
\end{align}
is a relatively $\beta_{k}$-compact fibration. That is, because the natural transformation
$(p_{\beta_{\bullet}})_k\colon k\rightarrow\mathbb{M}$ is cartesian and each $p_{\beta_{i_1\leq i_2}}$ is a monomorphism by assumption; it 
follows that each cocone map $p_i\rightarrow\mathrm{colim}_{i<k}p_{\beta_{i}}$ is cartesian by the descent properties of Grothendieck toposes 
(\cite[Proposition 2.4.2]{rezkhtytps}). Thus, relative $\beta_{k}$-compactness of the colimit (\ref{equlemmafunctsequenceofclass1}) follows 
readily. It also follows that each base map $U_{\beta_i}\rightarrow \mathrm{colim}_{i<k}U_{\lambda_{i}}$ is monic again. Fibrancy follows 
similarly in virtue of local representability of the fibered structure $\mathbb{F}$ (\cite[Proposition 3.18]{shulmanuniverses}). Now, let 
$q_{\beta_{k}}\colon\tilde{V}_{\beta_{k}}\twoheadrightarrow V_{\beta_{k}}$ again be a universe for the class of relatively
$\beta_{k}$-compact fibrations via \cite{shulmanuniverses}. We thus 
obtain a classifying cartesian square
\[\xymatrix{
\mathrm{colim}_{i<k}\tilde{U}_{\beta_{i}}\ar@{->>}[d]_{\mathrm{colim}_{i<k}p_{\beta_{i}}}\ar[r] & \tilde{V}_{\beta_{k}}\ar@{->>}[d]^{q_{\beta_{k}}}\\
\mathrm{colim}_{i<k}U_{\beta_{i}}\ar[r] & V_{\beta_{k}}
}\]
in $\mathbb{M}$. We again factor the base map into a cofibration $\mathrm{colim}_{i<k}U_{\beta_{i}}\hookrightarrow U_{\beta_k}$ 
followed by an acyclic fibration $U_{\beta_k}\twoheadrightarrow V_{\beta_{k}}$. Then pullback of $q_{\beta_k}$ along 
$U_{\beta_k}\twoheadrightarrow V_{\beta_{k}}$ yields a relatively $\beta_k$-compact fibration
\[p_{\beta_k}\colon\tilde{U}_{\beta_k}\twoheadrightarrow U_{\beta_k}\]
with fibrant base which satisfies 1. It extends the diagram $(p_{\beta_{\bullet}})_k\colon k\rightarrow\mathcal{F}_{\mathbb{B}}^{\times}$ to 
a diagram $(p_{\bullet})_{k+1}\colon k+1\rightarrow\mathcal{F}_{\mathbb{B}}^{\times}$ such that 2.\ holds by construction. This finishes the 
induction.

\end{proof}

With Lemma~\ref{lemmafunctsequenceofclass} at hand, we may consider the composition
\begin{align}\label{equdefobjclasssequence}
N(p_{\bullet})\colon\mathrm{Shl}(\mathfrak{b})\xrightarrow{p_{\bullet}}\mathcal{F}_{\mathbb{B}}^{\times}\xrightarrow{N}\mathbb{B}^{\Delta^{op}},
\end{align}
where $N$ is the internal nerve construction from \cite[Section 2]{rs_uc}. Reedy fibrant replacement of the composition (\ref{equdefobjclasssequence}) 
yields a diagram 
\[\mathbb{U}_{\bullet}:=\mathbb{R}(N(p_{\bullet}))\colon \mathrm{Shl}(\mathfrak{b})\rightarrow\mathbf{Cat}_{\infty}(\mathbb{B})\]
which itself is not only Reedy fibrant, but Reedy cofibrant as well. That is, first, because $N(p_{\bullet})$ is Reedy cofibrant by Lemma~\ref{lemmadirectedmc}, 
Lemma~\ref{lemmafunctsequenceofclass} and the fact that $N$ preserves monomorphisms. And second, because there is an acyclic cofibration
$N(p_{\bullet})\overset{\sim}{\hookrightarrow}\mathbb{R}(N(p_{\bullet}))$ in
$\mathrm{Fun}(\mathrm{Shl}(\mathfrak{b}),\mathbb{B}^{\Delta^{op}})$ equipped with the Reedy model structure by construction. Thus, each value
$\mathbb{U}_{\kappa}\in\mathbb{B}^{\Delta^{op}}$ is hence in particular Reedy fibrant and Reedy equivalent to the strict Segal object 
$N(p_{\kappa})$; it follows that each $\mathbb{U}_{\kappa}$ is a complete Segal object in $\mathbb{B}$ by \cite[Section 4]{rs_uc}. The 
diagram $\mathbb{U}_{\bullet}$ is therefore an object in the $\infty$-cosmos
$\mathbf{Fun}(\mathrm{Shl}(\mathfrak{b}),\mathbf{Cat}_{\infty}(\mathbb{B}))$.

For each $\kappa\in \mathrm{Shl}(\mathfrak{b})$, the equivalence $p_{\kappa}\simeq\pi_{\kappa}$ in $\mathcal{B}$, the equivalences
$\mathbb{U}_{\kappa}\simeq N(p_{\kappa})\simeq\mathcal{U}_{\kappa}$ in $\mathrm{Fun(\Delta^{op},\mathcal{B})}$, and 
Proposition~\ref{propcosmicext} induce an equivalence
\[\mathrm{Ext}(\mathbb{U}_{\kappa})\simeq \mathrm{Ext}(\mathcal{U}_{\kappa})\]
in $\mathrm{Cart}(\mathcal{B})$. Via (\ref{equextobjclass}) we obtain a composite equivalence 
\begin{align}\label{equtargetextptwequivs}
e_{\kappa}\colon\mathrm{Ext}(\mathbb{U}_{\kappa})\xrightarrow{\simeq}t_{\kappa}
\end{align}
in $\mathbf{Cart}(\mathcal{B})$. In the next lemma, we show that these equivalences can be chosen so to assemble into a natural equivalence 
indexed by $\mathrm{Shl}(\mathfrak{b})$.

\begin{remark}[Univalent $n$-cells]
Under Lemma~\ref{lemmaextyoneda}, the equivalences (\ref{equtargetextptwequivs}) correspond to elements
\[\gamma_{\kappa}\in\int\limits_{n\in\Delta^{op}}\left((\mathcal{B}_{/(\mathcal{U}_{\kappa})_n})^{\Delta^n}\right)^{\simeq}.\]
These may be informally thought of as tuples
$\gamma_{\kappa}=(e_{\kappa}(1_{(\mathcal{U}_{\kappa})_n})\in(\mathcal{B}_{/(\mathcal{U}_{\kappa})_n})^{\Delta^n}|n\geq 0)$
of homotopy-coherent $n$-simplices in $\mathcal{B}_{/(\mathcal{U}_{\kappa})_n}$, and be depicted as diagrams of the form
\[\xymatrix{
& & &  \\
(\mathrm{Fun}(\pi_{\kappa})/U_{\kappa})^n & (\mathcal{U}_{\kappa})_n\ar[l]_(.4){\simeq}\ar@{-->}@/^1pc/[dd]\ar@{}[u]|{\vdots} & \dots(\gamma_{\kappa})_n\dots\ar[l]& \\
 & & \{0\}^{\ast}(\tilde{U}_{\kappa})\ar[dr]^{d_2^{\ast}(\gamma_{\kappa})_1}\ar[dd]_{d_1^{\ast}(\gamma_{\kappa})_1}\ar@/_1pc/[dl] & \\
\mathrm{Fun}(\pi_{\kappa})\times_{U_{\kappa}}\mathrm{Fun}(\pi_{\kappa}) & (\mathcal{U}_{\kappa})_2\ar[l]_(.35){\simeq}\ar@{-->}@/^1pc/[uu]\ar@<2ex>[dd]\ar[dd]\ar@<-2ex>[dd] & & \{1\}^{\ast}(\tilde{U}_{\kappa})\ar[dl]^{d_0^{\ast}(\gamma_{\kappa})_1}\ar@/^1.5pc/[ll]|(.5)\hole \\
  & & \{2\}^{\ast}(\tilde{U}_{\kappa})\ar@/^/[ul]\ar@{}[uur]^(.45){(\gamma_{\kappa})_2} & \\
& \mathrm{Fun}(\pi_{\kappa})\ar@<1ex>[dd]\ar@<-1ex>[dd]\ar@<1ex>[uu]\ar@<-1ex>[uu] & d_1^{\ast}(\tilde{U}_{\kappa})\ar[l]\ar@/^/[d]^{(\gamma_{\kappa})_1}\\
& & d_0^{\ast}(\tilde{U}_{\kappa})\ar@/^/[ul]\\
& U_{\kappa}\ar[uu] & \tilde{U}_{\kappa}\ar[l]^{\pi_{\kappa}\simeq(\gamma_{\kappa})_0} 
}\]
such that for all simplicial operators $f\colon [n]\rightarrow [m]$ we have (suitably cohering) homotopies
$\mathcal{U}_{\kappa}(f)^{\ast}((\gamma_{\kappa})_n)\simeq f^{\ast}(\gamma_{\kappa})_m$. Each diagram
$(\gamma_{\kappa})_n\colon\Delta^n\rightarrow\mathcal{B}_{/(\mathcal{U}_{\kappa})_n}$ is the image of the identity
$1_{(\mathcal{U}_{\kappa})_n}\colon(\mathcal{U}_{\kappa})_n\rightarrow(\mathcal{U}_{\kappa})_n$ under the equivalence
\[\xymatrix{
\mathcal{B}_{/(\mathcal{U}_{\kappa})_n}\ar[rr]^{(e_{\kappa})_n}_{\simeq}\ar@/_1pc/@{->>}[dr] & & (t_{\kappa}^{\Delta^n})^{\times}\ar@/^1pc/@{->>}[dl]\\
 & \mathcal{B} & 
}\]
of right fibrations over $\mathcal{B}$. Here, the codomain $(t_{\kappa}^{\Delta^n})^{\times}$ denotes the largest right fibration
(i.e.\ the core) of the cotensor $t_{\kappa}^{\Delta^n}$ in $\mathrm{Cart}(\mathcal{B})$.  In particular, each such
$(\gamma_{\kappa})_n\in(t_{\kappa}^{\Delta^n})^{\times}$ is a terminal object (in the corresponding total $\infty$-category). In that 
sense, the section $\gamma_{\kappa}$ is a 
simplicial diagram of ``univalent'' $n$-cells $(\gamma_{\kappa})_n$ in $\mathcal{B}$, each of which represents the presheaf of diagrams 
of the form $B\mapsto((\mathcal{B}_{/B})^{\Delta^n})^{\simeq}$ in just the same way the object classifier
$\pi_{\kappa}\simeq(\gamma_{\kappa})_0$ represents the presheaf of $0$-cells $B\mapsto(\mathcal{B}_{/B})^{\simeq}$.
\end{remark}

\begin{proposition}\label{corextobjclass}
Suppose $\mathcal{B}$ is an $\infty$-topos. Let $p_{\bullet}\colon \mathrm{Shl}(\mathfrak{b})\rightarrow\mathcal{F}_{\mathbb{B}}^{\times}$ be 
the sequence of universal fibrations in $\mathbb{B}$ as constructed in Lemma~\ref{lemmafunctsequenceofclass}. 
Then there is a natural equivalence
\[\mathrm{Ext}_{\ast}(\mathbb{U}_{\bullet}))\simeq t_{\bullet}\]
in $\mathbf{Fun}(\mathrm{Shl}(\mathfrak{b}),\mathbf{Cart}(\mathcal{B}))$.
\end{proposition}
\begin{proof}
In the following, we construct a homotopy-coherent natural equivalence
$\mathrm{Ext}_{\ast}(\mathcal{U}_{\bullet})\simeq (\mathcal{B}_{/\blank})_{\bullet}$ of $\mathrm{Shl}(\mathfrak{b})$-indexed diagrams in
$\mathrm{Fun}(\mathcal{B}^{op},\mathrm{Cat}_{\infty})$, and then use general arguments about the 
strictification of such homotopy-coherent diagrams to obtain a (strict) natural equivalence of indexed cartesian fibrations as stated. We 
therefore will make extensive use of Lemma~\ref{lemmaextyoneda}. Yet, to avoid the daring task of constructing a path between fixed endpoints 
in the space $\int_{n\in\Delta^{op}}((\mathcal{B}_{/(\mathcal{U}_{\kappa})_n})^{\Delta^n})^{\simeq}$ by hand, we will exploit
the 1-categorical constructions we have at hand.

With that in mind, let us first show that for all cardinals $\beta_i\in\mathrm{Shl}(\mathfrak{b})$ there is a 
specific section $\gamma_{\beta_i}\in\int_{n\in\Delta^{op}}((\mathcal{B}_{/(\mathcal{U}_{\beta_i})_n})^{\Delta^n})^{\simeq}$ such 
that for all $i\leq j$, we have $\gamma_{\beta_i}\simeq\iota_{ij}^{\ast}\gamma_{\beta_j}$ in
$\int_{n\in\Delta^{op}}((\mathcal{B}_{/(\mathcal{U}_{\beta_i})_n})^{\Delta^n})^{\simeq}$ where
$\iota_{ij}\colon\mathcal{U}_{\beta_i}\rightarrow\mathcal{U}_{\beta_j}$ denotes the canonical embedding induced by the 
inclusion $U_{\beta_i}\hookrightarrow U_{\beta_j}$ in Lemma~\ref{lemmafunctsequenceofclass}.

Therefore, we recall that the fibrant replacement $\rho_{\bullet}\colon N(p_{\bullet})\xrightarrow{\sim}\mathbb{U}_{\bullet}$ induces a 
pointwise fibrant replacement $\rho_{\beta_i}\colon N(p_{\beta_i})\xrightarrow{\sim}\mathbb{U}_{\beta_i}$ in $\mathbb{B}^{\Delta^{op}}$ for 
each $i\in\mathrm{Ord}$. For all pairs $i\leq j$ we thus obtain commutative squares in $\mathcal{B}^{\Delta^{op}}$ as 
follows.
\[\xymatrix{
N(p_{\beta_i})\ar[r]^{\rho_{\beta_i}}\ar@{^(->}[d]_{\iota_{ij}} & \mathbb{U}_{\beta_i}\ar[d]^{\mathbb{R}(\iota_{ij})} \\
N(p_{\beta_j})\ar[r]_{\rho_{\beta_j}} & \mathbb{U}_{\beta_j}
}\]
Postcomposition of any simplicial object $X\in\mathbb{B}^{\Delta^{op}}$ with the simplicially enriched slice functor
$\mathbb{B}_{/\blank}\colon\mathbb{B}^{op}\rightarrow\mathbf{Cat}_{\Delta}$ yields a cosimplicial diagram
$\mathbb{B}_{/X}\colon\Delta\rightarrow\mathbf{Cat}_{\Delta}$ of simplicial model categories. The natural equivalence
$\rho_{\kappa}\colon N(p_{\beta_i})\xrightarrow{\sim}\mathbb{U}_{\beta_i}$  induces a pointwise Quillen right adjoint
$(\rho_{\beta_i})^{\ast}\colon\mathbb{B}_{/\mathbb{U}_{\beta_i}}\xrightarrow{\sim}\mathbb{B}_{/N(p_{\beta_i})}$, and hence, as all objects 
in both model categories are cofibrant (and the cosimplicial operators in each case are given by pullback action as well and so preserve 
fibrant objects), a natural equivalence
\[(\rho_{\beta_i})^{\ast}\colon(\mathbb{B}_{/\mathbb{U}_{\beta_i}})_{\mathrm{cf}}\xrightarrow{\sim}(\mathbb{B}_{/N(p_{\beta_i})})_{\mathrm{cf}}\]
between cosimplicial diagrams of fibrant objects in the model category of simplicial categories. As the simplicial nerve functor $N_{\Delta}\colon\mathbf{Cat}_{\Delta}\rightarrow(\mathbf{S},\mathrm{QCat})$ is a right Quillen functor itself, we obtain a natural equivalence
$N_{\Delta}(\rho_{\beta_i}^{\ast})\colon N_{\Delta}((\mathbb{B}_{/\mathbb{U}_{\beta_i}})_{\mathrm{cf}})\xrightarrow{\sim}N_{\Delta}((\mathbb{B}_{/N(p_{\beta_i})})_{\mathrm{cf}})$
of cosimplicial objects in quasi-categories. Since the Joyal model structure is cartesian closed, this induces a natural equivalence
\[N_{\Delta}(\rho_{\beta_i}^{\ast})^{\Delta^{\bullet}}\colon N_{\Delta}((\mathbb{B}_{/\mathbb{U}_{\beta_i}})_{\mathrm{cf}})^{\Delta^{\bullet}}\xrightarrow{\sim}N_{\Delta}((\mathbb{B}_{/N(p_{\beta_i})})_{\mathrm{cf}})^{\Delta^{\bullet}}\]
of functors of type $\Delta^{op}\times\Delta\rightarrow\mathrm{QCat}$, and eventually, by precomposition with the discrete fibration
$\mathrm{Tw}(\Delta)\twoheadrightarrow\Delta^{op}\times\Delta$ of twisted arrows on the one hand, and postcomposition with the core functor $(\cdot)^{\simeq}$ on the other hand, a natural equivalence of functors of type $\mathrm{Tw}(\Delta)\rightarrow\mathrm{Kan}$. This natural 
equivalence eventually induces a homotopy equivalence
\begin{align}\label{equcorextobjclass1}
\underset{\mathrm{Tw}(\Delta)}{\mathrm{holim}}\left(N_{\Delta}(\rho_{\beta_i}^{\ast})^{\Delta^{\bullet}}\right)^{\simeq}\colon\underset{\mathrm{Tw}(\Delta)}{\mathrm{holim}}\left(N_{\Delta}\left((\mathbb{B}_{/\mathbb{U}_{\beta_i}})_{\mathrm{cf}}\right)^{\Delta^{\bullet}}\right)^{\simeq}\xrightarrow{\sim}\underset{\mathrm{Tw}(\Delta)}{\mathrm{holim}}\left(N_{\Delta}\left((\mathbb{B}_{/N(p_{\beta_i})})_{\mathrm{cf}}\right)^{\Delta^{\bullet}}\right)^{\simeq}
\end{align}
of homotopy limits in the model category $(\mathbf{S},\mathrm{Kan})$ (say, by a functorial choice of injective fibrant replacements in the 
according model category of diagrams). We will refer to this equivalence short-hand by
\[\mathrm{htot}(\rho_{\beta_i}^{\ast})\colon\mathrm{htot}(\mathbb{B}_{/\mathbb{U}_{\beta_i}})\xrightarrow{\sim}\mathrm{htot}(\mathbb{B}_{/N p_{\beta_i}}).\]
Now, on the one hand, the domain of this equivalence presents the limit $\mathrm{lim}_{\mathrm{Tw}(\Delta)}((\mathcal{B}_{/\mathcal{U}_{\beta_i}})^{\Delta^{\bullet}})^{\simeq}$ in $\mathcal{S}$ and as such the end 
$\int_{n\in\Delta^{op}}((\mathcal{B}_{/(\mathcal{U}_{\beta_i})_n})^{\Delta^n})^{\simeq}$ via Remark~\ref{remaltslice}.
On the other hand, we have a sequence of canonical maps into the codomain as follows.
\begin{align}\label{equcorextobjclass2}
\notag N_{\Delta}\left(\underset{\mathrm{Tw}(\Delta)}{\mathrm{lim}}(\mathbb{B}_{/N(p_{\beta_i})})_{\mathrm{cf}}^{\Delta^{\bullet}}\right) & \xrightarrow{\cong} 
\underset{\mathrm{Tw}(\Delta)}{\mathrm{lim}}N_{\Delta}\left((\mathbb{B}_{/N(p_{\beta_i})})_{\mathrm{cf}}^{\Delta^{\bullet}}\right) \\
\notag & \rightarrow \underset{\mathrm{Tw}(\Delta)}{\mathrm{holim}}N_{\Delta}\left((\mathbb{B}_{/N(p_{\beta_i})})_{\mathrm{cf}}^{\Delta^{\bullet}}\right) \\
\notag & \rightarrow \underset{\mathrm{Tw}(\Delta)}{\mathrm{holim}}N_{\Delta}\left((\mathbb{B}_{/N(p_{\beta_i})})_{\mathrm{cf}}^{\mathfrak{C}(\Delta^{\bullet})}\right) \\
\notag & \rightarrow \underset{\mathrm{Tw}(\Delta)}{\mathrm{holim}}N_{\Delta}\left((\mathbb{B}_{/N(p_{\beta_i})})_{\mathrm{cf}}\right)^{N_{\Delta}\mathfrak{C}(\Delta^{\bullet})} \\
\notag & \xrightarrow{\sim} \underset{\mathrm{Tw}(\Delta)}{\mathrm{holim}}N_{\Delta}\left((\mathbb{B}_{/N(p_{\beta_i})})_{\mathrm{cf}}\right)^{\Delta^{\bullet}} \\
 & = \mathrm{htot}(\mathbb{B}_{/N p_{\beta_i}})
\end{align}
The limit $\mathrm{tot}(\mathbb{B}_{/N(p_{\beta_i})}):=\mathrm{lim}_{\mathrm{Tw}(\Delta)}((\mathbb{B}_{/N(p_{\beta_i})})_{\mathrm{cf}}^{\Delta^{\bullet}})^{\simeq}$ is an ordinary end in 
the category $\mathrm{Cat}_{\Delta}$ of simplicial categories. The fibration
$(\gamma_{\beta_i})_0:= p_{\beta_i}\colon\tilde{U}_{\beta_i}\twoheadrightarrow U_{\beta_i}$ together with 
the (strictly) universal arrow
$(\gamma_{\beta_i})_1:=\ulcorner 1_{\mathrm{Fun}(p_{\beta_i})}\urcorner\colon s^{\ast}\tilde{U}_{\beta_i}\rightarrow t^{\ast}\tilde{U}_{\beta_i}$ over $\mathrm{Fun}(p_{\kappa})$ between relatively $\beta_i$-compact fibrations yields a diagram
\[((\gamma_{\beta_i})_1)/(\gamma_{\beta_i})_0)^n\colon \Delta^n\rightarrow\mathbb{B}_{/N(p_{\beta_i})_n}\]
for every $n\geq 0$ by pullback along the according projections $N(p_{\beta_i})_n\rightarrow\mathrm{Fun}(p_{\beta_i})$. These diagrams are 
easily seen to assemble into a section
\[((\gamma_{\beta_i})_1)/(\gamma_{\beta_i})_0)^{\bullet}\in\mathrm{tot}(\mathbb{B}_{/N(p_{\beta_i})}).\]
Mapping this section into the homotopy limit along the sequence of maps (\ref{equcorextobjclass2}) yields a homotopically unique section
\[\gamma_{\beta_i}\in\mathrm{htot}(\mathbb{B}_{/\mathbb{U}_{\beta_i}})\]
by the homotopy equivalence (\ref{equcorextobjclass1}).

Then, for all $i\leq j$, we obtain a commutative diagram in $\mathbb{B}$ as follows.
\begin{align}\label{diagcorextobjclass1}
\begin{gathered}
\xymatrix{
 & \mathrm{htot}(\mathbb{B}_{/\mathbb{U}_{\beta_j}})\ar[rr]_{\sim}^{\mathrm{htot}(\rho_{\beta_j}^{\ast})}\ar[dl]_{\mathbb{R}(\iota_{ij})^{\ast}} & & \mathrm{htot}(\mathbb{B}_{/N\pi_{\beta_j}})\ar[dl]^{\iota_{ij}^{\ast}} \\
\mathrm{htot}(\mathbb{B}_{/\mathbb{U}_{\beta_i}})\ar[rr]_{\sim}^{\mathrm{htot}(\rho_{\beta_i}^{\ast})} & & \mathrm{htot}(\mathbb{B}_{/N\pi_{\beta_i}}) & \\
 & & & N_{\Delta}(\mathrm{tot}(\mathbb{B}_{/N(p_{\beta_j})}))\ar[uu]\ar[dl]_{\iota_{ij}^{\ast}} \\
 & & N_{\Delta}(\mathrm{tot}(\mathbb{B}_{/N(p_{\beta_i})}))\ar[uu]  & 
}
\end{gathered}
\end{align}
In virtue of the fact that $p_{\beta_i}\cong\iota_{ij}^{\ast}p_{\beta_j}$ in $\mathbb{B}$ as well as in virtue of strict 
universality of the maps $(\gamma_{\beta_k})_1\colon s^{\ast}\tilde{U}_{\beta_k}\rightarrow t^{\ast}\tilde{U}_{\beta_k}$ for $k=i,j$, it 
follows that the sections $\iota_{ij}^{\ast}(((\gamma_{\beta_j})_1)/(\gamma_{\beta_j})_0)^{\bullet})$ and
$((\gamma_{\beta_i})_1)/(\gamma_{\eta_i})_0)^{\bullet}$ are isomorphic in the simplicial category
$\mathrm{tot}(\mathbb{B}_{/N(p_{\beta_i})})$. It follows from commutativity of 
Diagram~(\ref{diagcorextobjclass1}) that $\iota_{ij}^{\ast}(\gamma_{\beta_j})\simeq\gamma_{\beta_i}$ in the space
$\mathrm{htot}(\mathbb{B}_{/\mathbb{U}_{\beta_i}})\simeq\int_{n\in\Delta^{op}}((\mathcal{B}_{/(\mathcal{U}_{\beta_i})_n})^{\Delta^n})^{\simeq}$ as we had claimed.

Thus, by Lemma~\ref{lemmaextyoneda}, for every pair $i\leq j$ of ordinals we obtain a homotopy-commutative square of the form
\begin{align}\label{diagcorextobjclass2}
\begin{gathered}
\xymatrix{
\mathrm{Ext}(\mathbb{U}_{\beta_i})\ar[r]_(.5){e_{\beta_i}}^(.7){}\ar@{^(->}[d]_{\mathrm{Ext}(\mathbb{R}(\iota_{ij}))} & (\mathcal{B}_{/\blank})_{\beta_i} \ar@{^(->}[d]^{} \\
\mathrm{Ext}(\mathbb{U}_{\beta_j})\ar[r]_(.5){e_{\beta_j}}^(.7){} & (\mathcal{B}_{/\blank})_{\beta_j}
}
\end{gathered}
\end{align}
in $\mathrm{Fun}(\mathcal{B}^{op},\mathrm{Cat}_{\infty})$. More precisely, these in fact are the 0-th rows of squares 
\begin{align*}
\begin{gathered}
\xymatrix{
\mathrm{Fun}(\blank,\mathbb{U}_{\beta_i})\ar[r]_(.5){e_{\beta_i}}^(.7){}\ar@{^(->}[d]_{y(\mathbb{R}(\iota_{ij}))} & (((\mathcal{B}_{/\blank})_{\beta_i})^{\Delta^{\bullet}})^{\simeq} \ar@{^(->}[d]^{} \\
\mathrm{Fun}(\blank,\mathbb{U}_{\beta_j})\ar[r]_(.5){e_{\beta_j}}^(.7){} & (((\mathcal{B}_{/\blank})_{\beta_j})^{\Delta^{\bullet}})^{\simeq}
}
\end{gathered}
\end{align*}
in $\mathrm{Fun}(\mathcal{B}^{op},\mathrm{Cat}_{\infty}(\mathcal{S}))$.
For $k\in\{i,j\}$, the natural transformation
\[(e_{\beta_k})_0\colon\mathcal{B}(\blank,U_{\beta_k})\rightarrow(\mathcal{B}_{/\blank})_{\beta_k}^{\simeq}\]
of $0$-th columns is determined by $1_{U_{\beta_k}}\mapsto p_{\beta_k}$ and the natural transformation
\[(e_{\beta_k})_1\colon\mathcal{B}(\blank,\mathrm{Fun}(p_{\beta_k}))\rightarrow(\mathcal{B}_{/\blank})_{\beta_k}^{\Delta^1})^{\simeq}\]
of first columns is determined by $1_{\mathrm{Fun}(p_{\beta_k})}\mapsto \ulcorner_{\mathrm{Fun}(p_{\beta_k})}\urcorner$. Both of these are 
natural equivalences of presheaves by virtue of universality of each. Furthermore, for all $n\geq 2$ the squares
\[\xymatrix{
\mathcal{B}(\blank,(\mathbb{U}_{\beta_k})_n)\ar[r]^{(e_{\beta_k})_n}\ar[d]^{\sim} & (\mathcal{B}_{/\blank})_{\beta_k}^{\Delta^n})^{\simeq}\ar[d]^{\sim}\\
(\mathcal{B}(\blank,\mathrm{Fun}(p_{\beta_k}))/\mathcal{B}(\blank,U_{\beta_k}))^n\ar[r]^(.65){(e_{\beta_k})_1^n}_(.65){\sim} & (\mathcal{B}_{/\blank})_{\beta_k}^{S^n})^{\simeq}
}\]
commute by naturality of $e_{\beta_k}$, and so it follows that all natural transformations $(e_{\beta_k})_n$ of presheaves are natural 
equivalences. It follows that the natural transformations $e_{\beta_k}$ are pointwise (in $\mathcal{B}$) equivalences of complete Segal 
spaces, and hence induce equivalences of quasi-categories after postcomposition with the first row
$(\cdot)_{\bullet 0}\colon\mathrm{Cat}_{\infty}(\mathcal{S})\rightarrow\mathrm{Cat}_{\infty}$. In particular, the cartesian morphisms
\[e_{\beta_k}\colon\mathrm{Ext}(\mathbb{U}_{\beta_k})\rightarrow (\mathcal{B}_{/\blank})_{\beta_k}\]
in (\ref{diagcorextobjclass2}) are equivalences themselves. By Unstraightening we eventually obtain homotopy-commutative squares
\begin{align}\label{diagcorextobjclass3}
\begin{gathered}
\xymatrix{
\mathrm{Ext}(\mathbb{U}_{\beta_i})\ar[r]_(.6){e_{\beta_i}}^(.7){}\ar@{^(->}[d]_{\mathrm{Ext}(\mathbb{R}(\iota_{ij}))} & t_{\beta_i} \ar@{^(->}[d]^{} \\
\mathrm{Ext}(\mathbb{U}_{\beta_j})\ar[r]_(.6){e_{\beta_j}}^(.7){} & t_{\beta_j}
}
\end{gathered}
\end{align}
in $\mathrm{Cart}(\mathcal{B})$.
We can strictify the squares (\ref{diagcorextobjclass3}) to obtain a natural transformation
\[\mathrm{Ext}(\mathbb{U}_{\bullet})\xrightarrow{\simeq}t_{\bullet}\]
in $\mathbf{Fun}(\mathrm{Shl}(\mathfrak{b}),\mathbf{Cart}(\mathcal{B}))$ now directly by abstract properties of the homotopy category functor. Namely, the equivalences $e_{\beta_i}$ together with the individual squares (\ref{diagcorextobjclass3}) yield a natural isomorphism $e_{\bullet}\colon\mathrm{Ext}(\mathbb{U}_{\bullet})\rightarrow t_{\bullet}$ in the functor category
$\mathrm{Fun}(\mathrm{Shl}(\mathfrak{b}),\mathrm{Ho}(\mathrm{Cart}(\mathcal{B}))$. But the canonical functor 
\[\mathrm{Ho}(\mathrm{Fun}(\mathrm{Shl}(\mathfrak{b}),\mathrm{Cart}(\mathcal{B}))\rightarrow\mathrm{Fun}(\mathrm{Shl}(\mathfrak{b}),\mathrm{Ho}(\mathrm{Cart}(\mathcal{B}))\]
is smothering, since $\mathrm{Shl}(\mathfrak{b})\cong\mathrm{Ord}$ is the free category on its underlying graph
(see e.g.\ \cite[Lemma 3.1.4]{riehlverityelements}). In particular, the functor is conservative. It follows that there is an  
equivalence between the two diagrams in the homotopy category $\mathrm{Ho}(\mathrm{Fun}(\mathrm{Shl}(\mathfrak{b}),\mathrm{Cart}(\mathcal{B}))$. As both diagrams are bifibrant, the equivalence can be lifted to a homotopy equivalence between $\mathrm{Ext}(\mathbb{U}_{\bullet})$ and $t_{\bullet}$ in $\mathbf{Fun}(\mathrm{Shl}(\mathfrak{b}),\mathbf{Cart}(\mathcal{B}))$ as was to show.
\end{proof}


We proceed to show that accessible cartesian reflective localizations of the target fibration yield Segal-representable sequences as well.
Therefore, given a left exact and accessible reflective localization $\rho\colon t\rightarrow\mathcal{E}$ in $\mathbf{Cart}(\mathcal{B})$ 
over an $\infty$-topos $\mathcal{B}$, let $a\colon\mathbb{B}\rightarrow\mathbb{E}$ be its associated left exact Bousfield localization via 
Theorem~\ref{thmlocmodopacc}. We denote the corresponding functorial fibrant replacement functor by $\mathbb{R}^{\rho}\colon\mathbb{B}^{[1]}\rightarrow\mathbb{B}^{[1]}$; as the localization is left exact, the functor restricts to
\[\mathbb{R}^{\rho}\colon\mathcal{F}_{\mathbb{B}}^{\times}\rightarrow\mathcal{F}_{\mathbb{E}}^{\times}.\]
On the one hand, applied to the sequence $p_{\bullet}$ of universal fibrations from Lemma~\ref{lemmafunctsequenceofclass} we obtain a 
sequence $\mathbb{R}^{\rho}(p_{\bullet})\colon \mathrm{Shl}(\mathfrak{b})\rightarrow\mathcal{F}_{\mathbb{E}}^{\times}$. This sequence is 
technically less well behaved than $p_{\bullet}$ itself, because its bases are not fibrant in $\mathbb{E}$ and because it furthermore 
consists of (merely) homotopy-cartesian squares only; these squares thus generally do not induce a cofibrant 
$\mathrm{Shl}(\mathfrak{b})$-indexed nerve in $\mathbf{Cat}_{\infty}(\mathbb{B})$. Therefore, on the other hand, we may apply Lemma~\ref{lemmafunctsequenceofclass} directly to the type
theoretic model topos $\mathbb{E}$ to obtain a sequence
\[(q_{\bullet}\colon\tilde{V}_{\bullet}\twoheadrightarrow V_{\bullet})\colon \mathrm{Shl}(\mathfrak{e})\rightarrow\mathcal{F}_{\mathbb{E}}^{\times}\]
satisfying the properties listed in the lemma (which we will consider canonically embedded into the model category $\mathbb{B}$). Here, 
$\mathfrak{e}\in\mathrm{Shl}(\mathfrak{b})$ is the regular cardinal such that, without loss of generality, the derived Bousfield 
localization $(\mathbb{L}a,\mathbb{R}\iota)\colon\mathbb{B}\rightarrow\mathbb{E}$ is $\mathfrak{e}$-accessible 
both 1-categorically as well as on underlying quasi-categories. As the sharply larger relation on regular cardinals is transitive, we have 
$\mathrm{Shl}(\mathfrak{e})\subseteq \mathrm{Shl}(\mathfrak{b})$.

\begin{proposition}\label{corextobjclassloc}
Suppose $\mathcal{B}$ is an $\infty$-topos. Let $p_{\bullet}\colon\mathrm{Shl}(\mathfrak{b})\rightarrow\mathcal{F}_{\mathbb{B}}^{\times}$ be 
the sequence of universal fibrations in $\mathbb{B}$ from Lemma~\ref{lemmafunctsequenceofclass}.
Let $\rho\colon t\rightarrow\mathcal{E}$ be a left exact reflective localization in $\mathbf{Cart}(\mathcal{B})$ and let the left exact 
Bousfield localization $a\colon\mathbb{B}\rightarrow\mathbb{E}$, the regular cardinal $\mathfrak{e}\in\mathrm{Shl}(\mathfrak{b})$, and the 
sequence of universal fibrations $q_{\bullet}\colon\mathrm{Shl}(\mathfrak{e})\rightarrow\mathcal{F}_{\mathbb{E}}^{\times}$ be as above. Then 
there is a natural equivalence
\[\mathrm{Ext}_{\ast}(\mathbb{R}(N(\iota(q_{\bullet})))\simeq \mathcal{E}_{\bullet}\]
in $\mathbf{Fun}(E,\mathbf{Cart}(\mathcal{B}))$.
\end{proposition}
\begin{proof}
Against the background of Proposition~\ref{corextobjclass} there are at least two ways to show the statement. For one, one may apply 
Proposition~\ref{corextobjclass} directly to the target fibration over the $\infty$-topos $\mathcal{E}(1)$ and pull-back the resulting 
equivalence along the left adjoint $L\colon\mathcal{E}(1)\rightarrow\mathcal{B}$. More concisely however, for every
pair $\kappa_1\leq\kappa_2$ in $\mathrm{Shl}(\mathfrak{e})$ we find a homotopy-coherent square
\begin{align}\label{diagcorextobjclassloc1}
\begin{gathered}
\xymatrix{
\mathrm{Ext}(\mathbb{R}N(q_{\kappa_1}))\ar[r]_(.7){e_{\kappa_1}}^(.7){\simeq}\ar@{^(->}[d]_{\mathrm{Ext}(\mathbb{R}N(q_{\kappa_1\leq\kappa_2}))} & \mathcal{E}_{\kappa_1} \ar@{^(->}[d]^{} \\
\mathrm{Ext}(\mathbb{R}N(q_{\kappa_2}))\ar[r]_(.7){e_{\kappa}}^(.7){\simeq} & \mathcal{E}_{\kappa_2}
}
\end{gathered}
\end{align}
by virtually the same construction given to construct the squares (\ref{diagcorextobjclass2}), using that each $q_{\kappa_i}$ is universal in
$\mathcal{E}$, and that the inclusion $\iota\colon\mathcal{E}\hookrightarrow\mathcal{B}$ is fully faithful (so in particular
$\iota(\mathrm{Fun}(q_{\kappa_i}))\cong\mathrm{Fun}(\iota(q_{\kappa_i}))$ over the base $V_{\kappa_i}\times V_{\kappa_i}$).
These can be again assembled into a natural equivalence
\[\mathrm{Ext}_{\ast}(\mathbb{R}(N(\iota(q_{\bullet})))\simeq \mathcal{E}_{\bullet}\]
in $\mathbf{Fun}(\mathrm{Shl}(\mathfrak{e}),\mathbf{Cart}(\mathcal{B}))$ in the very same way.
\end{proof}

Thus, both the domain and codomain of any accessible reflective localization of the target fibration over an $\infty$-topos $\mathcal{B}$ are 
contained in the essential image of the externalization functor. By Proposition~\ref{propcosmicext} it follows that the whole localization 
lies in the essential image, and so it can be lifted to a (polymorphic family of) reflective localizations of internal $\infty$-categories in 
$\mathcal{B}$.
Informally, the polymorphic family of reflective localizations in $\mathbf{Cat}_{\infty}(\mathbb{B})$ associated to an accessible reflective 
localization $\rho\colon t\rightarrow\mathcal{E}$ is a compatible sequence of accessible left exact reflections of the following form.
\[\xymatrix{
\mathbb{R}N(q_{\epsilon_0})\ar@{^(->}[r]\ar@{^(->}@/_1pc/[d] & \mathbb{R}N(q_{\epsilon_1})\ar@{^(->}[r]\ar@{^(->}@/_1pc/[d] & \dots\ar@{^(->}[r] & \mathbb{R}N(q_{\epsilon_i})\ar@{^(->}[r]\ar@{^(->}@/_1pc/[d] & \dots \\
\mathbb{R}N(p_{\epsilon_0})\ar@{^(->}[r]\ar@/_1pc/[u]\ar@{}[u]|{\vdash} & \mathbb{R}N(p_{\epsilon_1})\ar@{^(->}[r]\ar@/_1pc/[u]\ar@{}[u]|{\vdash} & \dots\ar@{^(->}[r] & \mathbb{R}N(p_{\epsilon_i})\ar@{^(->}[r]\ar@/_1pc/[u]\ar@{}[u]|{\vdash} & \dots
}\]

\begin{corollary}\label{corpolyfam1-1}
Let $\mathcal{B}$ be a $\infty$-topos. Then there is a bijection between the following two collections.
\begin{enumerate}
\item Equivalence classes of $\mathrm{Shl}(\mathfrak{b})$-indexed polymorphic families of $\mu$-accessible (and left exact) reflective 
localizations $\mathcal{E}_{\bullet}$ of $t_{\bullet}$ in $\mathbf{Cart}(\mathcal{B})$ (for some regular $\mu\in\mathrm{Shl}(\mathfrak{b})$).
\item Equivalence classes of polymorphic families of $\mu$-accessible (and left exact) reflective localizations
$\mathcal{E}_{\bullet}$ of $\mathbb{U}_{\bullet}$ in $\mathbf{Cat}_{\infty}(\mathbb{B})$ (for some regular $\mu\in\mathrm{Shl}(\mathfrak{b})$).
\end{enumerate}
\end{corollary}
\begin{proof}
Let $\mu\in\mathrm{Shl}(\mathfrak{b})$ and $\rho_{\bullet}\colon t_{\bullet}\rightarrow\mathcal{E}_{\bullet}$ be 
a $\mathfrak{b}$-polymorphic family of $\mu$-accessible (and left exact) reflective localizations in $\mathbf{Cart}(\mathcal{B})$. Via
Proposition~\ref{corextobjclass} and Proposition~\ref{corextobjclassloc} we obtain a cardinal $\mathfrak{e}\in\mathrm{Shl}(\mathfrak{b})$  
such $\rho_{\bullet}$ restricts to a reflective localization
$\rho_{\bullet}|_{\mathrm{Shl}(\mathfrak{e})}\colon t_{\bullet}|_{\mathrm{Shl}(\mathfrak{e})}\rightarrow\mathcal{E}_{\bullet}|_{\mathrm{Shl}(\mathfrak{e})}$ in $\mathrm{Fun}(\mathrm{Shl}(\mathfrak{e}),\mathbf{Cart}(\mathcal{B}))$, which again is (left exact and) accessible 
(however with a potentially sharply larger rank $\mu\sprime$). 
The cardinal $\mathfrak{e}$ is chosen in such a way that both $t_{\bullet}|_{\mathrm{Shl}(\mathfrak{e})}$ and $\mathcal{E}_{\bullet}|_{\mathrm{Shl}(\mathfrak{e})}$ are contained in the essential image of
$\mathrm{Ext}_{\ast}$. As the latter is a pseudo-cosmological embedding (Proposition~\ref{propcosmicext}, Corollary~\ref{corcosmicext}), we 
obtain a $\mu\sprime$-accessible (and left exact) reflective localization between the preimages 
$\mathbb{U}_{\bullet}$ and $\mathbb{R}N(q_{\bullet})$ in $\mathbf{Fun}(\mathrm{Shl}(\mathfrak{b}),\mathbf{Cat}_{\infty}(\mathbb{B}))$.

Vice versa, given a cardinal $\mathfrak{e}\in \mathrm{Shl}(\mathfrak{b})$, and a $\mu$-accessible (and left exact) $\mathfrak{e}$-polymorphic 
family of reflective localizations 
$\rho\colon\mathbb{U}_{\bullet}\rightarrow\mathcal{E}_{\bullet}$ in $\mathbf{Cat}_{\infty}(\mathbb{B})$, we obtain a $\mu$-accessible (and 
left exact) $\mathfrak{e}$-polymorphic family of reflective localizations
$\mathrm{Ext}_{\ast}(\rho)\colon\mathrm{Ext}_{\ast}(\mathbb{U}_{\bullet}))\rightarrow\mathrm{Ext}_{\ast}(\mathcal{E}_{\bullet})$ in
$\mathbf{Cart}(\mathcal{B})$ again by Proposition~\ref{propcosmicext} and Corollary~\ref{corcosmicext}. As
$\mathfrak{e}\in\mathrm{Shl}(\mathfrak{b})$, we 
have a natural equivalence
$t_{\bullet}|_{\mathrm{Shl}(\mathfrak{e})}\simeq\mathrm{Ext}_{\ast}(\mathbb{U}_{\bullet}|_{\mathrm{Shl}(\mathfrak{e})})$ by restriction of 
Proposition~\ref{corextobjclass}. The composite reflective localization 
$t_{\bullet}|_{\mathrm{Shl}(\mathfrak{e})}\rightarrow\mathrm{Ext}_{\ast}(\mathcal{E}_{\bullet})$ induces a $\mu$-accessible (and left exact) 
reflective localization $t\rightarrow\mathcal{E}\sprime$ for $\mathcal{E}\sprime:=\bigcup\mathrm{Ext}_{\ast}(\mathcal{E}_{\bullet})$ in
$\mathbf{Cart}(\mathcal{B})$ by Proposition~\ref{propmonadscatter}. An application of Proposition~\ref{propmonadscatter} again yields a
$\mu\sprime$-accessible (and left exact) reflective localization $t_{\bullet}\rightarrow\mathrm{Ext}_{\ast}(\mathcal{E}\sprime_{\bullet})$ in
$\mathbf{Fun}(\mathrm{Shl}(\mathfrak{b}),\mathbf{Cart}(\mathcal{B}))$ for some $\mu\sprime\in\mathrm{Shl}(\mathfrak{b})$.

The fact that the two constructions are mutually inverse (up to natural homotopy equivalence) is straight-forward.
\end{proof}

\begin{remark}
Corollary~\ref{corpolyfam1-1} together with Proposition~\ref{corextobjclassloc} in fact implies that every polymorphic family of
accessible reflective localizations $\mathcal{E}_{\bullet}$ of $\mathbb{U}_{\bullet}$ in $\mathbf{Cat}_{\infty}(\mathbb{B})$ is of the 
form $\mathbb{R}N(q_{\bullet})$ for some sequence of univalent universal fibrations $q_{\bullet}$.
\end{remark}

We are left to incorporate compositionality as a property of indexed internal $\infty$-categories in $\mathcal{B}$. We therefore make note 
of the following lemma. 

\begin{lemma}\label{lemmadefsumsintly}	
Let $\mathcal{B}$ be a $\infty$-topos and $\mathcal{E}_{\bullet}$ be a polymorphic family of reflective localizations of 
$t_{\bullet}$ in $\mathbf{Cart}(\mathcal{B})$. Let
$\mathbb{R}N(q_{\bullet})\colon\mathrm{Shl}(\mathfrak{e})\rightarrow\mathbf{Cat}_{\infty}(\mathbb{B})$ be the associated polymorphic family 
of reflective localizations of $\mathbb{U}_{\bullet}$ in
$\mathbf{Cat}_{\infty}(\mathbb{B})$. Then $\mathcal{E}_{\bullet}$ is compositional (Definition~\ref{defcomppolyfam}) if and only if for every 
$\kappa\in\mathrm{Shl}(\mathfrak{e})$, for every map $\alpha\colon C\rightarrow V_{\kappa}$ into the codomain of $q_{\kappa}$, and every map
$\beta\colon s(q_{\kappa}^{\ast}\alpha)\rightarrow V_{\kappa}$, the essentially unique name
$\ulcorner \beta^{\ast}q_{\kappa}\circ \alpha^{\ast}q_{\kappa}\urcorner\colon C\rightarrow U_{\kappa}$ of the 
composition $\beta^{\ast}q_{\kappa}\circ \alpha^{\ast}q_{\kappa}\in\mathcal{B}_{/C}$ factors through the inclusion
$V_{\kappa}\hookrightarrow U_{\kappa}$ up to equivalence.
\end{lemma}
\begin{proof}
This is but a translation of closure under compositions of each $\mathcal{E}_{\kappa}$ along the equivalence
\[\mathrm{Ext}(\mathbb{R}N(q_{\kappa}))\xrightarrow{\simeq}\mathcal{E}_{\kappa},\]
using univalence of the object classifier $p_{\kappa}\colon\tilde{U}_{\kappa}\rightarrow U_{\kappa}$.
\end{proof}

We say that a polymorphic family of reflective localizations $\mathbb{R}N(q_{\bullet})$ of $\mathbb{U}_{\bullet}$ in
$\mathbf{Cat}_{\infty}(\mathbb{B})$ \emph{has sums} if either of the conditions in Lemma~\ref{lemmadefsumsintly} is satisfied.

\begin{remark}
Given a polymorphic family of reflective localizations $(\rho_{\bullet},\iota_{\bullet})\colon t_{\bullet}\rightarrow \mathcal{E}_{\bullet}$ 
in $\mathbf{Cart}(\mathcal{B})$, by Corollary~\ref{corpolyfam1-1} we obtain an associated polymorphic family of reflective localizations
$\mathbb{R}N(q_{\bullet})$ of $\mathbb{U}_{\bullet}$ in $\mathbf{Cat}_{\infty}(\mathbb{B})$. For each regular cardinal $\kappa$ in the 
corresponding class $\mathrm{Shl}(\mathfrak{e})$, the cartesian fibration
$\mathcal{E}_{\kappa}\simeq\mathrm{Ext}_{\ast}(\mathbb{R}(N(q_{\kappa})))$ is a reflective localization of the cartesian fibration
$t_{\kappa}\colon(\mathcal{B}^{\Delta^1})_{\kappa}\twoheadrightarrow\mathcal{B}$. The fibration $t_{\kappa}$ is cocartesian, and hence so is
$\mathcal{E}_{\kappa}\twoheadrightarrow\mathcal{B}$ by Lemma~\ref{lemmarefllocbifib}. Thus, for every map $f\colon C\rightarrow D$ in
$\mathcal{B}$, we obtain an adjoint pair of the form
\[\xymatrix{
\mathrm{Ext}(\mathbb{R}(N(q_{\kappa})))(C)\ar@/^1pc/[r]^{\Sigma_f}\ar@{}[r]|{\bot} & \mathrm{Ext}(\mathbb{R}(N(q_{\kappa})))(D)\ar@/^1pc/[l]^{f^{\ast}}.
}\]
On vertices, the cartesian action $f^{\ast}$ associated to a map $f\colon C\rightarrow D$ takes a map $D\rightarrow V_{\kappa}$ to its 
precomposition with $f$; the cocartesian action takes a map $\gamma\colon C\rightarrow V_{\kappa}$ to the name
$\ulcorner \rho_D(f\circ\gamma^{\ast}(q_{\kappa})\urcorner\colon D\rightarrow V_{\kappa}$ of the localization of the
``$\Sigma$-type'' $f\circ\gamma^{\ast}(q_{\kappa})$ in $\mathcal{B}$.

By Lemma~\ref{lemmacharTglue} (and Lemma~\ref{corlexmodetale}), the cartesian subfibrations $\mathcal{E}_{\kappa}\subseteq t_{\kappa}$ are
closed under compositions if and only if the cocartesian action (of $\mathcal{E}_{\kappa}$-arrows) on $\mathcal{E}_{\kappa}$ -- and hence on
$\mathrm{Ext}(\mathbb{R}(N(q_{\kappa})))$ -- is conservative. Type theoretically, the fact that conservativity of these cocartesian 
actions is a consequence of $\Sigma$-closure of the localization $\mathcal{E}$ is stated in \cite[Lemma 1.40]{rss_hottmod}. The fact that 
the other direction holds as well can be phrased and shown accordingly.
\end{remark}

An application of the Monadicity Theorem in the $\infty$-cosmos $\mathbf{Fun}(\mathrm{Shl}(\mathfrak{b}),\mathbf{Cat}_{\infty}(\mathbb{B}))$ 
yields an equivalence between (left exact and accessible) reflective localizations and according idempotent monads in
$\mathbf{Fun}(\mathrm{Shl}(\mathfrak{b}),\mathbf{Cat}_{\infty}(\mathbb{B}))$. We may again say that such an idempotent monad is 
\emph{compositional} if its associated reflective localization has sums. We hence make the following definition.

\begin{definition}\label{defhlvops}
A \emph{higher Lawvere-Tierney operator} on an $\infty$-topos $\mathcal{B}$ consists of a cardinal
$\mathfrak{e}\in\mathrm{Shl}(\mathfrak{b})$ together with a compositional, left exact, and $\mu$-accessible idempotent monad $T$ on
$\mathbb{U}_{\bullet}|_{\mathrm{Shl}(\mathfrak{e})}\in\mathbf{Fun}(\mathrm{Shl}(\mathfrak{e}),\mathbf{Cat}_{\infty}(\mathbb{B}))$ for some $\mu\in\mathrm{Shl}(\mathfrak{b})$.
\end{definition}

We thus have proven the following theorem.

\begin{theorem}\label{thmcharlvtops}
Let $\mathcal{B}$ be an $\infty$-topos. Then there is a bijection between
\begin{enumerate}
\item equivalence classes of accessible left exact localizations of $\mathcal{B}$, and 	
\item equivalence classes of higher Lawvere Tierney operators on $\mathcal{B}$.
\end{enumerate}
\end{theorem}

\begin{proof}
The proof is a now a direct concatenation of the Beck Monadicity Theorem (in the reduced form of Corollary~\ref{corbeckmonadicity}) applied 
to the $\infty$-cosmoses $\mathrm{Fun}(\mathrm{Shl}(\mathfrak{e}),\mathbf{Cat}_{\infty}(\mathcal{B}))$, together with 
Lemma~\ref{lemmadefsumsintly}, Corollary~\ref{corpolyfam1-1}, Proposition~\ref{propmonadscatter} and Corollary~\ref{corlexfibmods} and 
Theorem~\ref{thmprelim}.
\end{proof}

We end this section with a few remarks. 

\begin{remark}
One can write down a compositionality formula for the monads in $\mathbf{Cat}_{\infty}(\mathcal{B})$ which characterizes its associated 
reflective localization to have sums in the same vein of Definition~\ref{defcompmonads} for the monads in $\mathbf{Cart}(\mathcal{B})$. 
Informally, the formula computes the value of the underlying pointed endofunctor
$T_{\kappa}\colon\mathbb{U}_{\kappa}\rightarrow\mathbb{U}_{\kappa}$ at the object of morphisms $\mathrm{Fun}(p_{\kappa})$ entirely by 
its values on the object $U_{\kappa}$ of objects. In that sense, compositional idempotent monads on $\mathbb{U}_{\kappa}$ are fully 
determined by their pointed endomap $(T_{\kappa})_0\colon U_{\kappa}\rightarrow U_{\kappa}$ in $\mathcal{B}$. This compares to the fact that 
traditional Lawvere-Tierney topologies are - while functors of an internal poset - fully determined by their values on objects of the 
subobject classifier. Indeed, here due to posetality, 
functoriality is a property, not additional structure. In that sense, higher Lawvere-Tierney operators are determined entirely 
by internal $\infty$-groupoidal (rather than $(\infty,1)$-categorical) data. This is responsible for the relation to type theoretic notions 
referred to in Remark~\ref{remlvtopstt}.
\end{remark}

\begin{remark}\label{remelemlvtops}
Although fairly involved as a definition, higher Lawvere-Tierney operators as in Definition~\ref{defhlvops} can be defined accordingly in 
more general settings. For example, given an elementary higher topos $\mathcal{B}$ (\cite{rasekheltops}) together with a directed poset $P$ 
and a diagram $\mathbb{U}\colon P\rightarrow\mathbb{B}^{\Delta^1}$ of univalent maps in $\mathcal{B}$, a higher Lawvere-Tierney operator on
$\mathcal{B}$ may be defined as a left exact and compositional idempotent monad on $\mathbb{R}N\circ\mathbb{U}$ in 
$\mathrm{Fun}(P,\mathrm{Cat}_{\infty}(\mathcal{B}))$. Here we drop the assumption of accessibility, 
because elementary toposes are generally not assumed to be accessible. The definition of such higher Lawvere-Tierney operators 
has the benefit of being a perfectly formal $\infty$-categorical definition -- although given the lack of an $\infty$-cosmological 
presentation of $\mathcal{B}$, the definition of a monad is now specifically the one of \cite[Section 4.7]{lurieha}.

Furthermore, there is no apparent way to internalize the global theory of left exact localizations (or left exact modalities for that matter) 
of a higher elementary topos $\mathcal{B}$ by ``scattering'' it along the indexed quasi-categories presented by the complete Segal objects
$\mathbb{R}N(\mathbb{U}(i))$. The assumption of accessibility in combination with the scattering along sizes indexed by regular cardinals was 
indeed crucial in the according internalization procedure in the case of (presentable) $\infty$-toposes. Therefore, such higher
Lawvere-Tierney operators appear to be a stronger notion on an elementary higher topos than left exact localizations of $\mathcal{B}$ (and 
hence left exact modalities in the sense of \cite{abjfsheavesI} are.
\end{remark}

\begin{remark}[Relation to type theoretical $\Sigma$-closed reflective subuniverses]\label{remlvtopstt}
The authors of \cite{rss_hottmod} characterize (left exact) modalities in Homotopy Type Theory in various ways, one of them being (left 
exact) $\Sigma$-closed reflective subuniverses (\cite[Section 1.3]{rss_hottmod}). By way of their definition, a left exact $\Sigma$-closed 
reflective subuniverse corresponds semantically in a suitable $\infty$-category $\mathcal{B}$ to a left exact reflective subcategory of the 
target fibration as in Corollary~\ref{corlexfibmods}.2 (\cite[Theorem A.7]{rss_hottmod}). 
These subuniverses are hence not ``polymorphically'' parametrized notions defined for each internal universe, but rather structures imposed 
by judgemental rules on the calculus as a whole (\cite[Section A.1]{rss_hottmod}). In particular, the associated $\Sigma$-closed reflective subuniverses are not types themselves, but rather classes of types closed under various type constructors. 

The fact that (left exact and compositional) reflective subcategories of the target fibration $t$ in $\mathbf{Cart}(\mathcal{B})$ (as in 
Corollary~\ref{corlexfibmods}.2) over an $\infty$-category $\mathcal{B}$ that arises from a type theoretic model category
$\mathbb{B}$ correspond to (lex and $\Sigma$-closed) ``weakly stable'' modalities in the associated syntax is essentially given in
\cite[Sections A.2 and A.3]{rss_hottmod}.
Whenever $\mathcal{B}$ is an $\infty$-topos and $\mathbb{B}$ is its ``excellent type theoretic model categorical presentation'' 
(\cite[Section A.4]{rss_hottmod}), then this establishes an equivalence of higher Lawvere-Tierney operators and such lex and
$\Sigma$-closed ``weakly stable'' modalities which furthermore are generated by a map, so to capture the condition of accessibility, see 
\cite[Section A.4]{rss_hottmod}.

\end{remark}

\section{Fibered structures over generating $\infty$-categories}\label{sectopdef}

In light of the 1-categorical background considered in Section~\ref{secfirstorder}, we will think of a small $\infty$-category $\mathcal{C}$ 
to come equipped with a canonical higher order hyperdoctrine $\mathcal{O}_{\mathcal{C}}$ which induces notions of topology on
$\mathcal{C}$ by considering higher order nuclei $\mathcal{E}$ (i.e.\ left exact modalities) on it. In analogy to 
Remark~\ref{remlocisnull}, we can describe a generalized topology $J$ associated to an indexed left exact modality $\mathcal{E}$ 
over a small $(\infty,1)$-category $\mathcal{C}$ as the associated collection of objects in $\mathcal{O}_{\mathcal{C}}(C)$ for
$C\in\mathcal{C}$ which are nullified in $\mathcal{E}(C)$. We will refer to such generalized topologies as \emph{geometric kernels} since the 
term ``generalized topology'' means different things in various settings. \\

Therefore, let $\mathcal{C}$ be a small $\infty$-category. While $\mathcal{C}$ itself may not have much of an 
intrinsic internal logic itself, it does possess various external sheaf semantics simply by being enriched in the $\infty$-cosmos
$\mathcal{S}$.
Indeed, the $\infty$-category $\mathcal{S}$ of spaces has an internal logic given by a suitable univalent version of Homotopy Type Theory 
(\cite{klvsimp}). Due to the ``propositions as types'' paradigm in Homotopy Type Theory, the predicate-classifier of $(\infty,1)$-category 
theory is given by $\mathcal{S}$ itself.
Thus, every presheaf over $\mathcal{C}$ classifies a $\mathcal{C}$-indexed predicate, represented by its associated right fibration 
over $\mathcal{C}$. 	
We may consider the indexed quasi-category of (external) predicates over $\mathcal{C}$ defined as follows.
\begin{align}\label{equdefOC}
\mathcal{O}_{\mathcal{C}}\colon\mathcal{C}^{op}\xrightarrow{\mathcal{C}_{/(\cdot)}}\mathrm{Cat}_{\infty}^{op}\xrightarrow{\mathrm{Fun}((\cdot)^{op},\mathcal{S})}\mathrm{CAT}_{\infty}
\end{align}
Here, $\mathrm{CAT}_{\infty}$ is the $\infty$-category of large $\infty$-categories, and we consider the canonical
``co-indexing'' $\mathcal{C}_{/(\cdot)}$ acting covariantly via postcomposition $f\mapsto \Sigma_f$. We have natural 
pointwise equivalences
\begin{align*}
\mathcal{O}_{\mathcal{C}}(C) & \simeq\widehat{\mathcal{C}_{/C}} \\
 & \simeq\hat{\mathcal{C}}_{/yC} \\
 & \simeq\mathrm{RFib}(\mathcal{C}_{/C}).
\end{align*}
This indexed $\infty$-topos is precisely the ``logical structure sheaf'' on $\mathcal{C}$ we will consider in 
Section~\ref{secsublogstrsheaf}.


\subsection{Sheaves of $\mathcal{O}_{\mathcal{C}}$-ideals}\label{secsublogstrsheaf}

In this section we define fibered local versions over a small $\infty$-category $\mathcal{C}$ associated to the structures 
fibered over $\hat{\mathcal{C}}$ considered in Section~\ref{secsubfsrl}, and show that they yield equivalent data. These local structures 
will be defined as substructures of the logical structure sheaf $\mathcal{O}_{\mathcal{C}}$ given in (\ref{equdefOC}), mirroring the 
corresponding first order notions from Section~\ref{secfirstorder}.

Therefore, we recall the (super-large) $\infty$-category of $\infty$-toposes and \'{e}tale geometric morphisms between 
such.
Let $\mathrm{RTop}$ be the $\infty$-category of $\infty$-toposes and geometric morphisms, and $\mathrm{LTop}$ be its 
opposite. Thus, the arrows in $\mathrm{LTop}$ are the left exact cocontinuous functors
(\cite[6.3.1]{luriehtt}).
A geometric morphism $f_{\ast}\colon\mathcal{F}\rightarrow\mathcal{E}$ between $\infty$-toposes is an \emph{embedding} 
if the right adjoint $f_{\ast}$ is fully faithful. 
A geometric morphism $f_{\ast}\colon\mathcal{F}\rightarrow\mathcal{E}$ between $\infty$-toposes is
\emph{\'{e}tale} if its left adjoint $f^{\ast}$ is equivalent to one of the form
$(\blank\times E)\colon\mathcal{E}\rightarrow\mathcal{E}_{/E}$ for some object $E\in\mathcal{E}$. By
\cite[Proposition 6.3.5.11]{luriehtt}, $f_{\ast}\colon\mathcal{F}\rightarrow\mathcal{E}$ is \'{e}tale if and only if 
the following three conditions hold.
\begin{enumerate}
\item The left adjoint $f^{\ast}$ admits a further left adjoint $f_!\colon\mathcal{F}\rightarrow\mathcal{E}$.
\item The left adjoint $f_!$ is conservative.
\item The pair $(f_!,f^{\ast})$ satisfies the \emph{projection formula}, i.e.\ for every $X\rightarrow Y$ in
$\mathcal{E}$, every object $Z\in\mathcal{F}$ and every morphism $f_!Z\rightarrow Y$, the induced square
\[\xymatrix{
f_!(f^{\ast}X\times_{f^{\ast}Y}Z)\ar[r]\ar[d] & f_!Z\ar[d] \\
X\ar[r] & Y
}\]
is cartesian in $\mathcal{E}$.
\end{enumerate}

Let $(\mathrm{RTop},\mathrm{Et})$ be the $\infty$-category of $\infty$-toposes and \'{e}tale geometric 
morphisms between them, and let $(\mathrm{LTop},\mathrm{Et})$ again be the opposite (\cite[Section 6.3.5]{luriehtt}).

More generally recall the notion of $\infty$-logoi from \cite{aneljoyaltopos}, which are left exact and cocomplete $\infty$-categories which 
satisfy descent. In particular, $\infty$-topos are exactly the presentable $\infty$-logoi. We define the notions of geometric morphisms, 
geometric embeddings and \'{etale} geometric morphisms between $\infty$-logoi in the same way: In the following, a geometric morphism 
(embedding) between $\infty$-logoi shall be a (fully faithful) functor with a left exact left adjoint. Such a geometric morphism between
$\infty$-logoi is defined to be \'{e}tale whenever the three conditions above are satisfied. This defines the $\infty$-categories
$\mathrm{RLog}$ and $(\mathrm{RLog},\mathrm{Et})$ together with their opposites
$\mathrm{Log}$ and $(\mathrm{Log},\mathrm{Et})$ analogously.

Let $\mathcal{C}$ be a small $\infty$-category. While $\mathcal{C}$ itself is generally not left exact and hence 
exhibits no canonical indexing over itself, its presheaf $\infty$-category $\hat{\mathcal{C}}$ does. We thus may 
consider the canonical indexing of $\hat{\mathcal{C}}$ over itself restricted to $\mathcal{C}$ as follows.
\[\mathcal{C}^{op}\xrightarrow{y^{op}}(\hat{\mathcal{C}})^{op}\xrightarrow{\hat{\mathcal{C}}_{/(\cdot)}}\mathrm{CAT}_{\infty}.\]
Each such value $\hat{\mathcal{C}}_{/yC}\simeq\widehat{\mathcal{C}_{/C}}$ is an $\infty$-topos, and the induced 
transition maps $f^{\ast}\colon \hat{\mathcal{C}}_{/yD}\rightarrow \hat{\mathcal{C}}_{/yC}$ for $f\colon C\rightarrow D$ 
in $\mathcal{C}$ are part of the \'{e}tale geometric morphisms $(\Sigma_f,f^{\ast},\prod_f)$. We will denote the induced 
composition by
\[\mathcal{O}_{\mathcal{C}}\colon\mathcal{C}^{op}\rightarrow(\mathrm{LTop},\mathrm{Et})\]
and refer to $\mathcal{O}_{\mathcal{C}}$ as the \emph{logical structure sheaf} (the term ``sheaf'' again will be justified in 
Proposition~\ref{propinftysemsheaf}).

\begin{definition}\label{defindexedmod}
Let $\mathcal{C}$ be a small $\infty$-category.
A sheaf $\mathcal{E}$ of $\mathcal{O}_{\mathcal{C}}$-ideals is a natural geometric embedding
$\mathcal{O}_{\mathcal{C}}\rightarrow\mathcal{E}$ in $\mathrm{Fun}(\mathcal{C}^{op},\mathrm{Log})$ such that
$\mathcal{E}\colon\mathcal{C}^{op}\rightarrow\mathrm{Log}$ factors through $(\mathrm{Log},\mathrm{Et})$.

A sheaf $\mathcal{E}$ of accessible $\mathcal{O}_{\mathcal{C}}$-ideals is a natural geometric embedding
$\mathcal{O}_{\mathcal{C}}\rightarrow\mathcal{E}$ in $\mathrm{Fun}(\mathcal{C}^{op},\mathrm{LTop})$ such that
$\mathcal{E}\colon\mathcal{C}^{op}\rightarrow\mathrm{LTop}$ factors through $(\mathrm{LTop},\mathrm{Et})$.

\end{definition}

\begin{proposition}\label{propetalemod}
Let $\mathcal{B}$ be an $\infty$-topos. Then there is a bijection between
\begin{enumerate}
\item left exact modalities on $\mathcal{B}$, and 
\item equivalence classes of natural geometric embeddings $\rho\colon\mathcal{B}_{/\blank}\rightarrow{\mathcal{E}}$ in
$\mathrm{Fun}(\mathcal{B}^{op},\mathrm{Log})$ such that $\mathcal{E}\colon\mathcal{C}^{op}\rightarrow\mathrm{Log}$ factors through
$(\mathrm{Log},\mathrm{Et})$.
\end{enumerate}
It restricts to a bijection between 
\begin{enumerate}
\item left exact modalities of small generation on $\mathcal{B}$, and 
\item equivalence classes of natural geometric embeddings $\rho\colon\mathcal{B}_{/\blank}\rightarrow{\mathcal{E}}$ in
$\mathrm{Fun}(\mathcal{B}^{op},\mathrm{LTop})$ such that $\mathcal{E}\colon\mathcal{C}^{op}\rightarrow\mathrm{LTop}$ factors through
$(\mathrm{LTop},\mathrm{Et})$.
\end{enumerate}
\end{proposition}

\begin{proof}
Via Parts 2 and 3 of Corollary~\ref{corlexfibmods}, Corollary~\ref{corlexmodetale} and Remark~\ref{remlawverecomp}, 
left exact modalities (of small generation) on $\mathcal{B}$ stand in 1-1 correspondence to equivalence classes of 
(accessible) left exact reflective localizations $\mathcal{E}\hookrightarrow\mathcal{B}^{\Delta^1}$ in $\mathbf{Cart}(\mathcal{B})$ 
such that the cocartesian transition functors $(\mathcal{E}_f)_{!}$ are conservative.
Under the Straightening equivalence to the $\infty$-category $\mathrm{Fun}(\mathcal{B}^{op},\mathrm{CAT}_{\infty})$ of large
$\mathcal{B}$-indexed $\infty$-categories, such correspond bijectively to equivalence classes of 
natural transformations $\rho\colon\mathcal{B}_{/\blank}\rightarrow{\mathcal{E}}$ in
$\mathrm{Fun}(\mathcal{C}^{op},\mathrm{CAT}_{\infty})$ such that each $\rho_C$ is left exact and admits an (accessible) right adjoint 
inclusion, and such that each $f^{\ast}\colon\mathcal{E}_D\rightarrow\mathcal{E}_C$ has a conservative left adjoint (via 
Proposition~\ref{propfiblex} and Remark~\ref{rempropfiblexgen2}).
We claim that these are exactly of the form in 2. Clearly, every natural geometric embedding as in 2.\ yields a natural transformation 
of this form by postcomposition with the forgetful functors
$(\mathrm{LTop},\mathrm{Et})\rightarrow(\mathrm{Log},\mathrm{Et})\rightarrow\mathrm{CAT}_{\infty}$. Vice versa, given a 
natural transformation $\rho\colon\mathcal{B}_{/\blank}\rightarrow{\mathcal{E}}$ corresponding to a left exact modality (of small generation) 
on $\mathcal{B}$ as above, the transition maps
$(\mathcal{E}_f)^{\ast}\colon\mathcal{E}_D\rightarrow\mathcal{E}_C$ for $f\colon C\rightarrow D$ in $\mathcal{B}$ fit into squares of 
the form
\begin{align*}
\begin{gathered}
\xymatrix{
\mathcal{E}_D\ar[d]_{(\mathcal{E}_f)^{\ast}}\ar@{^(-->}@/^/[r]^{\iota_D} & \mathcal{B}_{/D}\ar[d]^{f^{\ast}}\ar[l]^{\rho_D}\\
\mathcal{E}_C\ar@{^(-->}@/^/[r]^{\iota_C} & \mathcal{B}_{/C}\ar[l]^{\rho_C}
}
\end{gathered}
\end{align*}
which commute in both directions up to homotopy. As the horizontal pairs are left exact (accessible) reflective localizations, 
and the slice functor $\mathcal{B}_{/\blank}$ factors through $(\mathrm{LTop},\mathrm{Et})$,  each $\mathcal{E}_C$ is automatically an 
$\infty$-sublogos ($\infty$-subtopos) of $\mathcal{B}_{/C}$, and each $(\mathcal{E}_f)^{\ast}\colon\mathcal{E}_D\rightarrow\mathcal{E}_C$ is 
cocontinuous and left exact. Thus, the transition functors $(\mathcal{E}_f)^{\ast}$ are always part of a geometric morphism, and so we obtain 
a factorization $\mathcal{E}\colon\mathcal{C}^{op}\rightarrow\mathrm{Log}$ ($\mathcal{E}\colon\mathcal{C}^{op}\rightarrow\mathrm{LTop}$). These transition functors furthermore exhibit a 
conservative left adjoint $(\mathcal{E}_f)_{!}$ themselves by assumption.
The projection formula for $\mathcal{E}$ with respect to a given map $f\colon C\rightarrow D$ can be shown to hold using that, first, 
it holds for $\mathcal{B}_{/\blank}$, second, that the inclusions $\iota_C\colon\mathcal{E}_C\rightarrow\mathcal{B}_{/C}$ are fully 
faithful, and, third, that the natural transformations $\Sigma_f\iota_C\Rightarrow \iota_D f_{\sharp}$ between the functors of type
$\mathcal{E}_C\rightarrow\mathcal{B}_{/D}$ obtained as a mate to $\iota_C(\eta)\colon \iota_C\Rightarrow\iota_C f^{\ast}f_{\sharp}$  
are $\mathcal{E}_D$-local equivalences. It follows that $\mathcal{E}\colon\mathcal{C}^{op}\rightarrow\mathrm{CAT}_{\infty}$ factors 
through $(\mathrm{Log},\mathrm{Et})$ ($(\mathrm{LTop},\mathrm{Et})$) as well, and $\rho\colon\mathcal{B}_{/\blank}\rightarrow\mathcal{E}$ is 
a pointwise geometric embedding.
\end{proof}

\begin{remark}
Accordingly, one may define \emph{sheaves of $\mathcal{O}_{\mathcal{C}}$-modules} simply as natural geometric morphisms
$\mathcal{O}_{\mathcal{C}}\rightarrow\mathcal{E}$ in $\mathrm{Fun}(\mathcal{C}^{op},\mathrm{Log})$ such that
$\mathcal{E}\colon\mathcal{C}^{op}\rightarrow\mathrm{Log}$ factors through $(\mathrm{Log},\mathrm{Et})$ (and the same for $\mathrm{LTop}$). 
The intuition is that of commutative algebra: given two rings $R$ and $S$, the set of $R$-module structures on $S$ is in bijection to 
the set of ring homomorphisms from $R$ to $S$. We will leave the study of such however to another time.
\end{remark}

The following theorem establishes an $\infty$-categorical analogue to the equivalence of structures 1.\ and 3.\ in 
Proposition~\ref{proplocalglobalfirstorder}.

\begin{theorem}\label{thmlocaltoglobal}
Let $\mathcal{C}$ be a small $\infty$-category. Then there is a bijection between sheaves of (accessible)
$\mathcal{O}_{\mathcal{C}}$-ideals and equivalences classes of left exact modalities (of small generation) on $\hat{\mathcal{C}}$.
\end{theorem}

\begin{proof}
Given Proposition~\ref{propetalemod}, the proof is a fairly straight-forward argument by restriction and right Kan extension along the 
Yoneda embedding of $\mathcal{C}$, exploiting descent of $\hat{\mathcal{C}}$.
Indeed, since both $\mathrm{Log}$ and $\mathrm{LTop}$ are complete (\cite[Proposition 6.3.2.3]{luriehtt}), for any small $\infty$-category
$\mathcal{C}$ the Yoneda embedding $y\colon\mathcal{C}\hookrightarrow\hat{\mathcal{C}}$ induces reflective localizations
\[\xymatrix{
\mathrm{Fun}(\hat{\mathcal{C}}^{op},\mathrm{Log})\ar@/^.5pc/[r]^{y^{\ast}}_{\bot} & \mathrm{Fun}(\mathcal{C}^{op},\mathrm{Log})\ar@/^/[l]^{y_\ast},
}\]
\[\xymatrix{
\mathrm{Fun}(\hat{\mathcal{C}}^{op},\mathrm{LTop})\ar@/^.5pc/[r]^{y^{\ast}}_{\bot} & \mathrm{Fun}(\mathcal{C}^{op},\mathrm{LTop})\ar@/^/[l]^{y_\ast}.
}\]
Via \cite[Theorem 5.1.5.6.]{luriehtt}, they restrict to equivalences
\[\xymatrix{
\mathrm{Fun}(\hat{\mathcal{C}}^{op},\mathrm{Log})_{\mathrm{lim}}\ar@/^.5pc/[r]^{y^{\ast}}_{\simeq} & \mathrm{Fun}(\mathcal{C}^{op},\mathrm{Log})\ar@/^/[l]^{y_{\ast}},
}\]
\[\xymatrix{
\mathrm{Fun}(\hat{\mathcal{C}}^{op},\mathrm{LTop})_{\mathrm{lim}}\ar@/^.5pc/[r]^{y^{\ast}}_{\simeq} & \mathrm{Fun}(\mathcal{C}^{op},\mathrm{LTop})\ar@/^/[l]^{y_{\ast}},
}\]
where $\mathrm{Fun}(\hat{\mathcal{C}}^{op},\mathrm{LTop})_{\mathrm{lim}}\subset\mathrm{Fun}(\hat{\mathcal{C}}^{op},\mathrm{LTop})$ 
denotes the full $\infty$-subcategory of continuous functors (and same for $\mathrm{Log}$). By virtue of descent of $\hat{\mathcal{C}}$, the 
slice functor $\hat{\mathcal{C}}_{/\blank}\colon\hat{\mathcal{C}}\rightarrow\mathrm{LTop}$ is contained in
$\mathrm{Fun}(\hat{\mathcal{C}}^{op},\mathrm{LTop})_{\mathrm{lim}}$ (as the forgetful functor $\mathrm{LTop}\rightarrow\mathrm{CAT}_{\infty}$ creates limits, see \cite[Proposition 6.3.2.3]{luriehtt}), and so we obtain equivalences of undercategories
\[\xymatrix{
(\mathrm{Fun}(\hat{\mathcal{C}}^{op},\mathrm{Log})_{\mathrm{lim}})_{(\hat{\mathcal{C}}_{/\blank})/}\ar@/^.5pc/[r]^{y^{\ast}}_{\simeq} & \mathrm{Fun}(\mathcal{C}^{op},\mathrm{Log})_{\mathcal{O}_{\mathcal{C}}/}\ar@/^/[l]^{y_{\ast}}
}\]
and
\[\xymatrix{
(\mathrm{Fun}(\hat{\mathcal{C}}^{op},\mathrm{LTop})_{\mathrm{lim}})_{(\hat{\mathcal{C}}_{/\blank})/}\ar@/^.5pc/[r]^{y^{\ast}}_{\simeq} & \mathrm{Fun}(\mathcal{C}^{op},\mathrm{LTop})_{\mathcal{O}_{\mathcal{C}}/}.\ar@/^/[l]^{y_{\ast}}
}\]
This further restricts to equivalences
\[\xymatrix{
\mathrm{GeoEmb}(\hat{\mathcal{C}}_{/\blank})\ar@/^.5pc/[r]^{y^{\ast}}_{\simeq} & \mathrm{GeoEmb}(\mathcal{O}_{\mathcal{C}}).\ar@/^/[l]^{y_{\ast}}
}\]
between the full $\infty$-subcategories each spanned by the pointwise (accessible) geometric embeddings out of
$\hat{\mathcal{C}}_{/\blank}$ and $\mathcal{O}_{\mathcal{C}}$, respectively. Indeed, on the one hand, it is clear that 
the restriction functor $y^{\ast}$ preserves pointwise (accessible) geometric embeddings; on the other hand, the right Kan extension
$y_{\ast}$ does so, too, because the class of geometric embeddings in $\mathrm{Log}$ (and in $\mathrm{LTop}$) is closed under limits. Lastly, 
whenever $\mathcal{E}\colon\hat{\mathcal{C}}^{op}\rightarrow\mathrm{LTop}$ factors through $(\mathrm{LTop},\mathrm{Et})$, clearly so does the 
restriction $y^{\ast}\mathcal{E}\colon\mathcal{C}^{op}\rightarrow\mathrm{LTop}$ (and the same for $\mathrm{Log}$). Vice versa, whenever
$\mathcal{E}\colon\mathcal{C}^{op}\rightarrow\mathrm{LTop}$ factors through $(\mathrm{LTop},\mathrm{Et})$, so does
$y_{\ast}\mathcal{E}\colon\hat{\mathcal{C}}^{op}\rightarrow\mathrm{LTop}$, because the inclusion
$(\mathrm{LTop},\mathrm{Et})\subset\mathrm{LTop}$ preserves and reflects limits (\cite[Proposition 6.3.5.13]{luriehtt}). The same again holds 
for $\mathrm{Log}$.
\end{proof}

\begin{remark}
Whenever $\mathcal{C}$ has a terminal object $1\in\mathcal{C}$, a sheaf $\mathcal{E}$ of
$\mathcal{O}_{\mathcal{C}}$-ideals is fully determined by its ``underlying'' $\infty$-topos $\mathcal{E}(1)$ since all 
transition maps are assumed to be \'{e}tale. In this case, $\mathcal{E}(1)$ is exactly the accessible left exact localization of
$\hat{\mathcal{C}}$ associated to $\mathcal{E}$ via Theorem~\ref{thmlocaltoglobal} and Proposition~\ref{thmprelim}.
\end{remark}

\begin{remark}\label{remlargeocideals}
The proof of Theorem~\ref{thmlocaltoglobal} becomes problematic when one allows to assume that $\mathcal{C}$ itself is a locally small but
potentially large $\infty$-category, as the right Kan extension of the ``small'' Yoneda embedding
$y\colon\mathcal{C}\rightarrow\mathrm{Fun}(\mathcal{C}^{op},\mathcal{S})$ generally need not exist. A more hands-on proof however (using that 
$\hat{\mathcal{C}}:=\mathrm{Fun}(\mathcal{C},\mathcal{S})$ has descent for all colimits that exist in $\hat{\mathcal{C}}$ rather than for all 
small colimits only, and that the ``small'' Yoneda embedding $y\colon\mathcal{C}\rightarrow\mathrm{Fun}(\mathcal{C}^{op},\mathcal{S})$ is 
still dense) shows an appropriate ``large'' version of the theorem: there is an equivalence between sheaves of
$\mathcal{O}_{\mathcal{C}}$-ideals on locally small (but potentially large) $\infty$-categories $\mathcal{C}$ and equivalences classes of 
left exact modalities on $\mathrm{Fun}(\mathcal{C}^{op},\mathcal{S})$.
\end{remark}

\begin{remark}\label{remtoplideals}
It may be noteworthy to point out that the proof of Theorem~\ref{thmlocaltoglobal} as stated requires descent and hence is 
characteristic to $\infty$-topos theory. Similarly, the according equivalence of 1.\ and 3.\ in Proposition~\ref{proplocalglobalfirstorder} 
is characteristic to finite dimensional topos theory, as in $\infty$-topos theory not every accessible left exact localization of a 
presheaf $\infty$-topos is topological.

In the 1-categorical case, left exact reflective localizations of a presheaf topos $\hat{\mathcal{C}}$ have been characterized by 
closure operators on $\hat{\mathcal{C}}$, which are indexed structures on the hyperdoctrine $\Omega_{\mathcal{C}}$ (rather than on the 
restriction of the full slice $\hat{\mathcal{C}}_{/\blank}$ along $y$). They are not required to be \'{e}tale and in fact generally simply 
are not. That is, because although the cocartesian action on
$\Omega_{\mathcal{C}}$-ideals $\mathcal{E}$ exists (via the image factorization in $\mathcal{E}$), it forgets too much 
information to be conservative. Indeed, conservativity would require that for monomorphisms
$x,y\in\mathrm{Sub}(A)$ with $x\leq y$ and an arbitrary map $f\colon A\rightarrow B$ in $\hat{\mathcal{C}}$, 
if $fx\colon X\rightarrow B$ and $fy\colon Y\rightarrow B$ have the same image in $B$, then $x=y$. This is generally not true if $f$ is not 
monic itself. At the same time conservativity (and hence \'{e}taleness) is not required in the proof of 
Proposition~\ref{proplocalglobalfirstorder}, essentially because of the specific interaction between pullbacks and monomorphisms. 
\end{remark}

\begin{remark}
The proof of Theorem~\ref{thmlocaltoglobal} can also be given more concretely by introducing indexed versions of left exact modalities 
and congruences over $\mathcal{C}$. Say, a \emph{congruence of small generation on $\mathcal{O}_{\mathcal{C}}$} is a collection of 
congruences $\mathcal{L}_C\subseteq\mathcal{O}_{\mathcal{C}}(C)^{\Delta^1}$ of small generation such that 
$f^{\ast}[\mathcal{L}_D]\subseteq\mathcal{L}_C$, $\Sigma_f[\mathcal{L}_C]\subseteq \mathcal{L}_D$ and
$\Sigma_f^{-1}[\mathcal{L}_D]\subseteq\mathcal{L}_C$ for all maps $f\colon C\rightarrow D$ in $\mathcal{C}$.
A \emph{left exact modality (of small generation) on $\mathcal{O}_{\mathcal{C}}$} is a family of left exact 
modalities $(\mathcal{L}_C,\mathcal{R}_C)$ (of small generation) on $\mathcal{O}_{\mathcal{C}}(C)$ for $C\in\mathcal{C}$ such 
that for all $f\colon C\rightarrow D$ in $\mathcal{C}$ the following conditions hold.
\begin{enumerate}
\item $f^{\ast}[\mathcal{L}_D]\subseteq\mathcal{L}_C$,
\item $f^{\ast}[\mathcal{R}_D]\subseteq\mathcal{R}_C$,
\item $\Sigma_f^{-1}[\mathcal{L}_D]\subseteq\mathcal{L}_C$.
\end{enumerate}
Then one can show that the classes of sheaves of $\mathcal{O}_{\mathcal{C}}$-ideals, of congruences of small generation on
$\mathcal{O}_{\mathcal{C}}$, and of left exact modalities (of small generation) on $\mathcal{O}_{\mathcal{C}}$ stand in bijective 
correspondence to one another, and that left exact modalities (of small generation) on $\mathcal{O}_{\mathcal{C}}$ stand in bijective 
correspondence to left exact modalities (of small generation) in $\hat{\mathcal{C}}$.
\end{remark}

We so far have introduced notion to prove 1.\, 2.\, and 3.\	 of Proposition 2.3. We end this paper with a notion for 4.

\subsection{Geometric kernels}\label{subsectopdef}

In Remark~\ref{remlocisnull} we noted that Grothendieck topologies on an ordinary category $\mathcal{C}$ can be 
characterized as the collections of predicates which are locally nullified in their associated sheaf of first order
$\mathcal{O}_{\mathcal{C}}$-ideals. Analogously, given a small $\infty$-category $\mathcal{C}$ and a sheaf $\mathcal{E}$ of
$\mathcal{O}_{\mathcal{C}}$-ideals, we may consider the $\infty$-subcategory
\[J(C)=\rho_C^{-1}[(-2)\text{-types}]\subseteq\hat{\mathcal{C}}_{/yC}\]
of objects which are mapped to contractible objects via the left adjoint $\rho_C\colon\hat{\mathcal{C}}_{/yC}\rightarrow\mathcal{E}_C$.
While we are not merely considering contractibility of monomorphisms anymore, it stays true that $\mathcal{O}_{\mathcal{C}}$-ideals are 
entirely determined by these collections $J(C)$ of objects they nullify locally.
To not overuse the notion ``generalized topologies'' which in the literature already has been applied to denote different things in different 
contexts, we will refer to such collections $J$ associated to a sheaf $\mathcal{E}$ of $\mathcal{O}_{\mathcal{C}}$-ideals as 
\emph{geometric kernels} on $\mathcal{C}$. 

\begin{definition}\label{defgeokernel}
Let $\mathcal{C}$ be a small $(\infty,1)$-category. Let $J$ be a full cartesian subfibration 
\[\xymatrix{
J\ar@{^(->}[r]\ar@/_/@{->>}[dr]_j & \sum_{C\in\mathcal{C}}\hat{\mathcal{C}}_{/yC}\ar@{->>}[d]\ar@{}[dr]|(.3){\pbs}\ar@{^(->}[r] & \hat{\mathcal{C}}^{\Delta^1}\ar@{->>}[d]^t \\
 & \mathcal{C}\ar@{^(->}[r]_y & \hat{\mathcal{C}}
}\]
over $\mathcal{C}$ such that each fiber $J(C)\subseteq\hat{\mathcal{C}}_{/yC}$ is a finitely accessible and left exact subcategory. For each 
$C\in\mathcal{C}$ let
\[J_C:=\bigcup_{g\in\mathcal{C}_{/C}}\Sigma_g[J(t(g))]\subseteq(\hat{\mathcal{C}}_{/yC})^{\Delta^1}.\]
We say that $J$ is an \emph{accessible geometric kernel} on $\mathcal{C}$ whenever each $J(C)\subseteq\hat{\mathcal{C}}_{/yC}$ is spanned 
exactly by those arrows over $yC$ which are nullified in the localization $\hat{\mathcal{C}}_{/yC}[(J_C)^{-1}]$.
\end{definition}

The localization $\hat{\mathcal{C}}_{/yC}[(J_C)^{-1}]$ in Definition~\ref{defgeokernel} exists in virtue of the local accessibility assumption on $J$: Each $J(C)$ is generated by a set $M(C)$ under filtered colimits, and so
$\hat{\mathcal{C}}_{/yC}[(J_C)^{-1}]=\hat{\mathcal{C}}_{/yC}[(M_C)^{-1}]$, where $M_C$ is defined as the according union of $M$-maps. Indeed, 
every arrow of the form $\Sigma_g(f)$ in $J_C$ is still a colimit of arrows in $M_C$ essentially because each $\Sigma_g$ is a left adjoint, 
and because every inclusion of the form $(\hat{\mathcal{C}}_{/yC})_{/g}\hookrightarrow(\hat{\mathcal{C}}_{/yC})^{\Delta^1}$ preserves 
filtered colimits (which in turn follows from the fact that filtered $\infty$-categories are weakly contractible,
\cite[Lemma 5.3.1.18]{luriehtt}).

\begin{remark}\label{remlexmods}
In \cite{as_soa}, the authors introduce a notion of lex modulators on $\infty$-toposes. Such are shown to be generators of accessible left 
exact localizations of $\hat{\mathcal{C}}$ in the following sense. The authors show that, first, every 
accessible left exact localization of $\hat{\mathcal{C}}$ is generated by some lex modulator on $\hat{\mathcal{C}}$, 
and, second, that the localization of $\hat{\mathcal{C}}$ at a lex modulator $M$ is always accessible and left exact. 
Furthermore, they show that the class $\mathcal{L}$ of arrows in $\hat{\mathcal{C}}$ which are inverted by such a 
localization is not only the strong saturation of $M$ (as is always the case), but is the saturation already 
(\cite[Theorem 3.4.2]{as_soa}). Every accessible geometric kernel on a small $\infty$-category $\mathcal{C}$ as defined above is hence 
generated by a lex modulator on $\hat{\mathcal{C}}$ by raising the index of accessibility of each $J(C)$ large enough so the generating sets
$M(C)\subset J(C)$ assemble to a pointwise left exact full cartesian subfibration $M\subset J$. In fact, generation of $J(C)$ by a small 
(left exact) set under filtered colimits can be replaced by generation under the transfinite plus construction.
\end{remark}

\begin{theorem}\label{thmgeomkernel}
Let $\mathcal{C}$ be a small $(\infty,1)$-category. Then there is a bijection between the class of accessible geometric kernels on
$\mathcal{C}$ and the class of sheaves of accessible $\mathcal{O}_{\mathcal{C}}$-ideals.
\end{theorem}
\begin{proof}
First, given an accessible geometric kernel $J$ on $\mathcal{C}$, as the base $\mathcal{C}$ is small we find a small full subfibration 
$M\subset J$ such that each set $M(C)\subseteq J(C)$ is closed under finite limits and generates the $\infty$-category $J(C)$ under
$\kappa$-filtered colimits 
for some $\kappa$ large enough. In particular, $M$ is a lex modulator on $\hat{\mathcal{C}}$ in the sense of \cite{as_soa}. Therefore, it 
generates a left exact modality $(\mathcal{L}_J,\mathcal{R}_J)$ on $\hat{\mathcal{C}}$ and hence an according sheaf $\mathcal{E}_J$ of
$\mathcal{O}_{\mathcal{C}}$-ideals given pointwise by localization of the class $(\mathcal{L}_J)_{/yC}$ of $\hat{\mathcal{C}}_{/yC}$ for each 
$C\in\mathcal{C}$ (via Proposition~\ref{propetalemod} and Theorem~\ref{thmlocaltoglobal}).
The two classes $J_C\subseteq(\mathcal{L}_J)_{/yC}$ induce the same localization of $\hat{\mathcal{C}}_{/y(C)}$ as for instance can be seen 
from the fact that every element in $(\mathcal{L}_J)_{/yC}$ is equivalent to a transfinite iteration of the plus-construction with respect to 
$M(C)\subset J(C)$.

Second, given a sheaf $\rho\colon\mathcal{O}_{\mathcal{C}}\rightarrow\mathcal{E}$ of $\mathcal{O}_{\mathcal{C}}$-ideals, for
$C\in\mathcal{C}$ let $J	_{\mathcal{E}}(C)\subseteq\hat{\mathcal{C}}_{/yC}$ be the class of objects which are mapped to a terminal object in 
$\mathcal{E}_C$ via the localization $\rho_{C}$. Since 
each such localization is left exact, the full subcategories $J_{\mathcal{E}}(C)\subseteq\hat{\mathcal{C}}_{/yC}$ are closed under finite 
limits. The subcategories $J_{\mathcal{E}}(C)\subseteq\hat{\mathcal{C}}_{/y(C)}$ are finitely accessible, because the left adjoints
$\rho_C$ preserve all colimits and so we may apply 
\cite[Proposition A.2.6.5]{luriehtt}. For the remaining condition to hold, it suffices to show that the two constructions are mutually 
inverse. 

And indeed, given a sheaf $\rho\colon\mathcal{O}_{\mathcal{C}}\rightarrow\mathcal{E}$ of $\mathcal{O}_{\mathcal{C}}$-ideals, to show that
the sheaves $\mathcal{E}$ and $\mathcal{E}_J$ of $\mathcal{O}_{\mathcal{C}}$-ideals coincide, it suffices to show that each
$\mathcal{E}_C\subseteq\hat{\mathcal{C}}_{/y(C)}$ consists exactly of the $J_C$-local objects. This however follows from the definition of 
$\mathcal{O}_{\mathcal{C}}$-ideals in fairly straight-forward fashion.

Vice versa, given an accessible geometric kernel $J$, the fact that $J$ and $J_{\mathcal{E}_J}$ coincide follows from the equalities
$\mathcal{E}_J(C)=\hat{\mathcal{C}}_{/yC}[(J_C)^{-1}]$ and the assumption that each $J(C)$ is spanned exactly by the objects nullified in
$\hat{\mathcal{C}}_{/yC}[(J_C)^{-1}]$.
\end{proof}

\begin{remark}\label{remlargegeokerns}
Given a locally small but potentially large $\infty$-category $\mathcal{C}$, we can define a geometric kernel $J$ on
$\mathcal{C}$ to be a full cartesian subfibration over $\mathcal{C}$ as in Definition~\ref{defgeokernel} such that each fiber
$J(C)\subseteq\hat{\mathcal{C}}_{/yC}$ is a left exact subcategory, and each $J(C)\subseteq\hat{\mathcal{C}}_{/yC}$ is spanned exactly by 
those arrows over $yC$ which are nullified in the localization $\hat{\mathcal{C}}_{/yC}[(J_C)^{-1}]$ which now has to be assumed to exist 
explicitly. Then again we obtain a bijection between geometric kernels and sheaves of $\mathcal{O}_{\mathcal{C}}$-ideals as in 
Theorem~\ref{thmgeomkernel}.
\end{remark}

\begin{remark}\label{remgeomkernelandhlvops}
Given a small $\infty$-category $\mathcal{C}$, every higher Lawvere-Tierney operator $T$ on $\hat{\mathcal{C}}$ induces a sequence 
$(T_0)_{\bullet}\colon U_{\bullet}\rightarrow U_{\bullet}$ of endomorphisms on the bases of the object classifiers
$p_{\bullet}\colon\tilde{U}_{\bullet}\rightarrow U_{\bullet}$ in $\hat{\mathcal{C}}$. The fibers 
\[\xymatrix{
J_{\bullet}\ar@{^(->}[d]\ar[r]\ar@{}[dr]|(.3){\pbs} & 1\ar@{^(->}[d]^{\{1\}} \\
U_{\bullet}\ar[r]_{(T_0)_{\bullet}} & U_{\bullet}
}\]
yield a presheaf $J:=\mathrm{colim}J_{\bullet}$ in $\hat{\mathcal{C}}$ which exactly returns the associated accessible geometric kernel. That 
means, for each object $C\in\mathcal{C}$ the value $J(C)$ is exactly the class of arrows contained in the associated geometric kernel at $C$.
In that sense, accessible geometric kernels are exactly the (unions of the) fibers of their associated higher Lawvere-Tierney operators.
\end{remark}

\begin{remark}\label{remgeomkernelandlocequiv}
Given a geometric kernel $J$ on an $\infty$-category $\mathcal{C}$, let $\hat{\mathcal{C}}\rightarrow\mathcal{E}$ be its associated 
left exact reflective localization. Then $J$ consists exactly of the $\mathcal{E}$-local equivalences with representable codomain. In 
particular, the class of $\mathcal{E}$-local equivalences in $\hat{\mathcal{C}}$ (that is, the associated congruence) is exactly the class of 
maps which are locally contained in $J$.
\end{remark}

\begin{remark}\label{remgrothtop}
Every geometric kernel $J$ on a small $\infty$-category $\mathcal{C}$ induces a Grothendieck topology
$J_{-1}:=J\cap\mathcal{M}$ on $\mathcal{C}$. Indeed, unitality and pullback-stability are immediate; transitivity is straight-forward to 
show.
The respective localizations
$\hat{\mathcal{C}}\rightarrow\mathrm{Sh}_{J_{-1}}(\mathcal{C})\rightarrow\mathrm{Sh}_{J}(\mathcal{C})$ give a 
factorization of the composite into a topological followed by a cotopological localization.

A Grothendieck topology $J_{-1}$ on a small $\infty$-category $\mathcal{C}$ in the sense of \cite[6.2.2]{luriehtt} is 
not quite a geometric kernel on $\hat{\mathcal{C}}$ itself, but each such $J_{-1}$ is a lex modulator. By Remark~\ref{remlexmods}, it thus 
generates an accessible  geometric kernel $J$ on $\mathcal{C}$ such that $J_{-1}=J\cap\mathcal{M}$, and such that they both generate the same 
left exact (topological) localization on $\hat{\mathcal{C}}$. The definition of a geometric kernel however does restrict to the definition of 
a Grothendieck topology in the monic context in the following sense. Consider the pullback
\[\xymatrix{
\sum_{C\in\mathcal{C}}\mathrm{Sub}(yC)\ar@{->>}[d]\ar[r]\ar@{}[dr]|(.3){\pbs} & \sum_{Z\in\hat{\mathcal{C}}}\mathrm{Sub}(Z)\ar@{->>}[d]^t \\
\mathcal{C}\ar[r]_y & \hat{\mathcal{C}}.
}\]
Then a Grothendieck topology on $\mathcal{C}$ is a full cartesian subfibration $J\subseteq\sum_{C\in\mathcal{C}}\mathrm{Sub}(yC)$ such that 
each $J(C)\subseteq\mathrm{Sub}(yC)$ is, first, closed under meets, and second, is spanned exactly by those monomorphisms over $yC$ which are 
nullified in the localization $\hat{\mathcal{C}}_{/yC}[(J_C)^{-1}]$ for
\[J_C:=\bigcup_{g\in\mathcal{C}_{/C}}\Sigma_g[J(t(g))]\subseteq\mathrm{Sub}(yC).\]
\end{remark}

We end this section with a proof of the higher categorical version of Proposition~\ref{proplocalglobalfirstorder} which 
justifies the denotation ``\emph{sheaves} of $\mathcal{O}_{\mathcal{C}}$-ideals'' in as much as we show that every such
sheaf $\mathcal{E}$ of $\mathcal{O}_{\mathcal{C}}$-ideals is a stack of $\infty$-toposes (and \'{e}tale maps between such) 
for the geometric kernel $J$ associated to $\mathcal{E}$.

\begin{proposition}\label{propinftysemsheaf}
Given a small $\infty$-category  $\mathcal{C}$ together with a geometric kernel $J$ on $\mathcal{C}$, the associated sheaf of
$\mathcal{O}_{\mathcal{C}}$-ideals
\[\mathcal{E}\colon\mathcal{C}^{op}\rightarrow(\mathrm{Log},\mathrm{Et})\]
is a $J$-stack in the following sense: every $J$-cover $s\colon S\rightarrow yC$ of presheaves induces an equivalence
\begin{align*}
\mathcal{E}_C & \simeq\{yC,\mathcal{E}\}\xrightarrow{s^{\ast}}\{S,\mathcal{E}\}
\end{align*}
of weighted limits. In particular, whenever the geometric kernel $J$ is accessible, the sheaf of accessible
$\mathcal{O}_{\mathcal{C}}$-ideals
\[\mathcal{E}\colon\mathcal{C}^{op}\rightarrow(\mathrm{LTop},\mathrm{Et})\]
is a $J$-stack as well.
\end{proposition}
\begin{proof}
The proof is essentially the same as the proof of Proposition~\ref{proplocalglobalfirstorder}, only replacing monic 
descent by full descent.

Let $s\colon S\rightarrow yC$ be a map in $\hat{\mathcal{C}}$. Since the $\infty$-category
$(\mathrm{Log},\mathrm{Et})$ is complete (via \cite[Theorem 6.3.5.13]{luriehtt}), we first note that by
\cite[Theorem 2.33]{rovelliweights}, the weighted limit $\{S,\mathcal{E}\}$ may be computed as the ordinary 
limit
\begin{align*}
\mathrm{lim}(\mathrm{Un}(S)^{op}\overset{\mathrm{Un}(s)}{\twoheadrightarrow}\mathcal{C}^{op}\xrightarrow{\mathcal{E}}(\mathrm{Log},\mathrm{Et})),
\end{align*}
where $\mathrm{Un}(s)\colon\mathrm{Un}(S)\twoheadrightarrow\mathcal{C}$ is the Unstraightening of the presheaf $S$ and 
hence a right fibration. To compute this limit, it suffices to compute the limit of its postcomposition with the inclusion
$(\mathrm{Log},\mathrm{Et})\hookrightarrow\mathrm{Cat}_{\infty}$ also via
\cite[Proposition 6.3.2.3, Theorem 6.3.5.13]{luriehtt}. Therefore, we are to consider the (large) $\infty$-category
\begin{align}\label{equinftysemsheaf1}
\mathrm{lim}(\mathrm{Un}(S)^{op}\twoheadrightarrow\mathcal{C}^{op}\xrightarrow{\mathcal{E}}\mathrm{Cat}_{\infty}).
\end{align}
Since $\mathcal{E}\colon\mathcal{C}^{op}\rightarrow\mathrm{Cat}_{\infty}$ is a pointwise fully faithful subfunctor of 
the composition
$\hat{\mathcal{C}}_{/y(\cdot)}\colon\mathcal{C}^{op}\rightarrow\hat{\mathcal{C}}^{op}\rightarrow\mathrm{Cat}_{\infty}$, 
the limit (\ref{equinftysemsheaf1}) is a full $\infty$-subcategory of the limit
\begin{align}\label{equinftysemsheaf1.5}
\mathrm{lim}(\mathrm{Un}(S)^{op}\twoheadrightarrow\mathcal{C}^{op}\xrightarrow{\hat{\mathcal{C}}_{/y(\cdot)}}\mathrm{Cat}_{\infty}).
\end{align}
By \cite[Corollary 3.3.3.2]{luriehtt} this limit is given by the $\infty$-category $\mathrm{CSec}(\mathrm{gl})$ of 
cartesian sections of the Artin glueing ``$\mathrm{gl}$'' of the horizontal bottom composition in the following diagram.
\begin{align}\label{equinftysemsheaf2}
\begin{gathered}
\xymatrix{
\hat{{C}}\downarrow(y\circ \mathrm{Un}(s)) \ar[rr]\ar@{->>}[d]_{\mathrm{gl}}\ar@{}[dr]|(.3){\pbs} & & \mathrm{Fun}(\Delta^1,\hat{\mathcal{C}})\ar@{->>}[d]^t \\
\mathrm{Un}(S)\ar@{->>}[r]_{\mathrm{Un}(s)} & \mathcal{C}\ar[r]_{y} & \hat{\mathcal{C}}
}
\end{gathered}
\end{align}
Basically by construction, the $\infty$-category $\mathrm{CSec}(\mathrm{gl})$ in turn is equivalent to the $\infty$-
category $\mathrm{Cat}_{\infty}(\mathrm{Un}(S),\hat{\mathcal{C}})\downarrow^{\times}(y\circ \mathrm{Un}(s))$
of cartesian natural transformations over the bottom composition in (\ref{equinftysemsheaf2}). By descent of $\hat{\mathcal{C}}$, taking pointwise colimits yields an equivalence
\[\mathrm{colim}\colon\left(\mathrm{Cat}_{\infty}(\mathrm{Un}(S),\hat{\mathcal{C}})\downarrow^{\times}(y\circ \mathrm{Un}(s))\right)\xrightarrow{\simeq}\hat{\mathcal{C}}_{/\mathrm{colim}(y\circ\mathrm{Un}(s))}.\]
Because the source fibration $\hat{\mathcal{C}}_{/S}\twoheadrightarrow\hat{\mathcal{C}}$ is the Unstraightening of the representable $\hat{\mathcal{C}}(\blank,S)\colon\hat{\mathcal{C}}^{op}\rightarrow\mathcal{S}$, and because we have an equivalence $\hat{\mathcal{C}}(\blank,S)\circ y\simeq S$ due to the Yoneda Lemma, the square
\[\xymatrix{
\mathrm{Un}(S)\ar@{->>}[d]_{\mathrm{Un}(s)}\ar[r] & \hat{\mathcal{C}}_{/S}\ar@{->>}[d] \\
\mathcal{C}\ar[r]_y & \hat{\mathcal{C}}
}\]
is homotopy-cartesian. Thus, $\mathrm{Un}(S)\simeq(\mathcal{C}\downarrow S)$ is the $\infty$-category of elements of S, 
and
\[\mathrm{colim}(y\circ\mathrm{Un}(s))\simeq\mathrm{colim}(\mathrm{C}\downarrow S\hookrightarrow\hat{\mathcal{C}}_{/S}\twoheadrightarrow\hat{\mathcal{C}}).\]
This colimit is $S$ itself by \cite[Lemma 5.1.5.3]{luriehtt}, and so it follows that the limit in 
(\ref{equinftysemsheaf1.5}) is equivalent to the slice $\hat{\mathcal{C}}_{/S}$.

Accordingly, the limit (\ref{equinftysemsheaf1}) is equivalent to the full $\infty$-subcategory
$\mathcal{E}_{\mathrm{loc}}(S)\subseteq\hat{\mathcal{C}}_{/S}$, generated by the maps $T\rightarrow S$ in
$\hat{\mathcal{C}}$ which are locally contained in $\mathcal{E}$. In summary, we obtain an equivalence of the form
\[\{S,\mathcal{E}\}\simeq\mathcal{E}_{\mathrm{loc}}(S).\]
The map $s\colon S\rightarrow yC$ induces an adjoint pair
\begin{align}\label{equinftysemsheaf3}
(s_{\sharp},s^{\ast})\colon\mathcal{E}_{\mathrm{loc}}(S)\rightarrow\mathcal{E}_C
\end{align}
by Lemma~\ref{lemmarefllocbifib} (i.e.\ via the sharp-construction from \cite[Proposition 3.1.22]{abjfsheavesI}, see 
Remark~\ref{remcocartcharfacsys}) and pullback along $s$.
Now, if the map $s\colon S\rightarrow yC$ is in $J(C)$, the pair (\ref{equinftysemsheaf3}) is an adjoint equivalence by 
an application of \cite[Proposition 3.3.5]{abjfsheavesI}. This finishes the proof.
\end{proof}

\subsection{Classifications of sheaf theories over some simple examples}\label{subsecexples}

We end this section with a few simple examples of classifications of the left exact localizations of a given presheaf $\infty$-category.

\begin{example}\label{explelocsspaces}
The $\infty$-topos $\mathcal{S}$ has no non-trivial left exact localizations.
We show this in detail in two different ways to describe different proof techniques in the simplest environment. 

First, every left exact localization $\mathcal{S}\rightarrow\mathcal{E}$ can be factored through its topological 
and cotopological part $\mathcal{S}\rightarrow\mathcal{E}_{-1}\rightarrow\mathcal{E}$. Since $\mathcal{S}\simeq\hat{\ast}$, the topological 
localization $\mathcal{S}\rightarrow\mathcal{E}_{-1}$ is presented by a Grothendieck topology 
on the terminal $\infty$-category $\ast$. However the only Grothendieck topologies on $\ast$ are the two trivial ones, and so it follows that 
$\mathcal{E}_{-1}\in\{\ast,\mathcal{S}\}$. Either way, $\mathcal{E}_{-1}$ is hypercomplete, and hence has no non-trivial cotopological 
localizations.

Second, given a left exact localization $\mathcal{S}\rightarrow\mathcal{E}$, let $J$ be its associated (large) geometric kernel. Since 
$J$ is a left exact class of maps over the terminal space (and hence is but a class of spaces closed under finite limits),
and $\mathcal{E}$ is the localization of $\mathcal{S}$ at $J$, it follows that $\mathcal{E}$ is in fact the nullification of the class of 
spaces contained in $J$. Assume this nullification is not the identity. 
Then there is a non-contractible space $X$ which is contracted in $\mathcal{E}$. If $X$ is empty, then $\mathcal{E}\simeq\ast$ as (up to
homotopy) only the terminal space is $\emptyset$-local. If $X$ is non-empty, let $x\colon 1\rightarrow X$ be a point. Since $X$ is assumed to 
be non-trivial, there is some $n\geq 0$ such that $\pi_n(X,x)$ is 
non-trivial. It follows that the $n$-th loop space $\Omega_n(X,x)$ is not connected. We thus find points $p,q\in\Omega_n(X,x)$ in different 
path components, so the path space $\{p\}\times_{\Omega_n(X,x)}\{q\}$ is empty. However, since $\mathcal{S}\rightarrow\mathcal{E}$ is left 
exact, the space $\Omega_n(X,x)$ is mapped to a contractible object in $\mathcal{E}$, and hence so is the empty space. It follows again that
$\mathcal{E}\simeq \ast$. 
\end{example}

The first argument in Example~\ref{explelocsspaces} is an instance of a more general situation: whenever all (non-trivial) sieves on a small
$(\infty,1)$-category are representable, the $\infty$-topos $\hat{\mathcal{C}}$ has no non-topological left exact localizations.
This indeed applies to the following second simplest example.

\begin{example}\label{explesierpinski}
All left exact localizations of the Sierpi\'{n}ski $\infty$-topos $\mathcal{S}^{(\Delta^1)^{op}}$ are topological.
Therefore, we are to show that every topological localization of $\mathcal{S}^{(\Delta^1)^{op}}$ is hypercomplete. We note that 
every non-trivial subobject of a representable in $\mathcal{S}{(\Delta^1)^{op}}$ is again representable: the frame $\mathrm{Sub}(y(0))$ is 
trivial, while $\mathrm{Sub}(y(1))$ consists of the three subobjects given by $\emptyset$, $y(\mathrm{id}_1)$ and $y(0\rightarrow 1)$.
It follows from \cite[Remark 4.1]{rs_hgst} that every topological localization of $\mathcal{S}{(\Delta^1)^{op}}$ 
is hypercomplete. Thus, $\mathcal{S}^{(\Delta^1)^{op}}$ has exactly one non-trivial left exact localization, characterized by 
the atomic Grothendieck topology. The fact that the latter is hypercomplete also follows directly from the fact that the $\infty$-category of 
sheaves for this topology is equivalent to the $\infty$-category of spaces.
\end{example}

\begin{example}
More generally, the argument of Example~\ref{explesierpinski} for $\Delta^1$ applies to all complete linear orders $\mathbb{P}$. 
\end{example}

\begin{example}\label{explefreemonoid}
Let $F(\ast)$ be the free monoid generated by one element $g$ considered as a category with one object $\ast$. Then all left exact 
localizations of the $\infty$-topos $\widehat{F(\ast)}$ again are topological, as all non-trivial sieves on the object $\ast\in F(\ast)$ are 
generated by an arrow $g^n$ for some $n\geq 0$, and hence are representable.
\end{example}

The second argument in Example~\ref{explelocsspaces} can be generalized as follows.

\begin{example}\label{expllocsgrpds}
For any space $X$, all left exact localizations of $\hat{X}$ are topological. Indeed, given any (large) geometric kernel $J$ on $X$, and 
given any object $x\in X$, we have $\hat{X}_{/y(x)}\simeq\widehat{X_{/x}}\simeq\hat{\ast}\simeq\mathcal{S}$, and so each $J(x)\in\mathcal{S}$ 
is a left exact class of spaces. It again follows that $J(x)$ is either exactly the class of 
contractible spaces, or $J(x)$ contains the empty space. The latter case implies that the associated localization of $\mathcal{S}$ 
(equivalent to $\hat{X}_{/y(x)}[J_x^{-1}]$) is terminal. This in turn implies that
$J(x)=\mathcal{S}$, because $J(x)$ consists by assumption of all spaces contracted in the associated localization
$\mathcal{S}\rightarrow\ast$. In other words, each $J(x)\subseteq\hat{X}_{/y(x)}$ consists either exactly of the terminal objects in
$\hat{X}_{/y(x)}$, or it consists of all maps over $y(x)$. Thus, the associated sheaf theory is generated by the
Grothendieck topology $J_{-1}=J\cap\mathcal{M}$ on $X$, which is given (up to equivalence) by $J_{-1}(x)=\{1_x\}$ whenever $J(x)$ is the 
class of terminal objects, and by $J_{-1}(x)=\mathrm{Sub}(y(x))$ whenever $J(x)=\hat{X}_{/y(x)}$.
\end{example}

\begin{example}\label{expllocslocgrpds}
For any space $X$ consider the simplicial category $F(X)$ with two objects $0,1$ and hom-spaces $F(X)(0,0)=F(X)(1,1)=\Delta^0$ and
$F(X)(0,1)=X$. This includes for example the case of the free parallel pair of arrows 
$F(\partial\Delta^1)=\xymatrix{0\ar@<.5ex>[r]\ar@<-.5ex>[r] & 1}$, and the Sierpi\'{n}ski $\infty$-topos from 
Example~\ref{explesierpinski}. Then, $F(X)$ again has exactly one non-trivial Grothendieck topology $J$, given by $J(0)=\{1_0\}$
and $J(1)=\{1_1,\bigwedge_{x\in\pi_0(X)}y(x)\}$. A presheaf $G$ -- given by a pair $(G(0),G(1))$ of spaces together with a functor 
$G_1\colon X\rightarrow\mathrm{Fun}(G(1),G(0))$ -- is a sheaf for $J$ if and only if the natural map
\[(G_1(x)\mid x\in\pi_0(X))\colon G(1)\rightarrow\prod_{x\in\pi_0(X)}G(0)\]
is an equivalence. It follows that evaluation $\mathrm{ev}_0\colon\mathrm{Sh}_J(F(X))\rightarrow\mathcal{S}$ is an equivalence of
$\infty$-categories, with an inverse given by the apparent functor mapping a space $Z$ to the functor
\[(\pi_{(\cdot)})\colon X\rightarrow\pi_0(X)\rightarrow\mathrm{Fun}(\prod_{x\in\pi_0(X)}Z,Z).\]
Thus, in particular, it is again hypercomplete, and so all left exact localizations of $\widehat{F(X)}$ are topological.
However, the argument via representable modulators from the earlier examples does not apply. 
\end{example}

\begin{example}
Recall that a left exact localization $\rho\colon\mathcal{B}\rightarrow\mathcal{E}$ of an $\infty$-topos is locally connected whenever the 
associated cartesian reflective localization $\rho_{\bullet}\colon t_{\mathcal{B}}\rightarrow\mathcal{E}_{\bullet}$ has a (further) left 
adjoint in $\mathbf{Cart}(\mathcal{B})$. This implies that the functor $\rho\colon\mathcal{B}\rightarrow\mathcal{E}$ preserves 
dependent products. But every left exact localization
$\rho\colon\mathcal{B}\rightarrow\mathcal{E}$ of an $\infty$-topos $\mathcal{B}$ which 
preserves dependent products is topological -- in fact, even is the nullification at a set of subterminal objects: to every map $f\colon X\rightarrow Y$ in $\mathcal{B}$ we may assign the $(-1)$-truncated object 
$\mathrm{isequiv}(f)$ (\cite[Chapter 4]{hott}) such that $f$ is an equivalence in $\mathcal{B}$ if and only $\mathrm{isequiv}(f)$ is contractible. Whenever
$\rho\colon\mathcal{B}\rightarrow\mathcal{E}$ preserves dependent products, we have
$\rho(\mathrm{isequiv}(f))\simeq\mathrm{isequiv}(\rho(f))$ in $\mathcal{E}$. It follows that $\mathcal{E}$ is already the localization at the 
class of monomorphisms it inverts. It in particular is accessible.
\end{example}

Given the fact that all sheaf theories in the examples considered above are generated by a Grothendieck 
topology, one may wonder how complicated an $\infty$-category $\mathcal{C}$ has to be for $\hat{\mathcal{C}}$ to exhibit a non-topological 
(accessible) left exact localization. Famously, the locale $2^{S^{\infty}}$ is complicated enough already (\cite[Section 11.3]{rezkhtytps}), 
but whether it is in some sense minimal with this regard is unclear to the author, as the argument of the authors is very 
specifically designed for the locale at hand. 
 

\section{Morphisms of cartesian sites and the Comparison Lemma}\label{secsitemorphisms}

In ordinary topos theory there is a fair amount of theory regarding the category of Grothendieck sites, which reduces statements about
geometric morphisms between toposes to statements about morphisms of their according presentation as sites over a given generating category. 
In this section we consider morphisms of cartesian $(\infty,1)$-sites both in a general and in a ``topological'' context, the latter of which 
is rather poorly behaved. We exemplify this by proving a version of the Comparison Lemma for $(\infty,1)$-sites, but providing a
counter-example to a topological version (for Grothendieck topologies) of it. Therefore we 
recall that in 1-topos theory every suitably dense functor $F\colon\mathcal{C}\rightarrow\mathcal{D}$ from a small category
$\mathcal{C}$ to a category $\mathcal{D}$ equipped with a Grothendieck site $K$ induces a Grothendieck topology $J$ on $\mathcal{C}$ such 
that the restriction functor $F^{\ast}\colon\hat{\mathcal{D}}\rightarrow\hat{\mathcal{C}}$ restricts to an equivalence
$F^{\ast}\colon\mathrm{Sh}_K(\mathcal{D})\rightarrow\mathrm{Sh}_J(\mathcal{C})$ of sheaf theories (see e.g.\ \cite[Theorem 2.2.3]{elephant} 
where this is called the Comparison Lemma). We show that if we strengthen the notion of density, we obtain an according statement for left 
exact $\infty$-categories equipped with a sheaf theory (Theorem~\ref{thmcomplemma}). However, if we consider the special case of Grothendieck 
topologies and require the according ``topological'' density condition, we find counter-examples to a corresponding Comparison Lemma even in 
the best case scenario (Proposition~\ref{propcomplemmafail}). The proof implies that the canonical Grothendieck topologies relative to a left 
exact functor from \cite[Section 6.2.4]{luriehtt} generally do not induce equivalences of sheaf theories (even if the $\infty$-topos we 
start with is presented by a topological localization of a presheaf $\infty$-category to begin with).

\begin{notation}
For the following, a \emph{small $(\infty,1)$-site} is a tuple which consists of a small $\infty$-category $\mathcal{C}$ together with any of 
the equivalent structures:
\begin{itemize}
\item An accessible left exact reflective localization $\mathcal{E}$ of $\hat{\mathcal{C}}$.
\item A left exact modality $(\mathcal{L},\mathcal{R})$ of small generation on $\hat{\mathcal{C}}$.
\item A congruence $\mathcal{L}$ of small generation on $\hat{\mathcal{C}}$.
\item A higher Lawvere-Tierney operator $T$ on $\hat{\mathcal{C}}$.
\item A sheaf $\mathcal{E}$ of accessible $\mathcal{O}_{\mathcal{C}}$-ideals.
\item An (accessible) geometric kernel $J$ on $\mathcal{C}$.
\end{itemize} 
Any of these will be referred to as a sheaf theory on $\mathcal{C}$.
If not otherwise specified, a small $(\infty,1)$-site will be given by a small $\infty$-category $\mathcal{C}$ together with an accessible 
geometric kernel $J$ on $\mathcal{C}$.
A \emph{Grothendieck site} will denote a tuple $(\mathcal{C},J)$ where $\mathcal{C}$ is a small $\infty$-category and $J$ is a Grothendieck 
topology on $J$ (\cite[Section 6.2.2]{luriehtt}, \cite{abjfsheavesII}). As noted in Remark~\ref{remgrothtop}, a Grothendieck site is not
an $(\infty,1)$-site itself, but generates a unique such. The Grothendieck topologies thereby single out exactly the topological
$(\infty,1)$-sites.
\end{notation}

\begin{definition}
A \emph{(large) $(\infty,1)$-site} is a locally small $\infty$-category $\mathcal{C}$ together with either of the following equivalent data:
\begin{itemize}
\item A left exact reflective localization $\mathcal{E}$ of $\hat{\mathcal{C}}$.
\item A left exact modality $(\mathcal{L},\mathcal{R})$ on $\hat{\mathcal{C}}$.
\item A sheaf $\mathcal{E}$ of $\mathcal{O}_{\mathcal{C}}$-ideals.
\item A geometric kernel $J$ on $\mathcal{C}$ in the sense of Remark~\ref{remlargegeokerns}.
\end{itemize}
A tuple $(\mathcal{C},J)$ referred to as an $(\infty,1)$-site will be given by a locally small $\infty$-category $\mathcal{C}$ together with 
a geometric kernel $J$ on $\mathcal{C}$. Its associated left exact localization will be denoted by
$\hat{\mathcal{C}}\rightarrow\mathrm{Sh}_J(\mathcal{C})$.
\end{definition}

As every small $(\infty,1)$-site is a large $(\infty,1)$-site, in the rest of this section we will use the unspecified term
``$(\infty,1)$-site'' to refer to large $(\infty,1)$-sites.

\begin{definition}\label{defmorphismsites}
Given a small $(\infty,1)$-site $(\mathcal{C},J)$ and an $(\infty,1)$-site $(\mathcal{D},K)$, a functor
$F\colon\mathcal{C}\rightarrow\mathcal{D}$ is \emph{cover-preserving} if the associated left Kan extension
$F_!\colon\hat{\mathcal{C}}\rightarrow\hat{\mathcal{D}}$ takes all elements in $J(C)$ for any $C\in\mathcal{C}$ to elements in $K(F(C))$.
Whenever $\mathcal{C}$ and $\mathcal{D}$ are left exact $\infty$-categories, a \emph{morphism
$F\colon(\mathcal{C},J)\rightarrow (\mathcal{D},K)$ of cartesian sites} is a left exact cover-preserving functor. We denote the full subspace 
of morphisms of cartesian sites in $\mathrm{Fun}^{\mathrm{lex}}(\mathcal{C},\mathcal{D})$ by
$\mathrm{CSite}((\mathcal{C},J),(\mathcal{D},K))$. In particular, restricting this definition to such morphisms with small codomain
$\mathcal{D}$, this defines the $\infty$-category $\mathrm{CSite}$ of small cartesian $(\infty,1)$-sites. We say that
a functor of cartesian sites $F\colon(\mathcal{C},J)\rightarrow(\mathcal{D},K)$ is an \emph{equivalence of cartesian sites} if 
the associated geometric functor $F^{\ast}\colon\mathrm{Sh}_K(\mathcal{D})\rightarrow\mathrm{Sh}_J(\mathcal{C})$ is an equivalence of
$\infty$-categories.
\end{definition}

\begin{definition}
Say, a cartesian $(\infty,1)$-site $(\mathcal{D},K)$ \emph{has a presentation with small base} if there is an equivalence
$F\colon(\mathcal{C},J)\rightarrow(\mathcal{D},K)$ of cartesian sites such that $\mathcal{C}$ is a small $\infty$-category. The 
$(\infty,1)$-site $(\mathcal{D},K)$ is \emph{essentially small} if furthermore the geometric kernel $J$ is accessible.
\end{definition}

\begin{remark}
The distinction between presentations with small base and essential smallness arises, because it is open to date whether every left exact 
reflective localization of an $\infty$-topos is accessible (\cite[Remark 6.1.0.5]{luriehtt}). If it turns out that accessibility is indeed 
derivable, the distinction can be ignored. The same applies to the distinction between geometric kernels and accessible geometric kernels on 
a small $\infty$-category $\mathcal{C}$, etc.
\end{remark}

\begin{definition}\label{defdensemorphismsites}
Given an $(\infty,1)$-site $(\mathcal{D},K)$, a functor $F\colon\mathcal{C}\rightarrow\mathcal{D}$ out of a small $\infty$-category
$\mathcal{C}$ is \emph{$K$-dense} if for all 
objects $D\in\mathcal{D}$ the gap map out of the canonical colimit
\begin{align}\label{diagdefdense}
\mathrm{colim}\left(F\downarrow D\xrightarrow{\pi_D}\mathcal{D}_{/D}\xrightarrow{y_{/D}}\hat{\mathcal{D}}_{/yD}\right)\rightarrow yD
\end{align}
is contained in $K(D)$. Given a Grothendieck site $(\mathcal{D},K)$, a functor $F\colon\mathcal{C}\rightarrow\mathcal{D}$ is
\emph{topologically $K$-dense} if the $(-1)$-truncation of the colimit (\ref{diagdefdense}) is a $K$-cover.
\end{definition}

\begin{remark}
To relate these notions to the conventional synonyms for Grothendieck sites, we first note that every functor
$F\colon\mathcal{C}\rightarrow\mathcal{D}$ out of a small $\infty$-category induces an adjoint pair
$(F_!,F^{\ast})\colon\hat{\mathcal{C}}\rightarrow\hat{\mathcal{D}}$ on presheaf $\infty$-categories. As every adjoint pair between
$\infty$-categories induces adjoint pairs of according slices, and as the right adjoint $F^{\ast}$ preserves monomorphisms, for every object $C\in\mathcal{C}$ we obtain an adjoint pair
\[(F_!)_{-1}\colon\mathrm{Sub}(yC)\rightarrow\mathrm{Sub}(yFC)\colon F^{\ast},\]
where the left adjoint maps a subobject $A\hookrightarrow yC$ to its image under $F_!$ and then $(-1)$-truncates. The right adjoint maps a 
subobject $B\hookrightarrow yFC$ to the subobject $F^{\ast}B\hookrightarrow F^{\ast}y(FC)$ and then pulls it back along the unit at $yC$. 
One says that $(F_!)_{-1}(S\hookrightarrow yC)$ is the \emph{sieve generated by $S\hookrightarrow yC$ under $F$}, that is, the sieve 
generated by the set $\{F(s)|s\in S\}$ on $FC$ in $\mathcal{D}$ .

Now, whenever $(\mathcal{C},J)$ is a Grothendieck site and $(\mathcal{D},K)$ is an $(\infty,1)$-site, then a functor 
$F\colon\mathcal{C}\rightarrow \mathcal{D}$ is cover-preserving when considered as functor between associated $(\infty,1)$-sites if and only 
if every $J$-cover is mapped to an element in $K$ (as the left adjoint $F_!$ preserves colimits).
Whenever $(\mathcal{D},K)$ is a Grothendieck site itself and the functor $F$ is cover-preserving, then it is cover-preserving in the 
conventional sense: i.e.\ for every $J$-cover $S$ on an object $C$ the sieve generated by $S$ 
under $F$ is a $K$-cover on $FC$. Indeed, this follows from the fact that for any given sieve $S\hookrightarrow yC$ in $J(C)$ the map
$F_!(S)\rightarrow yFC$ is a $\mathrm{Sh}_K(\mathcal{D})$-local equivalence by assumption, and hence so is its $(-1)$-truncation. That means, $(F_!)_{-1}(S\hookrightarrow yC)$ is contained in $K(F(C))$.
 The other direction holds whenever $F$ is a 
morphism of cartesian sites, since in this case the left adjoint $F_!$ preserves monomorphisms as well and so the sieve generated by a
$J$-cover $S$ under $F$ is just $F_!(S)$. Generally however, a cover-preserving morphism of Grothendieck sites in the conventional sense only 
maps $J$-covers $S\hookrightarrow yC$ to effective epimorphisms in $\mathrm{Sh}_K(\mathcal{D})$ rather than to equivalences. An according 
remark applies to the two notions of density from Definition~\ref{defdensemorphismsites}.
\end{remark}

Clearly, the definition of a cover-preserving functor $F\colon(\mathcal{C},J)\rightarrow (\mathcal{D},K)$ between $(\infty,1)$-sites is 
defined in such a way that the induced precomposition $F^{\ast}\colon\hat{\mathcal{D}}\rightarrow\hat{\mathcal{C}}$ of presheaves restricts to a precomposition
\begin{align}\label{equcovpresfctr}
F^{\ast}\colon\mathrm{Sh}_K(\mathcal{D})\rightarrow\mathrm{Sh}_J(\mathcal{C})
\end{align}
of according sheaf theories. The same goes for the notion of cover-preserving functors of Grothendieck sites.

\begin{theorem}[Comparison Lemma]\label{thmcomplemma}
Let $(\mathcal{D},K)$ be a cartesian $(\infty,1)$-site, $\mathcal{C}$ be a small left exact $\infty$-category and
$F\colon\mathcal{C}\rightarrow\mathcal{D}$ be a cartesian functor.
\begin{enumerate}
\item Whenever the localization $\mathrm{Sh}_K(\mathcal{D})$ is locally presentable, the ordered class of accessible geometric kernels $J$ on
$\mathcal{C}$ such that the functor $F\colon(\mathcal{C},J)\rightarrow(\mathcal{D},K)$ is cover-preserving has a maximal element $J_K$.
\item Whenever $F$ is fully faithful and $K$-dense, there is a geometric kernel $J_K$ on $\mathcal{C}$ such that
$F\colon(\mathcal{C},J_K)\rightarrow(\mathcal{D},K)$ is an equivalence of cartesian $(\infty,1)$-sites.
\item Whenever $F$ is fully faithful and $K$-dense, and $\mathrm{Sh}_K(\mathcal{D})$ is locally presentable as well, the $(\infty,1)$-sites
$(\mathcal{C},J_K)$ from Parts 1 and 2 coincide.
\end{enumerate}
\end{theorem}
\begin{proof}
For Part 1,  given a left exact functor $F\colon\mathcal{C}\rightarrow\mathcal{D}$ and an object $C\in\mathcal{C}$, the 
left Kan extension $F_!\colon\hat{\mathcal{C}}\rightarrow\hat{\mathcal{D}}$ is left exact and hence induces a left exact left adjoint 
composition
\[\hat{\mathcal{C}}\xrightarrow{F_!}\hat{\mathcal{D}}\xrightarrow{\rho}\mathrm{Sh}_K(\mathcal{D})\]
of $\infty$-topoi. It follows from \cite[Lemma 4.2.7]{abjfsheavesI} then that the preimage of this precomposition is a congruence
$\mathcal{L}_K$ of small generation. If we denote the associated accessible geometric kernel on $\mathcal{C}$ by $J_K$,
then $F_!$ maps $J_K$-covers to $K$-covers by construction. Thus, the right adjoint
$F^{\ast}\colon\hat{\mathcal{D}}\rightarrow\hat{\mathcal{C}}$ restricts to a functor
\[F^{\ast}\colon\mathrm{Sh}_K(\mathcal{D})\rightarrow\mathrm{Sh}_{J_K}(\mathcal{C})\]
with left exact left adjoint $\rho F_!\iota\colon\mathrm{Sh}_{J_K}(\mathcal{C})\rightarrow\mathrm{Sh}_K(\mathcal{D})$.
The fact that $J_K$ is the maximal geometric kernel $J$ with this property is similarly straight-forward: any object $f\in J(C)$ is left 
orthogonal to $F^{\ast}(X)$ for all $K$-sheaves $X$ on
$\mathcal{D}$. That means that the map $F_!(f)\in\hat{\mathcal{D}}_{/y((C))}$ is left orthogonal to all $K$-sheaves on $\mathcal{D}$, which 
which in turn implies that $F_!(f)$ is contained in $K(F(C))$. Thus, by definition, $f\in J_K(C)$, and so $J\subseteq J_K$. This finishes 
Part 1.

Now, for Part 2, let $(\mathcal{L}_K,\mathcal{R}_K)$ be the left exact modality on $\hat{\mathcal{D}}$ associated to $K$. We first show that 
under the given assumptions the counit $\epsilon_S\colon F_! F^{\ast}S\rightarrow S$ is contained in $\mathcal{L}_K$ for every presheaf 
$S\in\hat{\mathcal{D}}$.
Therefore, we note that the counit $\varepsilon_S\colon F_!F^{\ast}S\rightarrow S$ is the gap map 
\[\mathrm{colim}(yF\downarrow S\xrightarrow{\simeq}\mathcal{C}\downarrow F^{\ast}S\rightarrow\mathcal{C}\xrightarrow{F}\mathcal{D}\xrightarrow{y}\hat{\mathcal{D}})\xrightarrow{\varepsilon_S} S\]
which represents the associated canonical cocone. By a simple rearrangement of indexing $\infty$-categories, one sees that the map
$\varepsilon_S$ itself is the colimit of the composition
\[\mathcal{D}\downarrow S\rightarrow\mathcal{D}\xrightarrow{\varepsilon y}\hat{\mathcal{D}}^{\Delta^1}.\]
Thus, the counit $\varepsilon_S$ is a colimit of the counits
\[\varepsilon_{yD}\colon F_!F^{\ast}yD\rightarrow yD\]
for objects $D\in\mathcal{D}$. But by virtue of $K$-density of the functor $F$ and fully faithfulness of the Yoneda embedding, each of these 
maps is contained in $K(D)$. Thus, the colimit $\varepsilon_S$ is contained in $\mathcal{L}_K$. 	

Thus, for every object $D\in\mathcal{S}$ and every map $S\rightarrow yD$, the horizontal maps of the counit
\[\xymatrix{
F_!F^{\ast}S\ar[d]_{F_!F^{\ast}s}\ar[r]^(.6){\varepsilon_S} & S\ar[d]^s\\
F_!F^{\ast}yD\ar[r]_(.6){\varepsilon_D} & yD 
}\]
are $\mathcal{L}_K$-maps. By 2-out-of-3, it follows that $S\rightarrow yD$ is contained in $K$ if and only if
$F_! F^{\ast}S\rightarrow F_! F^{\ast}yD$ is contained in $\mathcal{L}_K$. Since the left adjoint $F_!$ is fully faithful, it follows that the 
preimage $\mathcal{L}_{F}=F_!^{-1}[\mathcal{L}_{K}]$ and the image $F^{\ast}[\mathcal{L}_K]$ in $\hat{\mathcal{C}}$ coincide.
In particular, we obtain a left exact modality on $\hat{\mathcal{C}}$ given by the class $\mathcal{L}_F$ and the class of retracts of the 
image $F^{\ast}[\mathcal{R}_K]$. This pair induces a left exact reflective localization
$\hat{\mathcal{C}}\rightarrow\mathrm{Sh}_{\mathcal{L}_F}(\mathcal{C})$ together with a factorization
\begin{align}\label{diagresfac}
\begin{gathered}
\xymatrix{
\hat{\mathcal{C}} & \hat{\mathcal{D}}\ar[l]_{F^{\ast}} \\
\mathrm{Sh}_{\mathcal{L}_F}(\mathcal{C})\ar@{^(->}[u]_{\iota} & \mathrm{Sh}_K(\mathcal{D}).\ar[l]_{F^{\ast}}\ar@{^(->}[u]_{\iota}
}
\end{gathered}
\end{align}
To show that the bottom horizontal functor is part of an equivalence, one can basically copy the argument via the right Kan extension 
$F_{\ast}\colon\mathrm{Sh}_{\mathcal{L}_F}(\mathcal{C})\rightarrow\mathrm{Sh}_K(\mathcal{D})$ from the proof of
\cite[Theorem 2.2.3]{elephant}, or directly work with the pair $(F_!, F^{\ast})$ at hand instead. Indeed, the left adjoint
$F_!\colon\hat{\mathcal{C}}\rightarrow\hat{\mathcal{D}}$ induces a left adjoint $\rho F_!\colon\mathrm{Sh}_{\mathcal{L}_F}(\mathcal{C})\rightarrow\mathrm{Sh}_K(\mathcal{D})$ to $F^{\ast}$ via \cite[Proposition 5.5.4.20]{luriehtt}. For every presheaf $X\in\hat{\mathcal{C}}$, we obtain a triangle
\begin{align}\label{diagresfac2}
\begin{gathered}
\xymatrix{
X\ar[r]_{\eta_X}^{\simeq}\ar[dr] & F^{\ast}F_! X\ar[d]^{F^{\ast}(\eta_{F_!(X)})} \\
 & F^{\ast}\iota\rho F_! X,
}
\end{gathered}
\end{align}
where the top arrow is an equivalence by fully faithfulness of $F_!$, and the vertical map is contained in
$\mathcal{L}_F=F^{\ast}[\mathcal{L}_K]$. The bottom vertex is always an $\mathcal{L}_F$-sheaf by the factorization in (\ref{diagresfac}). 
Thus, whenever $X$ is an $\mathcal{L}_F$-sheaf itself, the composition in (\ref{diagresfac2}) is an equivalence, and thus
$F^{\ast}\circ(\rho F_!)\simeq 1$. Vice versa, we have already argued that the counit
\[\varepsilon_S\colon F_! F^{\ast}(S)\rightarrow S\]
is contained in $\mathcal{L}_K$ for all $S\in\hat{\mathcal{D}}$. Thus, it is mapped to an equivalence under the sheafification
$\rho\colon\hat{\mathcal{D}}\rightarrow\mathrm{Sh}_K(\mathcal{D})$. In particular, the map $\rho(\varepsilon_S)\colon\rho F_! F^{\ast}(S)\rightarrow S$ is an equivalence for all $S\in\mathrm{Sh}_K(\mathcal{D})$.

Lastly, for Part 3, let $(\mathcal{C},(J_K)_1)$ and $(\mathcal{C},(J_K)_2)$ be the cartesian sites obtained from Parts 1 and 
2, respectively. By virtue of the maximality property of $(J_K)_1$, we obtain an inclusion
$\mathrm{Sh}_{(J_K)_1}(\mathcal{C})\subseteq\mathrm{Sh}_{(J_K)_2}(\mathcal{C})$. But for every $X\in\mathrm{Sh}_{(J_K)_2}{\mathcal{C}}$, we 
have $F^{\ast}(\rho F_!)(X)\simeq X$ by Part 2. As $(\rho F_!)(X)\in\mathrm{Sh}_K(\mathcal{D})$, and
$F^{\ast}\colon\mathrm{Sh}_K(\mathcal{D})\rightarrow\hat{\mathcal{C}}$ factors through $\mathrm{Sh}_{(J_K)_1}{\mathcal{C}}$, we have
$X\in\mathrm{Sh}_{(J_K)_1}{\mathcal{C}}$. Thus, the two sheaf theories coincide.

\end{proof}

\begin{remark}
The assumption of local presentability of the localization $\mathrm{Sh}_K(\mathcal{D})$ in Theorem~\ref{thmcomplemma}.1 is a means to assure 
that the associated localization $\mathrm{Sh}_{J_K}(\mathcal{C})$ exists in the first place. An argument without any such assumption would be 
possible for instance when given a factorization of the geometric morphism
$(\rho F_!,\iota F^{\ast})\colon\hat{\mathcal{C}}\rightarrow\mathrm{Sh}_K(\mathcal{D})$ of $\infty$-logoi into a ``surjection'' followed by 
an embedding (in the sense of \cite[Chapter A.4.2]{elephant} for ordinary toposes). Given the Beck Monadicity Theorem~\ref{thmbeckmonadicity}, 
a construction of (surjection,embedding)-factorizations for $\infty$-logoi is almost a straight-forward matter, however there is a critical 
caveat: in order for the usual construction to succeed, the left adjoint of any geometric morphism $(f^{\ast},f_{\ast})$ between $\infty$-
logoi ought not only to be left exact but furthermore to preserve limits of $f^{\ast}$-split cosimplicial objects. Such however are limits 
of infinite index. The reason this problem does not arise in the context of ordinary topos theory is that here totalizations of cosimplicial 
objects are simply equalizers, and being finite limits such are automatically preserved by the left adjoints of geometric morphisms. The 
existence of (surjection,embedding)-factorizations for geometric morphisms between $\infty$-logoi (or $\infty$-toposes for that matter)
is hence unclear to the author.
\end{remark}

The assumption of local presentability of the localization $\mathrm{Sh}_K(\mathcal{D})$ in Theorem~\ref{thmcomplemma}.1 is not the only way 
to assure the existence of the associated localization $\mathrm{Sh}_{J_K}(\mathcal{C})$ however. Whenever $(\mathcal{D},K)$ is instead 
assumed to be a (large) Grothendieck site, a similar argument yields the existence of a (maximal) Grothendieck topology $J_K$ on
$\mathcal{C}$ such that the cartesian functor $F\colon(\mathcal{C},J_K)\rightarrow(\mathcal{D},K)$ is cover-preserving. In the special case 
where $(\mathcal{D},\mathrm{Geo}_{-1})$ is a semi-topos $\mathcal{D}$ equipped with its ordinary geometric topology $\mathrm{Geo}_{-1}$, this 
is the ``canonical Grothendieck topology`` $J$ on $\mathcal{C}$ with respect to $F$ constructed in \cite[Section 6.2.4]{luriehtt}. If one 
furthermore replaces $K$-density with topological $K$-density in Theorem~\ref{thmcomplemma}.2, one also obtains that
\[F^{\ast}\colon\mathrm{Sh}_K(\mathcal{D})\rightarrow\mathrm{Sh}_{J_K}(\mathcal{C})\]
is a geometric embedding, which in ordinary topos theory is part of \cite[Section C.2.2]{elephant}. In fact, for 1-toposes in this situation 
(even with less assumptions), the functor $F\colon(\mathcal{C},J_K)\rightarrow(\mathcal{D},K)$ is an equivalence of cartesian sites.
In \cite[Theorem C.2.2.3]{elephant}, this is called the Comparison Lemma. Further relaxations of conditions on the diagram 
$F\colon\mathcal{C}\rightarrow\mathcal{D}$ are considered for example in
\cite[Section 11]{shulmanexact}).
However, we will give a counter-example to the Comparison Lemma for Grothendieck topologies in 
the corresponding context of topological density in Proposition~\ref{propcomplemmafail}. 

Therefore, we recall that a small $\infty$-category $\mathcal{E}$ is called $\kappa$-geometric (for some 
cardinal $\kappa$) in \cite[Section 5]{rs_hgst} whenever it has finite limits and universal $\kappa$-small colimits. Every $\kappa$-geometric 
$\infty$-category $\mathcal{E}$ can be augmented to its canonical ``higher geometric'' small cartesian $(\infty,1)$-site
$(\mathcal{E},\mathrm{Geo}_{\kappa})$, whose associated sheaves are exactly those presheaves $\mathcal{E}^{op}\rightarrow\mathcal{S}$ which 
take colimits of $\kappa$-small ``higher covering diagrams'' in $\mathcal{E}$ to limits in $\mathcal{S}$. We say that a (potentially large 
but locally small) $\infty$-category $\mathcal{E}$ is geometric if it is $\kappa$-geometric for all cardinals $\kappa$. 
Accordingly, the $\infty$-category $\mathrm{Sh}_{\mathrm{Geo}}(\mathcal{E})\subseteq\hat{\mathcal{E}}$ of higher geometric sheaves consists 
exactly of those presheaves (of small spaces) which take colimits of small higher covering diagrams in $\mathcal{E}$ to limits in
$\mathcal{S}$. In the following we show that $(\mathcal{E},\mathrm{Geo})$ is a (large) cartesian site as well whenever $\mathcal{E}$ is an
$\infty$-topos. This higher geometric $(\infty,1)$-site is the canonical one associated to an $\infty$-topos $\mathcal{E}$ in the precise 
sense that the Yoneda embedding of $\mathcal{E}$ factors through an equivalence
$y\colon\mathcal{E}\xrightarrow{\simeq}\mathrm{Sh}_{\mathrm{Geo}}(\mathcal{E})$ (\cite[Theorem 5.19]{rs_hgst}). 

For the following, as is often practice, for a given $\infty$-topos $\mathcal{E}$ we will assume that there is a Grothendieck 
universe $\mathcal{V}$ (e.g.\ given by an inaccessible cardinal $\nu$) that $\mathcal{E}$ belongs to. We thus let $\mathcal{S}^+$ denote the 
$\infty$-category of all spaces, and identify $\mathcal{S}$ with the $\infty$-category $\mathcal{S}^+_{\nu}$ of $\nu$-small spaces.

\begin{corollary}\label{corgeoesssmall}
Let $\mathcal{E}$ be an $\infty$-topos. Then the inclusion $\mathrm{Sh}_{\mathrm{Geo}}(\mathcal{E})\hookrightarrow\hat{\mathcal{E}}$ has a 
left exact left adjoint. Furthermore, the associated cartesian site $(\mathcal{E},\mathrm{Geo})$ of higher geometric sheaves on $\mathcal{E}$ 
is essentially small.
\end{corollary}

\begin{proof}
Suppose $\kappa<\nu$ is a regular cardinal such that $\mathcal{E}$ is $\kappa$-presentable and such that the subcategory
$\mathcal{E}_{\kappa}\subset\mathcal{E}$ of $\kappa$-compact objects is closed under finite limits. If
$\iota\colon\mathcal{E}_{\kappa}\hookrightarrow\mathcal{E}$ denotes the canonical inclusion, the adjoint pair
$\iota^{\ast}\adj\iota_{\ast}$ given by restriction and right Kan extension along $\iota$, respectively, induces a diagram of
$\infty$-categories as follows, to be explained from front to back in the following step by step.
\begin{align}\label{diagcorgeoesssmall}
\begin{gathered}
\xymatrix{
\mathcal{E}_{\kappa}\ar@{^(->}[rr]^{\iota}\ar[dr]^y\ar@/_/[dddr]_y & & \mathcal{E}\ar[dr]^y\ar@/_/[dddr]_y & & \\
 & \hat{\mathcal{E}_{\kappa}}\ar@{^(->}[rr]^{\iota_{\ast}}\ar@{^(->}[dr] & & \hat{\mathcal{E}}\ar@{^(->}[dr] & \\
 & & \mathrm{Fun}(\mathcal{E}_{\kappa}^{op},\mathcal{S}^+)\ar@{^(->}[rr]^(.6){\iota_{\ast}}\ar@{-->}@/_1pc/[dd]_(.6){L_{\kappa}^+}\ar@{}[dd]|(.6){\adj} & & \mathrm{Fun}(\mathcal{E}^{op},\mathcal{S}^+)\ar@{-->}@/_1pc/[dd]_{L^+}\ar@{}[dd]|{\adj} \\
 & \mathrm{Sh}_{\mathrm{Geo}}(\mathcal{E}_{\kappa})\ar@{^(->}[uu]_(.4){j_{\kappa}}\ar@{^(->}[rr]^{\iota_{\ast}}\ar@{^(->}[dr] & & \mathrm{Sh}_{\mathrm{Geo}}(\mathcal{E})\ar@{^(->}[uu]| \hole_(.4){j}\ar@{^(->}[dr] & \\
 & & \mathrm{Sh}_{\mathrm{Geo}}(\mathcal{E}_{\kappa})^+\ar[rr]^{\iota_{\ast}}_{\simeq}\ar@{^(->}@/_1pc/[uu]_(.4){j_{\kappa}^+}  & & \mathrm{Sh}_{\mathrm{Geo}}(\mathcal{E})^+\ar@{^(->}@/_1pc/[uu]_{j^+} \\
 }
\end{gathered}
\end{align}
Here, the bottom right corner $\mathrm{Sh}_{\mathrm{Geo}}(\mathcal{E})^+:=\mathrm{Sh}_{\nu\text{-Geo}}(\mathcal{E})$ denotes the
$\infty$-category of higher $\nu$-geometric sheaves on 
the $\nu$-geometric $\infty$-category $\mathcal{E}$ (which itself is $\nu^+$-small and hence contained in $\mathcal{S}^+$). We obtain a left 
exact left adjoint
$L^+\colon\mathrm{Fun}(\mathcal{E}^{op},\mathcal{S}^+)\rightarrow\mathrm{Sh}_{\mathrm{Geo}}(\mathcal{E})^+$ by \cite[Theorem 5.11]{rs_hgst} 
(applied in $\mathcal{S}^+$). The inclusion $\iota\colon\mathcal{E}_{\kappa}\hookrightarrow\mathcal{E}$ is fully faithful 
and generates $\mathcal{E}$ under canonical colimits (\cite[Lemma 5.3.5.8, Proposition 5.4.7.4]{luriehtt}). That means
for every object $E\in\mathcal{E}$, the natural gap map
\[\mathrm{colim}(\mathcal{E}_{\kappa}\downarrow E\twoheadrightarrow\mathcal{C}\hookrightarrow\mathcal{E})\rightarrow E\]
is an equivalence in $\mathcal{E}$. As $\mathcal{E}_{\kappa}$ is left exact, the comma $\infty$-category $\mathcal{E}_{\kappa}\downarrow E$ 
has pullbacks which furthermore are preserved by the composition $\mathcal{E}_{\kappa}\downarrow E\rightarrow\mathcal{E}$.
It follows that $\mathcal{E}_{\kappa}\downarrow E\rightarrow \mathcal{E}$ is a higher covering diagram
with colimit $E$ (\cite[Lemma 5.17]{rs_hgst}). That means that the inclusion
$\iota\colon\mathcal{E}_{\kappa}\hookrightarrow\mathcal{E}$ is $\mathrm{Geo}$-dense. We hence may apply Theorem~\ref{thmcomplemma}.2 to the 
inclusion $\iota\colon\mathcal{E}_{\kappa}\hookrightarrow\mathcal{E}$ in $\mathcal{S}^+$. This yields the left exact reflective localization
\[L_{\kappa}^+\colon\mathrm{Fun}(\mathcal{E}_{\kappa}^{op},\mathcal{S}^+)\rightarrow\mathrm{Sh}_{\mathrm{Geo}}(\mathcal{E}_{\kappa})^+\]
together with the equivalence on the bottom of Diagram~(\ref{diagcorgeoesssmall}) which makes the front square (of solid arrows) commute.

The $\infty$-category $\mathrm{Sh}_{\mathrm{Geo}}(\mathcal{E}_{\kappa})$ shall be defined so to make the bottom square (or equivalently the 
right hand side square) of Diagram~(\ref{diagcorgeoesssmall}) cartesian by construction. It follows that all four side faces of the front 
cube in Diagram~(\ref{diagcorgeoesssmall}) are pullback squares (in fact all six faces are).
As replete inclusions between $\infty$-categories are isofibrations, it follows that the pullback
\[\iota_{\ast}\colon\mathrm{Sh}_{\mathrm{Geo}}(\mathcal{E}_{\kappa})\hookrightarrow\mathrm{Sh}_{\mathrm{Geo}}(\mathcal{E})\]
of the equivalence $\iota_{\ast}$ on the bottom is an equivalence as well. Now, the inclusion
\[\iota_{\ast}\colon\hat{\mathcal{E}}_{\kappa}\rightarrow\hat{\mathcal{E}}\]
on the top of the cube is right adjoint to the restriction functor $\iota^{\ast}$ and hence exhibits $\hat{\mathcal{E}}_{\kappa}$ as a left 
exact reflective localization of $\hat{\mathcal{E}}$. To finish the proof it thus suffices to show that the inclusion
$\mathrm{Sh}_{\mathrm{Geo}}(\mathcal{E}_{\kappa})\hookrightarrow\hat{\mathcal{E}_{\kappa}}$ has a left exact left adjoint. In that case, the 
composition $\hat{\mathcal{E}}\rightarrow\hat{\mathcal{E}_{\kappa}}\rightarrow\mathrm{Sh}_{\mathrm{Geo}}(\mathcal{E}_{\kappa})$ exhibits
$\mathrm{Sh}_{\mathrm{Geo}}(\mathcal{E})\simeq\mathrm{Sh}_{\mathrm{Geo}}(\mathcal{E}_{\kappa})$ as a left exact reflective localization of
$\hat{\mathcal{E}}$, with a presentation over the small base $\mathcal{E}_{\kappa}$ given by
$\hat{\mathcal{E}_{\kappa}}\rightarrow\mathrm{Sh}_{\mathrm{Geo}}(\mathcal{E}_{\kappa})$ itself. As we have equivalences
$\mathcal{E}\xrightarrow{\sim}\mathrm{Sh}_{\mathrm{Geo}}(\mathcal{E})\xleftarrow{\sim}\mathrm{Sh}_{\mathrm{Geo}}(\mathcal{E}_{\kappa})$, and
$\mathcal{E}$ is locally presentable by assumption, so is $\mathrm{Sh}_{\mathrm{Geo}}(\mathcal{E}_{\kappa})$. It follows that the localization
$\hat{\mathcal{E}_{\kappa}}\rightarrow\mathrm{Sh}_{\mathrm{Geo}}(\mathcal{E}_{\kappa})$ is accessible, which means that the site $(\mathcal{E}_{\kappa},\mathrm{Geo})$ is small.

To construct the left adjoint of $\mathrm{Sh}_{\mathrm{Geo}}(\mathcal{E}_{\kappa})\hookrightarrow\hat{\mathcal{E}_{\kappa}}$, we use that 
the Yoneda embedding $\mathcal{E}\rightarrow\hat{\mathcal{E}}$ factors through $\mathrm{Sh}_{\mathrm{Geo}}(\mathcal{E})$ (as representables 
take all colimits in $\mathcal{E}$ to limits of spaces). Indeed, by 
construction, this implies that the Yoneda embedding $\mathcal{E}_{\kappa}\rightarrow\hat{\mathcal{E}_{\kappa}}$ factors through
$\mathrm{Sh}_{\mathrm{Geo}}(\mathcal{E}_{\kappa})$ as well. As we just argued that $\mathrm{Sh}_{\mathrm{Geo}}(\mathcal{E}_{\kappa})$ is 
locally presentable, it is cocomplete; we hence can construct a left Kan extension
\[L_{\kappa}:=\mathrm{Lan}_{y}(y)\colon\hat{\mathcal{E}_{\kappa}}\rightarrow\mathrm{Sh}_{\mathrm{Geo}}(\mathcal{E}_{\kappa})\]
which is automatically left adjoint to the inclusion $j_{\kappa}$. We are left to show that $L_{\kappa}$ is left exact. Therefore simply note 
that the diagram
\[\xymatrix{
\mathcal{E}_{\kappa}\ar[r]^y\ar[dr]_y\ar@/^2pc/[rr]^y & \hat{\mathcal{E}_{\kappa}}\ar@{^(->}[r]\ar[d]^{L_{\kappa}} & \mathrm{Fun}(\mathcal{E}_{\kappa}^{op},\mathcal{S}^+)\ar[d]^{L_{\kappa}^+} \\
 & \mathrm{Sh}_{\mathrm{Geo}}(\mathcal{E}_{\kappa})\ar@{^(->}[r] & \mathrm{Sh}_{\mathrm{Geo}}(\mathcal{E}_{\kappa})^+
}\]
commutes, given that both $L_{\kappa}$ and $L_{\kappa}^+$ are the left Kan extension of (the respective factorization of) $y$ along $y$, and 
that both vertical inclusions preserve colimits (as noted in \cite[Remark 6.3.5.17]{luriehtt}). Both vertical 
inclusions furthermore preserve and reflect small limits, since limits are simply computed pointwise in all four cases. Thus, as the left 
adjoint $L_{\kappa}^+$ is left exact, it follows that so is $L_{\kappa}$. 
\end{proof}

%
%

We furthermore recall that for every uncountable cardinal $\kappa$ and every $\kappa$-geometric $\infty$-category $\mathcal{E}$ the localization
$\hat{\mathcal{E}}\rightarrow\mathrm{Sh}_{\mathrm{Geo}_{\kappa}}(\mathcal{E})$ 
admits a factorization
\[\hat{\mathcal{E}}\rightarrow\mathrm{Sh}_{(\mathrm{Geo}_{\kappa})_{-1}}(\mathcal{E})\rightarrow\mathrm{Sh}_{\mathrm{Geo}_{\kappa}}(\mathcal{E}),\]
where the first localization is topological and the second one is cotopological (\cite[Proposition 5.13]{rs_hgst}).
Here, the site $(\mathcal{E},\mathrm{Geo}_{-1})$ is the ``ordinary $\kappa$-geometric Grothendieck site'' associated to $\mathcal{E}$. 
Whenever $\mathcal{E}$ is an $\infty$-topos, we obtain an according factorization
\[\hat{\mathcal{E}}\rightarrow\mathrm{Sh}_{\mathrm{Geo}_{-1}}(\mathcal{E})\rightarrow\mathrm{Sh}_{\mathrm{Geo}}(\mathcal{E})\]
of the localization constructed in Corollary~\ref{corgeoesssmall}.
The two sites $(\mathcal{E},\mathrm{Geo}_{-1})$ and $(\mathcal{E},\mathrm{Geo})$ are generally not the same 
(\cite[Proposition 5.22]{rs_hgst}). Despite this distinction, we obtain the following two straight-forward translations of 
ordinary topos theoretic results.

\begin{proposition}\label{propmorphsitescan}
Let $(\mathcal{C},J)$ be a small cartesian $(\infty,1)$-site and $\mathcal{E}$ be an $\infty$-topos. Then the sheafified Yoneda embedding
$\rho y\colon\mathcal{C}\rightarrow\mathrm{Sh}_J(\mathcal{C})$ induces an equivalence
\[\mathrm{CSite}((\mathcal{C},J),(\mathcal{E},\mathrm{Geo}))\simeq\mathrm{RTop}(\mathcal{E},\mathrm{Sh}_J(\mathcal{C}))\]
\end{proposition}
\begin{proof}
Left Kan extension induces a map
\[\mathrm{Fun}^{\mathrm{lex}}(\mathcal{C},\mathcal{E})\rightarrow\mathrm{LGeo}(\hat{\mathcal{C}},\hat{\mathcal{E}}),\]
where the right hand side denotes the space of left exact left adjoints from $\hat{\mathcal{C}}$ to $\hat{\mathcal{E}}$.
This Kan extension restricts to a map
\[\mathrm{CSite}((\mathcal{C},J),(\mathcal{E},\mathrm{Geo}))\rightarrow\mathrm{LTop}(\mathrm{Sh}_J(\mathcal{C}),\mathrm{Sh}_{\mathrm{Geo}}(\mathcal{E}))\]
via \cite[Proposition 5.5.4.20]{luriehtt} with the according right adjoints given by (\ref{equcovpresfctr}). Also note that, being equivalent to $\mathcal{E}$ itself, the $\infty$-category
$\mathrm{Sh}_{\mathrm{Geo}}(\mathcal{E})$ is presentable and hence an $\infty$-topos. This restriction has an inverse given by
\[(\rho y)^{\ast}\colon\mathrm{LTop}(\mathrm{Sh}_J(\mathcal{C}),\mathrm{Sh}_{\mathrm{Geo}}(\mathcal{E}))\rightarrow\mathrm{CSite}((\mathcal{C},J),(\mathcal{E},\mathrm{Geo}))\]
as $y\colon\mathcal{E}\rightarrow\mathrm{Sh}_{\mathrm{Geo}}(\mathcal{E})$ is an equivalence.
\end{proof}

\begin{remark}
Informally, Proposition~\ref{propmorphsitescan} states that the pair
\[\xymatrix{
\mathrm{LTop}\ar@/^1pc/@{-->}[r]^{((\cdot),\mathrm{Geo})}|{\times}\ar@{}[r]|{\top} & \mathrm{CSite}\ar@/^1pc/[l]^{\mathrm{Sh}}
}\]
would exhibit the $\infty$-category $\mathrm{LTop}$ of $\infty$-toposes as a reflective localization of the $\infty$-category
$\mathrm{CSite}$ of cartesian sites if the cartesian site $(\mathcal{E},\mathrm{Geo})$ was in fact small. This of course is not the case, 
however we will see in Corollary~\Ref{corgeoesssmall} that the site $(\mathcal{E},\mathrm{Geo})$ is \emph{essentially small} which still 
implies that the functor 
\[\mathrm{CSite}\xrightarrow{\mathrm{Sh}}\mathrm{LTop}\]
is essentially surjective.
\end{remark}

\begin{proposition}\label{propyonedacoverpres}
For every Grothendieck site $(\mathcal{C},J)$, the sheafified Yoneda embedding
$\rho y\colon (\mathcal{C},J)\rightarrow(\mathrm{Sh}_J(\mathcal{C}),\mathrm{Geo}_{-1})$ is cover-preserving. It hence is a morphism of cartesian sites if and only if the $\infty$-category $\mathcal{C}$ is cartesian. 
\end{proposition}
\begin{proof}
Given a $J$-cover $S\hookrightarrow yC$ in $\mathcal{C}$, the sieve generated by $(\rho y)_!(S\hookrightarrow yC)$ in
$\mathrm{Sh}_J(\mathcal{C})$ is the sieve generated by 
the family $\{\rho y(s)\colon \rho y(\mathrm{dom}(s))\rightarrow \rho y(C)|s\in S\}$. This however is an ordinary geometric covering family, 
as the arrow
\[\coprod_{s\in S}\rho y(s)\colon\coprod_{s\in S}\rho y(\mathrm{dom}(s))\rightarrow \rho y(C)\]
is an effective epimorphism in $\mathrm{Sh}_J(\mathcal{C})$ (given that its image is the equivalence $\rho (S)\hookrightarrow \rho y(C)$).
\end{proof}

While Proposition~\ref{propyonedacoverpres} is fairly innocuous, it is noteworthy that the statement is stronger than the direct 
transcription of the according ordinary topos theoretic result which considers $\mathrm{Sh}_J(\mathcal{C})$ equipped with its 
\emph{canonical} (and hence in this case higher geometric) site; the canonical counterpart to Proposition~\ref{propyonedacoverpres} indeed 
follows directly from Proposition~\ref{propmorphsitescan} applied to the identity on $\mathcal{E}=\mathrm{Sh}_J(\mathcal{C})$. With this	 
discrepancy for Grothendieck sites in mind, we end this paper with a counter-example to the following ``topological'' Comparison 
Lemma.

\begin{proposition}[Failure of the topological Comparison Lemma]\label{propcomplemmafail}
There is a left exact and fully faithful functor $F\colon\mathcal{C}\hookrightarrow\mathcal{D}$ from a small left exact $\infty$-category
$\mathcal{C}$ into an $\infty$-topos $\mathcal{D}$ equipped with the ordinary geometric Grothendieck topology $\mathrm{Geo}_{-1}$ such that
\begin{enumerate}
\item $F$ is topologically $\mathrm{Geo}_{-1}$-dense, but
\item there is no sheaf theory $J$ on $\mathcal{C}$ such that $F\colon(\mathcal{C},J)\rightarrow(\mathcal{D},\mathrm{Geo}_{-1})$ is 
cover-preserving and such that the induced functor
\[F^{\ast}\colon\mathrm{Sh}_{\mathrm{Geo}_{-1}}(\mathcal{D})\rightarrow\mathrm{Sh}_J(\mathcal{C})\]
is an equivalence of sheaf theories. In particular, there is no such Grothendieck topology.
\end{enumerate}
\end{proposition}

\begin{proof}
Let $\mathcal{C}$ be the locale $2^{S^{\infty}}$ from \cite[Section 11.3]{rezkhtytps} and $J$ be its associated canonical geometric 
Grothendieck topology. 
The sheafified Yoneda embedding $F:=\rho y\colon\mathcal{C}\rightarrow\mathrm{Sh}_J(\mathcal{C})$ is left exact, fully faithful and generates 
the $\infty$-topos $\mathcal{D}:=\mathrm{Sh}_J(\mathcal{C})$ under canonical colimits (and hence is dense).
The $\infty$-topos $\mathcal{D}$ of $J$-sheaves itself (or the $\kappa$-logos of $\kappa$-compact objects in $\mathcal{D}$ for any regular 
cardinal $\kappa$ large enough) may be equipped with the Grothendieck topology $\mathrm{Geo}_{-1}$ on 
the one hand, and with the geometric kernel $\mathrm{Geo}$ generated by the ($\kappa$-small) higher covering diagrams in $\mathcal{D}$ on the 
other hand. Basically by definition, as $\mathcal{C}$ has pullbacks, the functor $F$ is $\mathrm{Geo}$-dense (\cite[Lemma 5.17]{rs_hgst}). It 
hence is in particular topologically $\mathrm{Geo}_{-1}$-dense.

We show that there can be no geometric kernel $K$ on $\mathcal{C}$ such that the functor $F\colon\mathcal{C}\rightarrow\mathcal{D}$ induces 
an equivalence
\[(\rho y)^{\ast}\colon\mathrm{Sh}_{\mathrm{Geo}_{-1}}({\mathcal{D}})\rightarrow\mathrm{Sh}_K(\mathcal{C})\]
of sheaf theories.
As $F$ is assumed to be fully faithful, so is the left Kan extension $F_!\colon\hat{\mathcal{C}}\rightarrow\hat{\mathcal{D}}$.
We hence obtain a diagram of embeddings as follows.
\begin{align}\label{diagcomplemma1}
\begin{gathered}
\xymatrix{
\mathcal{C}\ar[r]^{F}\ar[d]^y & \mathcal{D}\ar[d]^{y} \\
\hat{\mathcal{C}}\ar[r]_{F_!}\ar[d] & \hat{\mathcal{D}}\ar[d]^{\rho}\\
\mathrm{Sh}_J(\mathcal{C})\ar[r]_(.45){\rho F_!}^(.45){\simeq} & \mathrm{Sh}_{\mathrm{Geo}}(\mathcal{D})
}
\end{gathered}
\end{align}
Indeed, the functor
$F\colon(\mathcal{C},J)\rightarrow(\mathcal{D},\mathrm{Geo}_{-1})$ is cover-preserving by Proposition~\ref{propyonedacoverpres}, and in 
particular so is $F\colon(\mathcal{C},J)\rightarrow(\mathcal{D},\mathrm{Geo})$. It follows that the composition $\rho F_!$ restricts to a 
functor of sheaf theories as on the bottom of (\ref{diagcomplemma1}) by Proposition~\ref{propmorphsitescan}. Both horizontal functors $F_!$ 
and $\rho F_!$ are the left adjoint part of a geometric morphism $(F_!,F^{\ast})$. Furthermore, the geometric morphism on the bottom is an 
equivalence by a 2-out-of-3 argument: the composition
\[\mathcal{D}\xrightarrow{y}\mathrm{Sh}_{\mathrm{Geo}}(\mathcal{D})\xrightarrow{F^{\ast}}\mathrm{Sh}_J(\mathcal{C})\] 
takes an object $D\in\mathcal{D}$ to the $J$-sheaf $\mathcal{D}(F(\cdot),D)$ and is equivalent to the identity on $\mathcal{D}$. The 
(sheafified) Yoneda embedding $y\colon\mathcal{D}\rightarrow\mathrm{Sh}_{\mathrm{Geo}}(\mathcal{D})$ is an equivalence itself as well
(\cite[Theorem 5.19]{rs_hgst}), and hence so are $F^{\ast}\colon\mathrm{Sh}_{\mathrm{Geo}}(\mathcal{D})\rightarrow\mathrm{Sh}_J(\mathcal{C})$ 
and its left adjoint $\rho F_!$. 

Now, assume that there is a geometric kernel $K$ on $\mathcal{C}$ such that $F^{\ast}\colon\hat{\mathcal{D}}\rightarrow\hat{\mathcal{C}}$ 
restricts to an equivalence $F^{\ast}\colon\mathrm{Sh}_{\mathrm{Geo}_{-1}}(\mathcal{D})\rightarrow\mathrm{Sh}_K(\mathcal{C})$; then
$F\colon(\mathcal{C},K)\rightarrow(\mathcal{D},\mathrm{Geo}_{-1})$ is cover-preserving as well and we obtain a factorization of Diagram (\ref{diagcomplemma1}) as follows.
\begin{align}\label{diagcomplemma2}
\begin{gathered}
\xymatrix{
\mathcal{C}\ar[r]^{F}\ar[d]^y & \mathcal{D}\ar[d]^{y_{}} \\
\hat{\mathcal{C}}\ar[r]_{F_!}\ar[d] & \hat{\mathcal{D}}\ar[d]^{\rho_{-1}}\ar@{-->}@/_1pc/[l]_{F^{\ast}}\\
\mathrm{Sh}_K(\mathcal{C})\ar[r]_(.4){\rho_{-1} F_!}^(.4){\simeq}\ar[d]  & \mathrm{Sh}_{\mathrm{Geo}_{-1}}(\mathcal{D})\ar[d]^{\rho_{\infty}} \ar@{-->}@/_1pc/[l]_{F^{\ast}}\\
\mathrm{Sh}_J(\mathcal{C})\ar[r]_{\rho F_!}^{\simeq} & \mathrm{Sh}_{\mathrm{Geo}}(\mathcal{D})\ar@{-->}@/_1pc/[l]_{F^{\ast}}
}
\end{gathered}
\end{align}
Indeed, every element $f\in K$ is an equivalence in $\mathrm{Sh}_J(\mathcal{C})$, because $F_!(f)$ is mapped to an equivalence in
$\mathrm{Sh}_{\mathrm{Geo}_{-1}}(\mathcal{E})$, and hence so it is in $\mathrm{Sh}_{\mathrm{Geo}}(\mathcal{E})$. By commutativity of the 
diagram, it follows that $f$ is mapped to an equivalence in $\mathrm{Sh}_J(\mathcal{C})$. We now make the following observation: by 
\cite[Proposition 6.2.4.2]{luriehtt}, the canonical Grothendieck topology relative to $F$ on $\mathcal{C}$ exists, and will be denoted by
$J_F$. Since $F_!$ is left exact, the topology $J_F$ is defined in such a way that a sieve $S\hookrightarrow yC$ in $\mathcal{C}$ is
$J_F$-covering if and only if $F_!(S)\hookrightarrow yFC$ is an ordinary geometric cover in $\mathcal{E}$. But, whenever a presheaf 
$X\in\hat{\mathcal{C}}$ is a $K$-sheaf, then $\rho_{-1} F_!(X)\in\hat{\mathcal{E}}$ is an ordinary geometric sheaf and hence
$F^{\ast}\rho_{-1} F_!(X)\simeq X$ is a $J_F$-sheaf (\cite[Proposition 6.2.4.6]{luriehtt}). The same argument applies in the other direction. 
Thus, the sheaf theories
$\mathrm{Sh}_{J_F}(\mathcal{C})$ and $\mathrm{Sh}_K(\mathcal{C})$ coincide. In particular, the localization
\[\hat{\mathcal{C}}\rightarrow\mathrm{Sh}_K(\mathcal{C})\]
is topological. Furthermore, as the localization
$\rho_{\infty}\colon\mathrm{Sh}_{\mathrm{Geo}_{-1}}(\mathcal{E})\rightarrow\mathrm{Sh}_{\mathrm{Geo}}(\mathcal{E})$ is cotopological
(\cite[Proposition 5.13]{rs_hgst}), so is the equivalent localization
\[\mathrm{Sh}_K(\mathcal{C})\rightarrow\mathrm{Sh}_J(\mathcal{C}).\]
But the localization $\hat{\mathcal{C}}\rightarrow\mathrm{Sh}_J(\mathcal{C})$ is topological already, and so it follows that in fact
$\mathrm{Sh}_K(\mathcal{C})\rightarrow\mathrm{Sh}_J(\mathcal{C})$ is the identity by uniqueness of these kinds of factorizations 
(\cite[Remark 6.5.2.20]{luriehtt}). In turn it follows that the localization
\[\mathrm{Sh}_{\mathrm{Geo}_{-1}}(\mathcal{E})\rightarrow\mathrm{Sh}_{\mathrm{Geo}}(\mathcal{E})\]
is the identity. This however contradicts \cite[Proposition 5.22]{rs_hgst}, where it is shown that this localization is non-trivial.
\end{proof}

\begin{remark}
The specific counter-example in Proposition~\ref{propcomplemmafail} shows that the canonical Grothendieck topology construction 
from \cite[Section 6.2.4]{luriehtt} generally does not induce an equivalence of sheaf theories (but an embedding at best). Furthermore, the 
proposition also shows that Corollary~\ref{corgeoesssmall} (or at least its proof) does not apply to the ordinary geometric $(\infty,1)$-site 
on an $\infty$-topos $\mathcal{E}$, which makes it questionable whether $(\mathcal{E},\mathrm{Geo}_{-1})$ has a small presentation at all.
\end{remark}

\bibliographystyle{amsplain}
\bibliography{../BSBib}

\newcommand{\noopsort}[1]{}
\providecommand{\bysame}{\leavevmode\hbox to3em{\hrulefill}\thinspace}
\providecommand{\MR}{\relax\ifhmode\unskip\space\fi MR }
\providecommand{\MRhref}[2]{%
  \href{http://www.ams.org/mathscinet-getitem?mr=#1}{#2}
}
\providecommand{\href}[2]{#2}
\begin{thebibliography}{10}

\bibitem{adamekrosicky}
J.~Ad\'{a}mek and J.~Rosick\'{y}, \emph{Locally presentable and accessible
  categories}, London Mathematical Society Lecture Note Series, vol. 189,
  Cambridge University Press, 1994.

\bibitem{aneljoyaltopos}
M.~Anel and A.~Joyal, \emph{Topo-logie}, New Spaces in Mathematics: Formal and
  Conceptual Reflections (G.~Catren M.~Anel, ed.), Cambridge University Press,
  2021, pp.~155--257.

\bibitem{abjfsheavesII}
M.~Anel, A.~Joyal, E.~Finster, and G.~Biedermann, \emph{Left-exact
  localizations of $\infty$-topoi {II}: {G}rothendieck topologies},
  \url{https://arxiv.org/abs/2201.01236}.

\bibitem{abjfsheavesI}
\bysame, \emph{Left-exact localizations of $\infty$-topoi {I}: {H}igher
  sheaves}, Advances in Mathematics \textbf{400} (2022).

\bibitem{as_soa}
M.~Anel and C.~Leena Subramaniam, \emph{Small object arguments,
  plus-construction and left-exact localizations},
  \url{https://arxiv.org/abs/2004.00731}, Last update 12 April 2020.

\bibitem{barwickthesis}
C.~Barwick, \emph{$(\infty,n)$-cat as a closed model category}, Ph.D. thesis,
  University of Pennsylvania, Philadelphia, PA 19104, United States, 2005.

\bibitem{barwicketalparahct}
C.~Barwick, E.~Dotto, S.~Glasman, D.~N., and J.~Shah, \emph{Parametrized higher
  category theory and higher algebra: {E}xpos\'{e} {I} - {E}lements of
  parametrized higher category theory}, \url{https://arxiv.org/abs/1608.03657},
  2016.

\bibitem{harpazprasmamodfib}
Y.~Harpaz and M.~Prasma, \emph{The {G}rothendieck construction for model
  categories}, Advances in Mathematics \textbf{281} (2015), 1306--–1363.

\bibitem{binimkelly}
G.~Bin Im and G.~M. Kelly, \emph{On classes of morphisms closed under limits},
  Journal of the Korean Mathematical Society \textbf{23} (1986), no.~1, 1--18.

\bibitem{jacobscompcats}
B.~Jacobs, \emph{Comprehension categories and the semantics of type
  dependency}, Theoretical Computer Science \textbf{107} (1993), 169--207.

\bibitem{jacobsttbook}
\bysame, \emph{Categorical logic and type theory}, Studies in Logic and the
  Foundations of Mathematics, vol. 141, Elsevier Science B.V., 1999.

\bibitem{johnstone_ss}
P.T. Johnstone, \emph{Stone spaces}, Cambridge Studies in Advanced Mathematics,
  vol.~3, Cambridge University Press, 1982.

\bibitem{elephant}
\bysame, \emph{Sketches of an elephant: A topos theory compendium}, Oxford
  Logic Guides, vol.~43, Clarendon Press, 2003.

\bibitem{jtqcatvsss}
A.\ Joyal and M.\ Tierney, \emph{Quasi-categories vs {S}egal spaces},
  Categories in {A}lgebra, {G}eometry and {M}athematical {P}hysics, American
  Mathematical Society, 2006, pp.~277--326.

\bibitem{klvsimp}
C.~Kapulkin, P.L. Lumsdaine, and V.~Voevodsky, \emph{The simplicial model of
  {U}nivalent {F}oundations}, \url{http://arxiv.org/abs/1211.2851}, 2012,
  [Online, accessed 15 Apr 2014].

\bibitem{kellyfinlimits}
G.M. Kelly, \emph{Structures defined by finite limits in the enriched context,
  i}, Cahiers de Topologie et g\'{e}om\'{e}trie diff\'{e}rentielle
  cat\'{e}goriques \textbf{23} (1982), no.~1, 3--42.

\bibitem{mlmsheaves}
S.~Mac Lane and I.~Moerdijk, \emph{Sheaves in {G}eometry and {L}ogic: {A} first
  introduction to topos theory}, Universitext, Springer-Verlag New York, Inc.,
  1992.

\bibitem{lawverecomp}
W.~Lawvere, \emph{Equality in hyperdoctrines and comprehension schema as an
  adjoint functor}, Proceedings of the AMS Symposium on Pure Mathematics
  \textbf{XVII} (1970).

\bibitem{lurieha}
J.~Lurie, \emph{Higher algebra},
  \url{http://www.math.harvard.edu/~lurie/papers/HA.pdf}, Last update September
  2017.

\bibitem{luriehtt}
\bysame, \emph{Higher topos theory}, Annals of Mathematics Studies, no. 170,
  Princeton University Press, 2009.

\bibitem{lgood}
\bysame, \emph{$(\infty,2)$-categories and the {G}oodwillie calculus {I}},
  \url{https://arxiv.org/abs/0905.0462}, 2009, [Online, last revised 08 May
  2009].

\bibitem{martiniwolf_ihtt}
L.~Martini and S.~Wolf, \emph{Internal higher topos theory},
  \url{https://arxiv.org/abs/2303.06437}, v1, uploaded 11 Mar 2023.

\bibitem{nguyencontracov}
H.K. Nguyen, \emph{Covariant \& {C}ontravariant {H}omotopy {T}heories}, v1
  uploaded 19 Aug 2019.

\bibitem{hott}
The Univalent~Foundations Program, \emph{Homotopy {T}ype {T}heory: {U}nivalent
  {F}oundations of {M}athematics}, \url{http://homotopytypetheory.org/book},
  2013.

\bibitem{rasekh}
N.~Rasekh, \emph{Complete {S}egal objects},
  \url{https://arxiv.org/abs/1805.03561}, 2018, [Online, v1 accessed 09 May
  2018].

\bibitem{rasekheltops}
\bysame, \emph{A theory of elementary higher toposes},
  \url{https://arxiv.org/abs/1805.03805}, 2018, [Online, v2 accessed 05 Sep
  2018].

\bibitem{rezkhtytps}
C.~Rezk, \emph{Toposes and homotopy toposes (version 0.15)},
  \url{https://www.researchgate.net/publication/255654755_Toposes_and_homotopy_toposes_version_015},
  2010.

\bibitem{riehlverityadjmon}
E.\ Riehl and D.\ Verity, \emph{Homotopy coherent adjunctions and the formal
  theory of monads}, Advances in Mathematics \textbf{286} (2016), 802--888.

\bibitem{rvyoneda}
\bysame, \emph{Fibrations and {Y}oneda’s lemma in an $\infty$-cosmos}, J.
  Pure Appl. Algebra \textbf{221} (2017), no.~3, 499–564.

\bibitem{riehlverityelements}
\bysame, \emph{Elements of $\infty$-category theory}, Cambridge Studies in
  Advanced Mathematics, vol. 194, Cambridge University Press, 2022.

\bibitem{rss_hottmod}
E.~Rijke, M.~Shulman, and B.~Spitters, \emph{Modalities in homotopy type
  theory}, Logical Methods in Computer Science \textbf{16} (2020), no.~1,
  2:1--2:79.

\bibitem{rovelliweights}
M.~Rovelli, \emph{Weighted limits in an $(\infty,1)$-category},
  \url{https://arxiv.org/abs/1902.00805}.

\bibitem{shulmanexact}
M.~Shulman, \emph{Exact completions and small sheaves}, Theory and
  {A}pplications of {C}ategories \textbf{27} (2012), no.~7, 97--173.

\bibitem{shulmanuniverses}
\bysame, \emph{{A}ll $(\infty,1)$-toposes have strict univalent universes},
  \url{https://arxiv.org/abs/1904.07004}, 2019, [Online, last revised 26 Apr
  2019].

\bibitem{thesis}
R.~Stenzel, \emph{On univalence, {R}ezk completeness and presentable
  quasi-categories}, Ph.D. thesis, University of Leeds, Leeds LS2 9JT, 3 2019.

\bibitem{rs_comp}
\bysame, \emph{$(\infty,1)$-categorical comprehension schemes},
  \url{https://arxiv.org/abs/2010.09663}, 2020.

\bibitem{rs_hgst}
\bysame, \emph{Higher geometric sheaf theories},
  \url{https://arxiv.org/abs/2205.08646}, 2022.

\bibitem{rs_mcext}
\bysame, \emph{The externalization functor in $\infty$-category theory}, to
  appear, 2023.

\bibitem{rs_small}
\bysame, \emph{On notions of compactness, object classifiers, and weak {T}arski
  universes}, Mathematical Structures in Computer Science (2023).

\bibitem{rs_uc}
\bysame, \emph{Univalence and completeness of {S}egal objects}, Journal of Pure
  and Applied Algebra \textbf{227} (2023), no.~4.

\bibitem{vergura_loctop}
M.~Vergura, \emph{Localization theory in an $\infty$-topos},
  \url{https://arxiv.org/abs/1907.03836}, v1, uploaded 8 Jul 2019.

\end{thebibliography}
\Address
\end{document}